\documentclass[11pt,reqno]{amsart}

\usepackage{geometry}

\title{$S=T$ for Shimura Varieties and moduli spaces of $p$-adic shtukas}
\author{Zhiyou Wu}

\usepackage{amsmath,amssymb,fancyhdr,url,amsthm,mathtools,mathrsfs,eufrak,setspace}
\usepackage[pdftex]{graphicx}
\usepackage{tikz-cd}
\usetikzlibrary{cd}
\usepackage[mathscr]{euscript}
\usepackage{tikz}
\usepackage{dsfont}
\usetikzlibrary{matrix}

\tikzset{
labl1/.style={anchor=north, rotate=90, inner sep=1.2mm}
}

\makeatletter
\newcommand*{\rom}[1]{\expandafter\@slowromancap\romannumeral #1@}
\makeatother

\newcommand{\nc}{\newcommand}
\nc{\on}{\operatorname}

\usepackage{relsize} 
\usepackage[bbgreekl]{mathbbol} 
\usepackage{amsfonts} 
\DeclareSymbolFontAlphabet{\mathbb}{AMSb} 
\DeclareSymbolFontAlphabet{\mathbbl}{bbold} 
\newcommand{\prism}{{\mathlarger{\mathbbl{\Delta}}}}

\usepackage[toc,page]{appendix}



\newtheorem{proposition}{Proposition}[section]
\newtheorem{theorem}[proposition]{Theorem}
\newtheorem{example}[proposition]{Example}
\newtheorem{corollary}[proposition]{Corollary}
\newtheorem{definition}[proposition]{Definition}
\newtheorem{remark}[proposition]{Remark}
\newtheorem{lemma}[proposition]{Lemma}

\newtheorem{notation}[proposition]{Notation}

\begin{document}

\begin{abstract}
We prove the $S=T$ conjecture proposed by  Xiao--Zhu in \cite{2017arXiv170705700X}, making use of Scholze's theory of diamonds and v-stacks and Fargues--Scholze's geometric Satake equivalence. Following \cite{2018arXiv180205299X}, we deduce the Eichler--Shimura relation for Shimura varieties of Hodge type.  
\end{abstract}

\maketitle

\tableofcontents

\section{Introduction}

Xiao--Zhu have raised a question on $S=T$ for Shimura varieties in  \cite[Conjecture 7.3.13]{2017arXiv170705700X}, which states that excursion operators on the compactly supported cohomology are the same as Hecke operators, and is an analogue for Lafforgue's $S=T$ theorem in \cite[Proposition 6.2]{Lafforgue2018}. This is a fundamental question which can be interpreted as a local-global compatibility in a categorical formulation of Langlands correspondence, see \cite[Conjecture 4.7.16]{2020arXiv200802998Z}. Moreover, it has consequences on the Eichler--Shimura relation for compactly supported cohomology of Shimura varieties of Hodge type, see \cite[Proposition 6.3.1]{2018arXiv180205299X}.

The main result of this paper is a proof of this conjecture.
\begin{theorem}
Let $X_K$ be a Shimura variety of Hodge type with respect to a reductive group $G$ and level $K$. Let $p$ be a prime over which $G$ is unramified, and $K$ is hyperspecial at $p$. Let $l$ be a prime different from $p$, then we have an identity
\[
T_V= S_V 
\]
as endormorphisms of 
$H^i_c(\mathfrak{X}_{K,\overline{\mathbb{F}_p}},\overline{\mathbb{Q}}_\ell)$,
where $\mathfrak{X}_K$ is the smooth integral model of $X_K$ constructed in \cite{Kisin2010}, $V$ is a representation of $^L G$, $T_V$ is the Hecke operator associated to $V$ as defined in section \ref{rieourioeu}, and $S_V$ is the excursion operator associated to $V$ as defined in \cite[Section 7.3.11]{2017arXiv170705700X}. 
\end{theorem}

An application of such result, as observed by Xiao--Zhu, is the generalization of the classical Eichler--Shimura relation (\cite{GoroSHIMURA1958}, \cite{Eichler1954}) to Shimura varieties of Hodge type, proving a conjecture of Blasius--Rogawski in this setting.

\begin{corollary} (\cite[Proposition 6.3.1]{2018arXiv180205299X})
We have the Eichler--Shimura relation for Shimura varieties of Hodge type. 
\end{corollary}

Strictly speaking, there are two versions of $S=T$ in mixed characteristics (while there is only one in equal characteristic). The first is $S=T$ for moduli spaces of mixed characteristic shtukas, and the second is the one for Shimura varieties. In \cite{2017arXiv170705700X}, Xiao--Zhu have proved $S=T$ for moduli spaces of Witt vector shtukas ($loc.cit. $ theorem 6.0.1 (2)), and  zero dimensional Shimura varieties ($loc.cit.$ proposition 7.3.14). However, even their version for moduli spaces of shtukas is strictly weaker than Lafforgue's $S=T$ theorem in equal characteristics, and is again zero dimensional in nature. More precisely, they have proved that the excursion operator associated to a representation of the Langlands dual group acts on the cohomology of a zero dimensional moduli space of shtukas by the function associated to the representation through classical Satake. While Lafforgue's $S=T$ theorem states that the Hecke correspondence associated to a representation $V$ is the same as the excursion operator  associated to $V$ on any moduli space of shtukas. 

The reason Xiao--Zhu only prove their results for zero dimensional objects is the same for both Shimura varieties and moduli spaces of shtukas. Namely, they work with the Witt vector affine Grassmannian and moduli spaces of Witt vector shtukas, which are living over a characteristic $p$ base, so they can only talk about the special fiber of Shimura varieties and moduli spaces of shtukas. However, Hecke operators at $p$ are only naturally visible (as étale correspondences) in characteristic 0, or at least prime to $p$. The usual treatment is to take the closure of the Hecke correspondence in an integral model and reduce to characteristic $p$, this usually creates a lot of degeneration, as the classical example 
$T_p \equiv \on{Frob} + V $
for modular curves shows. This problem only goes away in zero dimension, where the Hecke operators "flatten" to characteristic $p$. Since Xiao--Zhu can only work with the special fiber, they can not even define Hecke operators for general moduli space of Witt vector shtukas. Indeed, they only state their conjecture and results for cohomology, which is nothing but passing to zero dimension, where they have the obvious notion of Hecke operators, namely the action of a function. 

There is a natural way to proceed. Namely the  theory of diamonds and v-stacks developed by Scholze in \cite{SW17} and \cite{2017arXiv170907343S} allows us to talk about moduli spaces of mixed characteristic shtukas over generic fibers,  so we can  define Hecke operators for moduli spaces of shtukas. Further, we have integral models of moduli spaces of shtukas using this framework, whose special fiber is essentially the moduli spaces of Witt vector shtukas considered by Xiao--Zhu, and the proof in this article is in some sense a nearby cycle construction relating cohomology in characteristic $0$ to cohomology in characteristic $p$. 

More precisely, the excursion operators on moduli spaces of shtukas are defined by the composition  of  creation, partial Frobenius and annihilation correspondences, each of which are essentially pullbacks of cohomological correspondences on Hecke stacks. The key to the construction of these correspondences on the Hecke stacks is the geometric Satake established by Fargues and Scholze in \cite[Chapter \rom{6}]{2021arXiv210213459F}, which can be viewed as a globalization of the Witt vector geometric Satake established by  Zhu in \cite{Zhu2017}. We can perform these constructions on the integral models of moduli spaces of shtukas, whose restriction to the special fiber are the excursion operators for moduli spaces of Witt vector shtukas constructed by Xiao--Zhu.  

On the other hand, we can prove that the excursion operators on the generic fiber are the same as the Hecke operators. We follow the same strategy as Vincent Lafforgue's proof in \cite[Proposition 6.2]{Lafforgue2018}, which is basically to move legs to reduce to the zero dimensional case that has already been established by Xiao--Zhu in \cite[Proposition 6.0.1 (2)]{2017arXiv170705700X}. 

Now we have the $S=T$ for moduli spaces of $p$-adic shtukas, and we want to deduce the corresponding $S=T$ for Shimura varieties. Recall that excursion operators for the special fiber of Shimura varieties are defined as the pullback of the excursion operators for moduli spaces of Witt vector shtukas through crystalline period maps from (special fibers of) Shimura varieties to  moduli spaces of Witt vector shtukas, see \cite[Section 7.2.3]{2017arXiv170705700X}. We generalize this construction to period maps from integral models of Shimura varieties to (integral) moduli spaces of $p$-adic shtukas. The special fiber of this map is essentially the old crystalline period map, while the generic fiber is closely related to the (quotient of) Hodge--Tate period map in \cite[Section 2]{Caraiani_2017}   or \cite[Chapter \rom{3}.3]{Scholze2015}. The construction of this map makes essential use of the classification of $p$-divisible groups over integral perfectoid rings in terms of Breuil--Kisin--Fargues modules in the appendix of   \cite[Lecture 17]{SW17}. Now the pullbacks of excursion operators on moduli spaces of $p$-adic shtukas along these period maps produce the excursion operators on Shimura varieties, and the $S=T$ for Shimura varieties follows formally from the $S=T$ for moduli spaces of shtukas.

The integral period map also appears in the recent work   of Georgios Pappas   and Michael Rapoport, see \cite[Section 4]{2021arXiv210608270P}. Indeed, they construct it more generally for Shimura varieties with parahoric reduction, while we work only with the good reduction case. Moreover, they construct it on the full adic space associated to the integral model of Shimura varieties, while we work only with the good reduction locus. 

Under certain restrictive hypothesis, the Eichler--Shimura relation for Shimura varieties of Hodge type has also recently been established by Si Ying Lee in \cite{2020arXiv200611745L}. We provide an extensive remark on the comparison between the two approaches to congruence relations, which we believe to help clarify the ideas underlying this paper. 

The strategy of proof in \cite{2020arXiv200611745L} is similar to that of Faltings--Chai (\cite[Chapter \rom{7}.4]{Faltings2010-sw}) , Torsten Wedhorn (\cite{+2000+43+71}) and many others. The key to all approaches to congruence relations is to detect the degeneration of $p$-isogenies between abelian varieties from characteristic $0$ to characteristic $p$ using linear algebra data coming from cohomology theory, and the difference lies in which cohomology theory to use. The approach in \cite[Chapter \rom{7}.4]{Faltings2010-sw}   is to use 
étale cohomology  to detect the effect of specialization of isogenies.  For the ordinary reduction, this is easy since the specialization just adds a canonical filtration on the Tate module, and the effect on isogenies is simply to preserve this filtration.  Hence their proof makes crucial use of ordinary abelian varieties, which restricts the applicability of the method, namely one has to  add a condition on the density of  the ordinary isogeny locus. 

To get rid of this condition, it is necessary to look at the specialization of general abelian varieties and their isogenies, and detect them using linear algebra data. This is exactly what $p$-adic Hodge theory could help us with. In the ordinary case, $p$-adic Hodge theory is hidden since the comparison already takes places at the rational level, namely the canonical filtration. In general, the period rings manifest themselves. However, it is hard to find a direct linear algebra description of the specialization as in the ordinary case. Indeed,  it is well-known that  the relation between the Hodge--Tate period of a $p$-divisible group over $\mathcal{O}_{\mathbb{C}_p}$ and the Dieudonné module of its special fiber is mysterious. The way we proceed is exactly the way Scholze describes the mysterious relation in \cite[Lecture 12]{SW17} 
using intermediate integral objects, namely $p$-adic shtukas. 

The lack of explicit linear algebra models  in the previous approach is remedied by the high power functorial geometric models, which puts us in the framework of geometric Langlands. In particular, the degeneration is described by the comparison theorems in $p$-adic Hodge theory, in the disguise of the Fargues--Fontaine curve (or closely related objects), and the lack of explicit identifications is remedied by the functoriality of Fargues--Fontaine curve, which significantly mechanizes the situation by putting us in the framework of diamonds and v-stacks, so we can mimic what has been done in geometric Langlands. More precisely, the degeneration is reflected by the behaviour of excursion operators when the legs of shtukas collide, somehow the excursion operators "flatten" the situation by allowing extra freedom on adding legs and intermediate modifications, where the magic ultimately boils down to the fusion product in geometric Satake.

We learned from Scholze after initial progress on this project that Zhu also has a proof of the $S=T$ conjecture for some time, which in principle should be similar  to the proof presented here. 

We now briefly describe the content of this paper. In section 2, we provide definitions for all the spaces that we will use. In section 3,
we recall the geometric Satake proved by Fargues--Scholze, and interpret  morphisms between representations of the dual group as cohomological correspondences. In section 4, we introduce the Hecke operators on moduli spaces of shtukas. In section 5, 
we construct the excursion operators for moduli spaces of shtukas. In section 6, we prove the $S=T$ for moduli spaces of shtukas. In the last section, we construct the integral period maps from Shimura varieties to moduli spaces of shtukas, and  prove the $S=T$ for Shimura varieties. We collect some results on the comparison between cohomology of algebraic stacks and v-stacks, which is used  to compare the constructions given here with the constructions in \cite{2017arXiv170705700X}.

\subsection*{Convention}
We fix a prime $p$ throughout. We will work with the theory of v-stacks and diamonds as developed in \cite{2017arXiv170907343S}, by definition they are stacks over the category of characteristic $p$ perfectoid spaces with the v-topology.

We follow the notation in \cite[Section 15]{2017arXiv170907343S}  to denote $X^{\diamond}$ the diamond associated to  an adic space $X$ over $\mathbb{Z}_p$, which parametrizes untilts $S^{\sharp}$ of $S$ together with a map $S^{\sharp} \rightarrow X$. When $X=\on{Spa}(A,A^+)$, we denote as usual
$X^{\diamond} = \on{Spd}(A,A^+)$. When $A^+$ is clear from the situation, we abbreviate the notation by writing $\on{Spd}(A) =\on{Spd}(A,A^+) $. For example, when $K$ is a non-archimedean field, we write $\on{Spd}(K) =\on{Spd}(K,\mathcal{O}_K) $.

 Let $X$ be an algebraic stack or scheme locally of finite type over a discrete valuation ring $\mathcal{O}$ with perfect residue field, as in appendix \ref{serepoir}, we write $X^{\diamond}$ (resp. $X^{\diamond\diamond}$) for the stackification of the prestack sending a characteristic $p$ perfectoid space $\on{Spa}(R,R^+)$ to the groupoid of untilts $\on{Spa}(R^{\sharp},R^{\sharp,+})$ over $\mathcal{O}$ together with a map $\on{Spec}(R^{\sharp,+}) \rightarrow X$ (resp. $\on{Spec}(R^{\sharp}) \rightarrow X$) of algebraic stacks over $\mathcal{O}$. When $X$ is a scheme,
this deviates from the notation in \cite[Section 27]{2017arXiv170907343S}, where they denote $X^{\diamond}$ for our $X^{\diamond\diamond}$. Our notations of $X^{\diamond}$ and $X^{\diamond\diamond}$ are different from other authors such as \cite{2022arXiv220101234A}, where they use $X^{\diamond}$ and $X^{\Diamond} $.

For $E$ an extension of $\mathbb{Q}_p$, $\on{Spd}(E)$ is the diamond parametrizing characteristic 0 untilts with a morphism to $\on{Spa}(E,\mathcal{O}_E)$, and $\on{Spd}(\mathcal{O}_E)$ is the v-sheaf parametrizing all untilts  with a morphism to $\on{Spa}(\mathcal{O}_E,\mathcal{O}_E)$. 

Following the notation of \cite[Lecture 13]{SW17},  we denote 
\[
\mathcal{Y}_{[0,\infty)}(S) := \on{Spa}(W(R^+),W(R^+)) \setminus V([\omega]),
\]
where $S=\on{Spa}(R,R^+)$ is a characteristic $p$ affinoid perfectoid space with a topological nilpotent unit $\omega$. For general $S$, we simply glue along affinoid ones. 
We have a natural identification
\[
\mathcal{Y}_{[0,\infty)}(S)^{\diamond} = S \times \on{Spd}(\mathbb{Z}_p). 
\]
Then an untilt of $S$ is equivalent to a section of 
$\mathcal{Y}_{[0,\infty)}(S)^{\diamond} \rightarrow S$,
giving rise to an effective Cartier divisor on $\mathcal{Y}_{[0,\infty)}(S)$. 

We denote 
\[
\mathcal{X}_{\on{FF},S} := \mathcal{Y}_{(0,\infty)}(S) / \on{Frob}_S^{\mathbb{Z}}
\]
the relative Fargues--Fontaine curve with respect to $S$, where 
\[
\mathcal{Y}_{(0,\infty)}(S) := \on{Spa}(W(R^+),W(R^+)) \setminus V(p[\omega])
\]
for 
$S=\on{Spa}(R,R^+)$. We note that $\mathcal{Y}_{(0,\infty)}(S)$ is an open subspace of $\mathcal{Y}_{[0,\infty)}(S)$, the generic fiber of $\mathcal{Y}_{[0,\infty)}(S)$ over $\mathbb{Z}_p$. Moreover, the associated diamond
$\mathcal{Y}_{(0,\infty)}(S)^{\diamond}$
 can be identified with
$S^{\diamond} \times \on{Spd}(\mathbb{Q}_p)$.

Now let $S^{\sharp} =\on{Spa}(R^{\sharp},R^{\sharp,+})$ 
be an untilt of $S$, 
we denote 
\[
\mathcal{O}^{\land}_{\mathcal{Y}_{[0,\infty)}(S), S^{\sharp}} := W(R^+) [\frac{1}{[\omega]}]^{\wedge}_{\xi},
\]
i.e. the $\xi$-adic completion of 
$W(R^+) [\frac{1}{[\omega]}]$, where $\xi \in W(R^+)$ is a generator of the kernel of $W(R^+)\rightarrow R^{\sharp,+}$, and $\omega \in R^+$ is a pseudo-uniformizer. Note that 
\[
\mathcal{O}^{\land}_{\mathcal{Y}_{[0,\infty)}(S), S^{\sharp}} = \mathbb{B}_{\on{dR}}^+(R^{\sharp})
\]
if $R^{\sharp}$ is of characteristic $0$, and \[
\mathcal{O}^{\land}_{\mathcal{Y}_{[0,\infty)}(S), S^{\sharp}} = W(R)
\]
if $S^{\sharp}=S$, where $\mathbb{B}_{\on{dR}}^+(R^{\sharp}) $
is the $\xi$-adic completion of $W(R^+)[\frac{1}{p}]$
with $\xi \in W(R^+)$ being a generator of the kernel of the canonical map $W(R^{+})[\frac{1}{p}] \rightarrow R^{\sharp}$.

When we have finitely many untilts $S_i^{\sharp} = \on{Spa}(R^{\sharp}_i,R^{\sharp,+}_i)$ of $S$ (possibly not all distinct), we denote
\[
\mathcal{O}^{\land}_{\mathcal{Y}_{[0,\infty)}(S), \sum S_i^{\sharp}}
:=
W(R^+) [\frac{1}{[\omega]}]^{\wedge}_{\xi},
\]
where $\xi = \prod \xi_i$ with $\xi_i$ a generator of the kernel of $W(R^+)\rightarrow R_i^{\sharp,+}$. 

We will use frequently the formal properties of cohomological correspondences. Let
$X,Y$ and $Z$ be small v-stacks which form a diagram 
\[
\begin{tikzcd}
 & Z \arrow[dr,"t"] \arrow[dl,"s"'] &
 \\
X &  & Y,
\end{tikzcd}
\]
where we assume that $t$ is compactifiable, representable in locally spatial diamonds with locally dim.trg $< \infty$.  
Let $\Lambda = \mathcal{O}_K$ with $K$ a finite extension of $\mathbb{Q}_\ell$,  $A\in D_{\text{ét}}(X,\Lambda)$ and $B \in D_{\text{ét}}(Y,\Lambda)$, then
a cohomological correspondence
from 
$A$ to $B$ 
supported on $Z$, denoted by
\[
\mathscr{C}: (X,A) \longrightarrow (Y,B),
\]
is defined to be a morphism
\[
\mathscr{C}: s^*A \longrightarrow t^!B
\]
in $D_{\text{ét}}(Z,\Lambda)$.

Let us be given a commutative diagram
\[
\begin{tikzcd}
 & \mathcal{Z}  \arrow[dd,"f"]\arrow[dr,"t'"] \arrow[dl,"s'"'] &
 \\
\mathcal{X} \arrow[dd,"h"] &  & \mathcal{Y} \arrow[dd,"g"]
\\
 & Z \arrow[dr,"t"] \arrow[dl,"s"'] &
 \\
X &  & Y.
\end{tikzcd}
\]
of small $v$-stacks. Let $\mathscr{C} : (X,A) \rightarrow (Y,B)$ be a cohomological correspondence supported on $Z$, then we can pullback $\mathscr{C}$ along the diagram
to obtain a cohomological correspondence
\[
f^*\mathscr{C}:(\mathcal{X},h^*A) \longrightarrow (\mathcal{Y}, g^*B)
\]
supported on $\mathcal{Z}$, 
if either the right square of the diagram is Cartesian or $f,g$ and $h$ are cohomologically smooth, see \cite[Appendix A.2.11]{2017arXiv170705700X}.

On the other hand, let $\mathscr{C} : (\mathcal{X},A) \rightarrow (\mathcal{Y},B)$ be a cohomological correspondence supported on $\mathcal{Z}$, and suppose that $s'$ and $s$ are proper, then we can !-pushforward $\mathscr{C}$ to obtain a cohomological correspondence
\[
Rf_!\mathscr{C}: (X, Rh_!A) \longrightarrow (Y,Rg_!B)
\]
supported on $Z$, see \cite[Appendix A.2.6]{2017arXiv170705700X}.

Moreover, we can compose cohomological correspondences, and pullback and !-pushforward commute with composition, see appendix A.2 of $loc.cit.$. The cohomological correspondences and their relation with 6-functor formalism are studied systematically in the more recent works, see \cite{2024arXiv241013038H} and \cite{Lu_Zheng_2022}.

\subsection*{Acknowledgments}
I would like to thank Peter Scholze for his constant help and encouragement on this project. I am particularly grateful for the  early access of his fundamental work \cite{2021arXiv210213459F} with Fargues. I would like also to thank Xinwen Zhu for comments on the initial draft and sharing ideas with me. I would like to also thank the anonymous referee for many detailed suggestions and corrections.  I am  grateful to Max Planck Institute for Mathematics in Bonn for its hospitality and financial support.

\section{Basic Definitions}

Let $G$ be a reductive group over $\mathbb{Q}_p$, together with a reductive model $\mathcal{G}$  over $\mathbb{Z}_p$. Then $G$ is unramified, so split over $\Breve{\mathbb{Q}}_p:= \Breve{\mathbb{Z}}_p [\frac{1}{p}]$, where $ \Breve{\mathbb{Z}}_p :=W(\overline{\mathbb{F}_p})$.  Let 
$\mu $ be a conjugacy class of cocharacters
$\mathbb{G}_m \rightarrow G_{\overline{\mathbb{Q}_p}}$,
whose field of definition is $E$, which is contained in $\Breve{\mathbb{Q}}_p$ (by the unramifiedness of $G$). 

Let us first introduce notations that we will use frequently. Let $S$ be an affinoid perfectoid space of characteristic $p$,  $\mathcal{P}_1$ and $ \mathcal{P}_{2}$  $\mathcal{G}$-torsors  on 
$\mathcal{Y}_{[0,\infty)}(S)$, and $S^{\sharp}$  an untilt of $S$. We will use the dotted arrow
\[
\begin{tikzcd}
\mathcal{P}_1  \arrow[r,dashed, "\varphi"] &
\mathcal{P}_2
\end{tikzcd}
\]
to denote a modification from $\mathcal{P}_1$ to $\mathcal{P}_2$, i.e. $\varphi$ is an isomorphism
\[
\varphi:  \mathcal{P}_1|_{\mathcal{Y}_{[0,\infty)}(S)\setminus  S^{\sharp}} \cong \mathcal{P}_{2}|_{\mathcal{Y}_{[0,\infty)}(S)\setminus  S^{\sharp}}
\]
that is meromorphic at $S^{\sharp}$, in the sense of \cite[Definition 5.3.5]{SW17}. We can also have modifications with respect to finitely many untilts of $S$. 

Let $ \mu$ be a dominant cocharacter of $G$, we will say that $\varphi$ is bounded by $\mu$ at $S^{\sharp}$
if for any geometric rank 1 point $C$ of $S$, and  trivializations of 
$\mathcal{P}_1$
and $\mathcal{P}_2$
over  
$\mathcal{O}^{\land}_{\mathcal{Y}_{[0,\infty)}(C), C^{\sharp}} [\frac{1}{\xi}]$, the "punctured formal neighborhood of $C^{\sharp}$", 
$\varphi$ 
lies in 
$\mathcal{G}(\mathcal{O}^{\land}_{\mathcal{Y}_{[0,\infty)}(C), C^{\sharp}}) \lambda(\xi) \mathcal{G}(\mathcal{O}^{\land}_{\mathcal{Y}_{[0,\infty)}(C), C^{\sharp}})$
for some $\lambda \leq \mu$, 
where
$\xi \in \mathcal{O}^{\land}_{\mathcal{Y}_{[0,\infty)}(C), C^{\sharp}}$ 
is a uniformizer, and
$\lambda(\xi) \in G(\mathcal{O}^{\land}_{\mathcal{Y}_{[0,\infty)}(C), C^{\sharp}} [\frac{1}{\xi}])$
is the image of $\xi$ under ($\mathcal{O}^{\land}_{\mathcal{Y}_{[0,\infty)}(C), C^{\sharp}} [\frac{1}{\xi}]$-points of) $\lambda$. Note that this is a pointwise condition on $S$.

The definitions introduced below are only prestacks a priori, but it follows directly
 from \cite[Proposition 19.5.3]{SW17} and \cite[Corollary 17.1.9]{SW17}  that they are all stacks, with the exception on (global) Witt vector shtukas and Hecke stacks in section \ref{wittt}, where the stackification is necessary.

\subsection{Hecke stacks}

We first introduce Hecke stacks.

\begin{definition}

Let $I$ be a finite set, and $\mu_{\bullet}$ be a collection of dominant cocharacters of $G$ indexed by $I$.  Let $\{I_1, \cdots, I_k\}$ be a partition of $I$. We define 
\[
\on{Hecke}_{}^{(I_1, \cdots, I_k)}
\]
to be the stack over $\on{Spd}(\Breve{\mathbb{Z}}_p)^I $ ($:= \on{Spd}(\Breve{\mathbb{Z}}_p) \times_{\on{Spd}(\overline{\mathbb{F}_p})} \cdots \times_{\on{Spd}(\overline{\mathbb{F}_p})}  \on{Spd}(\Breve{\mathbb{Z}}_p)$) whose value at $S$ is the groupoid of 

$\bullet$ 
$\mathcal{G}$-torsors $\mathcal{P}_1, \cdots, \mathcal{P}_{k+1}$  on 
$\mathcal{Y}_{[0,\infty)}(S)$

$\bullet$
 untilts
$S^{\sharp}_i \hookrightarrow \mathcal{Y}_{[0,\infty)}(S)$
of $S$ defined over $\Breve{\mathbb{Z}}_p$ indexed by $i\in I$, and

$\bullet$
modifications
\[
\begin{tikzcd}
\mathcal{P}_1  \arrow[r,dashed, "\varphi_1"] &
\mathcal{P}_2 \arrow[r,dashed, "\varphi_2"] &
\cdots \cdots   \arrow[r,dashed,"\varphi_{k-1}" ]&
\mathcal{P}_k \arrow[r,dashed,"\varphi_k" ] &
\mathcal{P}_{k+1},
\end{tikzcd}
\]
where the dotted arrows $\varphi_i$ are isomorphisms
\[
\varphi_i:  \mathcal{P}_i|_{\mathcal{Y}_{[0,\infty)}(S)\setminus \underset{j \in I_i}{\cup} S^{\sharp}_j} \cong \mathcal{P}_{i+1}|_{\mathcal{Y}_{[0,\infty)}(S)\setminus \underset{j \in I_i}{\cup} S^{\sharp}_j}
\]
that are meromorphic along the  Cartier divisor 
$ \underset{j \in I_i}{\sum} S^{\sharp}_j $ of 
$ \mathcal{Y}_{[0,\infty)}(S)$.

We define 
\[
\on{Hecke}_{\mu_{\bullet}}^{(I_1, \cdots, I_k)}
\]
to be the closed substack of $\on{Hecke}_{}^{(I_1, \cdots, I_k)}$
parametrizing the  data as
$\on{Hecke}_{}^{(I_1, \cdots, I_k)}$ does with the additional condition that 
 at every geometric rank 1 point $C$ of $S$ and untilt $C^{\sharp}$ of $C$,
$\varphi_i$ is bounded by 
$\underset{
\begin{subarray}{c}
j \in I_i
\\
C^{\sharp} \in S^{\sharp}_j 
  \end{subarray}
  }{\sum} \mu_j$ 
at  $C^{\sharp}$, where $C^{\sharp} \in S^{\sharp}_j $ means that the canonical untilt (over $S_j^{\sharp}$) of $C$ over $S$, under the equivalence between perfectoid spaces over $S^{\sharp}_j$ and perfectoid spaces over $S$,  is identified with $C^{\sharp}$.

When the partition is trivial, we make the following simplification
\[
\on{Hecke}_{\mu_{\bullet}}:= \on{Hecke}_{\mu_{\bullet}}^{(I)},
\]
which parametrizes
\[
\varphi:
\mathcal{P}_1|_{\mathcal{Y}_{[0,\infty)}(S)\setminus \underset{i\in I}{\cup} S^{\sharp}_i} \cong \mathcal{P}_2|_{\mathcal{Y}_{[0,\infty)}(S)\setminus \underset{i\in I}{\cup} S^{\sharp}_i}
\]
that are meromorphic and bounded by $\mu_i$ at each untilt $S_i^{\sharp}$ of $S$.

\end{definition}

\begin{remark} \label{ruieuivcdeieurie}
We note that $\on{Hecke}^{(I_1,\cdots,I_k)}_{\mu_{\bullet}}$
is really defined over the product of  ring of integers of the defining fields of $\mu_{\bullet}$. The base can be further refined if we do not specify the boundedness condition. Indeed, $\on{Hecke}^{(I_1,\cdots,I_k)}$ can be naturally defined over $\on{Spd}(\mathbb{Z}_p)^I$, i.e. it is the base change to $\on{Spd}(\Breve{\mathbb{Z}}_p)^I$ of a v-stack defined over $\on{Spd}(\mathbb{Z}_p)^I$. 
\end{remark}

We observe that there are natural maps between Hecke stacks. 

Let 
$\zeta: I \rightarrow J$ 
be a map of finite sets, $(J_1, \cdots, J_k)$ be a partition of $J$, and $I_i = \zeta^{-1}J_i$, there exists a natural map
\begin{equation}  \label{eiurwdncmn}
i_{\zeta}: 
\on{Hecke}^{(I_1, \cdots, I_k)} \times_{\on{Spd}(\Breve{\mathbb{Z}}_p)^I, \zeta} \on{Spd}(\Breve{\mathbb{Z}}_p)^J \hookrightarrow
\on{Hecke}^{(J_1, \cdots, J_k)},
\end{equation}
where we abuse the notation by writing $\zeta: \on{Spd}(\Breve{\mathbb{Z}}_p)^J \rightarrow \on{Spd}(\Breve{\mathbb{Z}}_p)^I$
the permutation map corresponding to $\zeta$, i.e. sending $(x_j)_{j\in J}$ to $(x_{\zeta(i)})_{i\in I}$.   It is defined by sending
\[
\begin{tikzcd}
\mathcal{P}_1  \arrow[r,dashed, "\varphi_1"] &
\mathcal{P}_2 \arrow[r,dashed, "\varphi_2"] &
\cdots \cdots   \arrow[r,dashed,"\varphi_{k-1}" ]&
\mathcal{P}_k \arrow[r,dashed,"\varphi_k" ] &
\mathcal{P}_{k+1},
\end{tikzcd}
\]
to itself, except we just relabel the legs using $\zeta$. More precisely, it sends 
\[
\varphi_i:  \mathcal{P}_i|_{\mathcal{Y}_{[0,\infty)}(S)\setminus \underset{k \in I_i}{\cup} S^{\sharp}_{\zeta(k)}} \cong \mathcal{P}_{i+1}|_{\mathcal{Y}_{[0,\infty)}(S)\setminus \underset{k \in I_i}{\cup} S^{\sharp}_{\zeta(k)}}
\]
to 
\[
\varphi_i':  \mathcal{P}_i|_{\mathcal{Y}_{[0,\infty)}(S)\setminus \underset{j \in J_i}{\cup} S^{\sharp}_{j}} \cong \mathcal{P}_{i+1}|_{\mathcal{Y}_{[0,\infty)}(S)\setminus \underset{j \in J_i}{\cup} S^{\sharp}_{j}},
\]
which is the restriction of $\varphi_i$ to
$\mathcal{Y}_{[0,\infty)}(S)\setminus \underset{j \in J_i}{\cup} S^{\sharp}_{j}$, i.e.
the modification at $S^{\sharp}_{j}$ is given by $\varphi_i$ if $j \in \zeta(I_i)$, and it is trivial if $j \notin \zeta(I_i)$.

Moreover, we have the convolution map
\begin{equation} \label{7294389823}
\pi^{(I_1, \cdots, I_k)} : \on{Hecke}^{(I_1, \cdots, I_k)} \longrightarrow \on{Hecke}^{(I)},
\end{equation}
which is defined by sending 
\[
\begin{tikzcd}
\mathcal{P}_1  \arrow[r,dashed, "\varphi_1"] &
\mathcal{P}_2 \arrow[r,dashed, "\varphi_2"] &
\cdots \cdots   \arrow[r,dashed,"\varphi_{k-1}" ]&
\mathcal{P}_k \arrow[r,dashed,"\varphi_k" ] &
\mathcal{P}_{k+1}
\end{tikzcd}
\]
to 
\[
\begin{tikzcd}
\mathcal{P}_1  \arrow[r,dashed, "\varphi_k  \cdots  \varphi_1"] &
\mathcal{P}_{k+1}. 
\end{tikzcd}
\]

Lastly, we have the maps that split the partition. 
Let $(J_1, \cdots, J_l)$ be a partition of $I$ such that $J_j = I_{i_j} \cup \cdots \cup I_{i_{j+1}-1}$ for some partition $1=i_1 < i_2 < \cdots < i_{l+1}= k+1$ of  $k+1$. We define
\begin{equation} \label{24939923-043-204}
\kappa^{(I_1, \cdots, I_k)}_{(J_1, \cdots, J_l)}: \on{Hecke}^{(I_1, \cdots, I_k)} \longrightarrow \underset{j}{\prod} \on{Hecke}^{(I_{i_j}, \cdots, I_{i_{j+1}-1})}
\end{equation}
to be the forgetting map that sends 
\[
\begin{tikzcd}
\mathcal{P}_1  \arrow[r,dashed, "\varphi_1"] &
\mathcal{P}_2 \arrow[r,dashed, "\varphi_2"] &
\cdots \cdots   \arrow[r,dashed,"\varphi_{k-1}" ]&
\mathcal{P}_k \arrow[r,dashed,"\varphi_k" ] &
\mathcal{P}_{k+1}
\end{tikzcd}
\]
to 
\[
\underset{j}{\prod} (\begin{tikzcd}
\mathcal{P}_{i_j}  \arrow[r,dashed, "\varphi_{i_j}"] &
\mathcal{P}_{i_{j+1}} \arrow[r,dashed, "\varphi_{i_j+1}"] &
\cdots \cdots   \arrow[r,dashed,"\varphi_{i_{j+1}-2}" ]&
\mathcal{P}_{i_{j+1}-1} \arrow[r,dashed,"\varphi_{i_{j+1}-1}" ] &
\mathcal{P}_{i_{j+1}}
\end{tikzcd} 
).
\]

Now we introduce the correspondence version for Hecke stacks.

\begin{definition}
Let $I, J$ be finite sets, and $\mu_{\bullet}, \mu'_{\bullet}$ be collections of dominant cocharacters of $G$ indexed by $I$ and $J$ respectively. Let $\nu$ be a dominant cocharacter of $G$. Then we define 
 \[
 \on{Hecke}^{\nu}_{\mu_{\bullet}|\mu_{\bullet}'}
 \]
to be the stack  with value at a characteristic $p$ perfectoid space $S$ being the groupoid of the data of

$\bullet$ 
$\{\mathcal{P}_1, \mathcal{P}_2, \varphi \}\in \on{Hecke}_{\mu_{\bullet} }(S)$ 
and 
$\{\mathcal{P}'_1, \mathcal{P}_2',\varphi'\} \in \on{Hecke}_{\mu'_{\bullet} }(S)$,

$\bullet$ (meromorphic) modifications bounded by $\nu$ 
\[
\gamma_1:
\mathcal{P}_1|_{\mathcal{Y}_{(0,\infty)}(S)} \cong \mathcal{P}'_1|_{\mathcal{Y}_{(0,\infty)}(S)} 
\]
\[
\gamma_2:
\mathcal{P}_2|_{\mathcal{Y}_{(0,\infty)}(S)} \cong \mathcal{P}'_2|_{\mathcal{Y}_{(0,\infty)}(S)} 
\]
 at the characteristic $p$ untilt 
$S \hookrightarrow \mathcal{Y}_{[0,\infty)}(S)$
such that the following diagram is commutative
\[
\begin{tikzcd}
\mathcal{P}_1 \arrow[r, dashed,"\varphi"] \arrow[d, dashed,"\gamma_1"] 
&
\mathcal{P}_2 \arrow[d,dashed,"\gamma_2"] 
\\
\mathcal{P}'_1 \arrow[r,dashed,"\varphi'"] &
\mathcal{P}'_2.
\end{tikzcd}
\]

\end{definition}

We will occasionally use the following more general version.

\begin{definition} \label{badhec}
Let $I, J$ be finite sets, and $\mu_{\bullet}, \mu'_{\bullet}$ be collections of dominant cocharacters of $G$ indexed by $I$ and $J$ respectively. Let $\{I_1, \cdots, I_k\}$ and $\{J_1, \cdots, J_l\}$ be partitions of $I$ and $J$. Let $\nu$ be a dominant cocharacter of $G$. Then we define 
 \[
 \on{Hecke}^{(I_1, \cdots, I_k)|(J_1, \cdots, J_l), \nu}_{\mu_{\bullet}|\mu_{\bullet}'}
 \]
to be the stack  with value at a characteristic $p$ perfectoid space $S$ being the groupoid of the data of

$\bullet$ 
$\{\mathcal{P}_1,\cdots, \mathcal{P}_{k+1}, \varphi_1, \cdots \varphi_k \}\in \on{Hecke}_{\mu_{\bullet} }^{(I_1, \cdots, I_k)}(S)$ 
and 
$\{\mathcal{P}'_1, \cdots, \mathcal{P}_{l+1}',\varphi',\cdots \varphi_l'\} \in \on{Hecke}_{\mu'_{\bullet} }^{(J_1, \cdots, J_l)}(S)$,

$\bullet$ modifications bounded by $\nu$ (in the sense similar to the above)
\[
\gamma_1:
\mathcal{P}_1|_{\mathcal{Y}_{(0,\infty)}(S)} \cong \mathcal{P}'_{1}|_{\mathcal{Y}_{(0,\infty)}(S)} 
\]
\[
\gamma_2:
\mathcal{P}_{k+1}|_{\mathcal{Y}_{(0,\infty)}(S)} \cong \mathcal{P}'_{l+1}|_{\mathcal{Y}_{(0,\infty)}(S)} 
\]
at the characteristic $p$ untilt 
$S \hookrightarrow \mathcal{Y}_{[0,\infty)}(S)$
such that the following diagram is commutative
\[
\begin{tikzcd}
 \arrow[d, dashed,"\gamma_1"] \mathcal{P}_1  \arrow[r,dashed, "\varphi_1"] &
\mathcal{P}_2 \arrow[r,dashed, "\varphi_2"] &
\cdots \cdots   \arrow[r,dashed,"\varphi_{k-1}" ]&
\mathcal{P}_k \arrow[r,dashed,"\varphi_k" ] &
\mathcal{P}_{k+1} \arrow[d,dashed,"\gamma_2"] 
\\
\mathcal{P}_1'  \arrow[r,dashed, "\varphi_1'"] &
\mathcal{P}_2' \arrow[r,dashed, "\varphi_2'"] &
\cdots \cdots   \arrow[r,dashed,"\varphi_{l-1}'" ]&
\mathcal{P}'_l \arrow[r,dashed,"\varphi_l'" ] &
\mathcal{P}'_{l+1}.
\end{tikzcd}
\]

\end{definition}

We will also make use of a local version of Hecke stacks, which parametrizes modifications of torsors over the formal neighborhoods of Cartier divisors corresponding to untilts. 

\begin{definition}

Let $I$ be a finite set, and $\mu_{\bullet}$ be a collection of dominant cocharacters of $G$ indexed by $I$.  Let $\{I_1, \cdots, I_k\}$ be a partition of $I$. We define 
\[
\on{Hecke}_{\mu_{\bullet}}^{\on{loc},(I_1, \cdots, I_k)}
\]
to be the stack over $ \on{Spd}(\Breve{\mathbb{Z}}_p)^I$ whose value at $S$ is the groupoid of 

$\bullet$
 untilts
$S^{\sharp}_i$
of $S$ defined over $\Breve{\mathbb{Z}}_p$ indexed by $i\in I$,

$\bullet$ 
$\mathcal{G}$-torsors $\mathcal{P}_1, \cdots, \mathcal{P}_{k+1}$  on 
$Spec( \mathcal{O}^{\land}_{\mathcal{Y}_{[0,\infty)}(S), \sum S_i^{\sharp}})$, and,

$\bullet$
modifications
\[
\begin{tikzcd}
\mathcal{P}_1  \arrow[r,dashed, "\varphi_1"] &
\mathcal{P}_2 \arrow[r,dashed, "\varphi_2"] &
\cdots \cdots   \arrow[r,dashed,"\varphi_{k-1}" ]&
\mathcal{P}_k \arrow[r,dashed,"\varphi_k" ] &
\mathcal{P}_{k+1},
\end{tikzcd}
\]
where the dotted arrows $\varphi_i$ are isomorphisms
\[
\varphi_i:  \mathcal{P}_i|_{\mathcal{O}^{\land}_{\mathcal{Y}_{[0,\infty)}(S), \sum S_k^{\sharp}}[\frac{1}{\underset{j \in I_i}{\prod}\xi_j}]} \cong 
\mathcal{P}_{i+1}|_{\mathcal{O}^{\land}_{\mathcal{Y}_{[0,\infty)}(S), \sum S_k^{\sharp}}[\frac{1}{\underset{j \in I_i}{\prod}\xi_j}]}
\]
such that at every geometric rank 1 point $C$ of $S$ and untilt $C^{\sharp}$ of $C$,
$\varphi_i$ is bounded by 
$\underset{
\begin{subarray}{c}
j \in I_i
\\
C^{\sharp} \in S^{\sharp}_j 
  \end{subarray}
  }{\sum} \mu_j$ 
at  $C^{\sharp}$.

When the partition is trivial, we make the following simplification
\[
\on{Hecke}^{\on{loc}}_{\mu_{\bullet}}:= \on{Hecke}_{\mu_{\bullet}}^{{\on{loc}},(I)}. 
\]

Lastly, we denote
\[
\on{Hecke}^{{\on{loc}},(I_1, \cdots, I_k)} := \underset{\mu_{\bullet}}{\bigcup} \on{Hecke}_{\mu_{\bullet}}^{{\on{loc}},(I_1, \cdots, I_k)}.
\]
\end{definition}

\begin{remark}
When the partition of $I$ is trivial, the local Hecke stack is defined in \cite[Definition \rom{6}.1.6]{2021arXiv210213459F}. The general partition case also appears in the proof of \cite[Definition/proposition \rom{6}.9.4]{2021arXiv210213459F}.
\end{remark}

\begin{remark} \label{u803r4u3}
The drawback of local Hecke stacks is that there is no obvious correspondence between them, due to that the torsors are defined on different domains. However, this problem disappears if we fix the legs at characteristic $p$ untilts.  An important example is the Witt vector local Hecke stacks to be defined below, where we have the correspondences.
\end{remark}

There is a canonical map
\[
\on{Hecke}_{\mu_{\bullet}}^{(I_1, \cdots, I_k)}
\longrightarrow
\on{Hecke}_{\mu_{\bullet}}^{\on{loc},(I_1, \cdots, I_k)}
\]
being the restriction from global $G$-torsors to local $G$-torsors. We will use it to define the Satake sheaves on $\on{Hecke}_{\mu_{\bullet}}^{(I_1, \cdots, I_k)}$. 

Let $L^{+}\mathcal{G} $ be the group v-sheaf over $\on{Spd}(\Breve{\mathbb{Z}}_p)$
whose value at $S$ is
\[
L^{+}\mathcal{G}(S) = \{(x, S^{\sharp}) | S^{\sharp} \in \on{Spd}(\Breve{\mathbb{Z}}_p)(S), x \in \mathcal{G}(\mathcal{O}^{\land}_{\mathcal{Y}_{[0,\infty)}(S), S^{\sharp}}) \},
\]
and 
$Gr_{\mathcal{G},\on{Spd}(\Breve{\mathbb{Z}}_p),\leq \mu}$ 
the Beilinson--Drinfeld affine Grassmannian as defined in \cite[Definition 20.3.1]{SW17}, which classifies $\mathcal{G}$-torsors on the formal neighborhood of $S^{\sharp}$ together with a trivialization of the generic fiber that is bounded by $\mu$. By \cite[Proposition 20.5.4]{SW17}, 
$Gr_{\mathcal{G},\on{Spd}(\Breve{\mathbb{Z}}_p),\leq \mu} \rightarrow \on{Spd}(\Breve{\mathbb{Z}}_p)$
is representable in spatial diamonds, and so in particular a small v-sheaf. By definition, we have 
\[
\on{Hecke}_{\mu}^{{\on{loc}}} =L^{+}\mathcal{G} \setminus Gr_{\mathcal{G},\on{Spd}(\Breve{\mathbb{Z}}_p),\leq \mu}
\]
as a stack quotient,
where $L^{+}\mathcal{G}$ acts on the trivialization, so 
$\on{Hecke}_{\mu}^{{\on{loc}}}$ is a small v-stack.

We have similar expression for the global Hecke stack $\on{Hecke}_{\mu}$, which can be written as the quotient of 
$Gr_{\mathcal{G},\on{Spd}(\Breve{\mathbb{Z}}_p),\leq \mu}$ 
by the group v-sheaf parametrizing automorphisms of the trivial global $\mathcal{G}$-torsors on $\mathcal{Y}_{[0,\infty)}$. Moreover, for Hecke stacks with multiple legs, it can be written again as quotients of twisted products of affine Grassmannian, see the beginning of \cite[Chapter \rom{6}]{2021arXiv210213459F}  (for the version of local Hecke stacks).  

If  the  characteristic of the untilt $S^{\sharp}$ is $p$,  then $S^{\sharp} = S$, and  
\[
\mathcal{O}^{\land}_{\mathcal{Y}_{[0,\infty)}(S), S^{\sharp}} = W(R)
\]
if $S=\on{Spa}(R,R^{+})$. Then the fiber of the Beilinson--Drinfeld affine Grassmannian at the characteristic $p$ untilt is the Witt vector affine Grassmannian considered in \cite{Zhu2017}. \footnote{We can embed perfect schemes into the category of v-sheaves using either the $\diamond$- or $\diamond\diamond$-analytification functor as explained in appendix \ref{serepoir}. The Witt vector affine Grassmannian is ind-proper, so the two $v$-sheaves associated to it are the same.} 
On the other hand, if 
$S^{\sharp}=\on{Spa}(R^{\sharp},R^{\sharp,+})$
is a characteristic zero untilt, we have 
\[
\mathcal{O}^{\land}_{\mathcal{Y}_{[0,\infty)}(S), S^{\sharp}} = \mathbb{B}^+_{\text{dR}}(R^{\sharp}),
\]
and the fiber of 
$Gr_{\mathcal{G},\on{Spd}(\Breve{\mathbb{Z}}_p),\leq \mu}$
at $S^{\sharp}$ is the 
$\mathbb{B}^+_{\text{dR}}$-affine Grassmannian as defined in \cite[Lecture 19]{SW17}.

Similar to global Hecke stacks, there are natural maps between local Hecke stacks. 

Let 
$\zeta: I \rightarrow J$ 
be a map of finite sets, $(J_1, \cdots, J_k)$ be a partition of $J$, and $I_i = \zeta^{-1}J_i$, there exists a natural map
\begin{equation}  \label{a4}
i^{\on{loc}}_{\zeta}: 
\on{Hecke}^{\on{loc},(I_1, \cdots, I_k)} \times_{\on{Spd}(\Breve{\mathbb{Z}}_p)^I, \zeta} \on{Spd}(\Breve{\mathbb{Z}}_p)^J \hookrightarrow
\on{Hecke}^{\on{loc},(J_1, \cdots, J_k)},
\end{equation}
where we abuse the notation by writing $\zeta: \on{Spd}(\Breve{\mathbb{Z}}_p)^J \rightarrow \on{Spd}(\Breve{\mathbb{Z}}_p)^I$
the permutation map corresponding to $\zeta$, i.e. sending $(x_j)_{j\in J}$ to $(x_{\zeta(i)})_{i\in I}$.   It is defined by sending
\[
\begin{tikzcd}
\mathcal{P}_1  \arrow[r,dashed, "\varphi_1"] &
\mathcal{P}_2 \arrow[r,dashed, "\varphi_2"] &
\cdots \cdots   \arrow[r,dashed,"\varphi_{k-1}" ]&
\mathcal{P}_k \arrow[r,dashed,"\varphi_k" ] &
\mathcal{P}_{k+1},
\end{tikzcd}
\]
to itself, except we just relabel the legs using $\zeta$. More precisely, it sends 
\[
\varphi_i:  \mathcal{P}_i|_{\mathcal{O}^{\land}_{\mathcal{Y}_{[0,\infty)}(S), \sum S_k^{\sharp}}[\frac{1}{\underset{h \in I_i}{\prod}\xi_{\zeta(h)}}]} \cong \mathcal{P}_{i+1}|_{\mathcal{O}^{\land}_{\mathcal{Y}_{[0,\infty)}(S), \sum S_k^{\sharp}}[\frac{1}{\underset{h \in I_i}{\prod}\xi_{\zeta(h)}}]}
\]
to 
\[
\varphi_i':  \mathcal{P}_i|_{\mathcal{O}^{\land}_{\mathcal{Y}_{[0,\infty)}(S), \sum S_k^{\sharp}}[\frac{1}{\underset{j \in J_i}{\prod}\xi_j}]} \cong \mathcal{P}_{i+1}|_{\mathcal{O}^{\land}_{\mathcal{Y}_{[0,\infty)}(S), \sum S_k^{\sharp}}[\frac{1}{\underset{j \in J_i}{\prod}\xi_j}]},
\]
where the modification at $S^{\sharp}_{j}$ is given by $\varphi_i$ if $j \in \zeta(I_i)$, and it is trivial, i.e. no modification at all, if $j \notin \zeta(I_i)$.

Moreover, we have the convolution map
\begin{equation} \label{a5}
\pi^{\on{loc},(I_1, \cdots, I_k)} : \on{Hecke}^{\on{loc},(I_1, \cdots, I_k)} \longrightarrow \on{Hecke}^{\on{loc},(I)},
\end{equation}
which is defined by sending 
\[
\begin{tikzcd}
\mathcal{P}_1  \arrow[r,dashed, "\varphi_1"] &
\mathcal{P}_2 \arrow[r,dashed, "\varphi_2"] &
\cdots \cdots   \arrow[r,dashed,"\varphi_{k-1}" ]&
\mathcal{P}_k \arrow[r,dashed,"\varphi_k" ] &
\mathcal{P}_{k+1}
\end{tikzcd}
\]
to 
\[
\begin{tikzcd}
\mathcal{P}_1  \arrow[r,dashed, "\varphi_k  \cdots  \varphi_1"] &
\mathcal{P}_{k+1}. 
\end{tikzcd}
\]

Lastly, we have the maps that split the partition. 
Let $(J_1, \cdots, J_l)$ be a partition of $I$ such that $J_j = I_{i_j} \cup \cdots \cup I_{i_{j+1}-1}$ for some partition $1=i_1 < i_2 < \cdots < i_{l+1}= k+1$ of  $k+1$. We define
\begin{equation} \label{a6}
\kappa^{\on{loc},(I_1, \cdots, I_k)}_{(J_1, \cdots, J_l)}: \on{Hecke}^{\on{loc},(I_1, \cdots, I_k)} \longrightarrow \underset{j}{\prod} \on{Hecke}^{\on{loc},(I_{i_j}, \cdots, I_{i_{j+1}-1})}
\end{equation}
to be the forgetting map that sends 
\[
\begin{tikzcd}
\mathcal{P}_1  \arrow[r,dashed, "\varphi_1"] &
\mathcal{P}_2 \arrow[r,dashed, "\varphi_2"] &
\cdots \cdots   \arrow[r,dashed,"\varphi_{k-1}" ]&
\mathcal{P}_k \arrow[r,dashed,"\varphi_k" ] &
\mathcal{P}_{k+1}
\end{tikzcd}
\]
to 
\[
\underset{j}{\prod} (\begin{tikzcd}
\mathcal{P}_{i_j}  \arrow[r,dashed, "\varphi_{i_j}"] &
\mathcal{P}_{i_{j+1}} \arrow[r,dashed, "\varphi_{i_j+1}"] &
\cdots \cdots   \arrow[r,dashed,"\varphi_{i_{j+1}-2}" ]&
\mathcal{P}_{i_{j+1}-1} \arrow[r,dashed,"\varphi_{i_{j+1}-1}" ] &
\mathcal{P}_{i_{j+1}}
\end{tikzcd} 
).
\]

\begin{notation} \label{notehecke}
Throughout the paper, we use the following notation to denote the restriction  to the generic fiber,
\[
\on{Hecke}_{\mu_{\bullet},\eta_{\bullet}}^{(I_1,\cdots,I_k)} := \on{Hecke}_{\mu_{\bullet}}^{(I_1,\cdots,I_k)} |_{ \on{Spd}(\Breve{\mathbb{Q}}_p)^I}
\]
\[
\on{Hecke}_{\mu_{\bullet},\eta_{\bullet}}^{\on{loc},(I_1,\cdots,I_k)} := \on{Hecke}_{\mu_{\bullet}}^{\on{loc},(I_1,\cdots,I_k)} |_{ \on{Spd}(\Breve{\mathbb{Q}}_p)^I}
\]
\[
\on{Hecke}^{(I_1, \cdots, I_k)|(J_1, \cdots, J_l), \nu}_{\mu_{\bullet}|\mu_{\bullet}', \eta_{\bullet}|\eta_{\bullet}} :=
\on{Hecke}^{(I_1, \cdots, I_k)|(J_1, \cdots, J_l), \nu}_{\mu_{\bullet}|\mu_{\bullet}'} |_{ \on{Spd}(\Breve{\mathbb{Q}}_p)^I \times  \on{Spd}(\Breve{\mathbb{Q}}_p)^J}.
\]
Similarly for the restriction to the special fiber,
\[
\on{Hecke}_{\mu_{\bullet},s_{\bullet}}^{(I_1,\cdots,I_k)} := \on{Hecke}_{\mu_{\bullet}}^{(I_1,\cdots,I_k)} |_{ \on{Spd}(\overline{\mathbb{F}_p})^I}
\]
\[
\on{Hecke}_{\mu_{\bullet},s_{\bullet}}^{\on{loc},(I_1,\cdots,I_k)} := \on{Hecke}_{\mu_{\bullet}}^{\on{loc},(I_1,\cdots,I_k)} |_{ \on{Spd}(\overline{\mathbb{F}_p})^I}
\]
\[
\on{Hecke}^{(I_1, \cdots, I_k)|(J_1, \cdots, J_l), \nu}_{\mu_{\bullet}|\mu_{\bullet}', s_{\bullet} | s_{\bullet}} :=
\on{Hecke}^{(I_1, \cdots, I_k)|(J_1, \cdots, J_l), \nu}_{\mu_{\bullet}|\mu_{\bullet}'} |_{ \on{Spd}(\overline{\mathbb{F}_p})^I \times  \on{Spd}(\overline{\mathbb{F}_p})^J}. 
\]

We also have the mixed version, for example
\[
\on{Hecke}_{(\mu_{1}, \mu_2),s_{} \times \eta}^{(\{1\},\{2\})} := \on{Hecke}^{(\{1\},\{2\})}_{(\mu_{1}, \mu_2) }|_{ \on{Spd}(\overline{\mathbb{F}_p}) \times \on{Spd}(\Breve{\mathbb{Q}}_p)}.
\]
\end{notation}

\subsection{Moduli spaces of $p$-adic shtukas}

We now introduce the critical objects of study in this article, the moduli spaces of shtukas. 

\begin{definition}
The moduli space of shtukas  associated to $\mathcal{G}$ and $\mu$
is the stack 
\[
\on{Sht}_{\mu }
\]
over $\on{Spd}(\Breve{\mathbb{Z}}_p)$ whose value at a characteristic $p$ perfectoid space $S$ is the groupoid of the data of

$\bullet$  a $\mathcal{G}$-torsor $\mathcal{P}$ over $\mathcal{Y}_{[0,\infty)}(S)$,

$\bullet$ an (not necessarily of characteristic 0) untilt 
$S^{\sharp} \hookrightarrow \mathcal{Y}_{[0,\infty)}(S)$
of $S$ defined over $\Breve{\mathbb{Z}}_p$, and

$\bullet$ an isomorphism 
\[
\varphi_{\mathcal{P}}:  \mathcal{P}|_{\mathcal{Y}_{[0,\infty)}(S)\setminus S^{\sharp}} \cong \on{Frob}_S^*\mathcal{P}|_{\mathcal{Y}_{[0,\infty)}(S)\setminus S^{\sharp}}.
\]
that is meromorphic along the Cartier divisor 
$S^{\sharp} \hookrightarrow \mathcal{Y}_{[0,\infty)}(S)$, 
and bounded by $\mu$. 

\end{definition}

\begin{remark}
Although we work with the base $\Breve{\mathbb{Z}}_p$, the definition of moduli spaces of shtukas really depends on the group $\mathcal{G}$ over $\mathbb{Z}_p$ rather than its base change to $\Breve{\mathbb{Z}}_p$, as otherwise the Frobenius pullback $\on{Frob}_S^*\mathcal{P}$ does not make sense (it will be a torsor with respect to a Frobenius twist of the structure group over $\Breve{\mathbb{Z}}_p$). 
\end{remark}

\begin{remark}
The definition given here is close to the classical function field counterpart, which is slightly different from the ones in \cite[Definition 25.1.1 and 23.3.1]{SW17}  in that we do not fix a trivialization of the $\mathcal{G}(\mathbb{Z}_p)$-torsor as in $loc.cit.$, nor do we fix the type of shtukas in $B(G)$. The treatment in $loc.cit.$ is to emphasize the way moduli spaces of shtukas generalize Rapoport--Zink spaces \footnote{Indeed, the classical Rapoport Zink space can be identified with the fiber  at $b$ of a canonical map $\on{Sht_{\mu}} \longrightarrow \on{Bun}_G$.}, whereas we want to mimic the treatment of Lafforgue in \cite{Lafforgue2018}.   
\end{remark}

We will also make use of moduli spaces of shtukas with several legs.

\begin{definition}
Let $I$ be a finite set, and $\mu_{\bullet}$ be a collection of dominant cocharacters of $G$ indexed by $I$. Let $\{I_1, \cdots, I_k\}$ be a partition of $I$. 
The moduli space of shtukas  associated to $G$, $\mu_{\bullet}$ and $\{I_1, \cdots, I_k\}$ 
is the stack 
\[
\on{Sht}_{\mu_{\bullet}}^{(I_1, \cdots, I_k)}
\]
over $ \on{Spd}(\Breve{\mathbb{Z}}_p)^I$ whose value at a characteristic $p$ perfectoid space $S$ is the groupoid of the data of

$\bullet$  $\mathcal{G}$-torsors $\mathcal{P}_1, \cdots, \mathcal{P}_k$ over $\mathcal{Y}_{[0,\infty)}(S)$,

$\bullet$ (not necessarily of characteristic 0) untilts
$S^{\sharp}_i \hookrightarrow \mathcal{Y}_{[0,\infty)}(S)$
of $S$ defined over $\Breve{\mathbb{Z}}_p$ indexed by $i\in I$,  and

$\bullet$ modifications 
\[
\begin{tikzcd}
\mathcal{P}_1  \arrow[r,dashed, "\varphi_1"] &
\mathcal{P}_2 \arrow[r,dashed, "\varphi_2"] &
\cdots \cdots   \arrow[r,dashed,"\varphi_{k-1}" ]&
\mathcal{P}_k \arrow[r,dashed,"\varphi_k" ] &
\on{Frob}_S^*\mathcal{P}_1
\end{tikzcd}
\]
where the dotted arrows $\varphi_i$ are isomorphisms
\[
\varphi_i:  \mathcal{P}_i|_{\mathcal{Y}_{[0,\infty)}(S)\setminus \underset{j \in I_i}{\cup} S^{\sharp}_j} \cong \mathcal{P}_{i+1}|_{\mathcal{Y}_{[0,\infty)}(S)\setminus \underset{j \in I_i}{\cup} S^{\sharp}_j}
\]
with $\mathcal{P}_{k+1} =\on{Frob}_S^*\mathcal{P}_1$,  which are meromorphic along the  Cartier divisor 
$ \underset{j \in I_i}{\sum} S^{\sharp}_j $ of 
$ \mathcal{Y}_{[0,\infty)}(S)$. 
Moreover, at every geometric rank 1 point $C$ of $S$ and untilt $C^{\sharp}$ of $C$,
$\varphi_i$ is bounded by 
$\underset{
\begin{subarray}{c}
j \in I_i
\\
C^{\sharp} \in S^{\sharp}_j 
  \end{subarray}
  }{\sum} \mu_j$ 
at  $C^{\sharp}$.

When the partition is trivial, we drop the upper script to simplify the notation, i.e. 
\[
\on{Sht}_{\mu_{\bullet} } := 
\on{Sht}_{\mu_{\bullet}}^{(I)}, 
\]
which parametrizes isomorphisms of $\mathcal{G}$-torsors
\[
\mathcal{P}|_{\mathcal{Y}_{[0,\infty)}(S)\setminus \underset{i\in I}{\cup} S^{\sharp}_i} \cong \on{Frob}_S^*\mathcal{P}|_{\mathcal{Y}_{[0,\infty)}(S)\setminus \underset{i\in I}{\cup} S^{\sharp}_i}
\]
that are meromorphic and bounded by $\mu_i$ at each untilt $S_i^{\sharp}$ of $S$. 
\end{definition}

The following variant defines a correspondence between moduli spaces of shtukas, and will be very important for our purpose.

\begin{definition}
Let 
$\mu, \mu' $ and $\nu$ be conjugacy classes of cocharacters
$\mathbb{G}_m \rightarrow G_{\overline{\mathbb{Q}_p}}$. We define 
 \[
 \on{Sht}^{}_{\mu|\mu'}
 \]
 to be the stack over $\on{Spd}(\Breve{\mathbb{Z}}_p)^2$ with value at a characteristic $p$ perfectoid space $S$ being the groupoid of the data of

$\bullet$ $\mathcal{P} \in \on{Sht}_{\mu}(S)$ and 
$\mathcal{P}' \in \on{Sht}_{\mu'}(S)$,

$\bullet$ a meromorphic modification 
\[
\gamma:
\mathcal{P}|_{\mathcal{Y}_{(0,\infty)}(S)} \cong \mathcal{P}'|_{\mathcal{Y}_{(0,\infty)}(S)} 
\]
at the characteristic $p$ untilt 
$S \hookrightarrow \mathcal{Y}_{[0,\infty)}(S)$
as shtukas, i.e. the following diagram is commutative
\[
\begin{tikzcd}
\mathcal{P} \arrow[r, dashed,"\varphi_{\mathcal{P}}"] \arrow[d, dashed,"\gamma"] &
\on{Frob}_S^*\mathcal{P} \arrow[d,dashed,"\on{Frob}_S^*\gamma"] \\
\mathcal{P}' \arrow[r,dashed,"\varphi_{\mathcal{P}'}"] &
\on{Frob}_S^*\mathcal{P}'
\end{tikzcd}
\]
with the notation that dashed arrows denote generically defined modifications. 

We define
\[
\on{Sht}^{\nu}_{\mu|\mu'}
\]
to be the closed substack of
$\on{Sht}_{\mu|\mu'}$ 
parametrizing the data as $\on{Sht}_{\mu|\mu'}$ does with the additional condition that  $\gamma$ is bounded by $\nu$.

\end{definition}

Similarly, we have the version for moduli spaces of shtukas with several legs.

\begin{definition}
Let $I, J$ be finite sets, and $\mu_{\bullet}, \mu'_{\bullet}$ be collections of dominant cocharacters of $G$ indexed by $I$ and $J$ respectively. Let $\{I_1, \cdots, I_k\}$ and $\{J_1, \cdots, J_l\}$ be partitions of $I$ and $J$. Let $\nu$ be a dominant cocharacter of $G$. Then we define 
 \[
 \on{Sht}^{(I_1, \cdots, I_k)|(J_1, \cdots, J_l), \nu}_{\mu_{\bullet}|\mu_{\bullet}'}
 \]
to be the stack  with value at a characteristic $p$ perfectoid space $S$ being the groupoid of the data of

$\bullet$ 
$\{\mathcal{P}_1, \cdots, \mathcal{P}_k, \varphi_1,\cdots, \varphi_k \} \in \on{Sht}^{(I_1, \cdots, I_k)}_{\mu_{\bullet} }(S)$ 
and 
$\{\mathcal{P}_1', \cdots, \mathcal{P}_l', \varphi_1',
\cdots, \varphi_l'\} \in \on{Sht}^{(J_1, \cdots, J_l)}_{\mu'_{\bullet} }(S)$,

$\bullet$ a modification bounded by $\nu$ (in the sense similar to the above)
\[
\gamma:
\mathcal{P}_1|_{\mathcal{Y}_{(0,\infty)}(S)} \cong \mathcal{P}'_1|_{\mathcal{Y}_{(0,\infty)}(S)} 
\]
at the characteristic $p$ untilt 
$S \hookrightarrow \mathcal{Y}_{[0,\infty)}(S)$
as shtukas, i.e. the following diagram is commutative
\[
\begin{tikzcd}
 \arrow[d, dashed,"\gamma"] \mathcal{P}_1  \arrow[r,dashed, "\varphi_1"] &
\mathcal{P}_2 \arrow[r,dashed, "\varphi_2"] &
\cdots \cdots   \arrow[r,dashed,"\varphi_{k-1}" ]&
\mathcal{P}_k \arrow[r,dashed,"\varphi_k" ] &
\on{Frob}_S^*\mathcal{P}_1 \arrow[d,dashed,"\on{Frob}_S^*\gamma"] 
\\
\mathcal{P}_1'  \arrow[r,dashed, "\varphi_1'"] &
\mathcal{P}_2' \arrow[r,dashed, "\varphi_2'"] &
\cdots \cdots   \arrow[r,dashed,"\varphi'_{l-1}" ]&
\mathcal{P}'_l \arrow[r,dashed,"\varphi_l'" ] &
\on{Frob}_S^*\mathcal{P}_1'.
\end{tikzcd}
\]

When the partitions of $I$ and $J$ are trivial, we will simplify the notation by writing 
\[
\on{Sht}^{\nu}_{\mu_{\bullet}|\mu_{\bullet}'} := 
 \on{Sht}^{(I)|(J), \nu}_{\mu_{\bullet}|\mu_{\bullet}'},
\]
which parametrizes commutative diagrams
\[
\begin{tikzcd}
\mathcal{P} \arrow[r, dashed,"\varphi_{\mathcal{P}}"] \arrow[d, dashed,"\gamma"] &
\on{Frob}_S^*\mathcal{P} \arrow[d,dashed,"\on{Frob}_S^*\gamma"] 
\\
\mathcal{P}' \arrow[r,dashed,"\varphi_{\mathcal{P}'}"] &
\on{Frob}_S^*\mathcal{P}'
\end{tikzcd}
\]
where $\varphi_{\mathcal{P}}$ (resp. $\varphi_{\mathcal{P}'}$) is meromorphic and bounded by $\sum_i \mu_i$ (resp. $\sum_{j}\mu_j'$) at each untilt $S^{\sharp}$ of $S$, where $i $ (resp. $j$) ranges over $i\in I$ (resp. $j\in J$) for which  $S_i^{\sharp} = S^{\sharp}$ (resp. $S_j^{\sharp} = S^{\sharp}$). 
\end{definition}

\begin{remark}
Let $E_i$ be the reflex field of $\mu_i$, then $\on{Sht}_{\mu_{\bullet}}^{(I_1,\cdots,I_k)}$ is actually defined over $\prod \on{Spd}(\mathcal{O}_{E_i})$, and similarly for $\on{Sht}^{(I_1, \cdots, I_k)|(J_1, \cdots, J_l), \nu}_{\mu_{\bullet}|\mu_{\bullet}'}$. The same refinement works for all the spaces defined in this section. 
\end{remark}

\begin{notation} \label{nota}
We will use similar notations as in Notation \ref{notehecke} for the restriction of moduli spaces of shtukas to the generic fiber,
\[
\on{Sht}_{\mu_{\bullet},\eta_{\bullet}}^{(I_1, \cdots, I_k)} := \on{Sht}_{\mu_{\bullet} }^{(I_1, \cdots, I_k)}|_{\on{Spd}(\Breve{\mathbb{Q}}_p)^I},
\]
\[
\on{Sht}^{(I_1, \cdots, I_k)|(J_1, \cdots, J_l),\nu}_{\mu_{\bullet}|\mu'_{\bullet}, \eta_{\bullet}|\eta_{\bullet}} :=
\on{Sht}^{(I_1, \cdots, I_k)|(J_1, \cdots, J_l),\nu}_{\mu_{\bullet}|\mu'_{\bullet}}|_{ \on{Spd}(\Breve{\mathbb{Q}}_p)^I \times  \on{Spd}(\Breve{\mathbb{Q}}_p)^J},
\]
and in particular 
\[
\on{Sht}^{\nu}_{\mu|\mu', \eta_{\bullet}|\eta_{\bullet}} := \on{Sht}^{\nu}_{\mu|\mu'}|_{\on{Spd}(\Breve{\mathbb{Q}}_p) \times \on{Spd}(\Breve{\mathbb{Q}}_p)}.
\]
Note that in the last case, $\mu$ has to equal to $\mu'$. 

Similarly, we have the restriction to the special fiber
\[
\on{Sht}_{\mu_{\bullet},s_{\bullet}}^{(I_1, \cdots, I_k)} := \on{Sht}_{\mu_{\bullet} }^{(I_1, \cdots, I_k)} |_{ \on{Spd}(\overline{\mathbb{F}_p})^I},
\]
\[
\on{Sht}^{(I_1, \cdots, I_k)|(J_1, \cdots, J_l),\nu}_{\mu_{\bullet}|\mu'_{\bullet}, s_{\bullet}|s_{\bullet}} :=
\on{Sht}^{(I_1, \cdots, I_k)|(J_1, \cdots, J_l),\nu}_{\mu_{\bullet}|\mu'_{\bullet}}|_{ \on{Spd}(\overline{\mathbb{F}_p})^I \times   \on{Spd}(\overline{\mathbb{F}_p})^J}. 
\]

We also have the mixed version, for example
\[
\on{Sht}_{(\mu_{1}, \mu_2),s_{} \times \eta}^{(\{1\},\{2\})} := \on{Sht}^{(\{1\},\{2\})}_{(\mu_{1}, \mu_2) }|_{ \on{Spd}(\overline{\mathbb{F}_p}) \times \on{Spd}(\Breve{\mathbb{Q}}_p)}.
\]
\end{notation}

We have the following alternative characterization of moduli spaces of shtukas over the generic fiber.

\begin{theorem} \label{ffsht}
$\on{Sht}_{\mu, \eta}$ is equivalent to the stack over $\on{Spd}(\Breve{\mathbb{Q}}_p)$ sending a characteristic $p$ perfectoid space $S$ over $\Breve{\mathbb{Q}}_p$ to the groupoid of 

$\bullet$ 
$G$-torsors $\mathcal{E}_0$ and $\mathcal{E}$ over the Fargues--Fontaine curve 
$\mathcal{X}_{\on{FF}, S}$ 
such that $\mathcal{E}_0$ is trivial at every geometric point of $S$. The data $\mathcal{E}_0$ is equivalent to a pro-étale $G(\mathbb{Q}_p)$-torsor $\mathbb{L}_{\eta}$ on $S$ by \cite[Theorem 8.5.5]{KL},  

$\bullet$ 
an untilt
$S^{\sharp} \hookrightarrow \mathcal{Y}_{[0,\infty)}(S)$
of $S$ defined over $\Breve{\mathbb{Q}}_p$,

$\bullet$ 
an isomorphism 
\[
\alpha: 
\mathcal{E}_0|_{\mathcal{X}_{\on{FF}, S}\setminus S^{\sharp}} 
\cong
\mathcal{E}|_{\mathcal{X}_{\on{FF}, S}\setminus S^{\sharp}}
\]
that is meromorphic along $S^{\sharp}$ and bounded by $\mu$, 

$\bullet$ 
a sub $\mathcal{G}(\mathbb{Z}_p)$-torsor $\mathbb{L}$ inside $\mathbb{L}_{\eta}$.

\end{theorem}

\begin{proof}
The proof of \cite[Proposition 23.3.1]{SW17}  goes without change. Indeed, in $loc. cit.$ they fix the type of shtukas in $\on{B}(G)$, which corresponds to rigidifying $\mathcal{E}$. Dropping this extra condition gives exactly the data as in the statement.
\end{proof}

We have a natural correspondence diagram
\begin{equation} \label{etalecor}
\begin{tikzcd}
& \on{Sht}^{\nu}_{\mu|\mu, \eta} \arrow[dr, "p_2"] \arrow[dl,"p_1"'] & 
\\
\on{Sht}_{\mu,\eta}  & & \on{Sht}_{\mu,\eta}
\end{tikzcd}
\end{equation}
where $p_1$ (resp. $p_2$) sends $\{\mathcal{P}, \mathcal{P}', \gamma,S^{\sharp}\}$
to $\{\mathcal{P},S^{\sharp}\}$ (resp. $\{\mathcal{P}',S^{\sharp}\}$). 

\begin{proposition} \label{oewuriueoiiowiox}
The correspondence (\ref{etalecor}) is an étale correspondence, i.e. both $p_1$ and $p_2$ are finite étale morphisms.
\end{proposition}

\begin{proof}
With the characterization of theorem \ref{ffsht}, we see that $\on{Sht}_{\mu|\mu,\eta}^{\nu}(S)$ parametrizes 

$\bullet$ 
$\{ \mathcal{E}_0, \mathcal{E}, \alpha, \mathbb{L}, S^{\sharp} \} \in \on{Sht}_{\mu,\eta}(S)$,

$\bullet$ 
$\{ \mathcal{E}'_0, \mathcal{E}', \alpha', \mathbb{L}', S^{\sharp} \} \in \on{Sht}_{\mu,\eta}(S)$,

$\bullet$
isomorphisms 
\[
\beta_0: \mathcal{E}_0 \cong \mathcal{E}'_0
\]
\[
\beta: \mathcal{E} \cong \mathcal{E}'
\]
such that 
$\alpha' \circ \beta_0 = \beta \circ \alpha$, and 
under the isomorphism 
\[
\mathbb{L}_{\eta} \cong \mathbb{L}'_{\eta}
\]
of pro-étale $G(\mathbb{Q}_p)$-torsors on $S$ induced by $\beta_0$, the sub-$\mathcal{G}(\mathbb{Z}_p)$-torsors $\mathbb{L}$ and $\mathbb{L}'$ have relative position bounded by $\nu$ at every rank 1 geometric point of $S$. 

Now for a given $S$-point of $\on{Sht}_{\mu,\eta}$ specified by the data $\{ \mathcal{E}_0, \mathcal{E}, \alpha, \mathbb{L}, S^{\sharp} \} \in \on{Sht}_{\mu,\eta}(S)$, we see that the fiber of $p_1$ at this point is the data of 
sub-$\mathcal{G}(\mathbb{Z}_p)$-torsors $\mathbb{L}'$ 
of $\mathbb{L}_{\eta}$ which have relative position with respect to $\mathbb{L}$ bounded by $\nu$ at every geometric point of $S$. Since pro-étale locally $\mathbb{L}_{\eta}$ is the trivial torsor $S \times G(\mathbb{Q}_p)$, the relevant sub-$\mathcal{G}(\mathbb{Z}_p)$-torsors correspond to a finite subset of $G(\mathbb{Q}_p)/\mathcal{G}(\mathbb{Z}_p)$ determined by $\nu$
\footnote{It is the $\mathbb{F}_p$-points of the proper Witt vector Schubert variety with respect to $\mathcal{G}$ and $\nu$.},
which is visibly finite étale over $S$. Thus we see that $p_1$ is finite étale after base change to a pro-étale cover, hence it is finite étale by \cite[Corollary 9.11 (iii)]{2017arXiv170907343S}.  

The same argument shows that $p_2$ is finite étale. 
\end{proof}

\subsection{Moduli spaces of Witt vector shtukas and Witt vector Hecke stacks} \label{wittt}

We will need the Witt vector version of moduli spaces of shtukas and Hecke stacks as well.

\begin{definition}
Let $I$ be a finite set, and $\mu_{\bullet}$ be a collection of dominant cocharacters of $G$ indexed by $I$.  
The moduli space of Witt vector  shtukas over 
$ \on{Spd}(\overline{\mathbb{F}_p})$ associated to $G$ and $\mu_{\bullet}$, denoted by
\[
\on{Sht}_{\mu_{\bullet}}^{W},
\]
is the stackification of the prestack sending a characteristic $p$ affinoid perfectoid space $S=Spa(R,R^+)$ over $\overline{\mathbb{F}_p}$ to the groupoid of the data of

$\bullet$  $\mathcal{G}$-torsors $\mathcal{P}_1, \cdots, \mathcal{P}_k$ over $\on{Spec}(W(R^+))$, and 

$\bullet$ modifications 
\[
\begin{tikzcd}
\mathcal{P}_1  \arrow[r,dashed, "\varphi_1"] &
\mathcal{P}_2 \arrow[r,dashed, "\varphi_2"] &
\cdots \cdots   \arrow[r,dashed,"\varphi_{k-1}" ]&
\mathcal{P}_k \arrow[r,dashed,"\varphi_k" ] &
\on{Frob}_S^*\mathcal{P}_1
\end{tikzcd}
\]
where the dotted arrows $\varphi_i$ are isomorphisms
\[
\varphi_i:  \mathcal{P}_i|_{\on{Spec}(W(R^+)[\frac{1}{p}])} \cong \mathcal{P}_{i+1}|_{\on{Spec}(W(R^+)[\frac{1}{p}])}
\]
with $\mathcal{P}_{k+1} =\on{Frob}_S^*\mathcal{P}_1$,  which are bounded by $\mu_i$ at every geometric point of $S$. 
\end{definition}

We have the local version as well, which is already a stack   by \cite[Corollary 17.1.9]{SW17}. 

\begin{definition}
Let $I$ be a finite set, and $\mu_{\bullet}$ be a collection of dominant cocharacters of $G$ indexed by $I$.  
The moduli space of Witt vector local shtukas  associated to $G$ and $\mu_{\bullet}$ 
is the stack 
\[
\on{Sht}_{\mu_{\bullet}}^{\on{loc},W}
\]
over 
$ \on{Spd}(\overline{\mathbb{F}_p})$ whose value at a characteristic $p$ affinoid perfectoid space $S$ over $\overline{\mathbb{F}_p}$ is the groupoid of the data of

$\bullet$  $\mathcal{G}$-torsors $\mathcal{P}_1, \cdots, \mathcal{P}_k$ over $\on{Spec}(\mathcal{O}^{\wedge}_{\mathcal{Y}_{[0,\infty)}(S), S} )$, and 

$\bullet$ modifications 
\[
\begin{tikzcd}
\mathcal{P}_1  \arrow[r,dashed, "\varphi_1"] &
\mathcal{P}_2 \arrow[r,dashed, "\varphi_2"] &
\cdots \cdots   \arrow[r,dashed,"\varphi_{k-1}" ]&
\mathcal{P}_k \arrow[r,dashed,"\varphi_k" ] &
\on{Frob}_S^*\mathcal{P}_1
\end{tikzcd}
\]
where the dotted arrows $\varphi_i$ are isomorphisms
\[
\varphi_i:  \mathcal{P}_i|_{\on{Spec}(\mathcal{O}^{\wedge}_{\mathcal{Y}_{[0,\infty)}(S), S})\setminus S} \cong \mathcal{P}_{i+1}|_{\on{Spec}(\mathcal{O}^{\wedge}_{\mathcal{Y}_{[0,\infty)}(S), S})\setminus S}
\]
with $\mathcal{P}_{k+1} =\on{Frob}_S^*\mathcal{P}_1$, which are bounded by $\mu_i$ at every geometric point of $S$. 
\end{definition}

The possibility of having such a definition comes from the fact that Frobenius fixes the characteristic $p$ legs, which is not the case for the characteristic 0 ones.

\begin{remark}
Note that there is no extra data on the partition of $I$ since we fix legs at the characteristic $p$ untilt, so essentially the only possibility of the partition is $(\{1\}, \cdots, \{k\})$. For example, moduli spaces of Witt vector shtukas with respect to $(\{1,2 \}, \{3\} )$ and cocharacters $\{\mu_1, \mu_2, \mu_3\}$ is the same as moduli spaces of Witt vector shtukas with respect to
$(\{1 \}, \{2\} )$ and cocharacters $\{\mu_1 +\mu_2, \mu_3\}$. 
\end{remark}

Again, there is the correspondence version.

\begin{definition}
Let $I, J$ be finite sets, and $\mu_{\bullet}, \mu'_{\bullet}$ be collections of dominant cocharacters of $G$ indexed by $I$ and $J$ respectively.  Let $\nu$ be a dominant cocharacter of $G$. Then we define 
 \[
 \on{Sht}^{\nu,W}_{\mu_{\bullet}|\mu_{\bullet}'}
 \]
to be the stackification of the prestack sending   a characteristic $p$ affinoid perfectoid space $S=Spa(R,R^+)$ over $\overline{\mathbb{F}_p}$ to the groupoid of the data of

$\bullet$ 
$\{\mathcal{P}_1, \cdots, \mathcal{P}_k, \varphi_1,\cdots, \varphi_k \} \in \on{Sht}^{W}_{\mu_{\bullet} }(S)$ 
and 
$\{\mathcal{P}_1', \cdots, \mathcal{P}_l', \varphi_1',
\cdots, \varphi_l'\} \in \on{Sht}^{W}_{\mu'_{\bullet} }(S)$,

$\bullet$ a modification bounded by $\nu$ 
\[
\gamma:
\mathcal{P}_1|_{\on{Spec}(W(R^+)[\frac{1}{p}])} \cong \mathcal{P}_1'|_{\on{Spec}(W(R^+)[\frac{1}{p}])} 
\]
as shtukas, i.e. the following diagram is commutative
\[
\begin{tikzcd}
 \arrow[d, dashed,"\gamma"] \mathcal{P}_1  \arrow[r,dashed, "\varphi_1"] &
\mathcal{P}_2 \arrow[r,dashed, "\varphi_2"] &
\cdots \cdots   \arrow[r,dashed,"\varphi_{k-1}" ]&
\mathcal{P}_k \arrow[r,dashed,"\varphi_k" ] &
\on{Frob}_S^*\mathcal{P}_1 \arrow[d,dashed,"\on{Frob}_S^*\gamma"] 
\\
\mathcal{P}_1'  \arrow[r,dashed, "\varphi_1'"] &
\mathcal{P}_2' \arrow[r,dashed, "\varphi_2'"] &
\cdots \cdots   \arrow[r,dashed,"\varphi'_{l-1}" ]&
\mathcal{P}'_l \arrow[r,dashed,"\varphi_l'" ] &
\on{Frob}_S^*\mathcal{P}_1'.
\end{tikzcd}
\]
 
\end{definition}

We now define the Witt vector Hecke stacks. 

\begin{definition}
Let $I$ be a finite set, and $\mu_{\bullet}$ be a collection of dominant cocharacters of $G$ indexed by $I$.   We define 
\[
\on{Hecke}_{\mu_{\bullet}}^{W}
\]
to be the stackification   of the prestack sending $S=Spa(R,R^+)$ over $\overline{\mathbb{F}_p}$ to the groupoid of

$\bullet$ 
$\mathcal{G}$-torsors $\mathcal{P}_1, \cdots, \mathcal{P}_{k+1}$  on 
$Spec(W(R^+))$,

$\bullet$
modifications
\[
\begin{tikzcd}
\mathcal{P}_1  \arrow[r,dashed, "\varphi_1"] &
\mathcal{P}_2 \arrow[r,dashed, "\varphi_2"] &
\cdots \cdots   \arrow[r,dashed,"\varphi_{k-1}" ]&
\mathcal{P}_k \arrow[r,dashed,"\varphi_k" ] &
\mathcal{P}_{k+1},
\end{tikzcd}
\]
where the dotted arrows $\varphi_i$ are isomorphisms
\[
\varphi_i:  \mathcal{P}_i|_{\on{Spec}(W(R^+)[\frac{1}{p}])} \cong 
\mathcal{P}_{i+1}|_{\on{Spec}(W(R^+)[\frac{1}{p}])}
\]
that are  bounded by $\mu_i$ at every geometric point of $S$.

\end{definition}

\begin{remark}
   One can also define a local version of the Witt vector Hecke stack, but this coincides with the special fiber of $\on{Hecke}^{\on{loc}}_{\mu_{\bullet}}$. 
\end{remark}

\begin{definition}
Let $I, J$ be finite sets, and $\mu_{\bullet}, \mu'_{\bullet}$ be collections of dominant cocharacters of $G$ indexed by $I$ and $J$ respectively. Let $\nu$ be a dominant cocharacter of $G$. Then we define 
 \[
 \on{Hecke}^{\nu,W}_{\mu_{\bullet}|\mu_{\bullet}'}
 \]
to be the stackification  of the prestack sending a characteristic $p$ affinoid perfectoid space $S=Spa(R,R^+)$ over $\overline{\mathbb{F}_p}$ to the groupoid of the data of

$\bullet$ 
$\{\mathcal{P}_1,\cdots, \mathcal{P}_{k+1}, \varphi_1, \cdots \varphi_k \}\in \on{Hecke}_{\mu_{\bullet} }^{W}(S)$ 
and 
$\{\mathcal{P}'_1, \cdots, \mathcal{P}_{l+1}',\varphi',\cdots \varphi_l'\} \in \on{Hecke}_{\mu'_{\bullet} }^{W}(S)$,

$\bullet$ modifications bounded by $\nu$
\[
\gamma_1:
\mathcal{P}_1|_{\on{Spec}(W(R^+)[\frac{1}{p}])} \cong \mathcal{P}'_1|_{\on{Spec}(W(R^+)[\frac{1}{p}])} 
\]
\[
\gamma_2:
\mathcal{P}_{k+1}|_{\on{Spec}(W(R^+)[\frac{1}{p}])} \cong \mathcal{P}'_{l+1}|_{\on{Spec}(W(R^+)[\frac{1}{p}])} 
\]
such that the following diagram is commutative
\[
\begin{tikzcd}
 \arrow[d, dashed,"\gamma_1"] \mathcal{P}_1  \arrow[r,dashed, "\varphi_1"] &
\mathcal{P}_2 \arrow[r,dashed, "\varphi_2"] &
\cdots \cdots   \arrow[r,dashed,"\varphi_{k-1}" ]&
\mathcal{P}_k \arrow[r,dashed,"\varphi_k" ] &
\mathcal{P}_{k+1} \arrow[d,dashed,"\gamma_2"] 
\\
\mathcal{P}_1'  \arrow[r,dashed, "\varphi_1'"] &
\mathcal{P}_2' \arrow[r,dashed, "\varphi_2'"] &
\cdots \cdots   \arrow[r,dashed,"\varphi_{l-1}'" ]&
\mathcal{P}'_l \arrow[r,dashed,"\varphi_l'" ] &
\mathcal{P}'_{l+1}.
\end{tikzcd}
\]

\end{definition}

We observe that the moduli spaces of Witt vector shtukas and Hecke stacks are analytifications of perfect stacks. Let $\on{Hecke}^{\on{loc}(\infty)}_{\mu_{\bullet}}$ be the perfect Hecke stack as defined in \cite[Definition 5.1.1]{2017arXiv170705700X}, see also appendix \ref{qrerewrewrf}, then it is easy to see that
\[
\on{Hecke}^{\on{loc}(\infty),\diamond}_{\mu_{\bullet}} = \on{Hecke}^W_{\mu_{\bullet}}
\]
and 
\[
\on{Hecke}^{\on{loc}(\infty),\diamond\diamond}_{\mu_{\bullet}} = \on{Hecke}^{\on{loc},(\{1\},\cdots,\{n\})}_{\mu_{\bullet},s_{\bullet}}
\]
with notations as in \ref{serepoir}. On the other hand, we  have
\[
\on{Sht}_{\mu_{\bullet}}^{\on{loc}(\infty), \diamond} = \on{Sht}_{\mu_{\bullet}}^W
\]
and 
\[
\on{Sht}_{\mu_{\bullet}}^{\on{loc}(\infty), \diamond\diamond} = \on{Sht}_{\mu_{\bullet}}^{\on{loc},W}
\]
where 
$\on{Sht}_{\mu_{\bullet}}^{\on{loc}(\infty)}$ 
is the untruncated moduli space of Witt vector shtukas defined in \cite[Definition 5.2.1]{2017arXiv170705700X}. Similarly for the correspondence version $\on{Hecke}^{\nu,W}_{\mu_{\bullet}|\mu_{\bullet}'}$
and 
$ \on{Sht}^{\nu,W}_{\mu_{\bullet}|\mu_{\bullet}'}$. 

As a sanity check, all the stacks defined above are small v-stacks.

\begin{proposition}
$\on{Hecke}_{\mu_{\bullet}}^{(I_1,\cdots,I_k)}$, 
$\on{Hecke}_{\mu_{\bullet}}^{\on{loc},(I_1,\cdots,I_k)}$,
$ \on{Hecke}^{(I_1, \cdots, I_k)|(J_1, \cdots, J_l), \nu}_{\mu_{\bullet}|\mu_{\bullet}'}$, $\on{Hecke}^W_{\mu_{\bullet}}$, 
$\on{Hecke}^{\nu,W}_{\mu_{\bullet}|\mu_{\bullet}'}$,
$\on{Sht}_{\mu_{\bullet}}^{(I_1,\cdots,I_k)}$,
$\on{Sht}_{\mu_{\bullet}}^{\on{loc},W}$,
$\on{Sht}_{\mu_{\bullet}}^W$,
$\on{Sht}^{\nu,W}_{\mu_{\bullet}|\mu_{\bullet}'}$
and $\on{Sht}^{(I_1, \cdots, I_k)|(J_1, \cdots, J_l),\nu}_{\mu_{\bullet}|\mu'_{\bullet}} $ are small v-stack.
\end{proposition}

\begin{proof}
 See the proof of \cite[Proposition \rom{3}.1.3]{2021arXiv210213459F}.

\end{proof}

\section{The Geometric Satake} \label{satsection}

We follow the same notations as in the previous section. Moreover, we define the derived category of a small v-stack $X$ with $\overline{\mathbb{Q}}_\ell$-coefficient to be 
\[
D^b_{\text{ét}}(X,\overline{\mathbb{Q}}_\ell) := \underset{E}{\on{colim}} \ D^b_{\text{ét}}(X,E),
\]
where $E$ ranges over subfields of $\overline{\mathbb{Q}}_\ell$ that are finite extensions of $\mathbb{Q}_\ell$, $D^b_{\text{ét}}(X,E) $
is the idempotent completion of  $D^b_{\text{ét}}(X,\mathcal{O}_E) [\frac{1}{l}]$ 
and the colimit is taken in the $\infty$-categorical sense. Recall that the derived category $D^b_{\text{ét}}(X,\mathcal{O}_E)$ with adic coefficients is defined in \cite[Section 26]{2017arXiv170907343S}. The six functors formalism generalizes directly to this setting. 

\begin{remark}
The constructions in this paper work already with the coefficient ring $\mathcal{O}_{\mathbb{Q}_\ell(\sqrt{q})}$, and we can work with that instead. 
\end{remark}

\begin{remark}
The definition given here compares favorably with the constructible derived categories of schemes
\footnote{Indeed, it follows from \cite[Section 7.2]{2020arXiv201202853H} that for a qcqs scheme $X$, 
$D_{c}^b(X,\overline{\mathbb{Q}}_\ell) \cong  \underset{E}{\on{colim}} D_{c}^b(X,E)$,  
$D_{c}^b(X,E) $ is equivalent to the idempotent completion of $ D_{c}^b(X,\mathcal{O}_E)[\frac{1}{l}]$ 
and $D_{c}^b(X,\mathcal{O}_E) \cong \underset{n}{\on{lim}} D_{c}^b(X,\mathcal{O}_E/l^n)$.
}, see \cite[Section 6.8]{article}  or \cite{2020arXiv201202853H}, which suffices for our purpose in this paper. Moreover, the comparison functor from schemes to the associated diamonds is fully faithful, see \cite[Proposition 27.2]{2017arXiv170907343S} and the appendix. 
\end{remark}

Let us recall the geometric Satake proved by Fargues--Scholze.

\begin{theorem}  \label{ioiioewqw}
(Fargues--Scholze) 
Let $I $ be a finite set,  and $\mu_{\bullet}$ be a collection of  dominant geometric cocharacters of $G$ indexed by $I$.  Let $\{I_1, \cdots, I_k\}$ be a partition of $I$. There is a natural functor 
\[
U \longmapsto \mathscr{S}_U^{\on{loc},(I_1, \cdots, I_k)}
\]
from the category of representations of $\hat{G}^I$, the product of the dual group of $G$ (with $\overline{\mathbb{Q}}_\ell$-coefficient), to the category  $D_{\text{ét}}^b(\on{Hecke}^{\on{loc},(I_1, \cdots, I_k)}, \overline{\mathbb{Q}}_\ell)$. It has the following properties.

1) If $U$ is irreducible with highest weight $\mu_{\bullet}$, then $\mathscr{S}_U^{\on{loc},(I_1, \cdots, I_k)}$ is supported on $\on{Hecke}_{\mu_{\bullet}}^{\on{loc},(I_1, \cdots, I_k)}$.

2) Let 
$\zeta: I \rightarrow J$ 
be a map of finite sets, $(J_1, \cdots, J_k)$ be a partition of $J$, and $I_i = \zeta^{-1}J_i$,  then there exists a canonical isomorphism
\[
i_{\zeta,*}^{\on{loc}}p_{\zeta}^{\on{loc},*} \mathscr{S}_U^{\on{loc},(I_1, \cdots, I_k)}
\cong \mathscr{S}_{\Delta_{\zeta}^*U}^{\on{loc},(J_1, \cdots, J_k)},
\]
where 
\[
i^{\on{loc}}_{\zeta}: 
\on{Hecke}^{\on{loc},(I_1, \cdots, I_k)} \times_{\on{Spd}(\Breve{\mathbb{Z}}_p)^I, \zeta} \on{Spd}(\Breve{\mathbb{Z}}_p)^J \hookrightarrow
\on{Hecke}^{\on{loc},(J_1, \cdots, J_k)}
\]
is defined in (\ref{a4}),
\[
p_{\zeta}^{\on{loc}}:\on{Hecke}^{\on{loc},(I_1, \cdots, I_k)} \times_{\on{Spd}(\Breve{\mathbb{Z}}_p)^I, \zeta} \on{Spd}(\Breve{\mathbb{Z}}_p)^J \rightarrow  \on{Hecke}^{\on{loc},(I_1, \cdots, I_k)}
\]
is the projection on the first factor,
and
$\Delta_{\zeta}: \hat{G}^J \rightarrow \hat{G}^I$
is the map sending 
$(g_j)_{j\in J}$ to $(g_{\zeta(i)})_{i\in I}$.

3) 
There is a canonical isomorphism
\[
(\pi^{\on{loc},(I_1, \cdots, I_k)})_!  \mathscr{S}_U^{\on{loc},(I_1, \cdots, I_k)} \cong \mathscr{S}_U^{\on{loc},(I)},
\]
where 
\[
\pi^{\on{loc},(I_1, \cdots, I_k)} : \on{Hecke}^{\on{loc},(I_1, \cdots, I_k)} \longrightarrow \on{Hecke}^{\on{loc},(I)}
\]
is the convolution map defined in (\ref{a5}). 

4) Let $(J_1, \cdots, J_l)$ be a partition of $I$ such that $J_j = I_{i_j} \cup \cdots \cup I_{i_{j+1}-1}$ for some partition $1=i_1 < i_2 < \cdots < i_{l+1}= k+1$ of  $k+1$, 
then we have a canonical isomorphism
\[
(\kappa^{\on{loc},(I_1, \cdots, I_k)}_{(J_1, \cdots, J_l)})^* \underset{j}{\boxtimes} \mathscr{S}_{U_{j}}^{\on{loc},(I_{i_j}, \cdots, I_{i_{j+1}-1})} \cong
\mathscr{S}_{\boxtimes U_j}^{\on{loc},(I_1, \cdots, I_k)},
\]
where 
\[
\kappa^{\on{loc},(I_1, \cdots, I_k)}_{(J_1, \cdots, J_l)}: \on{Hecke}^{\on{loc},(I_1, \cdots, I_k)} \longrightarrow \underset{j}{\prod} \on{Hecke}^{\on{loc},(I_{i_j}, \cdots, I_{i_{j+1}-1})}
\]
is defined in (\ref{a6}), and $U_j$ is a representation of $\hat{G}^{J_j}$.

5)
We have a canonical isomorphism
\[
r^*c_X^*A_{\mu_{\bullet}} \cong \mathscr{S}_U^{\on{loc},(\{1\}, \cdots, \{n\})}
\]
compatible with the isomorphisms in 1), 2),3) and 4),  
where $r$ is the natural restriction map
\[
r: \on{Hecke}_{\mu_{\bullet}, s}^{\on{loc}, (\{1\}, \cdots, \{n\})} \longrightarrow \on{Hecke}^{\on{loc}(m),\diamond\diamond}_{\mu_{\bullet}}
\]
from the local Hecke stack to the truncated Witt vector Hecke stack
 $X=\on{Hecke}^{\on{loc}(m)}_{\mu_{\bullet}}$ 
as defined in \cite[Definition 5.1.2]{2017arXiv170705700X}  and recalled in appendix \ref{qrerewrewrf},
$U$ is irreducible with highest weight $\mu_{\bullet}$, $c_X$ is the comparison functor defined in appendix \ref{serepoir}, and $A_{\mu_{\bullet}}$ is the Satake sheaf on $\on{Hecke}^{\on{loc}(m)}_{\mu_{\bullet}}$. 
More precisely, $A_{\mu_{\bullet}}$ is the descent of the intersection sheaf $\on{IC}_{\mu_{\bullet}}$, as  defined in \cite[Section 3.4.1]{2017arXiv170705700X},
on the Witt vector affine Grassmannian $Gr_{\mu_{\bullet}}$,
which exists by the $L^+G$-equivariance of $\on{IC}_{\mu_{\bullet}}$.

\end{theorem}

\begin{proof}

It is enough to work with the coefficient of the ring of integers of $\mathbb{Q}_\ell[\sqrt{q}]$, which is treated in \cite{2021arXiv210213459F}. 
We follow the ideas in the proof of \cite[Theorem 1.17]{Lafforgue2018}.

Let us first consider the case that $U=\underset{j}{\boxtimes} U_j$, where $U_j$ is a representation of $\hat{G}^{I_j}$.
From the geometric Satake proved in \cite[Theorem \rom{6}.11.1]{2021arXiv210213459F}, there exists a canonical perverse sheaf $\mathscr{S}_V^{\on{loc},(J)}$ on $\on{Hecke}^{\on{loc},(J)}$ for each representation $V$ of $\hat{G}^J$. Let $\{ I_1, \cdots I_k\}$ be a partition of $I$,  and we define
\[
\mathscr{S}_U^{\on{loc},(I_1, \cdots, I_k)}:=( \kappa_{(I_1, \cdots, I_k)}^{\on{loc},(I_1, \cdots, I_k)})^* \underset{j}{\boxtimes} \mathscr{S}_{U_j}^{\on{loc},(I_j)}. 
\]

Let $\mathcal{R}_j$ be the algebra of regular functions on $\hat{G}^{I_j}$, which we view as an ind-(finite-dimensional) representation of $\hat{G}^{I_j}$ through the left regular representation, and we define 
\[
\mathscr{S}_{\underset{j}{\boxtimes} \mathcal{R}_j}^{\on{loc},(I_1, \cdots, I_k)}:= \on{colim }  \mathscr{S}_{\underset{j}{\boxtimes} U_j}^{\on{loc},(I_1, \cdots, I_k)},
\]
where the colimit is over all (finite-dimensional) sub-representations of $\underset{j}{\boxtimes} \mathcal{R}_j$ which are of the form
$\underset{j}{\boxtimes} U_j$ 
with $U_j$ representations of $\hat{G}^{I_j}$.

Now let $U$ be an arbitrary (finite-dimensional) representation of $\hat{G}^I$, we define 
\[
\mathscr{S}_{U}^{\on{loc},(I_1, \cdots, I_k)}:= (\mathscr{S}_{\underset{j}{\boxtimes} \mathcal{R}_j}^{\on{loc},(I_1, \cdots, I_k)} \otimes U)^{\hat{G}^{I}},
\]
where $\hat{G}^{I}$ acts diagonally on the tensor product, and 
$(-)^{\hat{G}^{I}}$ means taking $\hat{G}^{I}$-invariants. 
The action of $\hat{G}^{I}$ on
$\underset{j}{\boxtimes} \mathcal{R}_j$
is through right multiplication on $\hat{G}^{I}$, so it is an action in the category of representations of $\hat{G}^{I}$ (recall that we view $\mathcal{R}_j$ as a left regular representation), which gives rise to an action of $\hat{G}^{I}$ on 
$\mathscr{S}_{\underset{j}{\boxtimes} \mathcal{R}_j}^{\on{loc},(I_1, \cdots, I_k)} $
by the functoriality of the geometric Satake equivalence established in  \cite[Theorem \rom{6}.11.1]{2021arXiv210213459F}. 

Now it is enough to prove the claims in the case $U=\underset{j}{\boxtimes} U_j$, which we assume from now on.
Note that 4) is automatic from the definition, using the identification 
$\kappa_{(I_1, \cdots, I_k)}^{\on{loc},(I_1, \cdots, I_k)} =(\underset{j}{\prod} \kappa_{(I_{i_j}, \cdots, I_{i_{j+1}-1})}^{\on{loc},(I_{i_j}, \cdots, I_{i_{j+1}-1})} )\circ \kappa^{\on{loc},(I_1, \cdots, I_k)}_{(J_1, \cdots, J_l)}$. 

By the proof of \cite[Proposition \rom{6}.9.4]{2021arXiv210213459F}, 
$(\pi^{\on{loc},(I_1, \cdots, I_k)})_!\mathscr{S}_U^{\on{loc},(I_1, \cdots, I_k)}$
is the fusion product of $\mathscr{S}_{U_j}^{\on{loc},(I_j)}$, then \cite[Proposition \rom{6}.10.3]{2021arXiv210213459F}  tells us that the corresponding representation is $\underset{j}{\boxtimes} U_j =U$. In other words, 
\[
(\pi^{\on{loc},(I_1, \cdots, I_k)})_!\mathscr{S}_U^{\on{loc},(I_1, \cdots, I_k)} \cong 
\mathscr{S}_U^{\on{loc},(I)}, 
\]
which proves 3). 

For 1), we observe that 
\[
(\kappa_{(I_1, \cdots, I_k)}^{\on{loc},(I_1, \cdots, I_k)})^{-1} \underset{j}{\prod} \on{Hecke}_{\mu_{\bullet}}^{\on{loc},(I_j)} = \on{Hecke}_{\mu_{\bullet}}^{\on{loc},(I_1, \cdots, I_k)}.
\]
It follows from the proof of \cite[Theorem \rom{6}.11.1]{2021arXiv210213459F}  that $\mathscr{S}_U^{\on{loc},(\{1\})}$
is the intersection sheaf of $\on{Hecke}_{\mu}^{\on{loc},(\{1\})}$,
where $\mu$ is the highest weight of $U$. In particular, $\mathscr{S}_U^{\on{loc},(\{1\})}$ is supported on $\on{Hecke}_{\mu}^{\on{loc},(\{1\})}$, so the claim follows from our observation. 

To prove 2), we first prove that it holds for the trivial partition $(J)$, namely 
\begin{equation} \label{1asm,fdwe}
i_{\zeta,*}^{\on{loc}}p_{\zeta}^{\on{loc},*} \mathscr{S}_U^{\on{loc},(I)} \cong
\mathscr{S}_{\Delta_{\zeta}^*U}^{\on{loc},(J)}.
\end{equation}
We observe that both sides are natural with respect to compositions of $\zeta$. More precisely, 
let 
$\zeta_1: I \longrightarrow J$ 
and 
$\zeta_2: J \longrightarrow K $ 
be two maps between finite sets, then 
we see immediately from the definition that 
\[
i_{\zeta_2 \circ \zeta_1}^{\on{loc}} =
i^{\on{loc}}_{\zeta_2} \circ (i^{\on{loc}}_{\zeta_1} \times_{\on{Spd}(\Breve{\mathbb{Z}}_p)^J, \zeta_2} \on{Spd}(\Breve{\mathbb{Z}}_p)^K).
\]

We have the following lemma for the naturality of the left hand side of (\ref{1asm,fdwe}) with respect to $\zeta$.

\begin{lemma}
We have a canonical identification
\[
i_{\zeta_2\circ \zeta_1,*}^{\on{loc}}p_{\zeta_2\circ \zeta_1}^{\on{loc},*} \mathscr{S}_U^{\on{loc},(I)}
\cong 
i^{\on{loc}}_{\zeta_2,*}
p_{\zeta_2}^{\on{loc},*}
i^{\on{loc}}_{\zeta_1,*} 
p_{\zeta_1}^{\on{loc},*} \mathscr{S}_U^{\on{loc},(I)}. 
\]
\end{lemma}

\begin{proof}

We first note that the projection  
\[
p_{\zeta_2 \circ \zeta_1}^{\on{loc}}:\on{Hecke}^{\on{loc},(I)} \times_{\on{Spd}(\Breve{\mathbb{Z}}_p)^I, \zeta_2 \circ \zeta_1} \on{Spd}(\Breve{\mathbb{Z}}_p)^K \rightarrow  \on{Hecke}^{\on{loc},(I)}
\]
 factorizes through 
\[
\on{Hecke}^{\on{loc},(I)} \times_{\on{Spd}(\Breve{\mathbb{Z}}_p)^I, \zeta_2 \circ \zeta_1} \on{Spd}(\Breve{\mathbb{Z}}_p)^K
\overset{q}{\longrightarrow}
\on{Hecke}^{\on{loc},(I)} \times_{\on{Spd}(\Breve{\mathbb{Z}}_p)^I, \zeta_1} \on{Spd}(\Breve{\mathbb{Z}}_p)^J
\overset{p_{\zeta_1}^{\on{loc}}}{\longrightarrow} \on{Hecke}^{\on{loc},(I)},
\]
where
\[
q:
\on{Hecke}^{\on{loc},(I)} \times_{\on{Spd}(\Breve{\mathbb{Z}}_p)^I, \zeta_1} \on{Spd}(\Breve{\mathbb{Z}}_p)^J 
\times_{\on{Spd}(\Breve{\mathbb{Z}}_p)^J, \zeta_2} \on{Spd}(\Breve{\mathbb{Z}}_p)^K
\longrightarrow 
\on{Hecke}^{\on{loc},(I)} \times_{\on{Spd}(\Breve{\mathbb{Z}}_p)^I, \zeta_1} \on{Spd}(\Breve{\mathbb{Z}}_p)^J 
\]
is the projection onto the first factor. 

We have a commutative diagram 
\[
\begin{tikzcd}
\on{Hecke}^{\on{loc},(I)} \times_{\on{Spd}(\Breve{\mathbb{Z}}_p)^I, \zeta_2 \circ \zeta_1} \on{Spd}(\Breve{\mathbb{Z}}_p)^K
\arrow[d,"i^{\on{loc}}_{\zeta_1} \times_{\on{Spd}(\Breve{\mathbb{Z}}_p)^J, \zeta_2} \on{Spd}(\Breve{\mathbb{Z}}_p)^K"]
\arrow[r,"q"]
&
\on{Hecke}^{\on{loc},(I)} \times_{\on{Spd}(\Breve{\mathbb{Z}}_p)^I, \zeta_1} \on{Spd}(\Breve{\mathbb{Z}}_p)^J 
\arrow[d,"i^{\on{loc}}_{\zeta_1}"]
\arrow[r,"p_{\zeta_1}^{\on{loc}}"]
& 
\on{Hecke}^{\on{loc},(I)}
\\
\on{Hecke}^{\on{loc},(J)} \times_{\on{Spd}(\Breve{\mathbb{Z}}_p)^J, \zeta_2} \on{Spd}(\Breve{\mathbb{Z}}_p)^K
\arrow[d,"i^{\on{loc}}_{\zeta_2}"]
\arrow[r,"p_{\zeta_2}^{\on{loc}}"]
&
\on{Hecke}^{\on{loc},(J)}
&
\\
\on{Hecke}^{\on{loc},(K)}
& &
\end{tikzcd}
\]
with the middle square being Cartesian. 
Then we have
\[
i_{\zeta_2\circ \zeta_1,*}^{\on{loc}}p_{\zeta_2\circ \zeta_1}^{\on{loc},*} \mathscr{S}_U^{\on{loc},(I)}
\cong
i^{\on{loc}}_{\zeta_2,*}
(i^{\on{loc}}_{\zeta_1} \times_{\on{Spd}(\Breve{\mathbb{Z}}_p)^J, \zeta_2} \on{Spd}(\Breve{\mathbb{Z}}_p)^K)_* q^*p_{\zeta_1}^{\on{loc},*} \mathscr{S}_U^{\on{loc},(I)} 
\cong 
i^{\on{loc}}_{\zeta_2,*}
p_{\zeta_2}^{\on{loc},*}
i^{\on{loc}}_{\zeta_1,*} 
p_{\zeta_1}^{\on{loc},*} \mathscr{S}_U^{\on{loc},(I)},
\]
where the last identification is the base change isomorphism. 
\end{proof}

Therefore, if we have established the identification (\ref{1asm,fdwe}) for both $\zeta_1$ and $ \zeta_2$, then 
\[
i_{\zeta_2\circ \zeta_1,*}^{\on{loc}}p_{\zeta_2\circ \zeta_1}^{\on{loc},*} \mathscr{S}_U^{\on{loc},(I)}
\cong
i^{\on{loc}}_{\zeta_2,*}
p_{\zeta_2}^{\on{loc},*}
i^{\on{loc}}_{\zeta_1,*} 
p_{\zeta_1}^{\on{loc},*} \mathscr{S}_U^{\on{loc},(I)}
\cong
i^{\on{loc}}_{\zeta_2,*}
p_{\zeta_2}^{\on{loc},*}
\mathscr{S}_{\Delta_{ \zeta_1}^*U}^{\on{loc},(J)} 
\cong 
\mathscr{S}_{\Delta_{ \zeta_2}^*\Delta_{ \zeta_1}^*U}^{\on{loc},(J)} 
\cong 
\mathscr{S}_{\Delta_{\zeta_2 \circ  \zeta_1}^*U}^{\on{loc},(K)}, 
\]
proving  (\ref{1asm,fdwe}) for 
$\zeta_2 \circ \zeta_1$.

We know that $\zeta$ can be written as a composition of simple degenerations and injections, i.e. a simple degeneration is a surjective map collapsing only two elements, and a simple injection is an injective map adding only one element, and we have seen that it is enough to prove (\ref{1asm,fdwe}) for these two cases. Moreover, they can be written as the disjoint union of (identity map on) a finite set with either $\{1,2\} \twoheadrightarrow \{1 \}$
or $\emptyset \hookrightarrow \{1\}$.  Using \cite[Proposition \rom{6}.10.3]{2021arXiv210213459F}, we can further reduce to the case that $\zeta$ is either $\{1,2\} \twoheadrightarrow \{1 \}$
or $\emptyset \hookrightarrow \{1\}$. 

We check $\{1,2\} \twoheadrightarrow \{1 \}$ first. In this case, $i_{\zeta}^{\on{loc}}$ is an isomorphism and $p_{\zeta}^{\on{loc}}$ is the map 
\[
p^{\on{loc}}_{\zeta}: \on{Hecke}^{\on{loc},(\{1\})} \hookrightarrow \on{Hecke}^{\on{loc},(\{1,2\})}
\]
identifying $\on{Hecke}^{\on{loc},(\{1\})}$ with the substack of $\on{Hecke}^{\on{loc},(\{1,2\})}$ where the two legs are the same. We can assume that $U= U_1 \boxtimes U_2$ is decomposable, then there is a canonical identification
\[
p_{\zeta}^{\on{loc},*}\mathscr{S}_U^{\on{loc},(\{1,2\})} \cong \mathscr{S}^{\on{loc},(\{1\})}_{\Delta_{\zeta}^* U}
\]
where $\Delta_{\zeta}:  \hat{G} \hookrightarrow \hat{G}^2$ 
is the diagonal map and $U$ is a representation of $\hat{G}^2$. This is because $\Delta_{\zeta}^* U = U_1 \otimes U_2$ and 
$p_{\zeta}^*\mathscr{S}_U^{\on{loc},(\{1,2\})}$ is the convolution product of $\mathscr{S}_{U_1}^{\on{loc},(\{1\})}$
and $\mathscr{S}_{U_2}^{\on{loc},(\{2\})}$ (using $\mathscr{S}_U^{\on{loc},(\{1,2\})}$ is the fusion product of $\mathscr{S}_{U_1}^{\on{loc},(\{1\})}$
and $\mathscr{S}_{U_2}^{\on{loc},(\{2\})}$), see the paragraph following the proof of \cite[Proposition \rom{6}.9.4]{2021arXiv210213459F}. Then the identification follows from the geometric Satake (\cite[Theorem \rom{6}.11.1]{2021arXiv210213459F}). 

Next, we assume that $\zeta$ is  $\emptyset \hookrightarrow \{1\}$, then $p_{\zeta}^{\on{loc}}$ is 
\[
p^{\on{loc}}_{\zeta}: 
\on{Spd}(\Breve{\mathbb{Z}}_p) \rightarrow \ast,
\]
the unique map to the final object, 
and 
\[
i^{\on{loc}}_{\zeta}: \on{Spd}(\Breve{\mathbb{Z}}_p) \hookrightarrow \on{Hecke}^{\on{loc},(\{1\})}
\]
is the section of the structure map corresponding to the trivial modification, i.e. the modification is an isomorphism. Then 
\[
i_{\zeta,*}^{\on{loc}}p^{\on{loc},*}_{\zeta}\mathscr{S}_{1}^{\on{loc},(\emptyset)} \cong 
\mathscr{S}_1^{\on{loc},(\{1\})}
\]
where $1$ is the trivial representation (of either $\hat{G}$ or the trivial group $\hat{G}^{\emptyset}$). This follows from $\mathscr{S}_{1}^{\on{loc},(\emptyset)}$
being the constant local system, and $\mathscr{S}_1^{\on{loc},(\{1\})}$ being the constant local system supported on $\on{Hecke}^{\on{loc},(\{1\})}_0$ (by 1)). The general $U$ are direct sums of $1$, so the identification follows. 

Now for a general partition $(J_1, \cdots, J_k)$, we have a diagram with Cartesian squares
\[
\begin{tikzcd}
\on{Hecke}^{\on{loc},(I_1, \cdots, I_k)}
\arrow[d,"\kappa^{\on{loc},(I_1, \cdots, I_k)}_{(I_1, \cdots, I_k)}"]
&
\on{Hecke}^{\on{loc},(I_1, \cdots, I_k)} \times_{\on{Spd}(\Breve{\mathbb{Z}}_p)^I, \zeta} \on{Spd}(\Breve{\mathbb{Z}}_p)^J \arrow[r,hook, "i^{\on{loc}}_{\zeta}"] \arrow[d,"\kappa^{\on{loc},(I_1, \cdots, I_k)}_{(I_1, \cdots, I_k)}\times \on{id}"]
\arrow[l,"p_{\zeta}^{\on{loc}}"]
&
\on{Hecke}^{\on{loc},(J_1, \cdots, J_k)}
\arrow[d,"\kappa^{\on{loc},(J_1, \cdots, J_k)}_{(J_1, \cdots, J_k)}"]
\\
\underset{j}{\prod} \on{Hecke}^{\on{loc},(I_{j})}
&
\underset{j}{\prod} \on{Hecke}^{\on{loc},(I_{j})}\times_{\on{Spd}(\Breve{\mathbb{Z}}_p)^{I_j}, \zeta} \on{Spd}(\Breve{\mathbb{Z}}_p)^{J_j}
\arrow[r,hook, "\prod i_{\zeta}^{\on{loc}}"]
\arrow[l, "\prod p_{\zeta}^{\on{loc}}"]
 &
 \underset{j}{\prod} \on{Hecke}^{\on{loc},(J_{j})}
\end{tikzcd}
\]
and we can pullback the identification for trivial partition to obtain the desired one. 

For the comparison with Witt vector Hecke stacks. We note that when $I=\{1\}$, it follows from the paragraph before \cite[Proposition \rom{6}.7.5]{2017arXiv170907343S}  and the proof of $loc.cit.$ that under the inclusion 
$\on{Perv}(\on{Hecke}_{\mu,s}^{\on{loc},{(\{1\})}},\Lambda) \subset D_{\text{ét}}(Gr_{\mu}^{W,\diamond\diamond}, \Lambda)$ with $\Lambda = \mathcal{O}_{\mathbb{Q}_\ell(\sqrt{q})}$ (recall that $\on{Hecke}_{\mu,s}^{\on{loc},{(\{1\})}}=L^+G \setminus Gr_{\mu}^{W,\diamond\diamond}$), 
\[
c_X^{*} IC_{\mu} =\mathscr{S}_U^{\on{loc},(\{1\})} 
\]
where $X=Gr_{\mu}^W$ is the Witt vector affine Grassmannian. Now we work with the general $IC_{\mu_{\bullet
}}$, recall that   the construction of $IC_{\mu_{\bullet}}$ is by taking the twisted product of $IC_{\mu_i}$ on $Gr_{\mu_{\bullet}}=\overset{\sim}{\boxtimes} \ Gr_{\mu_i}$ ,
see \cite[Section 3.4.1]{2017arXiv170705700X}. By passing to Hecke stacks, this corresponds precisely to 
$( \kappa_{(\{1\}, \cdots, \{n\}}^{\on{loc},(\{1\}, \cdots, \{n\})})^* \underset{j}{\boxtimes} \mathscr{S}_{U_j}^{\on{loc},(\{1\})}$,
which is the definition of $\mathscr{S}_U^{\on{loc},(I_1,\cdots,I_k)}$.

\end{proof}

We also have the geometric Satake for global Hecke stacks.

\begin{theorem} \label{gsatake}
(Fargues--Scholze) Let $I $ be a finite set,  and $\mu_{\bullet}$ be a collection of  dominant geometric cocharacters of $G$ indexed by $I$.  Let $\{I_1, \cdots, I_k\}$ be a partition of $I$. There is a natural functor 
\[
U \longmapsto \mathscr{S}_U^{(I_1, \cdots, I_k)}
\]
from the category of representations of $\hat{G}^I$, the product of the dual group of $G$ (with $\overline{\mathbb{Q}}_\ell$-coefficient), to the category  $D_{\text{ét}}^b(\on{Hecke}^{(I_1, \cdots, I_k)}, \overline{\mathbb{Q}}_\ell)$. It has the following properties.

1) If $U$ is irreducible with highest weight $\mu_{\bullet}$, then $\mathscr{S}_U^{(I_1, \cdots, I_k)}$ is supported on $\on{Hecke}_{\mu_{\bullet}}^{(I_1, \cdots, I_k)}$.

2) Let 
$\zeta: I \rightarrow J$ 
be a map of finite sets, $(J_1, \cdots, J_k)$ be a partition of $J$, and $I_i = \zeta^{-1}J_i$,  then there exists a canonical isomorphism
\[
i_{\zeta,*}p_{\zeta}^* \mathscr{S}_U^{(I_1, \cdots, I_k)}
\cong \mathscr{S}_{\Delta_{\zeta}^*U}^{(J_1, \cdots, J_k)},
\]
where 
\[
i_{\zeta}: 
\on{Hecke}^{(I_1, \cdots, I_k)} \times_{\on{Spd}(\Breve{\mathbb{Z}}_p)^I, \zeta} \on{Spd}(\Breve{\mathbb{Z}}_p)^J \hookrightarrow
\on{Hecke}^{(J_1, \cdots, J_k)}
\]
is defined in (\ref{eiurwdncmn}),
\[
p_{\zeta}:\on{Hecke}^{(I_1, \cdots, I_k)} \times_{\on{Spd}(\Breve{\mathbb{Z}}_p)^I, \zeta} \on{Spd}(\Breve{\mathbb{Z}}_p)^J \rightarrow  \on{Hecke}^{(I_1, \cdots, I_k)}
\]
is the projection onto the first factor,
and
$\Delta_{\zeta}: \hat{G}^J \rightarrow \hat{G}^I$
is the map sending 
$(g_j)_{j\in J}$ to $(g_{\zeta(i)})_{i\in I}$.

3) 
There is a canonical isomorphism
\[
(\pi^{(I_1, \cdots, I_k)})_!  \mathscr{S}_U^{(I_1, \cdots, I_k)} \cong \mathscr{S}_U^{(I)},
\]
where 
\[
\pi^{(I_1, \cdots, I_k)} : \on{Hecke}^{(I_1, \cdots, I_k)} \longrightarrow \on{Hecke}^{(I)}
\]
is the convolution map defined in (\ref{7294389823}). 

4) Let $(J_1, \cdots, J_l)$ be a partition of $I$ such that $J_j = I_{i_j} \cup \cdots \cup I_{i_{j+1}-1}$ for some partition $1=i_1 < i_2 < \cdots < i_{l+1}= k+1$ of  $k+1$, 
then we have a canonical isomorphism
\[
(\kappa^{(I_1, \cdots, I_k)}_{(J_1, \cdots, J_l)})^* \underset{j}{\boxtimes} \mathscr{S}_{U_{j}}^{(I_{i_j}, \cdots, I_{i_{j+1}-1})} \cong
\mathscr{S}_{\boxtimes U_j}^{(I_1, \cdots, I_k)},
\]
where 
\[
\kappa^{(I_1, \cdots, I_k)}_{(J_1, \cdots, J_l)}: \on{Hecke}^{(I_1, \cdots, I_k)} \longrightarrow \underset{j}{\prod} \on{Hecke}^{(I_{i_j}, \cdots, I_{i_{j+1}-1})}
\]
is defined in (\ref{24939923-043-204}), and $U_j$ is a representation of $\hat{G}^{J_j}$.

5) Let $U$ be irreducible with highest weight $\mu_{\bullet}$, then the pullback of 
$\mathscr{S}_U^{(\{1\}, \cdots, \{n\})}$ 
restricted on $\on{Hecke}_{\mu_{\bullet}, s}^{(\{1\}, \cdots, \{n\})}$ 
through the canonical map
\[
\on{Hecke}^W_{\mu_{\bullet}} \longrightarrow 
\on{Hecke}_{\mu_{\bullet}, s}^{(\{1\}, \cdots, \{n\})}
\]
is identified with the analytification of $A_{\mu_{\bullet}}$ as defined in
theorem \ref{ioiioewqw} (5). 

 More precisely, as in appendix \ref{serepoir}, we have the complex of sheaves $a_X^*c_X^*A_{\mu_{\bullet}}$  on  $X=\on{Hecke}^{\on{loc}(m)}_{\mu_{\bullet}}$, which can be further viewed as a complex of sheaves  on $\on{Hecke}^W_{\mu_{\bullet}}$ through the identification in  appendix \ref{qrerewrewrf}, and this is canonically identified with the pullback of $\mathscr{S}_U^{(\{1\}, \cdots, \{n\})}$ 
along 
\[
\on{Hecke}^W_{\mu_{\bullet}} \longrightarrow 
\on{Hecke}_{\mu_{\bullet}, s}^{(\{1\}, \cdots, \{n\})}. 
\]

6) 
There is a natural restriction map 
\[
\on{Hecke}^{(I_1, \cdots, I_k)} \longrightarrow \on{Hecke}^{\on{loc},(I_1, \cdots, I_k)},
\]
along which the pullback of 
$\mathscr{S}_U^{\on{loc},(I_1,\cdots,I_k)}$
is identified with $\mathscr{S}_U^{(I_1,\cdots,I_k)}$.  Moreover, the isomorphisms in 1) to 5) are compatible under this restriction map with the corresponding isomorphisms in theorem \ref{ioiioewqw} for local Hecke stacks. 
\end{theorem}

\begin{remark}
We can state  3) for more general partitions on the target, at the cost of further complicating the notation. For example, we can consider the partial convolution map 
\[
\on{Hecke}^{(\{1\},\{2\},\{3\})} \longrightarrow \on{Hecke}^{(\{1,2\},\{3\})}
\]
sending 
\[
\begin{tikzcd}
\mathcal{P}_1  \arrow[r,dashed, "\varphi_1"] &
\mathcal{P}_2 \arrow[r,dashed, "\varphi_2"] &
\mathcal{P}_{3}
\arrow[r,dashed, "\varphi_3"] &
\mathcal{P}_{4}
\end{tikzcd}
\]
to 
\[
\begin{tikzcd}
\mathcal{P}_1  \arrow[r,dashed, "\varphi_2  \circ  \varphi_1"] &
\mathcal{P}_{3}
\arrow[r,dashed, "\varphi_3"] &
\mathcal{P}_{4}.
\end{tikzcd}
\]
The special case as stated above suffices for our purpose.  
\end{remark}

\begin{remark}
We can refine the theorem using Langlands $L$-groups, see \cite[Theorem VI.0.2]{2021arXiv210213459F}. More precisely, let $\on{Hecke}_{ \on{Spd}(\mathbb{Z}_p)^I}^{(I_1,\cdots,I_k)}$ be the Hecke stack over  $\on{Spd}(\mathbb{Z}_p)^I$, see remark \ref{ruieuivcdeieurie}. Then there is a functor from representations $U$ of the L-group $(^LG)^I$ to complexes of sheaves $\mathscr{S}_{U}^{(I_1,\cdots,I_k)}$ on $\on{Hecke}_{ \on{Spd}(\mathbb{Z}_p)^I}^{(I_1,\cdots,I_k)}$ with similar properties as in the theorem, which refines the functor in the theorem. 
\end{remark}

\begin{proof}
We define $\mathscr{S}_U^{(I_1,\cdots,I_k)}$ as the pullback of $\mathscr{S}_U^{\on{loc},(I_1,\cdots,I_k)}$
along the canonical map
\[
 \on{Hecke}^{(I_1, \cdots, I_k)} \longrightarrow \on{Hecke}^{\on{loc},(I_1, \cdots, I_k)} 
\]
from global Hecke stacks to local Hecke stacks, then (1), (2), (3), (4), and (6) follow immediately from the local version in theorem \ref{ioiioewqw}. 

For (5), we note that the composition of the canonical maps
\[
\on{Hecke}^W_{\mu_{\bullet}} \longrightarrow 
\on{Hecke}_{\mu_{\bullet}, s}^{(\{1\}, \cdots, \{n\})}
\overset{r}{\longrightarrow}
\on{Hecke}_{\mu_{\bullet}, s}^{\on{loc}, (\{1\}, \cdots, \{n\})} 
\]
is exactly the comparison map
\[
a_X: 
\on{Hecke}^{\on{loc}(\infty),\diamond}_{\mu_{\bullet}}
\longrightarrow
\on{Hecke}^{\on{loc}(\infty),\diamond\diamond}_{\mu_{\bullet}}
\]
in appendix \ref{serepoir} with $X=\on{Hecke}^{\on{loc}(\infty)}_{\mu_{\bullet}}$, and the claim follows readily from theorem \ref{ioiioewqw} (5) and the identifications in appendix \ref{qrerewrewrf}.

\end{proof}

Now we study morhpisms between the Satake sheaves  $\mathscr{S}_U^{(I_1,\cdots,I_k)}$. We aim to interpret the morphisms as cohomological correspondences, similar to \cite[Corollary 3.4.4]{2017arXiv170705700X}. The possibility of having multiple legs makes the situation more complicated. 

Let $U$ be a representation  of $\hat{G}^I$, we denote
\[
\on{Hecke}_U^{(I_1, \cdots, I_k)} := \underset{\mu_{\bullet}}{\cup} \on{Hecke}_{\mu_{\bullet}}^{(I_1, \cdots, I_k)},
\]
where the index $\mu_{\bullet}$ ranges over the highest weights of irreducible summands of $U$. Similarly for $\on{Hecke}_{U|U'}^{(I_1, \cdots, I_k)|(J_1, \cdots, J_l), 0}$.

\begin{proposition} \label{satcorr}
Let $I, J, K$ be finite sets,  and let $\{I_1, \cdots, I_k\}$ (resp. $\{J_1, \cdots, J_l\}$) be a partition of $I$  (resp.  $J$).  Let $\zeta: I \longrightarrow K$ and $\zeta': J \longrightarrow K$ be two maps. 

Let $U$ (resp. $U'$) be a representation of $\hat{G}^I$ (resp. $\hat{G}^J$), then there is a natural map
\[
\on{Hom}_{\hat{G}^K}(\Delta_{\zeta}^* U,  \Delta_{\zeta'}^* U') \longrightarrow  \on{Corr}_{
X
}
(p_{\zeta}^*\mathscr{S}_{U}^{(I_1, \cdots, I_k)}, p_{\zeta'}^* \mathscr{S}_{U'}^{(J_1, \cdots, J_l)})
\]
where $X$ is the space
\[
X:=
\on{Hecke}_{U|U'}^{(I_1, \cdots, I_k)|(J_1, \cdots, J_l), 0} \times_{\on{Spd}(\Breve{\mathbb{Z}}_p)^{I} \times  \on{Spd}(\Breve{\mathbb{Z}}_p)^{J}, \zeta \times \zeta'} \on{Spd}(\Breve{\mathbb{Z}}_p)^{K} ,
\]
and the right hand side denotes the set of cohomological correspondences from
\[
(\on{Hecke}_U^{(I_1, \cdots, I_k)} \times_{\on{Spd}(\Breve{\mathbb{Z}}_p)^{I}, \zeta} \on{Spd}(\Breve{\mathbb{Z}}_p)^{K} , p_{\zeta}^*\mathscr{S}_{U}^{(I_1, \cdots, I_k)})
\]
to 
\[
(\on{Hecke}_{U'}^{(J_1, \cdots, J_l)} \times_{\on{Spd}(\Breve{\mathbb{Z}}_p)^{J}, \zeta'} \on{Spd}(\Breve{\mathbb{Z}}_p)^{K} , p_{\zeta'}^*\mathscr{S}_{U'}^{(J_1, \cdots, J_l)})
\]
supported on 
$X$. 
\end{proposition}

Let us first explain where the correspondence 
$X$
comes from. The idea is that 
we want to  complete the correspondence diagram
\[
\begin{tikzcd} 
& \on{Hecke}_{U|U'}^{(I_1, \cdots, I_k)|(J_1, \cdots, J_l), 0}
\arrow[dr, ""] \arrow[dl,""'] & 
\\
\on{Hecke}^{(I_1, \cdots, I_k)}_U & & \on{Hecke}_{U'}^{(J_1, \cdots, J_l)}
\end{tikzcd}
\]
to a cartesian square. 

Observe that $\on{Hecke}_{U|U'}^{(I_1, \cdots, I_k)|(J_1, \cdots, J_l), 0} $
lies over 
$\on{Spd}(\Breve{\mathbb{Z}}_p)^{I} \times  \on{Spd}(\Breve{\mathbb{Z}}_p)^{J}$, 
and we want to use $\zeta$ and $\zeta'$ to pass to the same labeling set $K$, namely we base change along 
\[
\zeta \times \zeta': 
\on{Spd}(\Breve{\mathbb{Z}}_p)^{K} 
\longrightarrow
\on{Spd}(\Breve{\mathbb{Z}}_p)^{I} \times  \on{Spd}(\Breve{\mathbb{Z}}_p)^{J}
\]
to obtain the diagram
\[
\begin{tikzcd} 
X
\arrow[d, "p_1"] \arrow[r,"p_2"]  
& 
\on{Hecke}_{U'}^{(J_1, \cdots, J_l)} \times_{\on{Spd}(\Breve{\mathbb{Z}}_p)^{J}, \zeta'} \on{Spd}(\Breve{\mathbb{Z}}_p)^{K}
\\
\on{Hecke}^{(I_1, \cdots, I_k)}_{U} \times_{\on{Spd}(\Breve{\mathbb{Z}}_p)^{I}, \zeta} \on{Spd}(\Breve{\mathbb{Z}}_p)^{K}. & 
\end{tikzcd}
\]
Now there is a natural way to complete the diagram
\[
\begin{tikzcd} 
 X
\arrow[r, "p_2"] \arrow[d,"p_1"'] 
& 
\on{Hecke}_{U'}^{(J_1, \cdots, J_l)} \times_{\on{Spd}(\Breve{\mathbb{Z}}_p)^{J}, \zeta'} \on{Spd}(\Breve{\mathbb{Z}}_p)^{K}
\arrow[d, " \pi^{(J_1, \cdots, J_l)} \times \on{id}"]
\\
\on{Hecke}^{(I_1, \cdots, I_k)}_{U} \times_{\on{Spd}(\Breve{\mathbb{Z}}_p)^{I}, \zeta} \on{Spd}(\Breve{\mathbb{Z}}_p)^{K}
\arrow[d," \pi^{(I_1, \cdots, I_k)} \times \on{id}"]
& 
\on{Hecke}_{U'}^{(J)} \times_{\on{Spd}(\Breve{\mathbb{Z}}_p)^{J}, \zeta'} \on{Spd}(\Breve{\mathbb{Z}}_p)^{K}
\arrow[d,hook,"i_{\zeta'}"]
\\
\on{Hecke}_{U}^{(I)} \times_{\on{Spd}(\Breve{\mathbb{Z}}_p)^{I}, \zeta} \on{Spd}(\Breve{\mathbb{Z}}_p)^{K}
\arrow[r,hook,"i_{\zeta}"]
&
 \on{Hecke}_{}^{(K)},
\end{tikzcd}
\]
and it is easy to see that it gives rise to a Cartesian square
\begin{equation} \label{cnbfsbndas}
\begin{tikzcd} 
 X
\arrow[r, "p_2"] \arrow[d,"p_1"] 
& 
\on{Hecke}_{U'}^{(J_1, \cdots, J_l)} \times_{\on{Spd}(\Breve{\mathbb{Z}}_p)^{J}, \zeta'} \on{Spd}(\Breve{\mathbb{Z}}_p)^{K}
\arrow[d," i_{\zeta'} \circ (\pi^{(J_1, \cdots, J_l)} \times \on{id})"]
\\
\on{Hecke}^{(I_1, \cdots, I_k)}_{U} \times_{\on{Spd}(\Breve{\mathbb{Z}}_p)^{I}, \zeta} \on{Spd}(\Breve{\mathbb{Z}}_p)^{K} 
\arrow[r,hook, "i_\zeta \circ (\pi^{(I_1, \cdots, I_k)} \times \on{id})  "]
&
 \on{Hecke}_{}^{(K)}.
\end{tikzcd}
\end{equation}

\begin{proof}
We work on the derived level, and write $p^!$ and $p_!$ instead of $Rf^!$ and $Rf_!$.
From the Cartesian square
(\ref{cnbfsbndas}),
we compute
\[
\on{Hom}(p_1^*p_{\zeta}^*\mathscr{S}_U^{(I_1, \cdots, I_k)}, p_2^!p_{\zeta'}^*\mathscr{S}_{U'}^{(J_1, \cdots, J_l)}) 
= \on{Hom}(p_{2,!}p_1^*p_{\zeta}^*\mathscr{S}_U^{(I_1, \cdots, I_k)}, p_{\zeta'}^*\mathscr{S}_{U'}^{(J_1, \cdots, J_l)}) 
\]
\begin{align*}
= & \on{Hom}(p_{2,*}p_1^*p_{\zeta}^*\mathscr{S}_U^{(I_1, \cdots, I_k)}, p_{\zeta'}^*\mathscr{S}_{U'}^{(J_1, \cdots, J_l)})
\\
= &
\on{Hom}((\pi^{(J_1, \cdots, J_l)} \times \on{id})^*i^*_{\zeta'}
i_{\zeta,*} (\pi^{(I_1, \cdots, I_k)} \times \on{id})_* 
p_{\zeta}^*
\mathscr{S}_U^{(I_1, \cdots, I_k)}, p_{\zeta'}^*\mathscr{S}_{U'}^{(J_1, \cdots, J_l)})
\\
= &
\on{Hom}(i_{\zeta,*} (\pi^{(I_1, \cdots, I_k)} \times \on{id})_* 
p_{\zeta}^*
\mathscr{S}_U^{(I_1, \cdots, I_k)},
i_{\zeta',*}(\pi^{(J_1, \cdots, J_l)} \times \on{id})_*
p_{\zeta'}^*\mathscr{S}_{U'}^{(J_1, \cdots, J_l)})
\\
= &
\on{Hom}(i_{\zeta,*}
p_{\zeta}^* \pi^{(I_1, \cdots, I_k)}_*
\mathscr{S}_U^{(I_1, \cdots, I_k)},
i_{\zeta',*}
p_{\zeta'}^*\pi^{(J_1, \cdots, J_l)}_*\mathscr{S}_{U'}^{(J_1, \cdots, J_l)})
\\
= &
\on{Hom}(i_{\zeta,*}
p_{\zeta}^* 
\mathscr{S}_U^{(I)},
i_{\zeta',*}
p_{\zeta'}^*\mathscr{S}_{U'}^{(J)})
\\
= &
\on{Hom}( 
\mathscr{S}_{\Delta_{\zeta}^*U}^{(K)}, \mathscr{S}_{\Delta_{\zeta'}^* U'}^{(K)})
\end{align*}
where the second identity holds because $p_2$ is proper, the third and the fifth identity are by proper base change, the second to last identity is 3) of theorem \ref{gsatake} and the last identity is 2) of theorem \ref{gsatake}. Now  geometric Satake (\cite[Theorem \rom{6}.11.1]{2021arXiv210213459F}) tells us that 
\[
\on{Hom}_{\hat{G}^K}(\Delta_{\zeta}^* U,\Delta_{\zeta'}^*U')\cong 
\on{Hom}( 
\mathscr{S}_{\Delta_{\zeta}^*U}^{\on{loc},(K)}, \mathscr{S}_{\Delta_{\zeta'}^* U'}^{\on{loc}, (K)}),
\]
which pulls back to morphisms in 
$\on{Hom}( 
\mathscr{S}_{\Delta_{\zeta}^*U}^{(K)}, \mathscr{S}_{\Delta_{\zeta'}^* U'}^{(K)})$, giving rise to the desired map 
\[
\on{Hom}_{\hat{G}^K}(\Delta_{\zeta}^* U,\Delta_{\zeta'}^*U')
\longrightarrow 
\on{Hom}( 
\mathscr{S}_{\Delta_{\zeta}^*U}^{(K)}, \mathscr{S}_{\Delta_{\zeta'}^* U'}^{(K)}) 
=
\on{Hom}(p_1^*p_{\zeta}^*\mathscr{S}_U^{(I_1, \cdots, I_k)}, p_2^!p_{\zeta'}^*\mathscr{S}_{U'}^{(J_1, \cdots, J_l)}). 
\]
\end{proof}

We document some special cases of the proposition, which are variants of the constructions in \cite[chapter 5]{Lafforgue2018}.

\begin{example}
Let $I=J$ with the trivial partition $(I)$, $K=\{1\}$, and $\zeta=\zeta'$ be the unique map from $I$ to $K$. Let $U$ and $U'$ be irreducible with highest weights $\mu_{\bullet}$ and $\mu'_{\bullet}$, then the proposition says that
\begin{equation} \label{xmcnmadfrkj}
\on{Hom}_{\hat{G}}(U, U') :=
\on{Hom}_{\hat{G}}(\Delta^{*}_{\zeta}U,\Delta^{*}_{\zeta} U') \longrightarrow \on{Corr}_{\on{Hecke}_{U|U'}^{(I)|(I),0} \times_{\on{Spd}(\Breve{\mathbb{Z}}_p)^I \times \on{Spd}(\Breve{\mathbb{Z}}_p)^I, \zeta \times \zeta} \on{Spd}(\Breve{\mathbb{Z}}_p) }(p_{\zeta}^*\mathscr{S}_{U}^{(I)}, p_{\zeta}^*\mathscr{S}_{U'}^{(I)}).
\end{equation}
If we further restrict the legs at the characteristic $p$ untilt, then the support  becomes 
$\on{Hecke}_{\Delta^* U |\Delta^* U',s}^{(\{1\})|(\{1\}),0}$, and this is a version of \cite[Corollary 3.4.4]{2017arXiv170705700X}. 

More precisely, the $\diamond\diamond$-analytification of $loc.cit.$ is the version of (\ref{xmcnmadfrkj}) for local Hecke stacks (with fixed legs), and (\ref{xmcnmadfrkj}) is the pullback of the local version along the global to local map
$\on{Hecke}^{(I)} \longrightarrow \on{Hecke}^{\on{loc}, (I)}$. 
\end{example}

\begin{example} \label{examplecreation}
(creation correspondences)
Let $I= \{1,2\}$ with partition $(\{1\},\{2\})$, $J= \{1,2,3\}$ with partition $(\{1,2\},\{3\})$,  $K=\{1,2\}$, $\zeta= \on{id}$, and $\zeta'= \begin{cases}
\{1,3\} \rightarrow \{1\}
\\
\{2\} \rightarrow \{2\}
\end{cases}$. Let $U$ and $V$ be representations of $\hat{G}$, and we consider
\[
X=
\on{Hecke}_{1 \boxtimes U|V^* \boxtimes U \boxtimes V}^{(\{1\},\{2\})|(\{1,2\},\{3\}), 0} \times_{\on{Spd}(\Breve{\mathbb{Z}}_p)^{2} \times  \on{Spd}(\Breve{\mathbb{Z}}_p)^{3}, \zeta \times \zeta'} \on{Spd}(\Breve{\mathbb{Z}}_p)^{2}, 
\]
which parametrizes 
\[
\begin{tikzcd}
\mathcal{P}_1' \arrow[rr,dashed,"\varphi'"] \arrow[d,equal,"{\resizebox{0.9cm}{0.1cm}{$\sim$}}" labl1] & & 
\mathcal{P}_2' \arrow[d,equal, "{\resizebox{0.9cm}{0.1cm}{$\sim$}}" labl1]
\\
\mathcal{P}_1  \arrow[r,dashed,"\varphi_1"] &
\mathcal{P}_2 \arrow[r,dashed,"\varphi_2"] &
\mathcal{P}_3.
\end{tikzcd}
\]
More precisely, it parametrizes 3 legs labeled by 1,2 and 3 with the legs labeled by 1 and 2 being the same.  $\varphi'$ is a modification at the legs labeled by 1 and 2, with the modification at the leg labeled by 1 being  trivial, and the leg labeled by 2 bounded by $U$. $\varphi_1$ is a modification at the two legs labeled by 1 and 2 with the one labeled by 1  
being bounded by $V^*$ and the other  bounded by $U$, and $\varphi_2$ is a modification at the leg labeled by 3 bounded by $V$. 

The proposition then tells us that the canonical map 
\[
\sharp \boxtimes \on{id}: 
1 \boxtimes U \longrightarrow (V^* \otimes V)  \boxtimes U 
\]
gives rise to  a cohomological correspondence 
\begin{equation*} 
\mathscr{C}_{\sharp}^{++}: (\on{Hecke}_{1 \boxtimes U}^{(\{1\},\{2\})}, \mathscr{S}_{1\boxtimes U}^{(\{1\},\{2\})})
\longrightarrow 
(\on{Hecke}_{ V^* \boxtimes U \boxtimes V}^{(\{1,2\},\{3\})} \times_{\on{Spd}(\Breve{\mathbb{Z}}_p)^{3}, \zeta'} \on{Spd}(\Breve{\mathbb{Z}}_p)^{2}, p_{\zeta'}^*
\mathscr{S}_{ V^* \boxtimes U\boxtimes V}^{(\{1,2\},\{3\})})
\end{equation*}
supported on 
$X $, where 
$\sharp: 1 \rightarrow V^* \otimes V$ is the canonical map sending 1 to the identity map in $V^* \otimes V \cong \on{Hom}(V,V)$. 

In the application below, we will take a further base change along 
\[
s \times \on{id}:
\on{Spd}(\Breve{\mathbb{Z}}_p) \longrightarrow \on{Spd}(\Breve{\mathbb{Z}}_p)^2
\]
where $s: \on{Spd}(\overline{\mathbb{F}_p}) \longrightarrow \on{Spd}(\Breve{\mathbb{Z}}_p)$
is the section corresponding to the characteristic $p$ untilt. Then we obtain
\begin{equation} \label{ttytyry}
\mathscr{C}_{\sharp}^+: (\on{Hecke}_{U}^{(\{1\})}, \mathscr{S}_{ U}^{(\{1\})})
\longrightarrow 
(\on{Hecke}_{V^* \boxtimes U \boxtimes V}^{(\{1,2\},\{3\})}|_{\on{Spd}(\overline{\mathbb{F}_p})\times \on{Spd}(\Breve{\mathbb{Z}}_p) \times \on{Spd}(\overline{\mathbb{F}_p})},
\mathscr{S}_{ V^* \boxtimes U\boxtimes V}^{(\{1,2\},\{3\})}),
\end{equation}
supported on 
\[
\on{Hecke}_{1 \boxtimes U|V^* \boxtimes U \boxtimes V}^{(\{1\},\{2\})|(\{1,2\},\{3\}), 0} \times_{\on{Spd}(\Breve{\mathbb{Z}}_p)^{2} \times  \on{Spd}(\Breve{\mathbb{Z}}_p)^{3}, (\zeta \times \zeta') \circ (s \times \on{id})} \on{Spd}(\Breve{\mathbb{Z}}_p),
\]
where we abuse the notation by not distinguishing 
$\mathscr{S}_{ V^* \boxtimes U\boxtimes V}^{(\{1,2\},\{3\})}$, as a sheaf on $\on{Hecke}_{ V^* \boxtimes U \boxtimes V}^{(\{1,2\},\{3\})}$,
with its pullback to $\on{Hecke}_{ V^* \boxtimes U \boxtimes V}^{(\{1,2\},\{3\})} \times_{\on{Spd}(\Breve{\mathbb{Z}}_p)^{3}, \zeta' \circ (s \times \on{id})} \on{Spd}(\Breve{\mathbb{Z}}_p)^{}$. We hope the context will make it clear which  support the sheaf lives on. Moreover, we have identified 
$\on{Hecke}_{1 \boxtimes U}^{(\{1\},\{2\})} \times_{\on{Spd}(\Breve{\mathbb{Z}}_p)^{2}, (s \times \on{id})} \on{Spd}(\Breve{\mathbb{Z}}_p)^{}$
with
$\on{Hecke}_{U}^{(\{1\})}$, which is clear.  
\end{example}

\begin{example} \label{examannih}
(annihilation correspondences)
Let $I= \{1,2,3\}$ with partition $(\{1\},\{2,3\})$, $J= \{1,2\}$ with partition $(\{1\},\{2\})$,  $K=\{1,2\}$, $\zeta= \begin{cases}
\{1,2\} \rightarrow \{1\}
\\
\{3\} \rightarrow \{2\}
\end{cases}$, and $\zeta'= \on{id}$. Let $U$ and $V$ be representations of $\hat{G}$, and we consider 
\[
X=
\on{Hecke}_{V \boxtimes V^* \boxtimes U | 1 \boxtimes U}^{(\{1\},\{2,3\})|(\{1\},\{2\}), 0} \times_{\on{Spd}(\Breve{\mathbb{Z}}_p)^{3} \times  \on{Spd}(\Breve{\mathbb{Z}}_p)^{2}, \zeta \times \zeta'} \on{Spd}(\Breve{\mathbb{Z}}_p)^{2}.
\]
The proposition then tells us that the canonical map 
\[
\flat \boxtimes \on{id}: 
(V \otimes V^*)  \boxtimes U 
\longrightarrow 1 \boxtimes U
\]
gives rise to  a cohomological correspondence 
\begin{equation*} 
\mathscr{C}_{\flat}^{++}: (\on{Hecke}_{V \boxtimes V^* \boxtimes U }^{(\{1\},\{2,3\})} \times_{\on{Spd}(\Breve{\mathbb{Z}}_p)^{3}, \zeta } \on{Spd}(\Breve{\mathbb{Z}}_p)^{2}, p_{\zeta}^* \mathscr{S}_{V \boxtimes V^*\boxtimes U}^{(\{1\},\{2,3\})})
\longrightarrow 
(\on{Hecke}_{ 1 \boxtimes U}^{(\{1\},\{2\})}, 
\mathscr{S}_{ 1 \boxtimes U}^{(\{1\},\{2\})})
\end{equation*}
supported on 
$X $, where 
$\flat:  V \otimes V^* \rightarrow 1$ is the evaluation map.

In the application below, we will take a further base change along 
\[
s \times \on{id}:
\on{Spd}(\Breve{\mathbb{Z}}_p) \longrightarrow \on{Spd}(\Breve{\mathbb{Z}}_p)^2
\]
where $s: \on{Spd}(\overline{\mathbb{F}_p}) \longrightarrow \on{Spd}(\Breve{\mathbb{Z}}_p)$
is the section corresponding to the characteristic $p$ untilt, so we obtain
\begin{equation} \label{qpqpppq}
\mathscr{C}_{\flat}^+ : 
( \on{Hecke}_{V \boxtimes V^* \boxtimes U }^{(\{1\},\{2,3\})} |_{\on{Spd}(\overline{\mathbb{F}_p}) \times \on{Spd}(\overline{\mathbb{F}_p})\times \on{Spd}(\Breve{\mathbb{Z}}_p)},  \mathscr{S}_{ V \boxtimes V^* \boxtimes U}^{(\{1\},\{2,3\})})
\longrightarrow
(\on{Hecke}_{U}^{(\{1\})}, \mathscr{S}_{U}^{(\{1\})})
\end{equation}
supported on 
\[
\on{Hecke}_{V \boxtimes V^* \boxtimes U | 1 \boxtimes U}^{(\{1\},\{2,3\})|(\{1\},\{2\}), 0} \times_{\on{Spd}(\Breve{\mathbb{Z}}_p)^{3} \times  \on{Spd}(\Breve{\mathbb{Z}}_p)^{2}, (\zeta \times \zeta') \circ (s \times \on{id})} \on{Spd}(\Breve{\mathbb{Z}}_p)^{},
\]
where we abuse the notation by not distinguishing 
$\mathscr{S}_{ V \boxtimes V^* \boxtimes U}^{(\{1\},\{2,3\})}$, as a sheaf on 
$\on{Hecke}_{V \boxtimes V^* \boxtimes U }^{(\{1\},\{2,3\})}$,
with its pullback to 
$\on{Hecke}_{V \boxtimes V^* \boxtimes U }^{(\{1\},\{2,3\})} \times_{\on{Spd}(\Breve{\mathbb{Z}}_p)^{3}, \zeta \circ (s \times \on{id})} \on{Spd}(\Breve{\mathbb{Z}}_p)^{}$. We hope the context will make it clear which  support the sheaf lives on. Moreover, we have identified 
$\on{Hecke}_{1 \boxtimes U}^{(\{1\},\{2\})} \times_{\on{Spd}(\Breve{\mathbb{Z}}_p)^{2}, (s \times \on{id})} \on{Spd}(\Breve{\mathbb{Z}}_p)^{}$
with
$\on{Hecke}_{U}^{(\{1\})}$, which is clear.  
\end{example}

We will make use of the following simplified version of creation correspondences.

\begin{example}
Let $I= \{1\}$ with partition $(\{1\})$, $J= \{1,2\}$ with partition $(\{1\},\{2\})$,  $K=\{1\}$, $\zeta= \on{id}$, and $\zeta'$ be the unique map from $J$ to $K$. Let  $V$ be a representation of $\hat{G}$, and we consider
\[
X=
\on{Hecke}_{1 |V^* \boxtimes  V}^{(\{1\})|(\{1\},\{2\}), 0} \times_{\on{Spd}(\Breve{\mathbb{Z}}_p)^{} \times  \on{Spd}(\Breve{\mathbb{Z}}_p)^{2}, \zeta \times \zeta'} \on{Spd}(\Breve{\mathbb{Z}}_p)^{}, 
\]
which parametrizes 
\[
\begin{tikzcd}
\mathcal{P}_1' \arrow[d,equal,"{\resizebox{0.9cm}{0.1cm}{$\sim$}}" labl1] \arrow[rr,equal,"\sim"] & & 
\mathcal{P}_2' \arrow[d,equal, "{\resizebox{0.9cm}{0.1cm}{$\sim$}}" labl1]
\\
\mathcal{P}_1  \arrow[r,dashed,"\varphi_1"] &
\mathcal{P}_2 \arrow[r,dashed,"\varphi_2"] &
\mathcal{P}_3,
\end{tikzcd}
\]
where $\varphi_1$ is a modification at an untilt bounded by $V^*$, and $\varphi_1$ is a modification at the same untilt bounded by $V$.

We note that $X$ is naturally identified with the closed sub-stack
\[
i:
\on{Hecke}_{V^*\boxtimes V}^{(\{1\},\{2\}),0} \hookrightarrow \on{Hecke}_{V^*\boxtimes V}^{(\{1\},\{2\})} 
\]
of 
$\on{Hecke}_{V^*\boxtimes V}^{(\{1\},\{2\})} $
parametrizing 
\begin{equation} \label{nmcnmsqa;'}
\begin{tikzcd}
\mathcal{P}_1  \arrow[rd,dashed,"\varphi_1"] \arrow[rr,"\sim"] & & \mathcal{P}_3
\\
& 
\mathcal{P}_2  \arrow[ru,dashed,"\varphi_2"] &
\end{tikzcd}
\end{equation}
i.e. it is the closed substack of
$\on{Hecke}_{V^*\boxtimes V}^{(\{1\},\{2\})} $
defined by the condition that $\varphi_2 \circ \varphi_1$ is an isomorphism. The correspondence diagram is then 
\begin{equation} \label{lepp}
\begin{tikzcd}
& \on{Hecke}_{V^*\boxtimes V}^{(\{1\},\{2\}),0}  \arrow[dr,hook,"i" ] \arrow[dl, "\pi"'] & 
\\
\on{Hecke}_{1}^{(\{1\})}  & & 
\on{Hecke}_{V^*\boxtimes V}^{(\{1\},\{2\})},
\end{tikzcd}
\end{equation}
where $\pi$ maps (\ref{nmcnmsqa;'})  to the isomorphism $\varphi_2 \circ \varphi_1$, and $i$ is the closed immersion. 

The proposition then tells us that the canonical map 
\[
\sharp : 
1 \longrightarrow (V^* \otimes V)  
\]
gives rise to  a cohomological correspondence 
\begin{equation*} 
\mathscr{C}_{\sharp}^{++,0}: (\on{Hecke}_{1 }^{(\{1\})}, \mathscr{S}_{1}^{(\{1\})})
\longrightarrow 
(\on{Hecke}_{ V^*  \boxtimes V}^{(\{1\},\{2\})} \times_{\on{Spd}(\Breve{\mathbb{Z}}_p)^{2}, \zeta'} \on{Spd}(\Breve{\mathbb{Z}}_p)^{}, p_{\zeta'}^*
\mathscr{S}_{ V^* \boxtimes  V}^{(\{1\},\{2\})})
\end{equation*}
supported on 
$ \on{Hecke}_{V^*\boxtimes V}^{(\{1\},\{2\}),0}  $, where 
$\sharp: 1 \rightarrow V^* \otimes V$ is the canonical map sending 1 to the identity map in $V^* \otimes V \cong \on{Hom}(V,V)$. 

In the application below, we will take a further base change along 
\[
s :
\on{Spd}(\overline{\mathbb{F}_p}) \longrightarrow \on{Spd}(\Breve{\mathbb{Z}}_p)
\]
where $s$
is the section corresponding to the characteristic $p$ untilt. Then we obtain the cohomological correspondence
\begin{equation} \label{178924789}
\mathscr{C}_{\sharp}^{+,0}: (\on{Hecke}_{1,s}^{(\{1\})}, \mathscr{S}_{ 1}^{(\{1\})})
\longrightarrow 
(\on{Hecke}_{ V^*  \boxtimes V,s \times s}^{(\{1\},\{2\})},
\mathscr{S}_{ V^* \boxtimes  V}^{(\{1\},\{2\})})
\end{equation}
supported on 
\[
\on{Hecke}_{V^*\boxtimes V,s \times s}^{(\{1\},\{2\}),0},
\]
where we abuse the notation by not distinguishing 
$\mathscr{S}_{ V^* \boxtimes  V}^{(\{1\},\{2\})}$ (resp. $\mathscr{S}_{ 1}^{(\{1\})}$), as a sheaf on $\on{Hecke}_{ V^* \boxtimes  V}^{(\{1\},\{2\})}$ (resp. $ \on{Hecke}_{1}^{(\{1\})} $),
with its pullback to $\on{Hecke}_{ V^*  \boxtimes V,s \times s}^{(\{1\},\{2\})} $ (resp. $  \on{Hecke}_{1,s}^{(\{1\})} $). We hope the context will make it clear which  support the sheaf lives on.   
\end{example}

The construction in the last example can be performed on local Hecke stacks. Note that in the local case, it is important that we restrict to legs at the characteristic $p$ untilt, as otherwise the correspondence Hecke stacks $Hekce^{\on{loc}, 0}_{1|V^*\otimes V}$ need not exist, see remark \ref{u803r4u3}.

\begin{lemma} \label{lemaamamm}

(1) The canonical map
$\sharp : 
1 \longrightarrow (V^* \otimes V) $
gives rise to a cohomological correspondence
\begin{equation} \label{sureqhrjhrh}
\mathscr{C}_{\sharp}^{\on{loc},+,0}: 
(\on{Hecke}_{1,s}^{\on{loc},(\{1\})}, \mathscr{S}_1^{\on{loc},(\{1\})}) \longrightarrow 
(\on{Hecke}_{V^*\boxtimes V,s \times s}^{\on{loc},(\{1\},\{2\})}, \mathscr{S}_{V^* \boxtimes V}^{\on{loc},(\{1\},\{2\})})
\end{equation}
supported on 
\[
\begin{tikzcd}
& \on{Hecke}_{V^*\boxtimes V,s \times s}^{\on{loc},(\{1\},\{2\}),0}  \arrow[dr,hook,"i^{\on{loc}}" ] \arrow[dl, "\pi^{\on{loc}}"'] & 
\\
\on{Hecke}_{1,s}^{\on{loc},(\{1\})}  & & 
\on{Hecke}_{V^*\boxtimes V,s \times s}^{\on{loc},(\{1\},\{2\})},
\end{tikzcd}
\]
where 
$\on{Hecke}_{V^*\boxtimes V,s \times s}^{\on{loc},(\{1\},\{2\}),0} $
is the closed substack of 
$\on{Hecke}_{V^*\boxtimes V,s\times s}^{\on{loc},(\{1\},\{2\})}$
parametrizing 
\[
\begin{tikzcd}
\mathcal{P}_1  \arrow[rd,dashed,"\varphi_1"] \arrow[rr,"\sim"] & & \mathcal{P}_3
\\
& 
\mathcal{P}_2  \arrow[ru,dashed,"\varphi_2"], &
\end{tikzcd}
\]
$\pi^{\on{loc}}$ sends the data to the isomorphism $\varphi_2 \circ \varphi_1$, and $i^{\on{loc}}$ is the closed immersion.

(2) $\mathscr{C}_{\sharp}^{+,0}$
is the pullback of $\mathscr{C}_{\sharp}^{\on{loc},+,0}$
along the  diagram
\[
\begin{tikzcd}
& \on{Hecke}_{V^*\boxtimes V,s \times s}^{(\{1\},\{2\}),0}  \arrow[dr,hook,"i" ] \arrow[dl, "\pi"']  \arrow[dd]& 
\\
\on{Hecke}_{1,s}^{(\{1\})} \arrow[dd]  & & 
\on{Hecke}_{V^*\boxtimes V,s \times s}^{(\{1\},\{2\})} \arrow[dd]
\\
& \on{Hecke}_{V^*\boxtimes V,s \times s}^{\on{loc},(\{1\},\{2\}),0}  \arrow[dr,hook,"i^{\on{loc}}" ] \arrow[dl, "\pi^{\on{loc}}"'] & 
\\
\on{Hecke}_{1,s}^{\on{loc},(\{1\})}  & & 
\on{Hecke}_{V^*\boxtimes V,s \times s}^{\on{loc},(\{1\},\{2\})}.
\end{tikzcd}
\]

(3) We have a cohomological correspondence on truncated Witt vector Hecke stacks \begin{equation} \label{sureqhrjhrh}
\mathscr{C}_{\sharp}^{\on{loc}(m),+,0}: 
(\on{Hecke}_{1}^{\on{loc}(m),(\{1\})}, \mathscr{S}_1^{\on{loc}(m),(\{1\})}) \longrightarrow 
(\on{Hecke}_{V^*\boxtimes V}^{\on{loc}(m),(\{1\},\{2\})}, \mathscr{S}_{V^* \boxtimes V}^{\on{loc}(m),(\{1\},\{2\})})
\end{equation}
supported on 
\[
\begin{tikzcd}
& \on{Hecke}_{V^*\boxtimes V}^{\on{loc}(m),(\{1\},\{2\}),0}  \arrow[dr,hook,"i^{\on{loc}(m)}" ] \arrow[dl, "\pi^{\on{loc}(m)}"'] & 
\\
\on{Hecke}_{1}^{\on{loc}(m),(\{1\})}  & & 
\on{Hecke}_{V^*\boxtimes V}^{\on{loc}(m),(\{1\},\{2\})},
\end{tikzcd}
\] 
and we have  
\[
\mathscr{C}_{\sharp}^{\on{loc},+,0}=\mathscr{C}_{\sharp}^{\on{loc}(m,n),+,0,\diamond\diamond}
\]
under the identification in appendix \ref{qrerewrewrf} between sheaves on truncated and untruncated moduli spaces of shtukas, i.e. 
$\mathscr{C}_{\sharp}^{\on{loc},+,0}$
is the $\diamond\diamond$-analytification of $\mathscr{C}_{\sharp}^{\on{loc}(m,n),+,0}$, see appendix \ref{sdkkdkdkd} for the notation. 
\end{lemma}

\begin{proof}
As in the proof of proposition \ref{satcorr}, we have the Cartesian square 
\[
\begin{tikzcd}
 \on{Hecke}_{V^*\boxtimes V,s \times s}^{\on{loc},(\{1\},\{2\}),0}  \arrow[r,hook,"i^{\on{loc}}" ] \arrow[dd, "\pi^{\on{loc}}"'] & 
 \on{Hecke}_{V^*\boxtimes V,s \times s}^{\on{loc},(\{1\},\{2\})}
\arrow[d,"\pi^{(\{1\},\{2\})}"]
\\
& 
\on{Hecke}_{V^*\boxtimes V,s \times s}^{\on{loc},(\{1,2\})}
\arrow[d,"i_{\zeta}"]
\\
\on{Hecke}_{1,s}^{\on{loc},(\{1\})} 
\arrow[r,"i"]
&
\on{Hecke}_{s}^{\on{loc},(\{1\})},
\end{tikzcd}
\]
where $\zeta: \{1,2\} \longrightarrow \{1\}$ is the unique map. We can compute 
\footnote{We have used that $f_! \mathscr{S}= f_* \mathscr{S}$ with $f=i,i_{\zeta}, i^{\on{loc}}$ and $\mathscr{S}$ perverse in the computation and the rest of the paper, this needs explanation since $f$ as such is not a closed immersion between Hecke stacks, see the footnote of \cite[page 227]{2021arXiv210213459F}. However, the corresponding map on affine Grassmannians are closed immersion, and we obtain the desired equality by fully faithfulness of the pullback functor from perverse sheaves on Hecke stacks to equivariant perverse sheaves on affine Grassmannians. The author would like to thank an anonymous referee to point this out.}
\[
\on{Hom}(\pi^{\on{loc},*} \mathscr{S}_1^{\on{loc},(\{1\})},
i^{\on{loc},!}\mathscr{S}_{V \boxtimes V^*}^{\on{loc},(\{1\},\{2\})}  ) = 
    \on{Hom}(i^{\on{loc}}_!\pi^{\on{loc},*} \mathscr{S}_1^{\on{loc},(\{1\})},
\mathscr{S}_{V \boxtimes V^*}^{\on{loc},(\{1\},\{2\})}  ) 
\]
\begin{align*}
= &
\on{Hom}(i^{\on{loc}}_*\pi^{\on{loc},*} \mathscr{S}_1^{\on{loc},(\{1\})},
\mathscr{S}_{V \boxtimes V^*}^{\on{loc},(\{1\},\{2\})}  )
\\
= &
\on{Hom}(\pi^{(\{1\},\{2\}),*}i_{\zeta}^*i_* \mathscr{S}_1^{\on{loc},(\{1\})},
\mathscr{S}_{V \boxtimes V^*}^{\on{loc},(\{1\},\{2\})}  )
\\
= &
\on{Hom}(i_* \mathscr{S}_1^{\on{loc},(\{1\})},
i_{\zeta,*}
\pi^{(\{1\},\{2\})}_*
\mathscr{S}_{V \boxtimes V^*}^{\on{loc},(\{1\},\{2\})}  )
\\
= &
\on{Hom}(i_* \mathscr{S}_1^{\on{loc},(\{1\})},
i_{\zeta,*}
\mathscr{S}_{V \boxtimes V^*}^{\on{loc},(\{1,2\})}  )
\\
= &
\on{Hom}( \mathscr{S}_1^{\on{loc},(\{1\})},
i_{\zeta,*}
\mathscr{S}_{V \otimes V^*}^{\on{loc},(\{1\})}  )
\\
= &
\on{Hom}_{\hat{G}}( 1,
V\otimes V^* )
\end{align*}
where the fourth equality is the base change isomorphism, the sixth equality is 3) of theorem \ref{ioiioewqw}, the second to last equality is 1) and 2) of theorem \ref{ioiioewqw}, and the last equality is the geometric Satake (\cite[Theorem \rom{6}.11.1]{2021arXiv210213459F}). 

The construction given here is parallel to the construction in proposition \ref{satcorr}, and (2)
follows from 6) of theorem \ref{gsatake}.

Lastly, (3) follows from 5) of theorem \ref{ioiioewqw}. Indeed, it is a special case of \cite[Corollary 3.4.4]{2017arXiv170705700X}.
\end{proof}

Let us consider the correspondence
\[
X=
\on{Hecke}_{1 \boxtimes U|V^* \boxtimes U \boxtimes V}^{(\{1\},\{2\})|(\{1,2\},\{3\}), 0} \times_{\on{Spd}(\Breve{\mathbb{Z}}_p)^{2} \times  \on{Spd}(\Breve{\mathbb{Z}}_p)^{3}, (\zeta \times \zeta') \circ (s \times \on{id})} \on{Spd}(\Breve{\mathbb{Z}}_p)
\]
in example \ref{examplecreation}, we note that if we restrict to the generic fiber, we have a canonical identification
\[
X_{\eta} 
\cong 
\on{Hecke}_{V^*\boxtimes V\boxtimes U, s \times s \times \eta}^{(\{1\},\{2\},\{3\}),0},
\]
where 
\[
\on{Hecke}_{V^*\boxtimes V\boxtimes U,s \times s \times \eta}^{(\{1\},\{2\},\{3\}),0} 
\overset{i}{\hookrightarrow} 
\on{Hecke}_{V^*\boxtimes V\boxtimes U,s \times s \times \eta}^{(\{1\},\{2\},\{3\})} 
\]
is the substack parametrizing
\begin{equation} \label{nxadfa.w}
\begin{tikzcd}
\mathcal{P}_1  \arrow[rd,dashed,"\varphi_1"] \arrow[rr,"\sim"] & & \mathcal{P}_3 \arrow[r,dashed,"\varphi_3"] &
\mathcal{P}_4
\\
& 
\mathcal{P}_2  \arrow[ru,dashed,"\varphi_2"] &
 &
\end{tikzcd}
\end{equation}
with $\varphi_1$ and $\varphi_2$ bounded by (highest weights of irreducible summand of) $V^*$ and $V$ respectively, and $\varphi_3$ bounded by $\mu$. The correspondence diagram becomes
\[
\begin{tikzcd}
&
\on{Hecke}_{V^*\boxtimes V\boxtimes U,s \times s \times \eta}^{(\{1\},\{2\},\{3\}),0} 
\arrow[dr,hook, "i"] \arrow[dl,"p"] & 
\\
\on{Hecke}_{1\boxtimes U, s \times \eta}^{(\{1\},\{2\})}
&
&
\on{Hecke}_{V^*\boxtimes V\boxtimes U, s \times s \times \eta}^{(\{1\},\{2\},\{3\})}.  
\end{tikzcd}
\]
where the morphism $p$ sends (\ref{nxadfa.w}) to the first line 
\[
\begin{tikzcd}
\mathcal{P}_1   \arrow[rr,"\sim"] & & \mathcal{P}_3 \arrow[r,dashed,"\varphi_3"] &
\mathcal{P}_4.
\end{tikzcd}
\]

\begin{lemma} \label{lemmahecke}
The restriction of the cohomological correspondence 
$\mathscr{C}_{\sharp}^+$
in example \ref{examplecreation}
to the generic fiber, namely
\[
\mathscr{C}_{\sharp,\eta}^+:
(\on{Hecke}_{U,\eta}^{(\{1\})}, \mathscr{S}_{ U}^{(\{1\})})
\longrightarrow 
(\on{Hecke}_{ V^* \boxtimes U \boxtimes V,s \times s \times \eta}^{(\{1,2\},\{3\})},
\mathscr{S}_{ V^* \boxtimes U\boxtimes V}^{(\{1,2\},\{3\})})
\]
supported on $\on{Hecke}_{ V^* \boxtimes U \boxtimes V,s \times s \times \eta}^{(\{1,2\},\{3\}),0}$, 
is canonically identified with the pullback of 
\[
\mathscr{C}_{\sharp}^{+,0} \boxtimes \on{id}: 
(\on{Hecke}_{1,s}^{(\{1\})}  \times \on{Hecke}_{U,\eta}^{(\{1\})}, 
\mathscr{S}_1^{(\{1\})} \boxtimes \mathscr{S}_U^{(\{1\})} 
)
\longrightarrow
(\on{Hecke}_{V^*\boxtimes V,s \times s}^{(\{1\},\{2\})} \times \on{Hecke}_{U,\eta}^{(\{1\})},
 \mathscr{S}_{V^*\boxtimes V}^{(\{1\},\{2\})}\boxtimes \mathscr{S}_U^{(\{1\})}
)
\]
along the diagram
\[
\begin{tikzcd}
&
\on{Hecke}_{V^*\boxtimes V\boxtimes U,s \times s \times \eta}^{(\{1\},\{2\},\{3\}),0} \arrow[dr,hook,"i" ] \arrow[dl, "p"'] \arrow[dd,"\kappa^{(\{1\},\{2\},\{3\}),0}_{(\{1,2\},\{3\})}"]
& 
\\
\on{Hecke}_{1 \boxtimes U,s \times \eta}^{(\{1\},\{2\})} \arrow[dd,"\kappa^{(\{1\},\{2\})}_{(\{1\},\{2\})}"]
& & 
\on{Hecke}_{V^*\boxtimes V\boxtimes U,s \times s \times \eta}^{(\{1\},\{2\},\{3\})} \arrow[dd,"\kappa^{(\{1\},\{2\},\{3\})}_{(\{1,2\},\{3\})}"]
\\
& \on{Hecke}_{V^* \boxtimes V, s\times s}^{(\{1\},\{2\}),0}  \times \on{Hecke}_{U,\eta}^{(\{1\})} \arrow[dr,hook,"i\times \on{id}" ] \arrow[dl, "\pi \times \on{id}"'] 
& 
\\
\on{Hecke}_{1,s}^{(\{1\})}  \times \on{Hecke}_{U,\eta}^{(\{1\})}
& & 
\on{Hecke}_{V^*\boxtimes V,s \times s}^{(\{1\},\{2\})} \times \on{Hecke}_{U,\eta}^{(\{1\})}.
\end{tikzcd}
\]
\end{lemma}

\begin{proof}
 We note that the diagram can be completed into 
\begin{equation} \label{qy82367921783498}
\begin{tikzcd}
&
\on{Hecke}_{V^*\boxtimes V\boxtimes U,s \times s \times \eta}^{(\{1\},\{2\},\{3\}),0} \arrow[dr,hook,"i" ] \arrow[dl, "p"'] \arrow[dd,"\kappa^{(\{1\},\{2\},\{3\}),0}_{(\{1,2\},\{3\})}"]
& 
\\
\on{Hecke}_{1 \boxtimes U,s \times \eta}^{(\{1\},\{2\})} \arrow[dd,"\kappa^{(\{1\},\{2\})}_{(\{1\},\{2\})}"]
& & 
\on{Hecke}_{V^*\boxtimes V\boxtimes U,s \times s \times \eta}^{(\{1\},\{2\},\{3\})} \arrow[dd,"\kappa^{(\{1\},\{2\},\{3\})}_{(\{1,2\},\{3\})}"]
\\
& \on{Hecke}_{V^* \boxtimes V, s\times s}^{(\{1\},\{2\}),0}  \times \on{Hecke}_{U,\eta}^{(\{1\})} \arrow[dr,hook,"i\times \on{id}" ] \arrow[dl, "\pi \times \on{id}"'] 
& 
\\
\on{Hecke}_{1,s}^{(\{1\})}  \times \on{Hecke}_{U,\eta}^{(\{1\})}
\arrow[dr,"i' \times \on{id}"]
& & 
\on{Hecke}_{V^*\boxtimes V,s \times s}^{(\{1\},\{2\})} \times \on{Hecke}_{U,\eta}^{(\{1\})}
\arrow[dl, "\pi^{(\{1\},\{2\})} \times \on{id}"]
\\
& 
\on{Hecke}_{V^*\boxtimes V,s \times s}^{(\{1,2\})} \times \on{Hecke}_{U,\eta}^{(\{1\})}
&.
\end{tikzcd}
\end{equation}
where $i': \on{Hecke}_{1,s}^{(\{1\})} \hookrightarrow \on{Hecke}_{V^*\otimes V,s }^{(\{1\})}  \cong 
\on{Hecke}_{V^*\boxtimes V,s \times s}^{(\{1,2\})}$
is the canonical closed immersion. We note that the lower square is Cartesian, and is also the square used in the construction of $\mathscr{C}_{\sharp}^{+,0} $ ($\boxtimes \on{id}$).

On the other hand, we have the Cartesian square
\[
\begin{tikzcd}
&
\on{Hecke}_{V^*\boxtimes V\boxtimes U,s \times s \times \eta}^{(\{1\},\{2\},\{3\}),0} \arrow[dr,hook,"i" ] \arrow[dl, "p"'] 
& 
\\
\on{Hecke}_{1 \boxtimes U,s \times \eta}^{(\{1\},\{2\})} 
\arrow[dr,hook,"i''"]
& & 
\on{Hecke}_{V^*\boxtimes V\boxtimes U,s \times s \times \eta}^{(\{1\},\{2\},\{3\})} 
\arrow[dl,"\pi^{(\{1\},\{2\},\{3\})}"]
\\
&
\on{Hecke}_{V^*\boxtimes V\boxtimes U,s \times s \times \eta}^{(\{1,2,3\})}
&
\end{tikzcd}
\]
that  is used in the construction of 
$\mathscr{C}_{\sharp}^+$, 
where 
$i'':\on{Hecke}_{1 \boxtimes U,s \times \eta}^{(\{1\},\{2\})} \hookrightarrow 
\on{Hecke}_{V^*\boxtimes V\boxtimes U,s \times s \times \eta}^{(\{1,2,3\})}
$
is the map sending 
\[
\begin{tikzcd}
\mathcal{P}_1   \arrow[r,"\sim"',"\varphi_1"] & \mathcal{P}_2 \arrow[r,dashed,"\varphi_2"] &
\mathcal{P}_3
\end{tikzcd}
\]
to the composition
\[
\begin{tikzcd}
\mathcal{P}_1    \arrow[r,dashed,"\varphi_2 \circ \varphi_1"] &
\mathcal{P}_3.
\end{tikzcd}
\]

This two squares are compatible in the sense that we have a commutative cube
\begin{equation} \label{18274918903490-23}
\begin{tikzcd}[row sep=scriptsize, column sep=scriptsize]
&
\on{Hecke}_{V^*\boxtimes V\boxtimes U,s \times s \times \eta}^{(\{1\},\{2\},\{3\}),0} \arrow[dr,hook,"i" ] \arrow[dl, "p"'] 
\arrow[d,"\kappa^{(\{1\},\{2\},\{3\}),0}_{(\{1,2\},\{3\})}"]
& 
\\
\on{Hecke}_{1 \boxtimes U,s \times \eta}^{(\{1\},\{2\})} \arrow[dd,"\kappa^{(\{1\},\{2\})}_{(\{1\},\{2\})}"']
&
\on{Hecke}_{V^* \boxtimes V, s\times s}^{(\{1\},\{2\}),0}  \times \on{Hecke}_{U,\eta}^{(\{1\})} \arrow[ddr,hook,"i\times \on{id}" near start ] \arrow[ddl,  "\pi \times \on{id}"' near start]
& 
\on{Hecke}_{V^*\boxtimes V\boxtimes U,s \times s \times \eta}^{(\{1\},\{2\},\{3\})} \arrow[dd,"\kappa^{(\{1\},\{2\},\{3\})}_{(\{1,2\},\{3\})}"]
\\
&
&
\\
\on{Hecke}_{1,s}^{(\{1\})}  \times \on{Hecke}_{U,\eta}^{(\{1\})}
\arrow[ddr,"i' \times \on{id}"']
& \on{Hecke}_{V^*\boxtimes V\boxtimes U,s \times s \times \eta}^{(\{1,2,3\})}
\arrow[dd,"f"]
\arrow[from=uul,hook,crossing over,"i''"' near end]
\arrow[from=uur,crossing over,"\pi^{(\{1\},\{2\},\{3\})}" near end]
& 
\on{Hecke}_{V^*\boxtimes V,s \times s}^{(\{1\},\{2\})} \times \on{Hecke}_{U,\eta}^{(\{1\})}
\arrow[ddl, "\pi^{(\{1\},\{2\})} \times \on{id}"]
\\
& & 
\\
& 
\on{Hecke}_{V^*\boxtimes V,s \times s}^{(\{1,2\})} \times \on{Hecke}_{U,\eta}^{(\{1\})}
&.
\end{tikzcd}
\end{equation} 
where $f$ is the separating morphism
sending
\[
\begin{tikzcd}
\mathcal{P}_1 \arrow[r,dashed,"\varphi"] 
&
\mathcal{P}_2
\end{tikzcd}
\]
to 
\[
\begin{tikzcd}
\mathcal{P}_1 \arrow[r,dashed, "\varphi_1"] 
&
\mathcal{P}_3
\end{tikzcd}
\]
and 
\[
\begin{tikzcd}
\mathcal{P}_3 \arrow[r,dashed, "\varphi_2"] 
&
\mathcal{P}_2
\end{tikzcd}
\]
such that $\varphi =\varphi_2 \circ \varphi_1$, where 
$\varphi_1$ is the modification at the characteristic $p$ untilt, and $\varphi_2$ is the modification at the characteristic 0 untilt. 
Now the lemma follows readily from  4) of theorem \ref{gsatake} and the construction of $\mathscr{C}_{\sharp}^{+,0} $ and $\mathscr{C}_{\sharp}^{+} $. 
\end{proof}

\section{The Hecke Operators} \label{rieourioeu}

In this section, we introduce the Hecke operators. We follow the same notations as above. 

There is a canonical map 
\[
\epsilon_{\mu}^{(\{1\})} : \on{Sht}_{\mu, \eta} \longrightarrow \on{Hecke}_{\mu,\eta} 
\]
that sends a shtuka $\{ \mathcal{P},\varphi_{\mathcal{P}}, S^{\sharp}\}$
to 
$\{\mathcal{P}, \on{Frob}_S^*\mathcal{P}, \varphi_{\mathcal{P}},S^{\sharp}\}$,
and we can use it to pullback the canonical sheaves in geometric Satake to sheaves on moduli spaces of shtukas. 

Let $U$ be an irreducible representation of $\hat{G}$ with highest weight $\mu$. We have a  diagram
\begin{equation}
\begin{tikzcd} \label{heckdia}
& \on{Sht}^{\nu}_{\mu|\mu,\eta} \arrow[dr, "p_2"] \arrow[dl,"p_1"'] & 
\\
\on{Sht}_{\mu,\eta} \arrow[dr,"\epsilon_{\mu}^{(\{1\})}"'] & & \on{Sht}_{\mu,\eta}
\arrow[dl,"\epsilon_{\mu}^{(\{1\})}"]
\\
& \on{Hecke}_{\mu,\eta} \arrow[d,"r"] &
\\
& \on{Hecke}_{\mu,\eta}^{\on{loc}} &
\end{tikzcd}
\end{equation}
where the last vertical map is the canonical restriction map. The square in the diagram is not commutative, but they become commutative after composing with the vertical map. Indeed, 
recall from notation \ref{nota}, the legs of shtukas in the diagram are in characteristic zero, while the modification bounded by $\nu$ takes place at characteristic $p$, hence it does not affect the local type of $\varphi_{\mathcal{P}}$, i.e. the diagram is commutative after composing with the vertical map.  

Now we have a cohomological correspondence $\mathscr{C}_{\nu}$ from $(\on{Sht}_{\mu,\eta}, \epsilon_{\mu}^{(\{1\}),*}\mathscr{S}_U^{(\{1\})})$ to itself supported on $ \on{Sht}^{\nu}_{\mu|\mu,\eta}$, namely
\begin{equation} \label{1222}
\mathscr{C}_{\nu}: 
p_1^*\epsilon_{\mu}^{(\{1\}),*}\mathscr{S}_U^{(\{1\})} \cong
p_1^*\epsilon_{\mu}^{(\{1\}),*}r^*\mathscr{S}_U^{\on{loc},(\{1\})} 
\cong
p_2^* \epsilon_{\mu}^{(\{1\}),*}r^*\mathscr{S}_U^{\on{loc},(\{1\})}  \cong
p_2^*\epsilon_{\mu}^{(\{1\}),*}\mathscr{S}_U^{(\{1\})}
\cong 
p_2^!\epsilon_{\mu}^{(\{1\}),*}\mathscr{S}_U^{(\{1\})},
\end{equation}
where we use $r^*\mathscr{S}_U^{\on{loc},(\{1\})} =\mathscr{S}_U^{(\{1\})}$ (see 6) of theorem \ref{gsatake}), and the étaleness of $p_2$ (proposition \ref{oewuriueoiiowiox}) for the last isomorphism. 

Now let $V$ be a representation of $\hat{G}$, and
\[
h_V:= \on{Sat}^{\on{cl}} ([V]) \in C_c(G(\Breve{\mathbb{Q}}_p)//\mathcal{G}(\Breve{\mathbb{Z}}_p), \overline{\mathbb{Q}}_\ell)
\]
be the function corresponding to $V$ through the classical Satake equivalence 
\[
\on{Sat}^{\on{cl}}: R(\hat{G}) \longrightarrow C_c(G(\Breve{\mathbb{Q}}_p)//\mathcal{G}(\Breve{\mathbb{Z}}_p), \overline{\mathbb{Q}}_\ell)
\]
as defined in \cite[Section 3.5.2]{2017arXiv170705700X}, where $R(\hat{G}) := K_0(\on{Rep}(\hat{G}))$ is the representation ring of $\hat{G}$, and $C_c(G(\Breve{\mathbb{Q}}_p)//\mathcal{G}(\Breve{\mathbb{Z}}_p), \overline{\mathbb{Q}}_\ell) $ is the space of compactly supported, i.e. supported on a finite subset, $\overline{\mathbb{Q}}_\ell$-valued functions on the double quotient set $\mathcal{G}(\Breve{\mathbb{Z}}_p)\setminus G(\Breve{\mathbb{Q}}_p)/\mathcal{G}(\Breve{\mathbb{Z}}_p) $. Note that we use that $G$ is unramified here, i.e. $G_{\Breve{\mathbb{Q}}_p}$ is split, so the Langlands dual group of $G_{\Breve{\mathbb{Q}}_p}$ can be reduced to the dual group $\hat{G}$. We know that   
$C_c(G(\Breve{\mathbb{Q}}_p)//\mathcal{G}(\Breve{\mathbb{Z}}_p), \overline{\mathbb{Q}}_\ell)$
has a basis indexed by dominant cocharacters $\nu$ ($\mathcal{G}$ is  reductive over $\mathbb{Z}_p$), i.e. 
$1_{\mathcal{G}(\Breve{\mathbb{Z}}_p)\nu(p)\mathcal{G}(\Breve{\mathbb{Z}}_p)}$'s 
form a basis of 
$C_c(G(\Breve{\mathbb{Q}}_p)//\mathcal{G}(\Breve{\mathbb{Z}}_p), \overline{\mathbb{Q}}_\ell)$, so $C_c(G(\Breve{\mathbb{Q}}_p)//\mathcal{G}(\Breve{\mathbb{Z}}_p), \overline{\mathbb{Q}}_\ell) = C_c(G(K)//\mathcal{G}(\mathcal{O}_K), \overline{\mathbb{Q}}_\ell)$ with $K$ a finite unramified extension of $\mathbb{Q}_p$ over which $G$ splits.
Let $\Gamma_V$ be the union of $\on{Sht}^{\nu}_{\mu|\mu,\eta}$ 
(as closed subspaces of $\on{Sht}^{}_{\mu|\mu,\eta}$)
with $\nu$ appearing in $h_V$. Then we define the Hecke operator $T_V$ as the cohomological correspondence supported on $\Gamma_V$, and is $h_V(\nu(p))$ times the cohomological correspondence (\ref{1222}) on $\on{Sht}^{\nu}_{\mu|\mu,\eta}$. More precisely, we define
\[
T_V:= \underset{h_V(\nu(p))\neq 0}{\sum} h_V(\nu(p)) \cdot \mathscr{C}_{\nu} 
\]
as a cohomological correspondence from 
$(\on{Sht}_{\mu,\eta}, \epsilon_{\mu}^*\mathscr{S}_U)$
to itself supported on 
\[
\Gamma_V:= \underset{ h_V(\nu(p))\neq 0}{\cup} \on{Sht}^{\nu}_{\mu|\mu,\eta}.
\]
Note that when $V$ is irreducible with highest weight $\delta$, then 
$h_V(\nu(p))\neq 0$ 
precisely when $\nu \leq \delta$, hence $\Gamma_V= \on{Sht}^{\nu}_{\mu | \mu,\eta}$. Since $\on{Sht}^{\nu}_{\mu | \mu,\eta} \subset \on{Sht}^{\delta}_{\mu | \mu,\eta}$ as a closed substack, we can view $\mathscr{C}_{\nu}$ as being supported on $\on{Sht}^{\delta}_{\mu | \mu,\eta}$, whence the sum in the definition of $T_V$ makes sense. In general $\Gamma_V$ is the union of $\on{Sht}^{\delta}_{\mu | \mu,\eta}$ where $\delta$ ranges over highest weights of irreducible summands of $V$. 

We observe that when $\mu$ is minuscule, $\mathscr{S}_U = \overline{\mathbb{Q}}_\ell[d]$ is the shifted constant sheaf with $d$ being the dimension of the Schubert variety with respect to $\mu$, and the cohomological correspondence $\mathscr{C}_{\nu}$ is the canonical isomorphism 
\[
p_1^*\overline{\mathbb{Q}}_\ell[d] \cong \overline{\mathbb{Q}}_\ell[d] \cong p_2^*\overline{\mathbb{Q}}_\ell[d] \cong p_2^!\overline{\mathbb{Q}}_\ell[d].
\]

We assume from now on that $V$ is a representation of $^LG$, which, by the proof of \cite[Lemma A.3]{Zhurami}, is equivalent to a representation of $\hat{G}$ that we still denote by $V$, together with an isomorphism to its Frobenius twist $^{\sigma} V$, where $^{\sigma} V$ is the representation of $\hat{G}$ obtained by twisting the action of $G$ on $V$ with the inverse Frobenius action on $\hat{G}$, see \cite[Section 3.4.5]{2017arXiv170705700X}. Then the (unramified) classical Satake equivalence 
\[
\on{Sat}^{\on{cl}}: R(^L G) \longrightarrow C_c(G(\mathbb{Q}_p)//\mathcal{G}(\mathbb{Z}_p)), \overline{\mathbb{Q}}_\ell)
\]
defined in \cite[Section 3.5.2]{2017arXiv170705700X} 
provides us with $h_V:= \on{Sat}^{\on{cl}} ([V]) \in C_c(G(\mathbb{Q}_p)//\mathcal{G}(\mathbb{Z}_p)), \overline{\mathbb{Q}}_\ell)$.

\section{The excursion operators}  \label{e888888}

We now start to define the excursion operators. We follow the same notations as in section \ref{rieourioeu}. In particular, let $U$ be an irreducible representation of $\hat{G}$ with highest weight $\mu$, and $V$ be a representation of $^LG$, viewed as a representation of $\hat{G}$ together with an isomorphism $^{\sigma}V \cong V$ with its Frobenius twist $^{\sigma}V$.

We have a canonical map 
\[
\epsilon_{\mu_{\bullet}}^{(I_1, \cdots, I_k)} : \on{Sht}_{\mu_{\bullet} }^{(I_1, \cdots, I_k)}
\longrightarrow 
\on{Hecke}_{\mu_{\bullet}}^{(I_1, \cdots, I_k)}
\]
which is simply viewing $\on{Frob}_S^*\mathcal{P}_1$ as $\mathcal{P}_{k+1}$. Similarly we have 
\[
\epsilon_{\mu_{\bullet}|\mu_{\bullet}'}^{(I_1, \cdots, I_k)| (J_1, \cdots, J_l)} :
 \on{Sht}^{(I_1, \cdots, I_k)| (J_1, \cdots, J_l),0}_{\mu_{\bullet}|\mu_{\bullet}'} 
\longrightarrow 
\on{Hecke}^{(I_1, \cdots, I_k)| (J_1, \cdots, J_l),0}_{\mu_{\bullet}|\mu_{\bullet}'},
\]
which forms a commutative diagram
\[
\begin{tikzcd}
&
 \on{Sht}^{(I_1, \cdots, I_k)| (J_1, \cdots, J_l),0}_{\mu_{\bullet}|\mu_{\bullet}'}  \arrow[dr ] \arrow[dl] \arrow[dd,"\epsilon_{\mu_{\bullet}|\mu_{\bullet}'}^{(I_1, \cdots, I_k)| (J_1, \cdots, J_l)}"]
& 
\\
\on{Sht}_{\mu_{\bullet} }^{(I_1, \cdots, I_k)} \arrow[dd,"\epsilon_{\mu_{\bullet}}^{(I)} "]
& & 
\on{Sht}_{\mu'_{\bullet}}^{(J_1, \cdots, J_l)} \arrow[dd,"\epsilon_{\mu'_{\bullet}}^{(I)}"]
\\
& \on{Hecke}^{(I_1, \cdots, I_k)| (J_1, \cdots, J_l),0}_{\mu_{\bullet}|\mu_{\bullet}'}  \arrow[dr ] \arrow[dl] 
& 
\\
 \on{Hecke}_{\mu_{\bullet}}^{(I_1, \cdots, I_k)}
& & 
 \on{Hecke}_{\mu'_{\bullet}}^{(J_1, \cdots, J_l)}.
\end{tikzcd}
\]
with the two squares Cartesian. 
 
Let $W$ be a representation  of $\hat{G}^I$, we denote
\[
\on{Sht}_W^{(I_1, \cdots, I_k)} := \underset{\mu_{\bullet}}{\cup} \on{Sht}_{\mu_{\bullet}}^{(I_1, \cdots, I_k)},
\]
where the index $\mu_{\bullet}$ ranges over the highest weights of irreducible summands of $W$. Similarly for $\on{Sht}_{W|W'}^{(I_1, \cdots, I_k)|(J_1, \cdots, J_l), 0}$, and the Witt vector version.

\subsection{The creation correspondences}

First we define the creation correspondences.

Let $\zeta: \{1\} \rightarrow \{1,2,3\} $
be the map defined by $\zeta(1)=2$, we define
\begin{equation} \label{geryuey}
\on{Sht}_{U|V^* \boxtimes U \boxtimes V, -|s \times - \times s}^{(\{1\})|(\{1,2\},\{3\}),0} \subset \on{Sht}^{(\{1\})|(\{1,2\},\{3\}), 0}_{U|V^* \boxtimes U \boxtimes V} \times_{\on{Spd}(\Breve{\mathbb{Z}}_p) \times \on{Spd}(\Breve{\mathbb{Z}}_p)^3, \zeta \times \on{id}} \on{Spd}(\Breve{\mathbb{Z}}_p)^3
\end{equation}
as the closed v-substack parametrizing modifications
\[
\begin{tikzcd}
\mathcal{P} \arrow[rr,dashed] \arrow[d,equal,"{\resizebox{0.9cm}{0.1cm}{$\sim$}}" labl1] & & 
\on{Frob}_S^*\mathcal{P} \arrow[d,equal,"{\resizebox{0.9cm}{0.1cm}{$\sim$}}" labl1]
\\
\mathcal{P}_1  \arrow[r,dashed] &
\mathcal{P}_2 \arrow[r,dashed] &
\on{Frob}_S^*\mathcal{P}_1
\end{tikzcd}
\]
with the first line in $\on{Sht}_{U}$ and the second line in 
\[
\on{Sht}_{V^* \boxtimes U \boxtimes V, s \times - \times s}^{(\{1,2\},\{3\})} :=
\on{Sht}_{V^* \boxtimes U \boxtimes V}^{(\{1,2\},\{3\})}|_{\on{Spd}(\overline{\mathbb{F}_p})\times \on{Spd}(\Breve{\mathbb{Z}}_p) \times \on{Spd}(\overline{\mathbb{F}_p})}.
\]
Moreover, the legs indexed by $U$ in both the upper and lower modifications are the same. 

Similarly, we define 
\[
\on{Hecke}^{(\{1\})|(\{1,2\},\{3\}),0}_{U| V^* \boxtimes U \boxtimes V,-|s \times - \times s} 
\subset 
 \on{Hecke}^{(\{1\})|(\{1,2\},\{3\}), 0}_{U| V^* \boxtimes U \boxtimes V} \times_{\on{Spd}(\Breve{\mathbb{Z}}_p) \times \on{Spd}(\Breve{\mathbb{Z}}_p)^3, \zeta \times \on{id}} \on{Spd}(\Breve{\mathbb{Z}}_p)^3
\]
as the v-substack parametrizing modifications
\[
\begin{tikzcd}
\mathcal{P}_1' \arrow[rr,dashed] \arrow[d,equal,"{\resizebox{0.9cm}{0.1cm}{$\sim$}}" labl1] & & 
\mathcal{P}'_2 \arrow[d,equal,"{\resizebox{0.9cm}{0.1cm}{$\sim$}}" labl1]
\\
\mathcal{P}_1  \arrow[r,dashed] &
\mathcal{P}_2 \arrow[r,dashed] &
\mathcal{P}_3
\end{tikzcd}
\]
with the first line in $\on{Hecke}_{U}^{(\{1\})}$ and the second line in 
\[
\on{Hecke}_{V^* \boxtimes U \boxtimes V,s \times - \times s}^{(\{1,2\},\{3\})}: = 
\on{Hecke}_{V^* \boxtimes U \boxtimes V}^{(\{1,2\},\{3\})}|_{\on{Spd}(\overline{\mathbb{F}_p})\times \on{Spd}(\Breve{\mathbb{Z}}_p) \times \on{Spd}(\overline{\mathbb{F}_p})}.
\]
Moreover, the legs indexed by $U$ in both the upper and lower modifications are the same.

We view 
$\on{Hecke}^{(\{1\})|(\{1,2\},\{3\}),0}_{U| V^* \boxtimes U \boxtimes V,-|s \times - \times s}$, 
$\on{Sht}_{\mu|V^* \boxtimes U \boxtimes V,s \times - \times s}^{(\{1\})|(\{1,2\},\{3\}),0}$
and 
$\on{Hecke}_{V^*\boxtimes U \boxtimes V, s \times - \times s }^{(\{1,2\},\{3\})}$
as stacks over $\on{Spd}(\Breve{\mathbb{Z}}_p)$, and
we have a commutative diagram over $\on{Spd}(\Breve{\mathbb{Z}}_p)$
\begin{equation} \label{asssseee}
\begin{tikzcd}
&
\on{Sht}_{U|V^* \boxtimes U \boxtimes V, -|s \times - \times s}^{(\{1\})|(\{1,2\},\{3\}),0}
\arrow[dr ] \arrow[dl] \arrow[dd,"\epsilon_{U| V^* \boxtimes U\boxtimes V}^{(\{1\})|(\{1,2\},\{3\})}"]
& 
\\
\on{Sht}_{U}^{(\{1\})}
\arrow[dd,"\epsilon_{U}^{(\{1\})} "]
& & 
\on{Sht}_{V^* \boxtimes U \boxtimes V,s \times - \times s}^{(\{1,2\},\{3\})}
\arrow[dd,"\epsilon_{ V^* \boxtimes U \boxtimes V}^{(\{1,2\},\{3\})}"]
\\
& \on{Hecke}^{(\{1\})|(\{1,2\},\{3\}),0}_{U| V^* \boxtimes U \boxtimes V, -|s \times - \times s}  \arrow[dr ] \arrow[dl] 
& 
\\
 \on{Hecke}_{U}^{(\{1\})}
& & 
 \on{Hecke}_{V^* \boxtimes U \boxtimes V, s \times - \times s}^{(\{1,2\},\{3\})}.
\end{tikzcd}
\end{equation}

Now since the squares in diagram (\ref{asssseee}) are Cartesian, we can pullback the cohomological correspondence 
\begin{equation*} 
\mathscr{C}_{\sharp}^+: (\on{Hecke}_{U}^{(\{1\})}, \mathscr{S}_{ U}^{(\{1\})})
\longrightarrow 
(\on{Hecke}_{ V^* \boxtimes U \boxtimes V, s \times - \times s}^{(\{1,2\},\{3\})},
\mathscr{S}_{ V^* \boxtimes U\boxtimes V}^{(\{1,2\},\{3\})})
\end{equation*}
supported on 
$\on{Hecke}^{(\{1\})|(\{1,2\},\{3\}),0}_{U|V^* \boxtimes U \boxtimes V} $
defined in example \ref{examplecreation}  (\ref{ttytyry})
to obtain the desired creation correspondence 
\begin{equation} \label{lleleeell}
\mathscr{C}_{\sharp} : (\on{Sht}_{U}^{(\{1\})}, \epsilon_{U}^{(\{1\}),*} \mathscr{S}_{U}^{(\{1\})}) \longrightarrow 
(\on{Sht}_{ V^* \boxtimes U \boxtimes V,s \times - \times s}^{(\{1,2\},\{3\})}, \epsilon_{ V^* \boxtimes U\boxtimes V}^{(\{1,2\},\{3\}),*} \mathscr{S}_{V^* \boxtimes U \boxtimes V}^{(\{1,2\},\{3\})})
\end{equation}
supported on 
$\on{Sht}_{U|V^*\boxtimes U \boxtimes V, -| s \times - \times s}^{(\{1\})|(\{1,2\},\{3\}),0}$.

\begin{remark}
The reason we use $\on{Sht}_{V^* \boxtimes U \boxtimes V, s \times - \times s}^{(\{1,2\},\{3\})}$
instead of $\on{Sht}_{ V^* \boxtimes U \boxtimes V, s \times - \times s}^{(\{1\},\{2\},\{3\})}$
is explained in the paragraph before \cite[Remark 6.1.10]{2017arXiv170705700X}, namely, we want the corresponding (integral) excursion operators to have  the same support as the Hecke operators, so we can pullback the cohomological correspondences to Shimura varieties.  
\end{remark}

\subsubsection{The generic fiber}
We now want to identify the generic fiber of the creation correspondences.

We first observe that 
$\on{Sht}_{V^*\boxtimes V\boxtimes U, s\times s \times \eta }^{(\{1\},\{2\},\{3\})}$
is identified with the generic fiber of
$\on{Sht}_{V^*\boxtimes U \boxtimes V, s \times \eta \times s }^{(\{1,2\},\{3\})}$, and similarly for $\on{Hecke}_{V^*\boxtimes V\boxtimes U, s\times s \times \eta }^{(\{1\},\{2\},\{3\})}$.

\begin{lemma} \label{1749903579012}
There are natural identifications
\[
\on{Sht}_{V^*\boxtimes U \boxtimes V, s \times \eta \times s }^{(\{1,2\},\{3\})} 
\cong
\on{Sht}_{V^*\boxtimes V\boxtimes U,s\times s\times \eta }^{(\{1\},\{2\},\{3\})}
\]
and
\[
\on{Hecke}_{V^*\boxtimes U \boxtimes V, s \times \eta \times s }^{(\{1,2\},\{3\})} 
\cong
\on{Hecke}_{V^*\boxtimes V\boxtimes U,s \times s \times \eta }^{(\{1\},\{2\},\{3\})}.
\]
\end{lemma}

\begin{proof}
When the two legs of 
$\mathcal{P}_1 \dashrightarrow  \mathcal{P}_2$ 
are different, it is the same as the data of 
\[
\mathcal{P}_1 \dashrightarrow \mathcal{P}_3 \dashrightarrow \mathcal{P}_2
\]
where the first modification takes places at the characteristic $p$ untilt, and the second over characteristic 0 untilts. Now as the two legs of 
\[
\mathcal{P}_3 \dashrightarrow \mathcal{P}_2 \dashrightarrow \on{Frob}_S^* \mathcal{P}_1
\]
are different, and by Beauville–-Laszlo it is the same as the data of  
\[
\mathcal{P}_3 \dashrightarrow \mathcal{P}_2' \dashrightarrow \on{Frob}_S^* \mathcal{P}_1
\]
where the first modification takes place at the characteristic $p$ untilt, and the second takes place at characteristic zero untilts. Similarly for 
$\on{Hecke}_{V^*\boxtimes U \boxtimes V, s \times \eta \times s }^{(\{1,2\},\{3\})}$. 
\end{proof}

Let $\zeta: \{1\} \rightarrow \{1,2,3\}$ be the map defined by $\zeta(1)=3$, and 
\begin{equation} \label{geryuey}
\on{Sht}_{U|V^*\boxtimes V \boxtimes U,- |s \times s \times -, \eta}^{(\{1\})|(\{1\},\{2\},\{3\}),0} \subset \on{Sht}^{(\{1\})|(\{1\},\{2\},\{3\}), 0}_{U|V^*\boxtimes V \boxtimes U}
\times_{\on{Spd}(\Breve{\mathbb{Z}}_p) \times \on{Spd}(\Breve{\mathbb{Z}}_p)^3, \zeta \times \on{id}} \on{Spd}(\Breve{\mathbb{Z}}_p)^3
\end{equation} 
be the v-substack parametrizing modifications
\[
\begin{tikzcd}
\mathcal{P} \arrow[rrr,dashed] \arrow[d,equal,"{\resizebox{0.9cm}{0.1cm}{$\sim$}}" labl1] & & &
\on{Frob}_S^*\mathcal{P} \arrow[d,equal, "{\resizebox{0.9cm}{0.1cm}{$\sim$}}" labl1]
\\
\mathcal{P}_1  \arrow[r,dashed] &
\mathcal{P}_2 \arrow[r,dashed] &
\mathcal{P}_3 \arrow[r,dashed] &
\on{Frob}_S^*\mathcal{P}_1
\end{tikzcd}
\]
with the first line in $\on{Sht}_{U,\eta}$ and the second line in 
$\on{Sht}_{V^* \boxtimes V \boxtimes U,s \times s \times \eta}^{(\{1\},\{2\},\{3\})}$. 
Moreover, the legs indexed by $U$ in both the upper and lower modifications are the same. 
In other words, it is the generic fiber of 
$\on{Sht}_{U|V^* \boxtimes U\boxtimes V ,- |s \times - \times s}^{(\{1\})|(\{1,2\},\{3\}),0}$
under the identification in lemma \ref{1749903579012}. 

By our convention, $\on{Sht}_{U,\eta}$ has its leg in characteristic zero, while  
the first two modifications of the second line
take place at characteristic $p$, hence the third modification takes place at the same leg as the first line, and
the diagram is equivalent to 
\begin{equation} \label{pppi}
\begin{tikzcd}
\mathcal{P}  \arrow[rd,dashed, "\varphi_1"] \arrow[rr,"\sim"] & & \mathcal{P}_2' \arrow[r,dashed, "\varphi_3"] &
\on{Frob}_S^*\mathcal{P} 
\\
& 
\mathcal{P}_1',  \arrow[ru,dashed,"\varphi_2"] &
 &
\end{tikzcd}
\end{equation}
where the first two dashed arrows are modifications at the characteristic $p$ untilt bounded by $V^*$ and $V$ respectively whose composition is an isomorphism, and the last dashed arrow is a modification at characteristic zero untilts bounded by $U$.  
It means that
$\on{Sht}_{U|V^*\boxtimes V \boxtimes U,-|s \times s \times -, \eta}^{(\{1\})|(\{1\},\{2\},\{3\}),0}$
is a closed v-substack 
\[
\on{Sht}_{V^*\boxtimes V\boxtimes U,s \times s \times \eta }^{(\{1\},\{2\},\{3\}),0} 
\]
of 
$\on{Sht}_{V^*\boxtimes V\boxtimes U,s \times s \times \eta }^{(\{1\},\{2\},\{3\})}$,
which parametrizes diagrams as in (\ref{pppi}) without the isomorphism condition, and the correspondence diagram becomes
\[
\begin{tikzcd}
&
\on{Sht}_{V^*\boxtimes V\boxtimes U,s \times s \times \eta}^{(\{1\},\{2\},\{3\}),0} \arrow[dr,hook,"p_2" ] \arrow[dl, "p_1"'] 
& 
\\
\on{Sht}_{1\boxtimes U,s \times \eta}^{(\{1\},\{2\})}
& & 
\on{Sht}_{V^*\boxtimes V\boxtimes U,s\times s\times \eta}^{(\{1\},\{2\},\{3\})} 
\end{tikzcd}
\]
where $p_2$ is the closed immersion, $p_1$ sends (\ref{pppi}) to the first row
\[
\begin{tikzcd}
\mathcal{P}   \arrow[rr,"\sim"] & & \mathcal{P}_2' \arrow[r,dashed, "\varphi_3"] &
\on{Frob}_S^*\mathcal{P}. 
\end{tikzcd}
\]

The restriction of $\mathscr{C}_{\sharp}$ 
to the generic fiber then gives rise to the cohomological correspondence 
\begin{equation} \label{qasdk;aworeu}
\mathscr{C}_{\sharp,\eta}: (\on{Sht}_{ U, \eta}^{(\{1\})}, \epsilon^{(\{1\}),*}_{U} \mathscr{S}_U^{(\{1\})}
)
\longrightarrow 
(\on{Sht}_{V^*\boxtimes V\boxtimes U,s\times s\times \eta}^{(\{1\},\{2\},\{3\})}, \epsilon^{(\{1\},\{2\},\{3\}),*}_{V^*\boxtimes V\boxtimes U} \mathscr{S}_{V^*\boxtimes V\boxtimes U}^{(\{1\},\{2\},\{3\})}
)
\end{equation} 
supported on 
$\on{Sht}_{V^*\boxtimes V\boxtimes U,s\times s\times \eta}^{(\{1\},\{2\},\{3\}),0}$.

\begin{lemma} \label{znxmbwa./rq}
$\mathscr{C}_{\sharp,\eta}$
is the pullback of 
$\mathscr{C}_{\sharp}^{\on{loc},+,0} \boxtimes \on{id}$
defined in lemma \ref{lemaamamm} (1)
along the following diagram
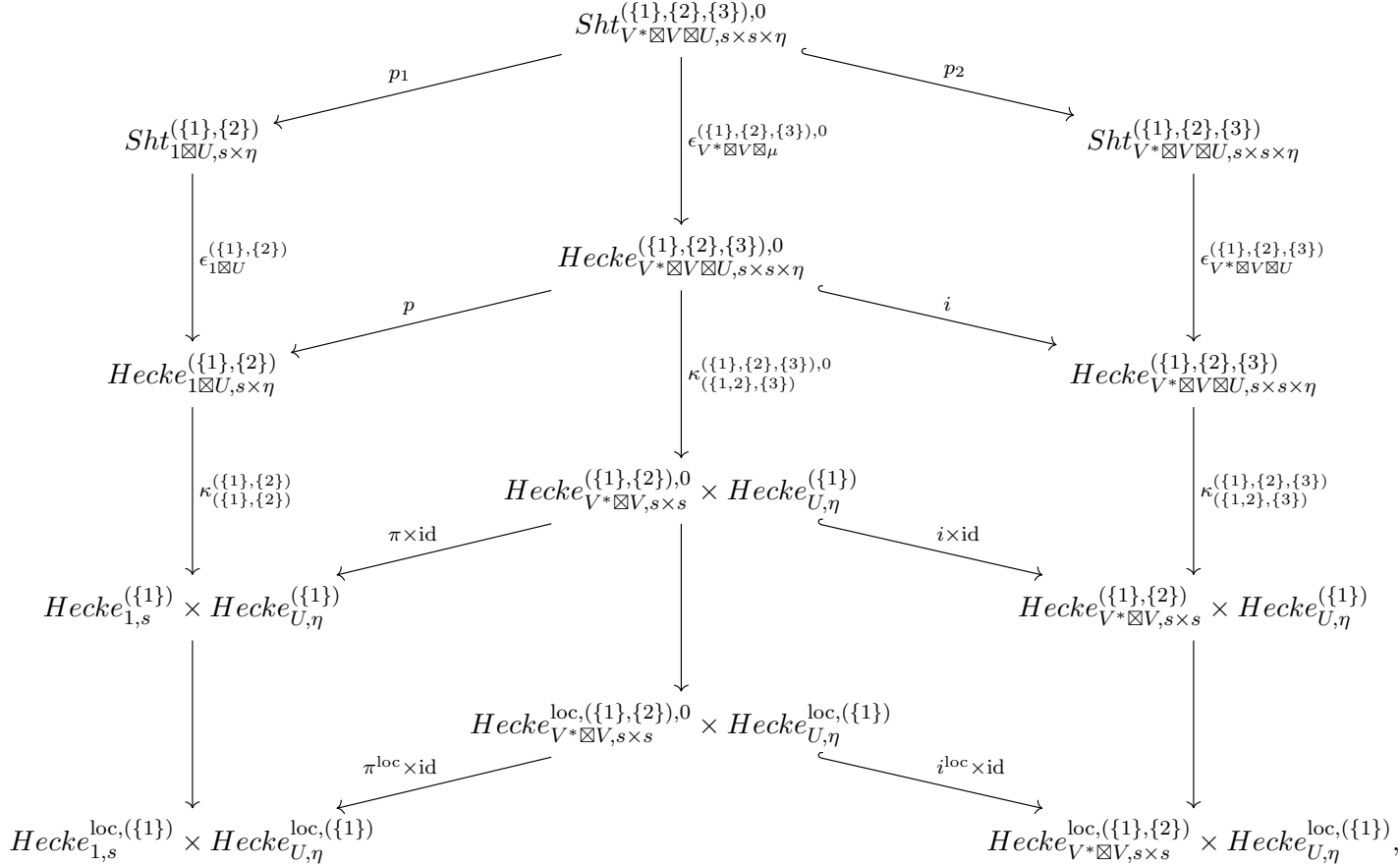
\[
\begin{tikzcd}
&
\on{Sht}_{V^*\boxtimes V\boxtimes U,s \times s \times \eta}^{(\{1\},\{2\},\{3\}),0} \arrow[dr,hook,"p_2" ] \arrow[dl, "p_1"'] \arrow[dd,"\epsilon^{(\{1\},\{2\},\{3\}),0}_{V^*\boxtimes V\boxtimes \mu}"]
& 
\\
\on{Sht}_{1\boxtimes U,s \times \eta}^{(\{1\},\{2\})} \arrow[dd,"\epsilon^{(\{1\},\{2\})}_{1\boxtimes U}"]
& & 
\on{Sht}_{V^*\boxtimes V\boxtimes U,s\times s\times \eta}^{(\{1\},\{2\},\{3\})} \arrow[dd,"\epsilon^{(\{1\},\{2\},\{3\})}_{V^*\boxtimes V\boxtimes U}"]
\\
&
\on{Hecke}_{V^*\boxtimes V\boxtimes U,s\times s\times \eta}^{(\{1\},\{2\},\{3\}),0} \arrow[dr,hook,"i" ] \arrow[dl, "p"'] \arrow[dd,"\kappa^{(\{1\},\{2\},\{3\}),0}_{(\{1,2\},\{3\})}"]
& 
\\
\on{Hecke}_{1\boxtimes U,s\times \eta}^{(\{1\},\{2\})} \arrow[dd,"\kappa^{(\{1\},\{2\})}_{(\{1\},\{2\})}"]
& & 
\on{Hecke}_{V^*\boxtimes V\boxtimes U,s \times s \times \eta}^{(\{1\},\{2\},\{3\})} \arrow[dd,"\kappa^{(\{1\},\{2\},\{3\})}_{(\{1,2\},\{3\})}"]
\\
& \on{Hecke}_{V^*\boxtimes V, s\times s}^{(\{1\},\{2\}),0}  \times \on{Hecke}_{U,\eta}^{(\{1\})} \arrow[dr,hook,"i\times \on{id}" ] \arrow[dl, "\pi \times \on{id}"'] \arrow[dd]
& 
\\
\on{Hecke}_{1, s}^{(\{1\})}  \times \on{Hecke}_{U,\eta}^{(\{1\})} \arrow[dd]
& & 
\on{Hecke}_{V^*\boxtimes V, s\times s}^{(\{1\},\{2\})} \times \on{Hecke}_{U,\eta}^{(\{1\})}
\arrow[dd]
\\
& \on{Hecke}_{V^*\boxtimes V, s \times s}^{\on{loc},(\{1\},\{2\}),0}  \times \on{Hecke}_{U,\eta}^{\on{loc},(\{1\})} \arrow[dr,hook,"i^{\on{loc}}\times \on{id}" ] \arrow[dl, "\pi^{\on{loc}} \times \on{id}"'] 
& 
\\
\on{Hecke}_{1,s}^{\on{loc},(\{1\})}  \times \on{Hecke}_{U,\eta}^{\on{loc},(\{1\})}
& & 
\on{Hecke}_{V^*\boxtimes V,s \times s}^{\on{loc},(\{1\},\{2\})} \times \on{Hecke}_{U,\eta}^{\on{loc},(\{1\})},
\end{tikzcd}
\]
where the maps between the last two rows are the global to local restriction maps ($\times \on{id}$). 
\end{lemma}

\begin{proof}

Clearly the squares are Cartesian, so we can pullback cohomological correspondences. By definition of $\mathscr{C}_{\sharp}$,  the restriction of $\mathscr{C}_{\sharp}$ 
to the generic fiber, 
\[
\mathscr{C}_{\sharp,\eta}: (\on{Sht}_{ U, \eta}^{(\{1\})}, \epsilon^{(\{1\}),*}_{U} \mathscr{S}_U^{(\{1\})}
)
\longrightarrow 
(\on{Sht}_{V^*\boxtimes V\boxtimes U,s\times s\times \eta}^{(\{1\},\{2\},\{3\})}, \epsilon^{(\{1\},\{2\},\{3\}),*}_{V^*\boxtimes V\boxtimes U} \mathscr{S}_{V^*\boxtimes V\boxtimes U}^{(\{1\},\{2\},\{3\})}
),
\]
is the pullback of 
$\mathscr{C}_{\sharp,\eta}^+$ 
defined in lemma \ref{lemmahecke}
along the 
first two rows.

By lemma \ref{lemmahecke},  $\mathscr{C}_{\sharp}^+$ is  the pullback 
of 
$\mathscr{C}_{\sharp}^{+,0} \boxtimes \on{id}$
along the middle two rows in the diagram. 

Further, we know from lemma \ref{lemaamamm} (2) that $\mathscr{C}_{\sharp}^{+,0}$
is the pullback of $\mathscr{C}_{\sharp}^{\on{loc},+,0}$ 
along the global to local restriction maps, thus $\mathscr{C}_{\sharp, \eta}$
is the pullback of $\mathscr{C}_{\sharp}^{\on{loc},+,0} \boxtimes \on{id}$ all the way from the bottom of the diagram. 
\end{proof}

The key observation for us is that the above diagram factorizes through moduli spaces of Witt vector shtukas. 

\begin{lemma}
We have the following commutative diagram whose composition is the same as that of the  diagram in lemma \ref{znxmbwa./rq}, 
\[
\begin{tikzcd}
&
\on{Sht}_{V^*\boxtimes V\boxtimes U, s \times s \times \eta}^{(\{1\},\{2\},\{3\}),0} \arrow[dr,hook,"p_2" ] \arrow[dl, "p_1"'] \arrow[dd,""]
& 
\\
\on{Sht}_{1\boxtimes U, s\times \eta}^{(\{1\},\{2\})} \arrow[dd,""]
& & 
\on{Sht}_{V^*\boxtimes V\boxtimes U,s \times s \times \eta}^{(\{1\},\{2\},\{3\})} \arrow[dd,""]
\\
&
\on{Sht}_{V^*\boxtimes V }^{\on{loc}, W,0} \times  \on{Hecke}_{U,\eta}^{\on{loc},(\{1\})} \arrow[dr,hook,"p_2 \times \on{id}" ] \arrow[dl, "p_1 \times \on{id}"'] \arrow[dd,"\epsilon^{\on{loc},W,0}_{V^*\boxtimes V} \times \on{id}"]
& 
\\
\on{Sht}_1^{\on{loc},W} \times  \on{Hecke}_{U,\eta}^{\on{loc},(\{1\})}
\arrow[dd,"\epsilon^{\on{loc},W}_{0} \times \on{id}"]
& & 
\on{Sht}_{V^*\boxtimes V}^{\on{loc},W} \times  \on{Hecke}_{U,\eta}^{\on{loc},(\{1\})} \arrow[dd,"\epsilon^{\on{loc},W}_{V^*\boxtimes V} \times \on{id}"]
\\
& \on{Hecke}_{V^*\boxtimes V, s \times s}^{\on{loc},(\{1\},\{2\}),0}  \times \on{Hecke}_{\mu,\eta}^{\on{loc},(\{1\})} \arrow[dr,hook,"i^{\on{loc}}\times \on{id}" ] \arrow[dl, "\pi^{\on{loc}} \times \on{id}"'] 
& 
\\
\on{Hecke}_{1,s}^{\on{loc},(\{1\})}  \times \on{Hecke}_{U,\eta}^{\on{loc},(\{1\})}
& & 
\on{Hecke}_{V^*\boxtimes V, s \times s}^{\on{loc},(\{1\},\{2\})} \times \on{Hecke}_{U,\eta}^{\on{loc},(\{1\})}.
\end{tikzcd}
\]
where the vertical maps along the first lines are restriction to the moduli spaces of local shtukas at the special fiber. 

More precisely,
the restriction map
\[
\on{Sht}_{V^*\boxtimes V\boxtimes U, s \times s \times \eta}^{(\{1\},\{2\},\{3\})} \longrightarrow
\on{Sht}_{V^*\boxtimes V, }^{\on{loc},W} \times \on{Hecke}_{U,\eta}^{\on{loc},(\{1\})}
\]
in the diagram is a shtuka version of $\kappa_{(\{1,2\},\{3\})}^{(\{1\},\{2\},\{3\})}$, namely sending 
\[
\begin{tikzcd}
\mathcal{P}_1  \arrow[r,dashed, "\varphi_1"] &
\mathcal{P}_2 \arrow[r,dashed, "\varphi_2"] &
\mathcal{P}_3 \arrow[r,dashed, "\varphi_3"] &
\on{Frob}_S^*\mathcal{P}_1
\end{tikzcd}
\]
to 
\[
\begin{tikzcd}
\mathcal{P}_1|_{\on{Spec}(\mathcal{O}^{\wedge}_{\mathcal{Y}_{[0,\infty)}(S), S})}  \arrow[r,dashed, "\varphi_1"] &
\mathcal{P}_2|_{\on{Spec}(\mathcal{O}^{\wedge}_{\mathcal{Y}_{[0,\infty)}(S), S})} \arrow[r,dashed, "\varphi_3 \circ \varphi_2"] &
\on{Frob}_S^*\mathcal{P}_1|_{\on{Spec}(\mathcal{O}^{\wedge}_{\mathcal{Y}_{[0,\infty)}(S), S})}
\end{tikzcd}
\]
and 
\[
\begin{tikzcd}
\mathcal{P}_3|_{\on{Spec}(\mathcal{O}^{\wedge}_{\mathcal{Y}_{[0,\infty)}(S), S^{\sharp}})}  \arrow[r,dashed, "\varphi_3"] &
\on{Frob}_S^*\mathcal{P}_1|_{\on{Spec}(\mathcal{O}^{\wedge}_{\mathcal{Y}_{[0,\infty)}(S), S^{\sharp}})},
\end{tikzcd}
\]
and similarly for the other two.
Since the modification $\varphi_3$ has its leg at characteristic 0 untilts, it is an isomorphism after restricting to 
$\on{Spec}(\mathcal{O}^{\wedge}_{\mathcal{Y}_{[0,\infty)}(S), S})$.  

On the other hand, the map 
\[
\epsilon_{\mu_{\bullet}}^{\on{loc}, W}: \on{Sht}^{\on{loc},W}_{\mu_{\bullet}} \longrightarrow 
\on{Hecke}^{\on{loc},(\{1\},\cdots,\{k\})}_{\mu{\bullet},s}
\]
is the forgetful map
sending 
\[
\begin{tikzcd}
\mathcal{P}_1  \arrow[r,dashed, "\varphi_1"] &
\mathcal{P}_2 \arrow[r,dashed, "\varphi_2"] &
\cdots \cdots   \arrow[r,dashed,"\varphi_{k-1}" ]&
\mathcal{P}_k \arrow[r,dashed,"\varphi_k" ] &
\on{Frob}_S^*\mathcal{P}_1
\end{tikzcd}
\]
to 
\[
\begin{tikzcd}
\mathcal{P}_1  \arrow[r,dashed, "\varphi_1"] &
\mathcal{P}_2 \arrow[r,dashed, "\varphi_2"] &
\cdots \cdots   \arrow[r,dashed,"\varphi_{k-1}" ]&
\mathcal{P}_k \arrow[r,dashed,"\varphi_k" ] &
\mathcal{P}_{k+1}
\end{tikzcd}
\]
with 
$\mathcal{P}_{k+1}:=
\on{Frob}_S^*\mathcal{P}_1$. 
\end{lemma}

\begin{proof}
Obvious. 
\end{proof}

Similar to the $p$-adic shtukas case, we can pullback $\mathscr{C}_{\sharp}^{\on{loc},+,0}$ in (\ref{sureqhrjhrh}) along 
\begin{equation} \label{eiurieurioe}
\begin{tikzcd}
&
\on{Sht}_{V^*\boxtimes V}^{\on{loc}, W,0}  \arrow[dr,hook,"p_2^{\on{loc}}" ] \arrow[dl, "p_1^{\on{loc}}"'] \arrow[dd,"\epsilon^{\on{loc},W,0}_{V^*\boxtimes V}"]
& 
\\
\on{Sht}_1^{\on{loc},W} 
\arrow[dd,"\epsilon^{\on{loc},W}_{0}"]
& & 
\on{Sht}_{V^*\boxtimes V}^{\on{loc},W}  \arrow[dd,"\epsilon^{\on{loc},W}_{V^*\boxtimes V} "]
\\
& \on{Hecke}_{V^*\boxtimes V, s \times s }^{\on{loc},(\{1\},\{2\}),0}   \arrow[dr,hook,"i^{\on{loc}}" ] \arrow[dl, "\pi^{\on{loc}}"'] 
& 
\\
\on{Hecke}_{1,s}^{\on{loc},(\{1\})} 
& & 
\on{Hecke}_{V^*\boxtimes V, s \times s }^{\on{loc},(\{1\},\{2\})} 
\end{tikzcd}
\end{equation}
to obtain the Witt vector creation cohomological correspondence
\[
\mathscr{C}_{\sharp}^{\on{loc},0,W} : p_1^{\on{loc},*} \epsilon^{\on{loc},W,*}_{0} \mathscr{S}_1^{(\{1\})} \longrightarrow 
p_2^{\on{loc},!} \epsilon^{\on{loc},W,*}_{V^*\boxtimes V} \mathscr{S}_{V^*\boxtimes V}^{(\{1\},\{2\})}
\]
supported on 
$\on{Sht}_{V^*\boxtimes V}^{\on{loc}, W,0}$.

\begin{proposition} \label{creprop}
The creation cohomological correspondence $\mathscr{C}_{\sharp, \eta}$ in (\ref{qasdk;aworeu}) is naturally identified with the pullback of 
$\mathscr{C}_{\sharp}^{\on{loc},0,W} \boxtimes \on{id}$ 
along 
\[
\begin{tikzcd}
&
\on{Sht}_{V^*\boxtimes V\boxtimes U,s \times s \times \eta}^{(\{1\},\{2\},\{3\}),0} \arrow[dr,hook,"p_2" ] \arrow[dl, "p_1"'] \arrow[dd,""]
& 
\\
\on{Sht}_{1\boxtimes U,s\times \eta}^{(\{1\},\{2\})} \arrow[dd,""]
& & 
\on{Sht}_{V^*\boxtimes V\boxtimes U,s \times s \times \eta}^{(\{1\},\{2\},\{3\})} \arrow[dd,""]
\\
&
\on{Sht}_{V^*\boxtimes V}^{\on{loc}, W,0} \times  \on{Hecke}_{U,\eta}^{\on{loc},(\{1\})} \arrow[dr,hook,"p_2^{\on{loc}} \times \on{id}" ] \arrow[dl, "p_1^{\on{loc}} \times \on{id}"']
& 
\\
\on{Sht}_1^{\on{loc},W} \times  \on{Hecke}_{U,\eta}^{\on{loc},(\{1\})}
& & 
\on{Sht}_{V^*\boxtimes V}^{\on{loc},W} \times  \on{Hecke}_{U,\eta}^{\on{loc},(\{1\})}. 
\end{tikzcd}
\]

Moreover, $\mathscr{C}_{\sharp}^{\on{loc},0,W}$ is the $\diamond\diamond$-analytification of the corresponding truncated Witt vector version as considered in 
\cite[example 6.1.9]{2017arXiv170705700X}. 
\end{proposition}

\begin{proof}

We use the truncated version of Shtukas and Hecke stacks as defined in \cite[section 5.3]{2017arXiv170705700X}, and recalled in appendix \ref{qrerewrewrf}. The truncated perfect stacks as defined in appendix \ref{qrerewrewrf} can be defined directly in the world of v-stacks, either by writing out the obvious moduli problem, or taking the corresponding $\diamond\diamond$-analytification.

We note that the truncated version of diagram (\ref{eiurieurioe}) is really the $\diamond\diamond$-analytification of the diagram below of perfect stacks 
\[
\begin{tikzcd}
&
\on{Sht}_{V^*\boxtimes V}^{\on{loc}(m,n),0}  \arrow[dr,hook,"p_2^{\on{loc}(m,n)}" ] \arrow[dl, "p_1^{\on{loc}(m,n)}"'] \arrow[dd,"\epsilon^{\on{loc}(m,n),0}_{V^*\boxtimes V}"]
& 
\\
\on{Sht}_0^{\on{loc}(m,n)} 
\arrow[dd,"\epsilon^{\on{loc}(m,n)}_{0}"']
& & 
\on{Sht}_{V^*\boxtimes V}^{\on{loc}(m,n)}  \arrow[dd,"\epsilon^{\on{loc}(m,n)}_{V^*\boxtimes V} "]
\\
& \on{Hecke}_{V^*\boxtimes V}^{\on{loc}(m),0}   \arrow[dr,hook,"i^{\on{loc}(m)}" ] \arrow[dl, "\pi^{\on{loc}(m)}"'] 
& 
\\
\on{Hecke}_{0}^{\on{loc}(m)} 
& & 
\on{Hecke}_{V^*\boxtimes V}^{\on{loc}(m)} 
\end{tikzcd}
\]
which defines  $\mathscr{C}_{\sharp}^{\on{loc}(m,n),0}$
by pullback of $\mathscr{C}_{\sharp}^{\on{loc}(m),+,0}$. We know that $\mathscr{C}_{\sharp}^{\on{loc}(m),+,0,\diamond\diamond} = \mathscr{C}_{\sharp}^{\on{loc},+,0}$
so it follows from proposition \ref{fjdklfsjdkeie} that 
\[
\mathscr{C}_{\sharp}^{\on{loc},0,W} = \mathscr{C}_{\sharp}^{\on{loc}(m,n),0,\diamond\diamond}
\]
under the identification in appendix \ref{qrerewrewrf}.

\end{proof}

\subsubsection{The special fiber}

Lastly, we  identify the creation correspondence at the special fiber. We want to prove that the pullback of $\mathscr{C}_{\sharp}$ (restricted on the special fiber) along the restriction map 
$\on{Sht}^W_{U} \longrightarrow \on{Sht}_{U,s}$ is precisely the $\diamond\diamond$-analytification of the creation correspondence constructed in \cite[Example 6.1.9]{2017arXiv170705700X}.

To simplify the notation, we denote
\[
\on{Sht}_{U|V^* \boxtimes U \boxtimes V,s}^{(\{1\})|(\{1,2\},\{3\}),0} 
:=
\on{Sht}_{U|V^* \boxtimes U \boxtimes V, s|s\times s\times s}^{(\{1\})|(\{1,2\},\{3\}),0},
\]
and 
\[
\on{Hecke}_{U|V^* \boxtimes U \boxtimes V,s}^{(\{1\})|(\{1,2\},\{3\})}
:=
\on{Hecke}_{U|V^* \boxtimes U \boxtimes V, s|s\times s\times s}^{(\{1\})|(\{1,2\},\{3\})}. 
\]
Similarly, we denote
\[
\on{Sht}_{V^* \boxtimes U \boxtimes V,s}^{(\{1,2\},\{3\})} 
:=
\on{Sht}_{V^* \boxtimes U \boxtimes V, s\times s\times s}^{(\{1,2\},\{3\})},
\]
and 
\[
\on{Hecke}_{V^* \boxtimes U \boxtimes V,s}^{(\{1,2\},\{3\})}
:=
\on{Hecke}_{V^* \boxtimes U \boxtimes V, s\times s\times s}^{(\{1,2\},\{3\})}. 
\]

As before, there is the diagram
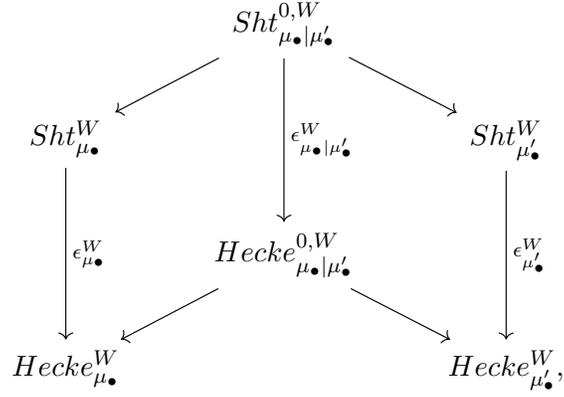
\[
\begin{tikzcd}
&
 \on{Sht}^{0,W}_{\mu_{\bullet}|\mu_{\bullet}'}  \arrow[dr ] \arrow[dl] \arrow[dd,"\epsilon_{\mu_{\bullet}|\mu_{\bullet}'}^{W}"]
& 
\\
\on{Sht}_{\mu_{\bullet} }^{W} \arrow[dd,"\epsilon_{\mu_{\bullet}}^{W} "]
& & 
\on{Sht}_{\mu'_{\bullet}}^{W} \arrow[dd,"\epsilon_{\mu'_{\bullet}}^{W}"]
\\
& \on{Hecke}^{0,W}_{\mu_{\bullet}|\mu_{\bullet}'}  \arrow[dr ] \arrow[dl] 
& 
\\
 \on{Hecke}_{\mu_{\bullet}}^{W}
& & 
 \on{Hecke}_{\mu'_{\bullet}}^{W},
\end{tikzcd}
\]
where the vertical maps are identifying $\mathcal{P}_{k+1}$ as $\on{Frob}_S^*\mathcal{P}_1$. 

On the other hand, we have the map
\[
\on{Sht}^{W}_{\mu_{\bullet}} \longrightarrow \on{Sht}_{\mu_{\bullet},s}^{(\{1\}, \{2\}, \cdots, \{k\})}
\]
which is the
restriction of moduli spaces of shtukas from $W(R^+)$ to $\mathcal{Y}_{[0,\infty)}(S)$, and similarly for $\on{Sht}_{\mu_{\bullet}|\mu_{\bullet}}^{0,W}$. They form a commutative diagram
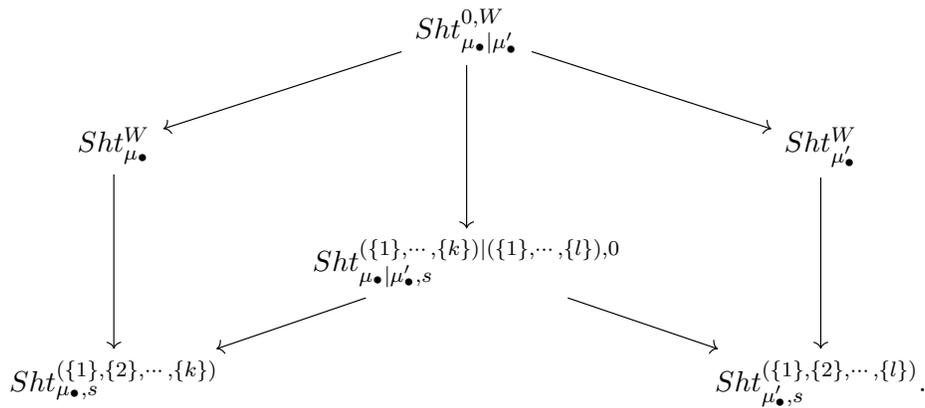
\[
\begin{tikzcd}
&
 \on{Sht}^{0,W}_{\mu_{\bullet}|\mu_{\bullet}'}  \arrow[dr ] \arrow[dl] \arrow[dd,""]
& 
\\
\on{Sht}_{\mu_{\bullet} }^{W} \arrow[dd," "]
& & 
\on{Sht}_{\mu'_{\bullet}}^{W} \arrow[dd,""]
\\
& \on{Sht}^{(\{1\},  \cdots, \{k\})|(\{1\},  \cdots, \{l\}), 0}_{\mu_{\bullet}|\mu_{\bullet}',s}  \arrow[dr ] \arrow[dl] 
& 
\\
 \on{Sht}_{\mu_{\bullet},s}^{(\{1\}, \{2\}, \cdots, \{k\})}
& & 
 \on{Sht}_{\mu'_{\bullet},s}^{(\{1\}, \{2\}, \cdots, \{l\})}.
\end{tikzcd}
\]

The diagram (\ref{asssseee}) restricted to the special fiber is 
\[
\begin{tikzcd}
&
\on{Sht}_{U|V^* \boxtimes U \boxtimes V,s}^{(\{1\})|(\{1,2\},\{3\}),0}
\arrow[dr ] \arrow[dl] \arrow[dd,"\epsilon^{(\{1\})|(\{1,2\},\{3\}),0}_{U| V^* \boxtimes \mu\boxtimes V}"]
& 
\\
\on{Sht}_{U,s}^{(\{1\})}
\arrow[dd,"\epsilon_{U}^{(\{1\})} "]
& & 
\on{Sht}_{V^* \boxtimes U \boxtimes V,s}^{(\{1,2\},\{3\})}
\arrow[dd,"\epsilon_{ V^* \boxtimes U \boxtimes V}^{(\{1,2\},\{3\})}"]
\\
& \on{Hecke}^{(\{1\})|(\{1,2\},\{3\}),0}_{U| V^* \boxtimes U \boxtimes V,s}  \arrow[dr ] \arrow[dl] 
& 
\\
 \on{Hecke}_{U,s}^{(\{1\})}
& & 
 \on{Hecke}_{V^* \boxtimes U \boxtimes V,s}^{(\{1,2\},\{3\})}.
\end{tikzcd}
\]
We can compose it with the previous one to obtain
\begin{equation} \label{fjiewurpo}
\begin{tikzcd}
&
 \on{Sht}^{0,W}_{U|(V^* \otimes U) \boxtimes V}  \arrow[dr ] \arrow[dl] \arrow[dd,""]
& 
\\
\on{Sht}_{U}^{W} \arrow[dd," "]
& & 
\on{Sht}_{(V^*\otimes U) \boxtimes  V}^{W} \arrow[dd,""]
\\
&
\on{Sht}_{U|V^* \boxtimes U \boxtimes V,s}^{(\{1\})|(\{1,2\},\{3\}),0}
\arrow[dr ] \arrow[dl] \arrow[dd,"\epsilon_{U| V^* \boxtimes U\boxtimes V}^{(\{1\})|(\{1,2\},\{3\}),0}"]
& 
\\
\on{Sht}_{U,s}^{(\{1\})}
\arrow[dd,"\epsilon_{U}^{(\{1\})} "]
& & 
\on{Sht}_{V^* \boxtimes U \boxtimes V,s}^{(\{1,2\},\{3\})}
\arrow[dd,"\epsilon_{ V^* \boxtimes U \boxtimes V}^{(\{1,2\},\{3\})}"]
\\
& \on{Hecke}^{(\{1\})|(\{1,2\},\{3\}),0}_{U| V^* \boxtimes U \boxtimes V,s}  \arrow[dr ] \arrow[dl] 
& 
\\
 \on{Hecke}_{U,s}^{(\{1\})}
& & 
 \on{Hecke}_{V^* \boxtimes U \boxtimes V,s}^{(\{1,2\},\{3\})},
\end{tikzcd}
\end{equation}
and we want to understand the pullback of the (restriction to the special fiber of)
creation correspondence $\mathscr{C}_{\sharp}$ in  (\ref{lleleeell}) along the upper half of the diagram.
We note that the diagram factorizes as
\begin{equation} \label{iweureyryei}
\begin{tikzcd}
&
 \on{Sht}^{0,W}_{U|(V^* \otimes U)\boxtimes V}  \arrow[dr ] \arrow[dl] \arrow[dd,"\epsilon_{U| (V^* \otimes U)\boxtimes V}^{W}"]
& 
\\
\on{Sht}_{U}^{W} \arrow[dd," \epsilon_{U}^{W} "]
& & 
\on{Sht}_{(V^* \otimes U)\boxtimes V}^{W} \arrow[dd,"\epsilon_{ (V^* \otimes U)\boxtimes V}^{W}"]
\\
&
\on{Hecke}_{U|(V^* \otimes U)\boxtimes V}^{0,W}
\arrow[dr ] \arrow[dl] \arrow[dd,"r_{(U|V^* \otimes U)\boxtimes V}"]
& 
\\
\on{Hecke}_{U}^{W}
\arrow[dd,"r_{U}"]
& & 
\on{Hecke}_{(V^* \otimes U)\boxtimes V}^{W}
\arrow[dd,"r_{(V^* \otimes U)\boxtimes V}"]
\\
& \on{Hecke}^{(\{1\})|(\{1,2\},\{3\}),0}_{U| V^* \boxtimes U \boxtimes V,s}  \arrow[dr ] \arrow[dl] 
& 
\\
 \on{Hecke}_{U,s}^{(\{1\})}
& & 
 \on{Hecke}_{V^* \boxtimes U \boxtimes V,s}^{(\{1,2\},\{3\})},
\end{tikzcd}
\end{equation}
where the vertical maps of the second row are restriction of torsors from $W(R^+)$ to $\mathcal{Y}_{[0,\infty)}(R,R^+)$.  

 We can pullback the  cohomological correspondence $\mathscr{C}_{\sharp}^+$ (restricted to the special fiber)  defined in example \ref{examplecreation} (\ref{ttytyry}) along the last diagram, which defines the creation correspondence on the global Witt vector shtuka
\begin{equation} \label{weewreerrr}
\mathscr{C}_{\sharp}^{W} : (\on{Sht}_{U}^{W}, \epsilon_{U}^{W,*}r_{U}^* \mathscr{S}_{U}^{(\{1\})}) \longrightarrow 
(\on{Sht}_{(V^* \otimes U)\boxtimes V}^{W}, (\epsilon_{ (V^* \otimes U)\boxtimes V}^{W})^*(r_{(V^* \otimes U)\boxtimes V})^* \mathscr{S}_{(V^* \otimes U )\boxtimes V}^{(\{1\},\{2\})}). 
\end{equation}
The importance of $\mathscr{C}_{\sharp}^{W}$
is that it is the $\diamond\diamond$-analytification of the creation correspondence on moduli spaces of Witt vector shtukas constructed in \cite[Example 6.1.9]{2017arXiv170705700X}.

\begin{proposition} \label{192089280932}
$\mathscr{C}_{\sharp}^{W}$ is the pullback of
the restriction of the creation correspondence $\mathscr{C}_{\sharp}^+$ in (\ref{lleleeell})  to $\on{Spd}(\overline{\mathbb{F}_p})$
along 
\[
\begin{tikzcd}
&
 \on{Sht}^{0,W}_{U|(V^* \otimes U)\boxtimes V}  \arrow[dr ] \arrow[dl] \arrow[dd,""]
& 
\\
\on{Sht}_{U}^{W} \arrow[dd," "]
& & 
\on{Sht}_{(V^* \otimes U)\boxtimes V}^{W} \arrow[dd,""]
\\
&
\on{Sht}_{U|V^* \boxtimes U \boxtimes V,s}^{(\{1\})|(\{1,2\},\{3\}),0}
\arrow[dr ] \arrow[dl]
& 
\\
\on{Sht}_{U,s}^{(\{1\})}
& & 
\on{Sht}_{V^* \boxtimes U \boxtimes V,s}^{(\{1,2\},\{3\})}.
\end{tikzcd}
\]

Moreover, we have 
\[
\mathscr{C}_{\sharp}^W=\mathscr{C}_{\sharp}^{\on{loc}(m,n), \diamond},
\]
the analytification of the creation correspondence in \cite[Example 6.1.9]{2017arXiv170705700X}. 
\end{proposition}

\begin{proof}
The factorization (\ref{fjiewurpo}) tells us that $\mathscr{C}_{\sharp}^{W} $ is the pullback of the creation correspondence (\ref{lleleeell}) along
\[
\begin{tikzcd}
&
 \on{Sht}^{0,W}_{\mu|(V^* \otimes U)\boxtimes V}  \arrow[dr ] \arrow[dl] \arrow[dd,""]
& 
\\
\on{Sht}_{U}^{W} \arrow[dd," "]
& & 
\on{Sht}_{(V^* \otimes U)\boxtimes V}^{W} \arrow[dd,""]
\\
&
\on{Sht}_{U|V^* \boxtimes U \boxtimes V,s}^{(\{1\})|(\{1,2\},\{3\}),0}
\arrow[dr ] \arrow[dl]
& 
\\
\on{Sht}_{U,s}^{(\{1\})}
& & 
\on{Sht}_{V^* \boxtimes U \boxtimes V,s}^{(\{1,2\},\{3\})}. 
\end{tikzcd}
\]

Next, we observe that the restriction of  $\mathscr{C}_{\sharp}^+$, as defined in example \ref{examplecreation} (\ref{ttytyry}), to the special fiber has a local counter part $\mathscr{C}_{\sharp}^{\on{loc},+}$, and $\mathscr{C}_{\sharp}^+$  is the pullback of $\mathscr{C}_{\sharp}^{\on{loc},+}$ along the canonical restriction diagram
\[
\begin{tikzcd}
&
 \on{Hecke}^{(\{1\})|(\{1,2\},\{3\}),0}_{U| V^* \boxtimes U\boxtimes V,s}
\arrow[dr ] \arrow[dl] \arrow[dd,""]
& 
\\
\on{Hecke}_{U,s}^{(\{1\})}
\arrow[dd,""]
& & 
\on{Hecke}_{V^* \boxtimes U \boxtimes V,s}^{(\{1,2\},\{3\})}
\arrow[dd,""]
\\
& \on{Hecke}^{\on{loc},(\{1\})|(\{1,2\},\{3\}),0}_{U| V^* \boxtimes U \boxtimes V,s}  \arrow[dr ] \arrow[dl] 
& 
\\
 \on{Hecke}_{U,s}^{\on{loc},(\{1\})}
& & 
 \on{Hecke}_{V^* \boxtimes U \boxtimes V,s}^{\on{loc},(\{1,2\},\{3\})}.
\end{tikzcd}
\]
The restriction to the special fiber is important here, as otherwise $\on{Hecke}^{\on{loc},(\{1\})|(\{1,2\},\{3\}),0}_{U| V^* \boxtimes U \boxtimes V,s} $ does not exist, see remark \ref{u803r4u3}. As always, $\mathscr{C}_{\sharp}^{+}$ being the pullback of $\mathscr{C}_{\sharp}^{\on{loc},+}$
follows ultimately from 6) of theorem \ref{gsatake}.  Now
we can compose the diagram (\ref{iweureyryei}) with the above restriction diagram, and obtain 
\[
\begin{tikzcd}
&
 \on{Sht}^{0,W}_{U|(V^* \otimes U)\boxtimes V}  \arrow[dr ] \arrow[dl] \arrow[dd,"\epsilon_{U| (V^* \otimes U)\boxtimes V}^{W}"]
& 
\\
\on{Sht}_{U}^{W} \arrow[dd," \epsilon_{U}^{W} "]
& & 
\on{Sht}_{(V^* \otimes U)\boxtimes V}^{W} \arrow[dd,"\epsilon_{ (V^* \otimes U)\boxtimes V}^{W}"]
\\
&
\on{Hecke}_{U|(V^* \otimes U)\boxtimes V}^{0,W}
\arrow[dr ] \arrow[dl] \arrow[dd,""]
& 
\\
\on{Hecke}_{U}^{W}
\arrow[dd,""]
& & 
\on{Hecke}_{(V^* \otimes U)\boxtimes V}^{W}
\arrow[dd,""]
\\
& \on{Hecke}^{\on{loc},(\{1\})|(\{1,2\},\{3\})0}_{U| V^* \boxtimes U \boxtimes V,s}  \arrow[dr ] \arrow[dl] 
& 
\\
 \on{Hecke}_{U,s}^{\on{loc},(\{1\})}
& & 
 \on{Hecke}_{V^* \boxtimes U \boxtimes V,s}^{\on{loc},(\{1,2\},\{3\})},
\end{tikzcd}
\]
and $\mathscr{C}_{\sharp}^{W}$ is the pullback of $\mathscr{C}_{\sharp}^{\on{loc},+}$ along the diagram. 
Up to truncation, this is precisely
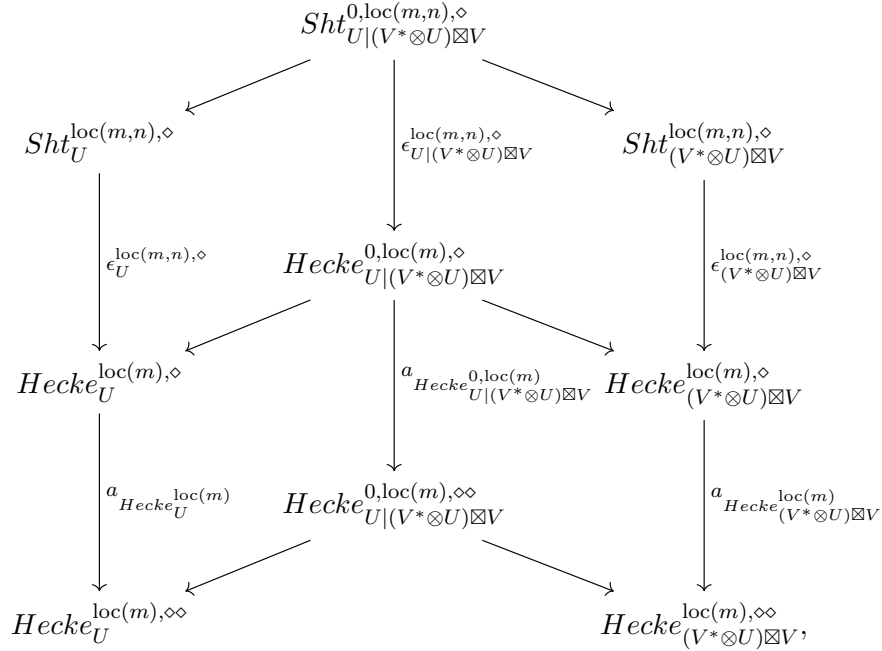
\[
\begin{tikzcd}
&
 \on{Sht}^{0,\on{loc}(m,n),\diamond}_{U|(V^* \otimes U)\boxtimes V}  \arrow[dr ] \arrow[dl] \arrow[dd,"\epsilon_{U| (V^* \otimes U)\boxtimes V}^{\on{loc}(m,n),\diamond}"]
& 
\\
\on{Sht}_{U}^{\on{loc}(m,n),\diamond} \arrow[dd," \epsilon_{U}^{\on{loc}(m,n),\diamond} "]
& & 
\on{Sht}_{(V^* \otimes U)\boxtimes V}^{\on{loc}(m,n),\diamond} \arrow[dd,"\epsilon_{(V^* \otimes U)\boxtimes V}^{\on{loc}(m,n),\diamond}"]
\\
&
\on{Hecke}_{U|(V^* \otimes U)\boxtimes V}^{0,\on{loc}(m),\diamond}
\arrow[dr ] \arrow[dl] \arrow[dd,"a_{\on{Hecke}_{U|(V^* \otimes U)\boxtimes V}^{0,\on{loc}(m)}}"]
& 
\\
\on{Hecke}_{U}^{\on{loc}(m),\diamond}
\arrow[dd,"a_{\on{Hecke}_{U}^{\on{loc}(m)}}"]
& & 
\on{Hecke}_{(V^* \otimes U)\boxtimes V}^{\on{loc}(m),\diamond}
\arrow[dd,"a_{\on{Hecke}_{(V^* \otimes U)\boxtimes V}^{\on{loc}(m)}}"]
\\
& \on{Hecke}^{0,\on{loc}(m),\diamond\diamond}_{U| (V^* \otimes U)\boxtimes V}  \arrow[dr ] \arrow[dl] 
& 
\\
 \on{Hecke}_{U}^{\on{loc}(m),\diamond\diamond}
& & 
 \on{Hecke}_{(V^* \otimes U)\boxtimes V}^{\on{loc}(m),\diamond\diamond},
\end{tikzcd}
\]
see appendix \ref{serepoir} for the notation of $a_{?}$, and \ref{qrerewrewrf} for the truncated Witt vector Hecke stacks and moduli spaces of Witt vector shtukas. We know from appendix \ref{qrerewrewrf} that the truncated version does not affect cohomology, so we can naturally view the cohomological correspondences on the (analytification of) truncated Witt vector Hecke stacks or moduli spaces of Witt vector shtukas  as being defined on the untruncated ones.  

We can perform the same construction as of $\mathscr{C}_{\sharp}^+$ (in the theory of perfect stacks) to obtain 
\[
\mathscr{C}_{\sharp}^{\on{loc}(m), +}: (\on{Hecke}_{U}^{\on{loc}(m)}, IC_{U}) \longrightarrow (\on{Hecke}_{(V^* \otimes U)\boxtimes V}^{\on{loc}(m)}, IC_{(V^* \otimes U)\boxtimes V})
\]
supported on 
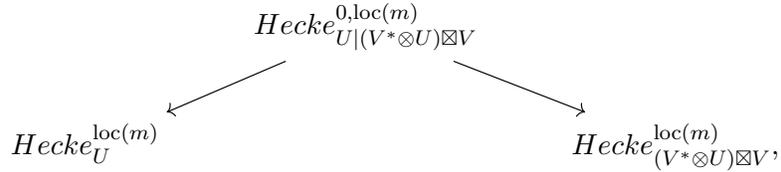
\[
\begin{tikzcd}
& \on{Hecke}^{0,\on{loc}(m)}_{U| (V^* \otimes U)\boxtimes V}  \arrow[dr ] \arrow[dl] 
& 
\\
 \on{Hecke}_{U}^{\on{loc}(m)}
& & 
 \on{Hecke}_{(V^* \otimes U)\boxtimes V}^{\on{loc}(m)},
\end{tikzcd}
\]
which is precisely the creation correspondence constructed in \cite[Example 6.1.9]{2017arXiv170705700X}.  Now 6) of theorem \ref{gsatake} (and the comparison results in appendices \ref{serepoir} and \ref{sdkkdkdkd}) tells us that 
\[
\mathscr{C}_{\sharp}^{\on{loc}(m),+,\diamond\diamond} = \mathscr{C}_{\sharp}^{\on{loc},+}
\]
under the identification of sheaves on truncated and untruncated Hecke stacks in \ref{qrerewrewrf}. By definition of $\mathscr{C}_{\sharp}^{\on{loc}(m),+,\diamond}$,
see appendix \ref{sdkkdkdkd}, and $\mathscr{C}_{\sharp}^{+}$ being the pullback of $\mathscr{C}_{\sharp}^{\on{loc},+}$, we have
\[
\mathscr{C}_{\sharp}^{\on{loc}(m),+,\diamond} = \mathscr{C}_{\sharp}^{+}.
\]
 Then the commutation between the analytification and pullback of cohomological correspondences (proposition \ref{fjdklfsjdkeie}) tells us that 
\[
\mathscr{C}_{\sharp}^{\on{loc}(m,n),\diamond} = \mathscr{C}_{\sharp}^W
\]
again under the identification of sheaves on truncated and untruncated Witt vector Hecke stacks and moduli spaces of Witt vector shtukas, where 
$\mathscr{C}_{\sharp}^{\on{loc}(m,n)}$
is the pullback of $\mathscr{C}_{\sharp}^{\on{loc}(m)}$ 
along 
\[
\begin{tikzcd}
&
 \on{Sht}^{0,\on{loc}(m,n)}_{U|(V^* \otimes U)\boxtimes V}  \arrow[dr ] \arrow[dl] \arrow[dd,"\epsilon_{U| (V^* \otimes U)\boxtimes V}^{\on{loc}(m,n)}"]
& 
\\
\on{Sht}_{U}^{\on{loc}(m,n)} \arrow[dd," \epsilon_{U}^{\on{loc}(m,n)} "]
& & 
\on{Sht}_{(V^* \otimes U)\boxtimes V}^{\on{loc}(m,n)} \arrow[dd,"\epsilon_{ (V^* \otimes U)\boxtimes V}^{\on{loc}(m,n)}"]
\\
&
\on{Hecke}_{U|(V^* \otimes U)\boxtimes V}^{0,\on{loc}(m)}
\arrow[dr ] \arrow[dl]
& 
\\
\on{Hecke}_{U}^{\on{loc}(m)}
& & 
\on{Hecke}_{(V^* \otimes U)\boxtimes V}^{\on{loc}(m)}.
\end{tikzcd}
\]
As usual, we use $\epsilon_?$ to denote the (in the sense of perfect stack) forgetting map (viewing the Frobenius twist of $\mathcal{P}_1$ as $\mathcal{P}_2$). $\mathscr{S}_{\sharp}^{\on{loc}(m,n)}$ is exactly the construction of the creation correspondence on moduli spaces of Witt vector shtukas in \cite[Example 6.1.9]{2017arXiv170705700X}, so we have proved that $\mathscr{C}_{\sharp}^W$ is the analytification of the creation correspondence in $loc.cit$.  
\end{proof}

\subsection{The annihilation correspondences}

Dually, we have the annihilation correspondences.

 Let $\zeta: \{1\} \longrightarrow \{1,2,3\}$ be a map defined by $\zeta(1)=3$,
and
\[
\on{Sht}_{V \boxtimes V^* \boxtimes U | U , s\times s\times -|-}^{(\{1\},\{2,3\}) |(\{1\}),0} \subset \on{Sht}^{(\{1\},\{2,3\}) |(\{1\}), 0}_{V \boxtimes V^* \boxtimes U | U }
\times_{\on{Spd}(\Breve{\mathbb{Z}}_p)^3 \times \on{Spd}(\Breve{\mathbb{Z}}_p),\on{id} \times \zeta } \on{Spd}(\Breve{\mathbb{Z}}_p)^3
\]
be the v-substack parametrizing modifications
\[
\begin{tikzcd}
\mathcal{P}_1  \arrow[r,dashed] \arrow[d,equal] &
\mathcal{P}_2 \arrow[r,dashed] &
\on{Frob}_S^*\mathcal{P}_1
\arrow[d,equal]
\\
\mathcal{P}_1 \arrow[rr,dashed]  & & 
\on{Frob}_S^*\mathcal{P}_1 
\end{tikzcd}
\]
with the first line in 
\[
\on{Sht}_{V\boxtimes V^* \boxtimes U, s\times s \times - }^{(\{1\},\{2,3\})} := \on{Sht}_{V\boxtimes V^* \boxtimes U }^{(\{1\},\{2,3\})}|_{\on{Spd}(\overline{\mathbb{F}_p}) \times \on{Spd}(\overline{\mathbb{F}_p}) \times \on{Spd}(\Breve{\mathbb{Z}}_p)},
\]
and the second line in $\on{Sht}_{U}^{(\{1\})}$. Moreover, the legs indexed by $\mu$ on the first and the second row are the same. 

Similarly, let
\[
\on{Hecke}_{V \boxtimes V^* \boxtimes U | U, s\times s\times -|- }^{(\{1\},\{2,3\}) |(\{1\}),0} \subset \on{Sht}^{(\{1\},\{2,3\}) |(\{1\}), 0}_{V \boxtimes V^* \boxtimes U | U}
\times_{\on{Spd}(\Breve{\mathbb{Z}}_p)^3 \times \on{Spd}(\Breve{\mathbb{Z}}_p),\on{id} \times \zeta } \on{Spd}(\Breve{\mathbb{Z}}_p)^3
\]
be the v-substack parametrizing modifications
\[
\begin{tikzcd}
\mathcal{P}_1  \arrow[r,dashed] \arrow[d,equal] &
\mathcal{P}_2 \arrow[r,dashed] &
\mathcal{P}_3
\arrow[d,equal]
\\
\mathcal{P}_1' \arrow[rr,dashed]  & & 
\mathcal{P}_2'
\end{tikzcd}
\]
with the first line in 
\[
\on{Hecke}_{V\boxtimes V^* \boxtimes U,s\times s\times - }^{(\{1\},\{2,3\})}
:=\on{Hecke}_{V\boxtimes V^* \boxtimes U }^{(\{1\},\{2,3\})}|_{\on{Spd}(\overline{\mathbb{F}_p}) \times \on{Spd}(\overline{\mathbb{F}_p}) \times \on{Spd}(\Breve{\mathbb{Z}}_p)},
\]
and the second line in $\on{Hecke}_{U}^{(\{1\})}$. Moreover, the legs indexed by $U$ on the first and the second row are the same. 

The annihilation correspondence is defined to be the cohomological correspondence 
\begin{equation} \label{weqqwww}
\mathscr{C}_{\flat} : 
(\on{Sht}_{ V \boxtimes V^* \boxtimes U, s\times s \times -}^{(\{1\},\{2,3\})}, \epsilon_{ V \boxtimes V^* \boxtimes U}^{(\{1\},\{2,3\}),*} \mathscr{S}_{ V \boxtimes V^* \boxtimes U}^{(\{1\},\{2,3\})})
\longrightarrow
(\on{Sht}_{U}^{(\{1\})}, \epsilon_{\mu}^{(\{1\}),*} \mathscr{S}_{U}^{(\{1\})})
\end{equation}
supported on 
$\on{Sht}_{ V \boxtimes V^* \boxtimes U| U, s\times s\times - |-}^0$ that is   the pullback
of 
\begin{equation} \label{qpqpppq}
\mathscr{C}_{\flat}^+ : 
( \on{Hecke}_{V \boxtimes V^* \boxtimes U,s\times s\times - }^{(\{1\},\{2,3\})},  \mathscr{S}_{ V \boxtimes V^* \boxtimes U}^{(\{1\},\{2,3\})})
\longrightarrow
(\on{Hecke}_{\mu}^{(\{1\})}, \mathscr{S}_{U}^{(\{1\})})
\end{equation}
defined in example \ref{examannih}
along 
\[
\begin{tikzcd}
&
\on{Sht}_{V \boxtimes V^* \boxtimes U | U,s\times s\times - |- }^{(\{1\},\{2,3\}) |(\{1\}),0}
\arrow[dr ] \arrow[dl] \arrow[dd,"\epsilon_{V \boxtimes V^* \boxtimes U|U}^{(\{1\},\{2,3\})|(\{1\})}"]
& 
\\
\on{Sht}_{V \boxtimes V^* \boxtimes U,s\times s\times -  }^{(\{1\},\{2,3\})}
\arrow[dd,"\epsilon_{V \boxtimes V^* \boxtimes U }^{(\{1\},\{2,3\})}"]
& &
\on{Sht}_{U}^{(\{1\})}
\arrow[dd,"\epsilon_{U}^{(\{1\})} "]
\\
& \on{Hecke}^{(\{1\},\{2,3\}) |(\{1\}),0}_{V \boxtimes V^* \boxtimes U | U,s\times s\times - |-}  \arrow[dr ] \arrow[dl] 
& 
\\
 \on{Hecke}_{V \boxtimes V^* \boxtimes U,s\times s\times -  }^{(\{1\},\{2,3\})}
 & &
 \on{Hecke}_{U}^{(\{1\})}.
\end{tikzcd}
\]

\subsubsection{The generic fiber}
We can identify the generic fiber of the annihilation correspondence, which is

\begin{equation} \label{annih}
\mathscr{C}_{\flat, \eta} : 
p_1^* \epsilon^{(\{1\},\{2\},\{3\}),*}_{V \boxtimes V^* \boxtimes U} \mathscr{S}_{V \boxtimes V^* \boxtimes U}^{(\{1\},\{2\},\{3\})}
\longrightarrow
p_2^! \epsilon_{U}^{(\{1\}),*} \mathscr{S}_U^{(\{1\})} 
\end{equation}
supported on
\[
\begin{tikzcd}
&
\on{Sht}_{V \boxtimes V^* \boxtimes U,s \times s \times \eta}^{(\{1\},\{2\},\{3\}),0} \arrow[dl,hook,"p_1"' ] \arrow[dr, "p_2"] 
& 
\\
\on{Sht}_{V \boxtimes V^* \boxtimes U,s\times s
\times\eta}^{(\{1\},\{2\},\{3\})} 
& &
\on{Sht}_{1\boxtimes U,s\times\eta}^{(\{1\},\{2\})}. 
\end{tikzcd}
\]
As before, we have the local Witt vector version, i.e. the cohomological correspondence 
\[
\mathscr{C}_{\flat}^{\on{loc},0,W} : p_1^*\epsilon^{\on{loc},W,*}_{V \boxtimes V^*} \mathscr{S}_{V \boxtimes V^*}^{\on{loc},(\{1\},\{2\})}  \longrightarrow 
p_2^!\epsilon^{\on{loc},W,*}_{1} \mathscr{S}_1^{\on{loc},(\{1\})}
\]
supported on
\[
\begin{tikzcd}
&
\on{Sht}_{V \boxtimes V^*}^{\on{loc}, W,0}  \arrow[dr,"p_2^{\on{loc}}" ] \arrow[dl,hook, "p_1^{\on{loc}}"'] 
& 
\\
\on{Sht}_{V \boxtimes V^*}^{\on{loc},W} 
& &
\on{Sht}_1^{\on{loc},W}
\end{tikzcd}
\]
corresponding to the evaluation map
\[
V \otimes V^* \longrightarrow 1.
\]

\begin{proposition} \label{annprop}
The annihilation cohomological correspondence $\mathscr{C}_{\flat}$ of (\ref{annih}) is naturally identified with the pullback of 
\[
\mathscr{C}_{\flat}^{\on{loc},0,W} \boxtimes \on{id}:
(\on{Sht}_{V \boxtimes V^*}^{\on{loc},W} \times  \on{Hecke}_{U,\eta}^{\on{loc},(\{1\})}, (\epsilon_{V \boxtimes V^*}^{\on{loc},W,*}
\mathscr{S}_{V \boxtimes V^*}^{\on{loc},(\{1\},\{2\})}) \boxtimes \mathscr{S}_U^{\on{loc},(\{1\})}) 
\longrightarrow
\]
\[
(\on{Sht}_1^{\on{loc},W} \times  \on{Hecke}_{U,\eta}^{\on{loc},(\{1\})} ,(\epsilon_1^{\on{loc},W,*} \mathscr{S}_1^{\on{loc},(\{1\})}) \boxtimes \mathscr{S}_U^{\on{loc},(\{1\})})
\]
along the restriction diagram
\[
\begin{tikzcd}
&
\on{Sht}_{V \boxtimes V^* \boxtimes U, s\times s\times \eta}^{(\{1\},\{2\},\{3\}),0} \arrow[dr,"p_2" ] \arrow[dl,,hook, "p_1"'] \arrow[dd,""]
& 
\\
\on{Sht}_{V \boxtimes V^* \boxtimes U,s\times s\times \eta}^{(\{1\},\{2\},\{3\})} \arrow[dd,""]
& &
\on{Sht}_{1\boxtimes U,s
\times \eta}^{(\{1\},\{2\})} \arrow[dd,""]
\\
&
\on{Sht}_{V \boxtimes V^*}^{\on{loc}, W,0} \times  \on{Hecke}_{U,\eta}^{\on{loc},(\{1\})} \arrow[dr,"p_2^{\on{loc}} \times \on{id}" ] \arrow[dl,hook, "p_1^{\on{loc}} \times \on{id}"']
& 
\\
\on{Sht}_{V \boxtimes V^*}^{\on{loc},W} \times  \on{Hecke}_{U,\eta}^{\on{loc},(\{1\})}
& &
\on{Sht}_1^{\on{loc},W} \times  \on{Hecke}_{U, \eta}^{\on{loc},(\{1\})}. 
\end{tikzcd}
\]

Moreover, $\mathscr{C}_{\flat}^{\on{loc},0,W}$ is the $\diamond\diamond$-analytification of the 
corresponding truncated Witt vector version $\mathscr{C}_{\flat}^{\on{loc}(m,n),0}$ constructed in \cite[Example 6.1.9]{2017arXiv170705700X}. 
\end{proposition}

\begin{proof}
    Same as the proof of proposition \ref{creprop}.
\end{proof}

\subsubsection{The special fiber}

We can identify the annihilation correspondence at the special fiber as well.

\begin{proposition} \label{2303490092}
$\mathscr{C}_{\flat}^{W}$  is the pullback of the restriction of the annihilation correspondence (\ref{weqqwww})  to $\on{Spd}(\overline{\mathbb{F}_p})$
along the diagram
\[
\begin{tikzcd}
&
 \on{Sht}^{0,W}_{V\boxtimes (V^* \otimes U)|U}  \arrow[dr ] \arrow[dl] \arrow[dd,""]
& 
\\
\on{Sht}_{V\boxtimes (V^* \otimes U)}^{W} \arrow[dd," "]
& & 
\on{Sht}_{U}^{W} \arrow[dd,""]
\\
&
\on{Sht}_{V \boxtimes V^* \boxtimes U | U,s}^{(\{1\},\{2,3\})|(\{1\}),0}
\arrow[dr ] \arrow[dl] 
& 
\\
\on{Sht}_{V \boxtimes V^* \boxtimes U ,s}^{(\{1\},\{2,3\})}
& &
\on{Sht}_{U, s}^{(\{1\})}
\end{tikzcd}
\]
where 
\begin{equation} \label{mm,sdmd}
\mathscr{C}_{\flat}^{W} : (\on{Sht}_{V\boxtimes (V^* \otimes U) }^{W},(\epsilon^{W}_{V\boxtimes (V^* \otimes U) } )^*(r_{V\boxtimes (V^* \otimes U)})^*\mathscr{S}_{V \boxtimes (V^*\otimes U) }^{(\{1\},\{2\})} ) \longrightarrow
(\on{Sht}_{U}^W, (\epsilon_{U}^{W})^*r_{U}^*\mathscr{S}_U^{(\{1\})})
\end{equation}
is the pullback of (\ref{qpqpppq}) (restricted to $\on{Spd}(\overline{\mathbb{F}_p})$) along
\[
\begin{tikzcd}
&
\on{Sht}_{V\boxtimes (V^* \otimes U)| U}^{0, W}  \arrow[dr,"p_2" ] \arrow[dl, "p_1"'] \arrow[dd,"\epsilon^{W}_{V\boxtimes (V^* \otimes U)| U}"]
& 
\\
\on{Sht}_{V\boxtimes (V^* \otimes U)}^{W}  \arrow[dd,"\epsilon^{W}_{V\boxtimes (V^* \otimes U)} "]
& &
\on{Sht}_{U}^{W} 
\arrow[dd,"\epsilon_{U}^{W}"]
\\
&
\on{Hecke}_{V\boxtimes (V^* \otimes U)|U}^{0,W}
\arrow[dr ] \arrow[dl] \arrow[dd,"r_{V\boxtimes (V^* \otimes U)|U}"]
& 
\\
\on{Hecke}_{V\boxtimes (V^* \otimes U)}^{W}
\arrow[dd,"r_{V\boxtimes (V^* \otimes U)}"]
& & 
\on{Hecke}_{U}^{W}
\arrow[dd,"r_{U}"]
\\
& \on{Hecke}^{(\{1\},\{2\})|(\{1\}),0}_{V\boxtimes (V^* \otimes U) | U,s}  \arrow[dr ] \arrow[dl] 
& 
\\ 
 \on{Hecke}_{ V\boxtimes (V^* \otimes U),s }^{(\{1\},\{2\})}
 & &
  \on{Hecke}_{U,s}^{(\{1\})}.
  \end{tikzcd}
\]

 Moreover, $\mathscr{C}_{\flat}^W$  is the analytification of the annihilation correspondence in \cite[Example 6.1.9]{2017arXiv170705700X}, i.e.
 \[
 \mathscr{C}_{\flat}^{W} = \mathscr{C}_{\flat}^{\on{loc}(m,n), \diamond}
 \]
 under the identification in \ref{sdkkdkdkd} between sheaves on truncated and untruncated moduli spaces of Witt vector shtukas.

\end{proposition}

\begin{proof}
    Same as the proof of proposition \ref{192089280932}. 
\end{proof}

\subsection{The partial Frobenius}
Lastly, we define the partial Frobenius morphism. Let $ \boxtimes U_i$ be a representation of $\hat{G}^I$, with each $U_i$ a representation of $\hat{G}^{I_i}$. Moreover, we assume that 
$U_1$ is a representation of $(^L G)^{I_1}$, so $U_1$ is isomorphic to its Frobenius twist $^{\sigma}U_1$, and we fix such an isomorphism from now on. 

We define the partial Frobenius to be the map
\[
F: \on{Sht}_{U_1\boxtimes \cdots \boxtimes U_k}^{(I_1,\cdots,I_k)} 
\longrightarrow 
\on{Sht}_{U_2\boxtimes \cdots \boxtimes U_k \boxtimes U_1}^{(I_2,\cdots,I_k,I_1)}
\]
sending 
\[
\begin{tikzcd}
\mathcal{P}_1  \arrow[r,dashed, "\varphi_1"] &
\mathcal{P}_2 \arrow[r,dashed, "\varphi_2"] &
\cdots \cdots   \arrow[r,dashed,"\varphi_{k-1}" ]&
\mathcal{P}_k \arrow[r,dashed,"\varphi_k" ] &
\on{Frob}_S^*\mathcal{P}_1
\end{tikzcd}
\]
to 
\[
\begin{tikzcd}
\mathcal{P}_2 \arrow[r,dashed, "\varphi_2"] &
\cdots \cdots   \arrow[r,dashed,"\varphi_{k-1}" ]&
\mathcal{P}_k \arrow[r,dashed,"\varphi_k" ] &
\on{Frob}_S^*\mathcal{P}_1
\arrow[r,dashed, "\on{Frob}_S^*\varphi_1"] 
&
\on{Frob}_S^*\mathcal{P}_2. 
\end{tikzcd}
\]
Note that we have used the assumption on $U_1$ so that $\on{Frob}_S^*\varphi_1$ is still bounded by $U_1$. In general, $\on{Frob}_S^*\varphi_1 $ 
would be bounded by $^{\sigma}U_1$. 

We have a commutative diagram
\[
\begin{tikzcd}
 \on{Sht}_{U_1\boxtimes \cdots \boxtimes U_k}^{(I_1,\cdots,I_k)} 
\arrow[r, "F"]
\arrow[d, "\kappa^{(I_1,\cdots,I_k)}_{(I_1,I_2 \cup \cdots \cup I_k)} \circ  \epsilon_{}^{(I_1,\cdots,I_k)}"]
& 
\on{Sht}_{U_2\boxtimes \cdots \boxtimes U_k\boxtimes U_1}^{(I_2,\cdots,I_k,I_1)} 
\arrow[d, "\kappa_{(I_2 \cup \cdots \cup I_k,I_1)}^{(I_2,\cdots,I_k,I_1)} \circ  \epsilon_{}^{(I_2,\cdots,I_k,I_1)}"]
\\
\on{Hecke}^{(I_1)}_{U_1} \times \on{Hecke}_{U_2 \boxtimes \cdots \boxtimes U_k }^{(I_2,\cdots,I_k)}  
\arrow[r, "\on{Frob} \times \text{id}"]
&
\on{Hecke}^{(I_1)}_{U_1} \times \on{Hecke}_{U_2 \boxtimes \cdots \boxtimes U_k }^{(I_2,\cdots,I_k)}. 
\end{tikzcd}
\]

We have a canonical isomorphism 
\[
\mathscr{C}_F^+: \on{Frob}^* \mathscr{S}_{U_1}^{(I_1)} \cong 
 \mathscr{S}_{U_1}^{(I_1)},
\]
from which we obtain 
\[
F^* \epsilon_{}^{(I_2,\cdots,I_k,I_1),*} \mathscr{S}_{U_2\boxtimes \cdots \boxtimes U_k\boxtimes U_1}^{(I_2,\cdots,I_k,I_1)} 
\cong 
F^* \epsilon_{}^{(I_2,\cdots,I_k,I_1),*} \kappa_{(I_2 \cup \cdots \cup I_k,I_1)}^{(I_2,\cdots,I_k,I_1),*} (\mathscr{S}_{U_1}^{(I_1)} \boxtimes \mathscr{S}_{ U_2\boxtimes \cdots \boxtimes U_k}^{(I_2,\cdots,I_k)})
\]
\begin{align*}
 \cong & 
  \epsilon_{}^{(I_1,\cdots,I_k),*}
\kappa^{(I_1,\cdots,I_k),*}_{(I_1,I_2 \cup \cdots \cup I_k)} (\on{Frob} \times \text{id})^*(\mathscr{S}_{U_1}^{(I_1)} \boxtimes \mathscr{S}_{U_2\boxtimes \cdots \boxtimes U_k}^{(I_2,\cdots,I_k)})
\\
 \overset{\mathscr{C}_F^+ \boxtimes \text{id}}{\cong} &
\epsilon_{}^{(I_1,\cdots,I_k),*}
\kappa^{(I_1,\cdots,I_k),*}_{(I_1,I_2 \cup \cdots \cup I_k)}
(\mathscr{S}_{U_1}^{(I_1)} \boxtimes \mathscr{S}_{U_2\boxtimes \cdots \boxtimes U_k}^{(I_2,\cdots,I_k)})
\\
 \cong &
\epsilon_{}^{(I_1,\cdots,I_k),*}
\mathscr{S}_{U_1\boxtimes  \cdots \boxtimes U_k}^{(I_1,\cdots,I_k)},
\end{align*}
where the first and the last isomorphism are 4) of theorem \ref{gsatake}.
It defines a pullback cohomological correspondence 
\[
\mathscr{C}_{F} :
(\on{Sht}_{U_2\boxtimes \cdots \boxtimes U_k\boxtimes U_1}^{(I_2,\cdots,I_k,I_1)} , \epsilon_{U_2\boxtimes \cdots \boxtimes U_k\boxtimes U_1}^{(I_2,\cdots,I_k,I_1),*} \mathscr{S}_{U_2\boxtimes \cdots \boxtimes U_k\boxtimes U_1}^{(I_2,\cdots,I_k,I_1)})
\longrightarrow
(\on{Sht}_{U_1\boxtimes \cdots \boxtimes U_k}^{(I_1,\cdots,I_k)} ,\epsilon_{U_1\boxtimes \cdots \boxtimes U_k}^{(I_1,\cdots,I_k),*}
\mathscr{S}_{U_1\boxtimes \cdots \boxtimes U_k}^{(I_1,\cdots,I_k)}  )
\]
supported on 
\[
\begin{tikzcd}
& \on{Sht}_{U_1\boxtimes \cdots \boxtimes U_k}^{(I_1,\cdots,I_k)}
\arrow[dr,equal ] \arrow[dl,"F"'] & 
\\
\on{Sht}_{U_2\boxtimes \cdots \boxtimes U_k\boxtimes U_1}^{(I_2,\cdots,I_k,I_1)}  & & 
\on{Sht}_{U_1\boxtimes \cdots \boxtimes U_k}^{(I_1,\cdots,I_k)} .
\end{tikzcd}
\]

Similarly, we have partial Frobenius on moduli spaces of local Witt vector shtukas,   
\[
F^{\on{loc},W}: \on{Sht}_{U_1\boxtimes \cdots \boxtimes U_k}^{\on{loc},W} 
\longrightarrow 
\on{Sht}_{U_2\boxtimes \cdots \boxtimes U_k \boxtimes U_1}^{\on{loc},W}
\]
sending 
\[
\begin{tikzcd}
\mathcal{P}_1  \arrow[r,dashed, "\varphi_1"] &
\mathcal{P}_2 \arrow[r,dashed, "\varphi_2"] &
\cdots \cdots   \arrow[r,dashed,"\varphi_{k-1}" ]&
\mathcal{P}_k \arrow[r,dashed,"\varphi_k" ] &
\on{Frob}_S^*\mathcal{P}_1
\end{tikzcd}
\]
to 
\[
\begin{tikzcd}
\mathcal{P}_2 \arrow[r,dashed, "\varphi_2"] &
\cdots \cdots   \arrow[r,dashed,"\varphi_{k-1}" ]
&
\mathcal{P}_k \arrow[r,dashed,"\varphi_k" ] &
\on{Frob}_S^*\mathcal{P}_1
\arrow[r,dashed, "\on{Frob}_S^*\varphi_1"] 
&
\on{Frob}_S^*\mathcal{P}_2,
\end{tikzcd}
\]
where we view each $U_i$ as a representation of $\hat{G}$ through the diagonal embedding $\hat{G} \hookrightarrow \hat{G}^{I_i}$. 

There is a commutative diagram
\[
\begin{tikzcd}
\on{Sht}^{\on{loc},W}_{U_1\boxtimes \cdots \boxtimes U_k} 
\arrow[r, "F^{\on{loc},W}"]
\arrow[d, "\kappa_{(I_1,I_2 \cup\cdots \cup I_k)}^{\on{loc},(I_1,\cdots,I_k)} \circ  \epsilon_{U_1\boxtimes \cdots \boxtimes U_k}^{\on{loc},W}"]
& 
\on{Sht}^{\on{loc},W}_{U_2\boxtimes \cdots \boxtimes U_k \boxtimes U_1}
\arrow[d, "\kappa_{(I_2 \cup\cdots \cup I_k,I_1)}^{\on{loc},(I_1,\cdots,I_k)} \circ  \epsilon_{U_2\boxtimes \cdots \boxtimes U_k \boxtimes U_1}^{\on{loc},W}"]
\\
\on{Hecke}_{U_1,s}^{\on{loc},(I_1)} \times \on{Hecke}_{U_2\boxtimes \cdots \boxtimes U_k,s }^{\on{loc},(I_2,\cdots , I_k)}  \arrow[r, "\on{Frob} \times \text{id}"]
&
\on{Hecke}_{U_1,s}^{\on{loc},(I_1)} \times \on{Hecke}_{ U_2\boxtimes \cdots \boxtimes U_k,s}^{\on{loc},(I_2,\cdots , I_k)}.
\end{tikzcd}
\]
As for the case of moduli space of shtukas, the diagram induces a pullback cohomological correspondence 
\[
\mathscr{C}_F^{\on{loc},W} : F^{\on{loc},W*} 
\epsilon_{U_2\boxtimes \cdots \boxtimes U_k \boxtimes U_1}^{\on{loc},W,*} \mathscr{S}_{U_2\boxtimes \cdots \boxtimes U_k \boxtimes U_1}^{\on{loc},(I_2,\cdots,I_k,I_1)}
\cong 
\epsilon_{U_1\boxtimes \cdots \boxtimes U_k}^{\on{loc},W,*} \mathscr{S}_{U_1\boxtimes \cdots \boxtimes U_k}^{\on{loc},(I_1,\cdots,I_k)} 
\]
supported on 
\[
\begin{tikzcd}
& \on{Sht}^{\on{loc},W}_{U_1\boxtimes \cdots \boxtimes U_k}  \arrow[dr,equal ] \arrow[dl,"F^{\on{loc},W}"'] & 
\\
\on{Sht}^{\on{loc},W}_{U_2\boxtimes \cdots \boxtimes U_k \boxtimes U_1}
& & 
\on{Sht}^{\on{loc},W}_{U_1\boxtimes \cdots \boxtimes U_k}.
\end{tikzcd}
\]
which is the $\diamond\diamond$-analytification of the  partial Frobenius correspondence constructed in \cite[Lemma 6.1.11]{2017arXiv170705700X}, namely, 
\[
\mathscr{C}_{F}^W=\mathscr{C}_{F}^{\on{loc}(m,n), \diamond\diamond}
\] 
under the identification in \ref{qrerewrewrf} between sheaves on truncated and untruncated moduli spaces of Witt vector shtukas,
where $\mathscr{C}_{F}^{\on{loc}(m,n)}$ is the partial Frobenius cohomological correspondence considered in \cite[Lemma 6.1.11]{2017arXiv170705700X}. Note that the inverse of the partial Frobenius is used in $loc.cit.$ as it is better suited with the truncated moduli spaces of Witt vector shtukas, after passing to the analytic untruncated version it agrees with our construction. 

We now specialize to the situation that is relevant to us.

\begin{example}
We consider 
\[
F: \on{Sht}_{V\boxtimes V^*\boxtimes U,s\times s\times -}^{(\{1\},\{2,3\})} 
\longrightarrow 
\on{Sht}_{V^*\boxtimes U \boxtimes V, s\times - \times s}^{(\{1,2\},\{3\})}.
\]
Note that we have used our assumption that $V$ is a representation of $^LG$ here. We have the partial Frobenius cohomological correspondence 
\begin{equation} \label{mvbcbvc}
\mathscr{C}_F : ( \on{Sht}_{ V^*\boxtimes U \boxtimes V,s\times - \times s}^{(\{1,2\},\{3\})},\epsilon_{V^*\boxtimes U \boxtimes V}^{(\{1,2\},\{3\}),*}
\mathscr{S}_{V^* \boxtimes U \boxtimes V}^{(\{1,2\},\{3\})}  )
\longrightarrow 
(\on{Sht}_{V \boxtimes V^*\boxtimes U,s\times s\times -}^{(\{1\},\{2,3\})} , \epsilon_{V\boxtimes V^*\boxtimes U}^{(\{1\},\{2,3\}),*} \mathscr{S}_{V \boxtimes V^* \boxtimes U }^{(\{1\},\{2,3\})})
\end{equation}
supported on 
\[
\begin{tikzcd}
& \on{Sht}_{V\boxtimes V^*\boxtimes U,s\times s\times -}^{(\{1\},\{2,3\})} \arrow[dr,equal ] \arrow[dl,"F"'] & 
\\
\on{Sht}_{V^*\boxtimes U \boxtimes V,s\times - \times s}^{(\{1,2\},\{3\})}   & & 
\on{Sht}_{V\boxtimes V^*\boxtimes U,s\times s\times -}^{(\{1\},\{2,3\})}.
\end{tikzcd}
\]
If we restrict to the generic fiber, then under the identification in lemma \ref{1749903579012}, we have
\[
\mathscr{C}_{F,\eta} :
(\on{Sht}_{V^*\boxtimes U \boxtimes V,s\times\eta \times s}^{(\{1\},\{2\},\{3\})} , \epsilon_{V\boxtimes V^*\boxtimes U}^{(\{1\},\{2\},\{3\}),*} \mathscr{S}_{V^* \boxtimes U \boxtimes V}^{(\{1\},\{2\},\{3\})})
\longrightarrow
( \on{Sht}_{V\boxtimes V^*\boxtimes U,s\times s\times\eta}^{(\{1\},\{2\},\{3\})},\epsilon_{V\boxtimes V^*\boxtimes U}^{(\{1\},\{2\},\{3\}),*}
\mathscr{S}_{V^* \boxtimes U \boxtimes V}^{(\{1\},\{2\},\{3\})}  )
\]
supported on 
\[
\begin{tikzcd}
& \on{Sht}_{V\boxtimes V^*\boxtimes U,s\times s\times\eta}^{(\{1\},\{2\},\{3\})} \arrow[dr,equal ] \arrow[dl,"F"'] & 
\\
\on{Sht}_{V^*\boxtimes U \boxtimes V,s\times\eta \times s}^{(\{1\},\{2\},\{3\})}   & & 
\on{Sht}_{V\boxtimes V^*\boxtimes U,s\times s\times\eta}^{(\{1\},\{2\},\{3\})}.
\end{tikzcd}
\]

We want to relate it with the local Witt vector partial Frobenius correspondence
\begin{equation} \label{xajgralkdflk}
\mathscr{C}_F^{\on{loc},0,W} : F^{\on{loc},W*} 
\epsilon_{V^*\boxtimes V}^{\on{loc},W,*} \mathscr{S}_{V^*\boxtimes V}^{\on{loc},(\{1\},\{2\})}
\cong 
\epsilon_{V\boxtimes V^*}^{\on{loc},W,*} \mathscr{S}_{V\boxtimes V^*}^{\on{loc},(\{1\},\{2\})} 
\end{equation}
supported on 
\[
\begin{tikzcd}
& \on{Sht}^{\on{loc},W}_{V \boxtimes V^*}  \arrow[dr,equal ] \arrow[dl,"F^{\on{loc},W}"'] & 
\\
\on{Sht}^{\on{loc},W}_{V^* \boxtimes V}
& & 
\on{Sht}^{\on{loc},W}_{V \boxtimes V^*}.
\end{tikzcd}
\]
As usual, this is the $\diamond\diamond$-analytification of the corresponding truncated Witt vector version $\mathscr{C}_F^{\on{loc}(m,n),0}$. 
\end{example}

\begin{proposition} \label{parfprop}
The cohomological correspondence 
$\mathscr{C}_{F,\eta}$ is the same as the pullback of $\mathscr{C}_F^{\on{loc},0,W} \boxtimes \on{id}$ 
along the diagram 
\[
\begin{tikzcd}
&
\on{Sht}_{V\boxtimes V^*\boxtimes U,s\times s\times\eta}^{(\{1\},\{2\},\{3\})} \arrow[dr,equal ] \arrow[dl, "F"'] \arrow[dd,""]
& 
\\
\on{Sht}_{V^*\boxtimes U \boxtimes V,s\times \eta \times s}^{(\{1\},\{2\},\{3\})}  \arrow[dd,""]
& & 
\on{Sht}_{V\boxtimes V^*\boxtimes U,s\times s\times\eta}^{(\{1\},\{2\},\{3\})} \arrow[dd,""]
\\
&
\on{Sht}^{\on{loc},W}_{V \boxtimes V^*} \times \on{Hecke}_{U,\eta}^{\on{loc},(\{1\})}  \arrow[dr,equal  ] \arrow[dl, "F^{\on{loc},W} \times \on{id}"']
& 
\\
\on{Sht}^{\on{loc},W}_{V^* \boxtimes V} \times  \on{Hecke}_{U,\eta}^{\on{loc},(\{1\})}
& & 
\on{Sht}_{V\boxtimes V^*}^{\on{loc},W} \times  \on{Hecke}_{U,\eta}^{\on{loc},(\{1\})}, 
\end{tikzcd}
\]
where the vertical arrows are as before restriction maps.

Moreover, \[
\mathscr{C}_F^{\on{loc},0,W} = \mathscr{C}_F^{\on{loc}(m,n),0,\diamond\diamond},
\]
being the $\diamond\diamond $-analytification of the truncated Witt vector version considered in \cite[Lemma 6.1.11]{2017arXiv170705700X}. 
\end{proposition}

\begin{proof}

It is clear that the two squares in the diagram are Cartesian, so we can pullback cohomological correspondences. We have the commutative diagram
\begin{equation*} 
\begin{tikzcd}
\on{Sht}_{V\boxtimes V^*\boxtimes U, s\times s\times \eta}^{(\{1\},\{2,3\})} 
\arrow[r, "F"]
\arrow[d, "\kappa_{(\{1\},\{2,3\})}^{(\{1\},\{2,3\})} \circ  \epsilon_{V\boxtimes V^*\boxtimes U}^{(\{1\},\{2,3\})}"]
& 
\on{Sht}_{V^*\boxtimes U \boxtimes V, s\times \eta \times s}^{(\{1,2\},\{3\})} 
\arrow[d, "\kappa_{(\{1,2\},\{3\})}^{(\{1,2\},\{3\})} \circ  \epsilon_{ V^*\boxtimes U \boxtimes V}^{(\{1,2\},\{3\})}"]
\\
\on{Hecke}_{V,s}^{(\{1\})} \times \on{Hecke}_{ V^*\boxtimes U,s\times \eta}^{(\{1,2\})}  \arrow[r, "\on{Frob} \times \text{id}"]
\arrow[d,"\text{id} \times \kappa_{(\{1\},\{2\})}^{(\{1\},\{2\})}"]
&
\on{Hecke}_{V,s}^{(\{1\})} \times \on{Hecke}_{ V^*\boxtimes U,s\times \eta}^{(\{1,2\})}
\arrow[d,"\text{id} \times \kappa_{(\{1\},\{2\})}^{(\{1\},\{2\})}"]
\\
\on{Hecke}_{V,s}^{(\{1\})} \times \on{Hecke}_{ V^*,s}^{(\{1\})} \times \on{Hecke}_{U,\eta}^{(\{1\})}  \arrow[r, "\on{Frob} \times \text{id}\times \text{id} "]
\arrow[d]
&
 \on{Hecke}_{V,s}^{(\{1\})} \times \on{Hecke}_{ V^*,s}^{(\{1\})} \times \on{Hecke}_{U,\eta}^{(\{1\})}
 \arrow[d]
\\
\on{Hecke}_{V,s}^{\on{loc},(\{1\})} \times \on{Hecke}_{ V^*,s}^{\on{loc},(\{1\})} \times \on{Hecke}_{U,\eta}^{\on{loc},(\{1\})}  \arrow[r, "\on{Frob} \times \text{id}\times \text{id} "]
&
 \on{Hecke}_{V,s}^{\on{loc},(\{1\})} \times \on{Hecke}_{ V^*,s}^{\on{loc},(\{1\})} \times \on{Hecke}_{U,\eta}^{\on{loc},(\{1\})}
\end{tikzcd}
\end{equation*}
where the bottom vertical maps are global to local restriction maps. 
We observe that it
factorizes through
\[
\begin{tikzcd}
\on{Sht}_{V\boxtimes V^*\boxtimes U,s\times s\times\eta}^{(\{1\},\{2\},\{3\})} 
\arrow[r, "F"]
\arrow[d, ""]
& 
\on{Sht}_{V^*\boxtimes U \boxtimes V,s\times\eta \times s}^{(\{1\},\{2\},\{3\})} 
\arrow[d, ""]
\\
\on{Sht}^{\on{loc},W}_{V \boxtimes V^*} \times  \on{Hecke}_{U,\eta}^{\on{loc},(\{1\})}
\arrow[r, "F^{\on{loc},W}  \times \text{id}"]
\arrow[d, "\kappa_{(\{1\},\{2\})}^{\on{loc},(\{1\},\{2\})} \circ  \epsilon_{V\boxtimes V^*}^{\on{loc},W}\times \text{id}"']
& 
\on{Sht}^{\on{loc},W}_{V^* \boxtimes V}\times  \on{Hecke}_{U,\eta}^{\on{loc},(\{1\})}
\arrow[d, "\kappa_{(\{1\},\{2\})}^{\on{loc},(\{1\},\{2\})} \circ  \epsilon_{V\boxtimes V^*}^{\on{loc},W} \times \text{id}"']
\\
\on{Hecke}_{V,s}^{\on{loc},(\{1\})} \times \on{Hecke}_{ V^*,s}^{\on{loc},(\{1\})} \times \on{Hecke}_{U,\eta}^{\on{loc},(\{1\})}  \arrow[r, "\on{Frob} \times \text{id}\times \text{id} "]
&
 \on{Hecke}_{V,s}^{\on{loc},(\{1\})} \times \on{Hecke}_{ V^*,s}^{\on{loc},(\{1\})} \times \on{Hecke}_{U,\eta}^{\on{loc},(\{1\})},
\end{tikzcd}
 \]
 where the top vertical arrows are restriction maps,
and the proposition follows from the construction of $\mathscr{C}_F$ and $\mathscr{C}_F^{\on{loc},W}$. Note that in the construction of $\mathscr{C}_F$, we can also use the local Hecke stacks. 
\end{proof}

Lastly, we want to identify the special fiber of the partial Frobenius $\mathscr{C}_F$, i.e.
\[
\mathscr{C}_{F,s} :
(\on{Sht}_{U_2\boxtimes \cdots \boxtimes U_k\boxtimes U_1,s}^{(I_2,\cdots,I_k,I_1)} , \epsilon_{U_2\boxtimes \cdots \boxtimes U_k\boxtimes U_1}^{(I_2,\cdots,I_k,I_1),*} \mathscr{S}_{U_2\boxtimes \cdots \boxtimes U_k\boxtimes U_1}^{(I_2,\cdots,I_k,I_1)})
\longrightarrow
(\on{Sht}_{U_1\boxtimes \cdots \boxtimes U_k,s}^{(I_1,\cdots,I_k)} ,\epsilon_{U_1\boxtimes \cdots \boxtimes U_k}^{(I_1,\cdots,I_k),*}
\mathscr{S}_{U_1\boxtimes \cdots \boxtimes U_k}^{(I_1,\cdots,I_k)}  )
\]
supported on 
\[
\begin{tikzcd}
& \on{Sht}_{U_1\boxtimes \cdots \boxtimes U_k,s}^{(I_1,\cdots,I_k)}
\arrow[dr,equal ] \arrow[dl,"F"'] & 
\\
\on{Sht}_{U_2\boxtimes \cdots \boxtimes U_k\boxtimes U_1,s}^{(I_2,\cdots,I_k,I_1)}  & & 
\on{Sht}_{U_1\boxtimes \cdots \boxtimes U_k,s}^{(I_1,\cdots,I_k)} .
\end{tikzcd}
\]

We first define the partial Frobenius on moduli spaces of  Witt vector shtukas. Similar to the other two cases, it is defined to be   
\[
F^W: \on{Sht}_{U_1\boxtimes \cdots \boxtimes U_k}^{W} 
\longrightarrow 
\on{Sht}_{U_2\boxtimes \cdots \boxtimes U_k \boxtimes U_1}^{W}
\]
sending 
\[
\begin{tikzcd}
\mathcal{P}_1  \arrow[r,dashed, "\varphi_1"] &
\mathcal{P}_2 \arrow[r,dashed, "\varphi_2"] &
\cdots \cdots   \arrow[r,dashed,"\varphi_{k-1}" ]&
\mathcal{P}_k \arrow[r,dashed,"\varphi_k" ] &
\on{Frob}_S^*\mathcal{P}_1
\end{tikzcd}
\]
to 
\[
\begin{tikzcd}
\mathcal{P}_2 \arrow[r,dashed, "\varphi_2"] &
\cdots \cdots   \arrow[r,dashed,"\varphi_{k-1}" ]&
\mathcal{P}_k \arrow[r,dashed,"\varphi_k" ] &
\on{Frob}_S^*\mathcal{P}_1
\arrow[r,dashed, "\on{Frob}_S^*\varphi_1"] 
&
\on{Frob}_S^*\mathcal{P}_2. 
\end{tikzcd}
\]

We have the commutative diagram
\begin{equation} \label{rpoeirope}
\begin{tikzcd}
\on{Sht}_{U_1\boxtimes \cdots \boxtimes U_k}^{W} \arrow[d] \arrow[r,"F^{W}"]
&
\on{Sht}_{U_2\boxtimes \cdots \boxtimes U_k \boxtimes U_1}^{W}
\arrow[d]
\\
\on{Sht}_{U_1\boxtimes \cdots \boxtimes U_k, s}^{(I_1,\cdots,I_k)}
\arrow[r, "F"]
\arrow[d, "\kappa_{(I_1, I_2 \cup \cdots \cup I_k)}^{(I_1,\cdots,I_k)} \circ  \epsilon_{U_1\boxtimes \cdots \boxtimes U_k}^{(I_1,\cdots,I_k)}"']
& 
\on{Sht}_{U_1\boxtimes \cdots \boxtimes U_k,s}^{(I_2,\cdots,I_k,I_1)}  
\arrow[d, "\kappa_{(I_2 \cup \cdots \cup I_k,I_1)}^{(I_2,\cdots,I_k,I_1)} \circ  \epsilon_{U_1\boxtimes \cdots \boxtimes U_k}^{(I_2,\cdots,I_k,I_1)} "]
\\
 \on{Hecke}_{W_1}^{(I_1)} \times \on{Hecke}_{ W_2\boxtimes \cdots \boxtimes W_k,s}^{(I_2,\cdots, I_k)}  \arrow[r, "\on{Frob} \times \text{id} "]
&
 \on{Hecke}_{W_1}^{(I_1)} \times \on{Hecke}_{ W_2\boxtimes \cdots \boxtimes W_k,s}^{(I_2,\cdots, I_k)} 
\end{tikzcd}
 \end{equation}
 which factorizes as 
 \[
 \begin{tikzcd}
 \on{Sht}_{U_1\boxtimes \cdots \boxtimes U_k}^{W}
 \arrow[d,"\kappa_{(I_1, I_2 \cup \cdots \cup I_k)}^{(I_1,\cdots,I_k),W} \circ \epsilon_{U_1\boxtimes \cdots \boxtimes U_k}^{W}"] \arrow[r,"F^{W}"]
&
\on{Sht}_{U_2\boxtimes \cdots \boxtimes U_k \boxtimes U_1}^{W}
\arrow[d,"\kappa_{(I_2 \cup \cdots \cup I_k,I_1)}^{(I_2,\cdots,I_k, I_1),W} \circ \epsilon_{U_2\boxtimes \cdots \boxtimes U_k \boxtimes U_1}^{W}"]
\\
 \on{Hecke}_{U_1}^{W} \times \on{Hecke}_{ U_2\boxtimes \cdots \boxtimes U_k}^{W}  \arrow[r, "\on{Frob} \times \text{id} "]
 \arrow[d,"r_{U_1} \times r_{U_2\boxtimes \cdots \boxtimes U_k}"]
&
 \on{Hecke}_{U_1}^{W} \times \on{Hecke}_{ U_2\boxtimes \cdots \boxtimes U_k}^{W} 
 \arrow[d,"r_{U_1} \times r_{U_2\boxtimes \cdots \boxtimes U_k}"]
\\
 \on{Hecke}_{U_1,s}^{(I_1)} \times \on{Hecke}_{ U_2\boxtimes \cdots \boxtimes U_k,s}^{(I_2,\cdots,I_k)}  \arrow[r, "\on{Frob} \times \text{id} "]
&
 \on{Hecke}_{U_1,s}^{(I_1)} \times \on{Hecke}_{U_2\boxtimes \cdots \boxtimes U_k,s}^{(I_2,\cdots,I_k)}. 
\end{tikzcd}
 \]
 
The last diagram defines the Witt vector partial Frobenius correspondence
\begin{equation*} 
\mathscr{C}_F^{W} : F^{W,*}  
\epsilon_{U_2\boxtimes \cdots \boxtimes U_k \boxtimes U_1}^{W,*}
\kappa_{(I_2 \cup \cdots \cup I_k,I_1)}^{(I_2,\cdots,I_k, I_1),W,*}  
(r_{U_1} \times r_{U_2\boxtimes \cdots \boxtimes U_k})^*
\mathscr{S}_{U_1}^{(I_1)} \boxtimes \mathscr{S}_{U_2\boxtimes \cdots \boxtimes U_k}^{(I_2,\cdots,I_k)}
\end{equation*}
\[
\cong 
\epsilon_{U_1\boxtimes \cdots \boxtimes U_k}^{W,*}
\kappa_{(I_1, I_2 \cup \cdots \cup I_k)}^{(I_1,\cdots,I_k),W,*}  
(r_{U_1} \times r_{U_2\boxtimes \cdots \boxtimes U_k})^*
\mathscr{S}_{U_1}^{(I_1)} \boxtimes \mathscr{S}_{U_2\boxtimes \cdots \boxtimes U_k}^{(I_2,\cdots,I_k)}, 
\]
which, similar as before, is the $\diamond$-analytification of the  partial Frobenius correspondence constructed in \cite[Lemma 6.1.11]{2017arXiv170705700X}.

Now the diagram (\ref{rpoeirope}) tells us that $\mathscr{C}_F^{W} $ is the pullback of $\mathscr{C}_{F,s}$
along the first square of (\ref{rpoeirope}). Specializing to the case of our interest, we have

\begin{proposition} \label{3984938498}
The correspondence
\begin{equation} 
\label{erjweewwe}
\mathscr{C}_F^{W} : F^{W,*}  
\epsilon_{(V^* \otimes U)\boxtimes V}^{W,*}
\kappa_{(\{1\},\{2\})}^{(\{1\},\{2\}),W,*} 
(r_V \times r_{(V^* \otimes U)})^*
\mathscr{S}_{(V^*\otimes U)\boxtimes V}^{(\{1\},\{2\})}
\end{equation}
\[
\cong 
\epsilon_{V\boxtimes (V^* \otimes U)}^{W,*}
\kappa_{(\{1\},\{2\})}^{(\{1\},\{2\}),W,*} 
(r_V \times r_{(V^* \otimes U)})^*
\mathscr{S}_{V\boxtimes (V^*\otimes U)}^{(\{1\},\{2\})}
\]
supported on
\[
\begin{tikzcd}
&
\on{Sht}_{V\boxtimes (V^* \otimes U)}^{W}  \arrow[dr,equal  ] \arrow[dl, "F^{W} "'] 
& 
\\
\on{Sht}_{(V^* \otimes U)\boxtimes V}^{W}
& & 
\on{Sht}_{V\boxtimes (V^* \otimes U)}^{W}
\end{tikzcd}
\]
 is the pullback of
$\mathscr{C}_{F,s}$ 
 along the diagram 
\[
\begin{tikzcd}
&
\on{Sht}_{V\boxtimes (V^* \otimes U)}^{W}  \arrow[dr,equal  ] \arrow[dl, "F^{W} "'] 
\arrow[dd,""]
& 
\\
\on{Sht}_{(V^* \otimes U)\boxtimes V}^{W}
\arrow[dd,""]
& & 
\on{Sht}_{V\boxtimes (V^* \otimes U)}^{W}
\arrow[dd,""]
\\
&
\on{Sht}_{V\boxtimes V^*\boxtimes U, s}^{(\{1\},\{2,3\})} \arrow[dr,equal ] \arrow[dl, "F"'] 
& 
\\
\on{Sht}_{V^*\boxtimes U \boxtimes V,s}^{(\{1,2\},\{3\})}  
& & 
\on{Sht}_{V\boxtimes V^*\boxtimes U, s}^{(\{1\},\{2,3\})}. 
\end{tikzcd}
\]

Moreover, $\mathscr{C}_{F}^W$ is the $\diamond$-analytification of the partial Frobenius cohomological  correspondence considered in \cite[Lemma 6.1.11]{2017arXiv170705700X}. 
\end{proposition}

\subsection{The excursion operators}
We can now define the excursion operator associated to $V$ as the cohomological correspondence 
\[
S_V:= \mathscr{C}_{\flat} \circ \mathscr{C}_F \circ \mathscr{C}_{\sharp}
\]
from 
$(\on{Sht}_{\mu}, \epsilon_{\mu}^{(\{1\}),*} \mathscr{S}_U^{(\{1\})})$ 
to itself supported on 
\[
\on{Sht}_{\mu | \mu}^{V} := 
\underset{\nu}{\bigcup} \on{Sht}_{\mu | \mu}^{\nu},
\]
where $\nu$ ranges over highest weights of irreducible components of $V$.
Indeed, we have 
\begin{lemma}
The support of $S_V$ is $\on{Sht}_{\mu | \mu}^{V}$.
\end{lemma}

\begin{proof}
By definition, the support of $S_V$ is the fiber product of
\[
\begin{tikzcd}
& & &
\on{Sht}_{V\boxtimes V^*\boxtimes U | U}^{0}
\arrow[dl]
\\
\on{Sht}_{U|  V^*\boxtimes U \boxtimes V}^{0}
\arrow[dr ]
& &
\on{Sht}_{V\boxtimes V^*\boxtimes U}^{(\{1\},\{2,3\})}
\arrow[dl,"F"']
\\
&
\on{Sht}_{V^*\boxtimes U \boxtimes V}^{(\{1,2\},\{3\})} 
&
\end{tikzcd}
\]
which by definition parametrizes 
\[
\begin{tikzcd}
\mathcal{P}_1   \arrow[r,dashed] \arrow[d,dashed, "\varphi_1"]  & 
\on{Frob}_S^*\mathcal{P}_1  
\arrow[d, dashed, "\on{Frob}_S^*\varphi_1"]
& 
\\
\mathcal{P}_2   \arrow[r,dashed]  \arrow[ur, dashed]
&
\on{Frob}_S^*\mathcal{P}_2
\end{tikzcd}
\]
where the upper left triangle lies in 
$\on{Sht}_{V\boxtimes V^*\boxtimes U | U}^{0}$, 
and the lower right triangle lies in 
$\on{Sht}_{U|  V^*\boxtimes U \boxtimes V}^{0}$. In other words, $\varphi_1$ has a leg at the characteristic $p$ untilt and is bounded by $V$, the two horizontal modifications have legs varying in $\on{Spd}(\Breve{\mathbb{Z}}_p)$ and is bounded by $U$, while the diagonal one has two legs, with one fixed at characteristic $p$ untilt bounded by $V^*$, and the other varying in $\on{Spd}(\Breve{\mathbb{Z}}_p)$  bounded by $U$. 
This is then the same as 
\[
\begin{tikzcd}
\mathcal{P}_1 \arrow[r,"\psi", dashed]
\arrow[d,"\varphi_1",dashed]
& 
\on{Frob}_S^*\mathcal{P}_1
\arrow[d,dashed, "\on{Frob}_S^*\varphi_1"]
\\
\mathcal{P}_2 \arrow[r, dashed]
& 
\on{Frob}_S^*\mathcal{P}_2
\end{tikzcd}
\]
with the two lines lying in $\on{Sht}_{\mu}$ and $\varphi_1$ bounded by  $V$,
which is nothing but what 
$\on{Sht}_{\mu | \mu}^{V} $ parametrizes. Indeed, the diagonal is just $\psi \circ \varphi_1^{-1}$, and the modification is taking place at the legs of $\varphi_1$ and $\psi$, bounded by the corresponding cocharcaters of $\varphi_1^{-1}$ and $\psi$. 
\end{proof}

Note that we have 
\[
\on{Sht}^V_{\mu |\mu,\eta}  = \Gamma_V, 
\]
thus $S_V$ restricted on $\on{Sht}_{\mu,\eta}$ is supported on $\Gamma_V$, the same as the support of the Hecke correspondence. 

Similarly, we have the local Witt vector version cohomological correspondence
\[
S_V^{\on{loc},0,W}:= \mathscr{C}_{\flat}^{\on{loc},0,W} \circ \mathscr{C}_F^{\on{loc},0,W} \circ \mathscr{C}_{\sharp}^{\on{loc},0,W}
\]
from 
$(\on{Sht}_{0}^{\on{loc},W}, \epsilon_{0}^{\on{loc},(\{1\}),*} \mathscr{S}_1^{\on{loc},(\{1\})})$ 
to itself supported on $\Gamma_V^{\on{loc},W}$, which can be identified with $ \on{Sht}^{\on{loc},V, W}_{0|0}$. Indeed, similar to the proof of previous lemma, we have that the support of $S_V^{\on{loc},0,W}$ parametrizes diagrams
\[
\begin{tikzcd}
\mathcal{P}_1  \arrow[d,dashed, "\varphi_1"] \arrow[r,"\sim"]  & 
\on{Frob}_S^*\mathcal{P}_1  
\arrow[d, dashed, "\on{Frob}_S^*\varphi_1"]
& &
\\
\mathcal{P}_2  \arrow[ru,dashed,"\varphi_2"] 
\arrow[r, "\sim"]
 &
\on{Frob}_S^*\mathcal{P}_2
\end{tikzcd}
\]
which is nothing but
\[
\begin{tikzcd}
\mathcal{P}_1 \arrow[rr, "\sim"', "\varphi_2\varphi_1"]
\arrow[d,"\varphi_1",dashed]
& &
\on{Frob}_S^*\mathcal{P}_1
\arrow[d,dashed, "\on{Frob}_S^*\varphi_1"]
\\
\mathcal{P}_2 \arrow[rr, "\sim"', "(\on{Frob}_S^*\varphi_1)\varphi_2"]
& &
\on{Frob}_S^*\mathcal{P}_2
\end{tikzcd}
\]
whence 
$\Gamma_V^{\on{loc},W} = \on{Sht}^{\on{loc},V, W}_{0|0} $.

\begin{theorem} \label{ssw}
The cohomological correspondence $S_{V,\eta}$, the restriction of $S_V$ on $\on{Sht}_{\mu,\eta}$, is the pullback of 
\[
S_V^{\on{loc},0,W} \boxtimes \on{id} : (\on{Sht}_0^{\on{loc},W} \times  \on{Hecke}_{\mu,\eta}^{\on{loc},(\{1\})}, (\epsilon_0^{\on{loc},W,*}\mathscr{S}_1^{\on{loc},(\{1\})}) \boxtimes \mathscr{S}_U^{\on{loc},(\{1\})} ) 
\longrightarrow
\]
\[
 (\on{Sht}_0^{\on{loc},W} \times  \on{Hecke}_{\mu,\eta}^{\on{loc},(\{1\})}, (\epsilon_0^{\on{loc},W,*}\mathscr{S}_1^{\on{loc},(\{1\})}) \boxtimes \mathscr{S}_U^{\on{loc},(\{1\})} ) 
\]
along the diagram
\[
\begin{tikzcd}
&
\Gamma_V \arrow[dr,"p_2" ] \arrow[dl, "p_1"'] \arrow[dd,""]
& 
\\
\on{Sht}_{\mu,\eta} \arrow[dd,""]
& & 
\on{Sht}_{\mu,\eta} \arrow[dd,""]
\\
&
\Gamma_V^{\on{loc},W} \times  \on{Hecke}_{\mu,\eta}^{\on{loc},(\{1\})} \arrow[dr,"p_2^{\on{loc}} \times \on{id}" ] \arrow[dl, "p_1^{\on{loc}} \times \on{id}"']
& 
\\
\on{Sht}_0^{\on{loc},W} \times  \on{Hecke}_{\mu,\eta}^{\on{loc},(\{1\})}
& & 
\on{Sht}_0^{\on{loc},W} \times  \on{Hecke}_{\mu,\eta}^{\on{loc},(\{1\})}. 
\end{tikzcd}
\]
where the vertical arrows are restriction maps, in particular, the middle one sends 
\[
\begin{tikzcd}
\mathcal{P}_1 \arrow[rr, dashed, "\varphi_1"]
\arrow[d,"\gamma",dashed]
& &
\on{Frob}_S^*\mathcal{P}_1
\arrow[d,dashed, "\on{Frob}_S^*\gamma"]
\\
\mathcal{P}_2 \arrow[rr, dashed, "\varphi_2"]
& &
\on{Frob}_S^*\mathcal{P}_2
\end{tikzcd}
\]
to
\[
\begin{tikzcd}
\mathcal{P}_1|_{\on{Spec}(\mathcal{O}^{\wedge}_{\mathcal{Y}_{[0,\infty)}(S), S})}  \arrow[rr,dashed, "\gamma|_{\on{Spec}(\mathcal{O}^{\wedge}_{\mathcal{Y}_{[0,\infty)}(S), S})} "] & &
\mathcal{P}_2|_{\on{Spec}(\mathcal{O}^{\wedge}_{\mathcal{Y}_{[0,\infty)}(S), S})} 
\end{tikzcd}
\]
and
\[
\begin{tikzcd}
\mathcal{P}_1|_{\on{Spec}(\mathcal{O}^{\wedge}_{\mathcal{Y}_{[0,\infty)}(S), S^{\sharp}})} \arrow[rr, dashed, "\varphi_1"]
& &
\on{Frob}_S^*\mathcal{P}_1|_{\on{Spec}(\mathcal{O}^{\wedge}_{\mathcal{Y}_{[0,\infty)}(S), S^{\sharp}})}.
\end{tikzcd}
\]

Moreover, $S_V^{\on{loc},0,W}$ is the $\diamond\diamond$-analytification of 
\[
S_V^{\on{loc}(m,n),0}:= \mathscr{C}_{\flat}^{\on{loc}(m,n,0)}\circ \mathscr{C}_{F}^{\on{loc}(m,n,0)} \circ
\mathscr{C}_{\sharp}^{\on{loc}(m,n,0)},
\]
the truncated Witt vector excursion operator  in \cite[Theorem 6.0.1 (1)]{2017arXiv170705700X}.  
\end{theorem}

\begin{proof}
    Since we know that pullback of cohomological correspondences respects composition, the theorem follows from proposition \ref{creprop}, \ref{annprop} and \ref{parfprop}.  
\end{proof}

\begin{remark}
The theorem holds with $\on{Sht}_{\mu,\eta}$ replaced by arbitary $\on{Sht}_{\mu_{\bullet},\eta}^{(I_1, \cdots, I_k)}$. The proof goes without change. We state it only for shtukas with a single leg to simplify the notation and that is the only case matters for Shimura varieties.   
\end{remark}

Lastly, we want to study the excursion operators at the special fiber. We have the global Witt vector excursion operator, which is the composition of (\ref{weewreerrr}), (\ref{erjweewwe}) and (\ref{mm,sdmd}), i.e. \[
S_V^{W}:= \mathscr{C}_{\flat}^{W} \circ \mathscr{C}_F^{W} \circ \mathscr{C}_{\sharp}^{W}
\]
is the cohomological correspondence from 
$(\on{Sht}^{W}_{\mu}, \epsilon_{\mu}^{W,*} r_{\mu}^* \mathscr{S}_U^{(\{1\})})$
to itself supported on 
$\on{Sht}_{\mu | \mu}^{V, W}$. 

\begin{proposition} \label{oeurieurio}
$S_V^{W}$ is the  pullback of $S_{V,s}$ along 
\[
\begin{tikzcd}
&
\on{Sht}_{\mu | \mu}^{V,W} \arrow[dr,"" ] \arrow[dl, ""] \arrow[dd]
& 
\\
\on{Sht}_{\mu}^{W} \arrow[dd]
& & 
\on{Sht}_{\mu}^{W} \arrow[dd]
\\
&
\on{Sht}_{\mu | \mu, s}^V \arrow[dr,"p_2" ] \arrow[dl, "p_1"] 
& 
\\
\on{Sht}_{\mu,s} 
& & 
\on{Sht}_{\mu,s}.
\end{tikzcd}
\]

Moreover, $S_V^W$ is the analytification of the truncated Witt vector excursion operator considered in \cite{2017arXiv170705700X}, under the equivalence between sheaves on truncatd and untruncated moduli spaces of Witt vector shtukas in \ref{qrerewrewrf}. 
\end{proposition}

\begin{proof}
    As $\mathscr{C}_{\sharp}^W$, $\mathscr{C}_{F}^W$ and $\mathscr{C}_{\flat}^W$ are all analytification of the truncated Witt vector counter part considered in \cite{2017arXiv170705700X}, so does the composite $S_V^W$ by proposition \ref{fjdklfsjdkeie}.    Then proposition \ref{3984938498}, proposition \ref{2303490092} and proposition \ref{192089280932} tells us that $S_V^{W}$ is the pullback of $S_{V,s}$, the restriction of $S_V$  to the special fiber. 
\end{proof}

\section{$S=T$ for moduli spaces of shtukas}

Let us first look more closely at $S_V^{\on{loc},W}$. Recall that it is supported on $\Gamma_V^{\on{loc},W} = \on{Hecke}_V^{\on{loc}} = L^{+,W}\mathcal{G} \setminus \text{Gr}^W_{G, \leq V}$, where $L^{+,W}\mathcal{G}$ is the v-sheaf whose value at a characteristic $p$ affinoid perfectoid space $\text{Spa}(R,R^+)$ is $\mathcal{G}(W(R))$, and $\text{Gr}^W_{G, \leq V}$
is the analytification of the Witt vector Schubert variety. 

Similar to Hecke operators on $\on{Sht}_{\mu,\eta}$, we also have them on $\on{Sht}_0^{\on{loc},W}$. There is the diagram
\[
\begin{tikzcd}
& \on{Sht}^{\on{loc},\nu,W}_{0|0}  \arrow[dr, "p_2"] \arrow[dl,"p_1"'] & 
\\
\on{Sht}_{0}^{\on{loc},W}  \arrow[dr,"\epsilon_{0}^{(\{1\})}"'] & & \on{Sht}_{0}^{\on{loc},W} 
\arrow[dl,"\epsilon_{0}^{(\{1\})}"]
\\
& \on{Hecke}_0^{\on{loc}} 
\arrow[d]
&
\\
& * & .
\end{tikzcd}
\]

\begin{lemma}
$p_1$ and $p_2$ are finite étale.
\end{lemma}

\begin{proof}

Given a morphism $\text{Spa}(R,R^+) \rightarrow \on{Sht}^{\on{loc},W}_0$ from an affinoid characteristic $p$ perfectoid space into $\on{Sht}_0^{\on{loc},W}$, which is determined by the data 
\[
\varphi_2: \mathcal{P}_2 \cong
\on{Frob}_S^*\mathcal{P}_2.
\]
Its fiber product with $p_2$ is the (analytification of) refined Witt vector affine Deligne-Lusztig variety $X^W_{0,\nu}(1)$ \footnote{See \cite[section 4.1.6]{2017arXiv170705700X}, where the refined affine Deligne-Lusztig variety $X_{\mu, \nu}(b)$ is defined as the intersection of the usual affine Deligne-Lusztig variety $X_{\mu}(b)$ with the Schubert cells $\on{Gr}_{\nu}$, i.e. $X_{\mu, \nu}(b)=X_{\mu}(b) \cap \on{Gr}_{\nu}$. } (base change to $R$), which parametrizes 
\[
\begin{tikzcd}
\mathcal{P}_1 \arrow[rr, "\sim"', "\varphi_1"]
\arrow[d,"\gamma",dashed]
& &
\on{Frob}_S^*\mathcal{P}_1
\arrow[d,dashed, "\on{Frob}_S^*\gamma"]
\\
\mathcal{P}_2 \arrow[rr, "\sim"', "\varphi_2"]
& &
\on{Frob}_S^*\mathcal{P}_2
\end{tikzcd}
\]
with $\gamma$ bounded by $\nu$ (the second line is fixed). This can be also characterized as the fiber product
\[
\begin{tikzcd}
X^W_{0,\nu}(1) \arrow[r,"a"]  \arrow[d, "b"]
&
Gr_{\nu}^W \arrow[d,"\Delta"]
\\
Gr_{\nu}^W \arrow[r,"\text{id} \times \varphi_2^{-1}\on{Frob}"]
&
Gr_{\nu}^W \times Gr_{\nu}^W
\end{tikzcd}
\]
where $a$ sends the above diagram to $\gamma\varphi_1^{-1}$, $b$ maps it to $\gamma$ 
and $\varphi_2^{-1} \on{Frob}$ is the map sending $\gamma : \mathcal{P}_1 \dashrightarrow \mathcal{P}_2$ to $\varphi_2^{-1}\on{Frob}_S^*\gamma$
(fixing a trivialization of $\mathcal{P}_2$), as $\varphi_2$ is an isomorphism, we see that $X^W_{0,\nu}(1) \cong Gr_{\nu}^W(\mathbb{F}_p)$ 
as zero dimensional varieties, which is obviously finite étale. Thus $p_2$ is finite étale, and similarly for $p_1$. 
\end{proof}

Moreover, we see that 
\[
\on{Sht}_{0|0}^{\on{loc},W} = L^{+,W}\mathcal{G} \setminus X^W_{0,\nu}(1) \cong L^{+,W}\mathcal{G} \setminus Gr_{\nu}^W(\mathbb{F}_p)
\]
which can be identified with the substack
\[
\Gamma_{\nu}^{\on{loc},W} := \on{Hecke}_{\nu}(\mathbb{F}_p) \subset \mathcal{G}(\mathbb{Z}_p) \setminus G(\mathbb{Q}_p) / \mathcal{G}(\mathbb{Z}_p)
\]
defined by the bounded by $\nu$ condition. Moreover, we have $\on{Sht}_0^{\on{loc},W} = L^{+,W}\mathcal{G} \setminus \bullet$ and $\mathscr{S}_1^{(\{1\})}= \overline{\mathbb{Q}}_\ell$, so any function $f \in C(\Gamma_{\nu}^W, \overline{\mathbb{Q}}_\ell) $ defines a cohomological correspondence $\mathscr{C}_f^{\on{loc},W}$ from 
$(\on{Sht}_{0}^{\on{loc},W}, \epsilon_{0}^{(\{1\}),*}\mathscr{S}_1^{\on{loc},(\{1\})})$ 
to itself supported on $\Gamma_{\nu}^{\on{loc},W}$
which is multiplication by $f(x)$ at $x \in \Gamma_{\nu}^{\on{loc},W}$. See \cite[Proposition 5.4.4]{2017arXiv170705700X}  for more details. 

We can do the same construction with truncated moduli spaces of Witt vector shtukas as in \cite{2017arXiv170705700X} to obtain $\mathscr{C}_{f}^{\on{loc}(m,n)}$, and we have
\[
\mathscr{C}_f^{\on{loc}(m,n),\diamond\diamond}= \mathscr{C}_f^{\on{loc},W}
\]
under the identification between sheaves on truncated and untruncated moduli spaces of Witt vector shtukas in appendix \ref{qrerewrewrf}. Note that this identification only holds at the sheaf level, and not at the space level. 

We have $\Gamma_V^{\on{loc},W} = \cup \Gamma_{\nu}^{\on{loc},W}$ where $\nu$ ranges over highest weights of irreducible summands of $V$, and we can similarly define $\mathscr{C}_f^{\on{loc},W}$ for $f \in C(\Gamma_{V}^{\on{loc},W}, \overline{\mathbb{Q}}_\ell)$
and 
\[
T_V^{\on{loc},W} := \mathscr{C}_{h_V}^{\on{loc},W}
\]
as cohomological correspondence from 
$(\on{Sht}_{0}^{\on{loc},W}, \epsilon_{0}^{\on{loc},(\{1\}),*}\mathscr{S}_1^{\on{loc},(\{1\})})$ 
to itself supported on $\Gamma_{V}^{\on{loc},W}$,
where $h_V$ is the function corresponding to $V$ under classical Satake isomorphism. Similarly, we can define $T_V^{\on{loc}(m,n)} = \mathscr{C}_{h_V}^{\on{loc}(m,n)}$, and we have 
\[
T_V^{\on{loc},W} = T_V^{\on{loc}(m,n),\diamond\diamond}
\]
under the canonical identification in \ref{qrerewrewrf} between sheaves on truncated and untruncated moduli spaces of Witt vector shtukas.

\begin{theorem} (Xiao--Zhu) \label{xiaozhu}
We have an equality 
\[
T_V^{\on{loc},W} = S_V^{\on{loc},0,W}
\]
as cohomological correspondence from 
$(\on{Sht}_{0}^{\on{loc},W}, \epsilon_{0}^{\on{loc},(\{1\}),*}\mathscr{S}_1^{\on{loc},(\{1\})})$ 
to itself supported on $\Gamma_{V}^{\on{loc},W}$.
\end{theorem}

\begin{proof}
This is the analytification of \cite[Theorem 6.0.1 (2)]{2017arXiv170705700X}. More precisely, Xiao--Zhu prove the equality as cohomological correspondences on perfect stacks, and we take the $\diamond\diamond$-analytification of the corresponding cohomological correspondences to get an equality of cohomological correspondences on v-stacks. The comparison results (between cohomological correspondences on perfect stacks and their $\diamond\diamond$-analytification) we proved in  section \ref{e888888}, recalled in the previous paragraph,  tells us that the cohomological correspondences obtained by analytification coincide with the ones in the statement of the theorem. 
\end{proof}

We can now prove the desired $S=T$ theorem for moduli spaces of $p$-adic shtukas. 

\begin{theorem} \label{840238493}
There is a canonical identification
\[
S_{V,\eta} = T_V
\]
as cohomological correspondence from 
$(\on{Sht}_{\mu,\eta}, \epsilon_{\mu}^{(\{1\}),*}\mathscr{S}_U^{(\{1\})})$ 
to itself supported on $\Gamma_V$. 
\end{theorem}

\begin{remark}
Again, there is no difficulty generalizing the theorem to shtukas with several legs (in characteristic zero). 
\end{remark}

\begin{proof}
Theorem \ref{ssw} and \ref{xiaozhu}  together imply that $S_{V,\eta}$ is the pullback of $T_V^{\on{loc},W} \boxtimes \text{id}$ along
\[
\begin{tikzcd}
&
\Gamma_V \arrow[dr,"p_2" ] \arrow[dl, "p_1"'] \arrow[dd,""]
& 
\\
\on{Sht}_{\mu,\eta} \arrow[dd,""]
& & 
\on{Sht}_{\mu,\eta} \arrow[dd,""]
\\
&
\Gamma_V^{\on{loc},W} \times  \on{Hecke}_{\mu,\eta}^{\on{loc},(\{1\})} \arrow[dr,"p_2^{\on{loc}} \times \text{id}" ] \arrow[dl, "p_1^{\on{loc}} \times \text{id}"']
& 
\\
\on{Sht}_0^{\on{loc},W} \times  \on{Hecke}_{\mu,\eta}^{\on{loc},(\{1\})}
& & 
\on{Sht}_0^{\on{loc},W} \times  \on{Hecke}_{\mu,\eta}^{\on{loc},(\{1\})},
\end{tikzcd}
\]
which is the same as $T_V$. Indeed, we have a commutative diagram with the upper two  squares Cartesian
\[
\begin{tikzcd}
& \on{Sht}^{\nu}_{\mu|\mu,\eta} \arrow[dr, "p_2"] \arrow[dl,"p_1"'] \arrow[dd,"\pi"] & 
\\
\on{Sht}_{\mu,\eta} \arrow[dd] & & \on{Sht}_{\mu,\eta}
\arrow[dd,"\pi_1"]
\\
& \on{Sht}^{\on{loc},\nu,W}_{0|0} \times \on{Hecke}_{\mu,\eta}^{\on{loc},(\{1\})} \arrow[dr, "p_2^{\on{loc}}\times \text{id}"] \arrow[dl,"p_1^{\on{loc}}\times \text{id}"'] & 
\\
\on{Sht}_{0}^{\on{loc},W} \times \on{Hecke}_{\mu,\eta}^{\on{loc},(\{1\})} \arrow[dr,"\epsilon_{0}^{\on{loc},(\{1\})}\times \text{id}"'] 
& & 
\on{Sht}_{0}^{\on{loc},W} \times \on{Hecke}_{\mu,\eta}^{\on{loc},(\{1\})}
\arrow[dl,"\epsilon_{0}^{\on{loc},(\{1\})}\times \text{id}"]
\\
& * \times  \on{Hecke}_{\mu,\eta}^{\on{loc},(\{1\})} &
\end{tikzcd}
\]
where we abuse the notation by writing $\epsilon_{0}^{\on{loc},(\{1\})}$ the unique map from $\on{Sht}_{0}^{\on{loc},W}$ to the final point $*$. Then under the canonical identification $* \times  \on{Hecke}_{\mu,\eta}^{\on{loc},(\{1\})} = \on{Hecke}_{\mu,\eta}^{\on{loc},(\{1\})}$, the composition is precisely (\ref{heckdia}). Let $\mathscr{C}_{\nu}^{\on{loc},W} = \mathscr{C}_{1_{\Gamma_{\nu}^W}}^{\on{loc},W}$ with 
$1_{\Gamma_{\nu}^{\on{loc},W}}$ the characteristic function of $\Gamma_{\nu}^{\on{loc},W}$.
Then the pullback of $\mathscr{C}_{\nu}^{\on{loc},W} \boxtimes \text{id}$ is
\[
p_1^* \epsilon_{\mu}^{(\{1\}),*} \mathscr{S}_U^{(\{1\})} \cong 
\pi^* (p_1^{\on{loc}}\times \text{id})^* (\epsilon_{0}^{\on{loc},(\{1\})}\times \text{id})^* \mathscr{S}_1^{\on{loc},(\{1\})} \boxtimes \mathscr{S}_U^{\on{loc},(\{1\})} 
\]
\[
\cong
\pi^* (p_2^{\on{loc}}\times \text{id})^* (\epsilon_{0}^{\on{loc},(\{1\})}\times \text{id})^* \mathscr{S}_1^{\on{loc},(\{1\})} \boxtimes \mathscr{S}_U^{\on{loc},(\{1\})} 
\]
\[
\cong
\pi^* (p_2^{\on{loc}}\times \text{id})^! (\epsilon_{0}^{\on{loc},(\{1\})}\times \text{id})^* \mathscr{S}_1^{\on{loc},(\{1\})} \boxtimes \mathscr{S}_U^{\on{loc},(\{1\})} 
\]
\[
\overset{\text{BC}}{\longrightarrow} 
p_2^! \pi_1^*  (\epsilon_{0}^{\on{loc},(\{1\})}\times \text{id})^* \mathscr{S}_1^{\on{loc},(\{1\})} \boxtimes \mathscr{S}_U^{\on{loc},(\{1\})}
\cong
p_2^!  \epsilon_{\mu}^{\on{loc},(\{1\}),*} \mathscr{S}_U^{\on{loc},(\{1\})}.
\]
Using that $p_2$ and $p_2^{\on{loc}}\times \text{id}$ are étale, we see that 
\[
\pi^* (p_2^{\on{loc}}\times \text{id})^* (\epsilon_{0}^{\on{loc},(\{1\})}\times \text{id})^* \mathscr{S}_1^{\on{loc},(\{1\})} \boxtimes \mathscr{S}_U^{\on{loc},(\{1\})} 
\cong
\pi^* (p_2^{\on{loc}}\times \text{id})^! (\epsilon_{0}^{\on{loc},(\{1\})}\times \text{id})^* \mathscr{S}_1^{\on{loc},(\{1\})} \boxtimes \mathscr{S}_U^{\on{loc},(\{1\})} 
\]
\[
\overset{\text{BC}}{\longrightarrow} 
p_2^! \pi_1^*  (\epsilon_{0}^{\on{loc},(\{1\})}\times \text{id})^* \mathscr{S}_1^{\on{loc},(\{1\})} \boxtimes \mathscr{S}_U^{\on{loc},(\{1\})}
\]
is nothing but 
\[
\pi^* (p_2^{\on{loc}}\times \text{id})^* (\epsilon_{0}^{\on{loc},(\{1\})}\times \text{id})^* \mathscr{S}_1^{\on{loc},(\{1\})} \boxtimes \mathscr{S}_U^{\on{loc},(\{1\})} 
\cong
p_2^* \pi_1^* (\epsilon_{0}^{\on{loc},(\{1\})}\times \text{id})^* \mathscr{S}_1^{\on{loc},(\{1\})} \boxtimes \mathscr{S}_U^{\on{loc},(\{1\})}
\]
\[
\cong
p_2^! \pi_1^*  (\epsilon_{0}^{\on{loc},(\{1\})}\times \text{id})^* \mathscr{S}_1^{\on{loc},(\{1\})} \boxtimes \mathscr{S}_U^{\on{loc},(\{1\})},
\]
thus the pullback of 
$\mathscr{C}_{\nu}^{\on{loc},W} \boxtimes \text{id}$ 
is simply $\mathscr{C}_{\nu}$, hence the pullback of $T_V^{\on{loc},W}\boxtimes \text{id}$ is $T_V$ by adding them together. 
\end{proof}

\section{$S=T$ for Shimura varieties}

In this section, we assume that $G / \mathbb{Q}$ is a reductive group equipped with a $G_{\mathbb{R}}$-conjugacy class $X$ of homomorphisms
$h: \text{Res}_{\mathbb{C}/\mathbb{R}} \mathbb{G}_m \rightarrow
 G_{\mathbb{R}}$
 that gives rise to a Hodge type Shimura variety. Namely, there is a closed embedding $G \hookrightarrow \text{GSp}(Q, \psi)$ such that $X$ is mapped to the canonical Siegel conjugacy class, where $Q$ is a finite dimensional $\mathbb{Q}$-vector space, $\psi$ is a non-degenerate alternating form on $Q$, and $\text{GSp}(Q, \psi)$ is the similitude sympletic group with respect to $(Q, \psi)$, i.e. the automorphism of $Q$ preserving $\psi$ up to a constant. We fix such an embedding throughout, then by \cite{Deligne1982}, there is a set of tensors 
 $(s_{\alpha}) $ in 
 \[
 Q^{\otimes} :=\underset{s,t \in \mathbb{N}}{\bigoplus} Q^{\otimes s} \otimes (Q^{\vee})^{\otimes t } 
 \]
such that $G$ is the stabilizer of $(s_{\alpha})$, which we also fix from now on.

We can base change $h$ to $\mathbb{C}$ to obtain $h_{\mathbb{C}}: \mathbb{G}_m \times \mathbb{G}_m \rightarrow G_{\mathbb{C}}$, where the two factors $\mathbb{G}_m$ correspond to the identity embedding and the conjugation embedding respectively. Let $\mu: \mathbb{G}_m \rightarrow G_{\mathbb{C}}$ be the restriction of $h_{\mathbb{C}}$ to the first factor, whose conjugacy class has field of definition $E$, which is a number field. 

Moreover, we fix a prime $p$ such that $G$ is hyperspecial at $p$, i.e. there is a reductive group $\mathcal{G}$ over $\mathbb{Z}_{(p)}$ with generic fiber $G$. Let $K=K^p K_p \subset G(\mathbb{A}_f)$ be a compact open subgroup where $K_p = \mathcal{G}(\mathbb{Z}_p) \subset G(\mathbb{Q}_p)$  and $K^p \subset G(\mathbb{A}^p_f)$ are compact open. We assume that $K^p$ is neat for convenience, so that the corresponding Shimura variety is a scheme instead of a stack. By enlarging $Q$, we may assume that $Q$ has a lattice $Q_{\mathbb{Z}_{(p)}}$ over $\mathbb{Z}_{(p)}$ on which $\psi$ induces a perfect pairing 
$Q_{\mathbb{Z}_{(p)}} \times Q_{\mathbb{Z}_{(p)}} \rightarrow \mathbb{Z}_{(p)}$. Moreover, we can assume that the embedding 
$G \hookrightarrow \text{GSp}(Q, \psi)$
induces an embedding
\[
\mathcal{G} \hookrightarrow \text{GSp}(Q_{\mathbb{Z}_{(p)}}, \psi)
\]
as reductive groups over $\mathbb{Z}_{(p)}$, and the tensors $s_{\alpha}$'s extend to tensors of $Q_{\mathbb{Z}_{(p)}}$ whose stabilizer group is $\mathcal{G}$,  see \cite[Section 2.3.2]{Kisin2010}  for details. 

By \cite[Lemma 2.1.2]{Kisin2010}, there is a compact open subgroup $\Tilde{K}^p \subset \text{GSp}(Q, \psi)(\mathbb{A}_f^p)$ such that 
$K^p = \Tilde{K}^p \cap G(\mathbb{A}_f^p)$, and the injection of $G$ into $\text{GSp}(Q, \psi)$ induces a closed embedding 
\[
S_K \hookrightarrow \Tilde{S}_{\Tilde{K}^p\Tilde{K}_p} \times_{\mathbb{Q}} E
\]
of Shimura varieties, where 
$\Tilde{K}_p := \text{GSp}(Q_{\mathbb{Z}_{(p)}}, \psi)(\mathbb{Z}_p)$,
$S_K$ is the canonical model over $E$ of the Shimura variety associated to $(G, X)$ with level $K$, and $\Tilde{S}_{\Tilde{K}^p\Tilde{K}_p}$
is the Siegel variety associated to $\text{GSp}(Q, \psi)$ with level $\Tilde{K}^p\Tilde{K}_p$ (over $\mathbb{Q}$). 

Let us now fix a prime $\mathfrak{p}$ of $E$ lying above $p$. 
By \cite{Kisin2010}, $S_K$ has a canonical smooth integral  model $\mathfrak{S}_K$ over $\Breve{\mathbb{Z}}_p$. It is defined as the normalization of the closure of $S_K$ in $\mathfrak{A}_{\Tilde{K}^p\Tilde{K}_p}$, where 
$\mathfrak{A}_{\Tilde{K}^p\Tilde{K}_p}$ is the canonical integral model of $\Tilde{S}_{\Tilde{K}^p\Tilde{K}_p}$
as moduli space of polarized abelian varieties (over $\Breve{\mathbb{Z}}_p$). By construction, there is a canonical abelian scheme $A$ over $\mathfrak{S}_K$, being the pullback of the universal one on $\mathfrak{A}_{\Tilde{K}^p\Tilde{K}_p}$. We denote $\mathfrak{X}_K$ the $p$-adic completion of $\mathfrak{S}_K$, and $\mathcal{X}_K$ the adic generic fiber of $\mathfrak{X}_K$, i.e. the good reduction locus in $(S_K \times_{E} \Breve{\mathbb{Q}}_p)^{\text{ad}}$. Moreover, we denote 
$\bar{\mathfrak{X}}_K := \mathfrak{S}_{K,\overline{\mathbb{F}_p}}$
the special fiber of $\mathfrak{S}_K$.

\begin{remark}
The integral model $\mathfrak{S}_K$ is really defined over $\mathcal{O}_{E_v}$, where $v$ is a place of $E$ lying over $p$, and similarly for the other spaces in the previous paragraph. We can upgrade everything below to this situation. 
\end{remark}

By construction, the representation $Q$ defines a Betti local system 
\[
\mathcal{V}_B := Q \times^{G(\mathbb{Q})} (X \times G(\mathbb{A}_f)/K)
\]
on 
\[
S_K(\mathbb{C}) \cong G(\mathbb{Q}) \setminus (X \times G(\mathbb{A}_f)/K)
\]
which can be identified with the first homology of $A$ over $S_K(\mathbb{C})$, i.e. $\mathcal{V}_B =(R^1h_{B,*} \mathbb{Q})^{\vee}$
where $h: A \rightarrow S_K$ is the structure map of $A$, and $h_{B,*}$ means that we take  pushforward in the classical topology.  

We can similarly define 
\[
\mathcal{V}_B^{\otimes} := \underset{s,t \in \mathbb{N}}{\bigoplus} \mathcal{V}_B^{\otimes s} \otimes (\mathcal{V}_B^{\vee})^{\otimes t } 
\]
with duals and tensor products in the category of Betti $\mathbb{Q}$-local systems over $S_K$. Similar notation will apply to other categories of local systems, namely étale, de Rham (filtered vector bundles with flat connections) and crystalline ($F$-cyrstals).
Now the tensors $s_{\alpha}$'s define global sections
\[
s_{\alpha,B}: 1 \longrightarrow \mathcal{V}_B^{\otimes}
\]
of $\mathcal{V}_B^{\otimes}$ since they are $G(\mathbb{Q})$-invariant, where $1$ denotes the unit object in the rigid Tannakian category, i.e. the 1 dimensional constant local system. 

Let $\ell$ be a prime that can be either equal or not equal to $p$. 
By \cite[Lemma 2.2.1]{Kisin2010}, $s_{\alpha,B}$'s determine global sections
\[
s_{\alpha,\ell}: 1 \longrightarrow \mathcal{V}_{\ell}^{\otimes} 
\]
with 
$\mathcal{V}_{\ell}:= (R^1h_{\text{ét},*} \mathbb{Q}_\ell )^{\vee}$,
characterized by $s_{\alpha,\ell}$  and $s_{\alpha, B}$ matching under the comparison isomorphism between Betti and étale cohomology, where $h_{\text{ét},*}$ denotes pushforward in the étale topology of $h: A\rightarrow S_K$ (over $E$). Note that the real content of the statement is that the sections defined by the comparison isomorphism descend to $E$.  

We also have the de Rham tensors
\[
s_{\alpha,\text{dR}}: 1 \longrightarrow \mathcal{V}_{\text{dR}}^{\otimes}
\]
where
$\mathcal{V}_{\text{dR}}$ := $(R^1h_{*} \Omega^*_{A/\mathfrak{S}_K})^{\vee}$
is the relative de Rham homology, which is a filtered vector bundle with flat connection. Over $\mathbb{C}$, these are defined by the comparison between the de Rham and Betti cohomology, which descend to $E$ by Deligne's theory of absolute Hodge and the equivalence of analytic vector bundles with a flat connection and algebraic vector bundles with a flat connection with regular singularities, see \cite[Corollary 2.2.2]{Kisin2010}  for details. Moreover, it is proved in $loc.cit.$ that $s_{\alpha, \text{dR}}$'s extend to the integral model. 

We can restrict (the analytification of) $\mathcal{V}_p$ to the good reduction locus $\mathcal{X}_K$, then by \cite[Theorem 1.10]{SCHOLZE2013},  it is a de Rham local system in the sense of \cite{SCHOLZE2013} with the associated filtered vector bundle with flat connection $\mathcal{V}_{\text{dR}}$ (restricted on $\mathcal{X}_K$), i.e. there is an isomorphism
\begin{equation} \label{2840384930}
\mathcal{V}_p \otimes_{\hat{\mathbb{Z}}_p} \mathcal{O} \mathbb{B}_{\text{dR}}  
\cong \mathcal{V}_{\text{dR}} \otimes_{\mathcal{O}_{\mathcal{X}_K}} \mathcal{O} \mathbb{B}_{\text{dR}},
\end{equation}
where $\hat{\mathbb{Z}}_p$ is the sheaf on
the proétale site $\mathbb{X}_{\text{proét}}$
associated to the topological ring $\mathbb{Z}_p$, and 
$\mathcal{O}\mathbb{B}_{\text{dR}}$ 
is the period sheaf as defined in \cite[Definition 6.8]{SCHOLZE2013}  and corrected in \cite{scholze_2016}. See \cite[Section 6]{SCHOLZE2013}  for the definition of other period sheaves used in the following. We also have similar comparison isomorphism between $\mathcal{V}_{p}^{\otimes}$ and $\mathcal{V}_{\text{dR}}^{\otimes}$. By the result of Blasius and Wintenberger (\cite{Blauis}), the isomorphism takes $s_{\alpha, p} \otimes 1$ 
to 
$s_{\alpha, \text{dR}} \otimes 1$ at points defined over number fields, so they are matched globally by the flatness of these sections. 

We have two $\mathbb{B}_{\text{dR}}^+$-local systems 
\[
\mathbb{M}:= \mathcal{V}_p \otimes_{\hat{\mathbb{Z}}_p} \mathbb{B}_{\text{dR}}^+ 
\]
and 
\[
\mathbb{M}_0 :=( \mathcal{V}_{\text{dR}} \otimes_{\mathcal{O}_{\mathcal{X}_K}} \mathcal{O} \mathbb{B}_{\text{dR}}^+)^{\nabla =0}
\]
such that 
\begin{equation} \label{uiwuriwur}
\mathbb{M}\otimes_{\mathbb{B}_{\text{dR}}^+} \mathbb{B}_{\text{dR}} \cong \mathbb{M}_0 \otimes_{\mathbb{B}_{\text{dR}}^+} \mathbb{B}_{\text{dR}},
\end{equation}
which is taking $\nabla =0$ on both sides of (\ref{2840384930}). Moreover, we
have
$\mathbb{M} \subset \mathbb{M}_0$, see \cite[Proposition 7.9]{SCHOLZE2013}. Note that we have used homology instead of cohomology, which corresponds to a reverse of the inclusion. Since the identification (\ref{2840384930}) matches $s_{\alpha,p} \otimes 1$ with $s_{\alpha,\text{dR}} \otimes 1$, we see that the same holds for the identification (\ref{uiwuriwur}). 

Moreover,  by \cite[Corollary 2.2.4]{Caraiani_2017}, we can take the 0th graded piece of (\ref{2840384930}) to obtain the Hodge--Tate sequence
\[
0 \longrightarrow 
(\text{Lie} \ A) \otimes_{\mathcal{O}_{\mathcal{X}_K}} \hat{\mathcal{O}}_{\mathcal{X}_K} (1)
\longrightarrow 
\mathcal{V}_p \otimes_{\mathbb{Z}_p} \hat{\mathcal{O}}_{\mathcal{X}_K}
\longrightarrow 
(\text{Lie} \ A^{\vee} )^{\vee} \otimes_{\mathcal{O}_{\mathcal{X}_K}} \hat{\mathcal{O}}_{\mathcal{X}_K}
\longrightarrow 0,
\]
where 
$\hat{\mathcal{O}}_{\mathcal{X}_K} $ 
is the $p$-adic completion of $\nu^* \mathcal{O}_{\mathcal{X}_{K,\text{ét}}}$ 
with $\nu: \mathcal{X}_{K,\text{proét}} \longrightarrow \mathcal{X}_{K,\text{ét}}$ 
the structure morphism from the proétale site to the étale site.

Let $D(A)$ be the covariant $F$-crystal associated to $A$ on the special fiber $\Bar{\mathfrak{X}}_K$. 
As explained in \cite[Section 7.1.5]{2017arXiv170705700X}, the de Rham tensors give rise to 
\[
s_{\alpha,0} : 1 \longrightarrow D(A)^{\otimes}.
\]
Indeed, as $\mathfrak{S}_K$ is smooth over $\Breve{\mathbb{Z}}_p$ (and $\Breve{\mathbb{Z}}_p$ is unramified), the crystal $D(A)$ is equivalent to the vector bundle with flat connections $\mathcal{V}_{\on{dR}}$, so $s_{\alpha, \on{dR}}$'s give rise to $s_{\alpha, 0}$'s as arrows of crystals. Then it is proved in \cite{Kisin2010} that pointwisely $s_{\alpha,0}$'s are arrows of F-crystals, which gives what we want.

We now want to construct a map 
\[
\on{loc}_p: \mathfrak{S}_K^{\diamond} \longrightarrow 
\on{Sht}_{\mu}
\]
over $\on{Spd}(\Breve{\mathbb{Z}}_p)$. Recall that $\mathfrak{S}_K^{\diamond}$ is the sheafification of the presheaf sending  an affinoid perfectoid  $S=\on{Spa}(R,R^+)$ over $\overline{\mathbb{F}_p}$ to the set of untilts $S^{\sharp}=\on{Spa}(R^{\sharp},R^{\sharp,+})$ over $\on{Spa}(\Breve{\mathbb{Z}}_p,\Breve{\mathbb{Z}}_p)$ together with a map $\on{Spec}(R^{\sharp,+}) \longrightarrow \mathfrak{S}_K$ of schemes over $\Breve{\mathbb{Z}}_p$. 

Given 
$f: \on{Spec}(R^{\sharp,+}) \rightarrow \mathfrak{S}_K$, 
we have the $p$-divisible group 
$f^*A[p^{\infty}]$
over $\on{Spec}(R^{\sharp,+})$,
which by \cite[Theorem 17.5.2]{SW17}  is equivalent to a  finite projective $A_{\on{inf}}(R^{\sharp,+}):= W(R^+)$-module $\mathcal{M}$ together with an isomorphism
\[
\varphi_{\mathcal{M}}:
\varphi^* \mathcal{M} [\frac{1}{\varphi(\xi)}] \overset{\sim}{\rightarrow}
\mathcal{M}[\frac{1}{\varphi(\xi)}]
\]
such that $\mathcal{M} \subset \varphi_{\mathcal{M}}(\varphi^* \mathcal{M}) \subset \frac{1}{\varphi(\xi)}\mathcal{M}$, 
where $\xi \in A_{\on{inf}}(R^{\sharp,+})$ 
is a generator of the kernel of $A_{\on{inf}}(R^{\sharp,+}) \rightarrow R^{\sharp,+}$, and $\varphi$ is the canonical lift of Frobenius on $A_{\on{inf}}(R^{\sharp,+})$. It gives rise to a minuscule $GL_n$-shtuka with one leg over $\on{Spa}(R^{\sharp}, R^{\sharp,+})$ by restricting on $\mathcal{Y}_{[0, \infty)}(R,R^+)$ (and twist by Frobenius), where $n$ is twice the relative dimension of $A$ over $\mathfrak{S}_K$. 

Let us first recall how $\mathcal{M}$ is constructed in \cite[Theorem 17.5.2]{SW17}. Let $G$ be a $p$-divisible group over an integral perfectoid ring  $\mathcal{S}$ (in the sense of \cite[Definition 17.5.1]{SW17}). If $p=0$ in $\mathcal{S}$, then $\mathcal{M}(G)$ is given by Dieudonn\'e theory. If $p\neq 0$ in $\mathcal{S}$,  then Dieudonn\'e theory applied to the restriction of $G$ to $\mathcal{S}/p$ gives us an $A_{\on{crys}}(\mathcal{S})$-module $\mathcal{M}_{\on{crys}}(G)$ together with an isomorphism
\[
\varphi^*\mathcal{M}_{\on{crys}}(G) [\frac{1}{p}] \overset{\sim}{\rightarrow} \mathcal{M}_{\on{crys}}(G)[\frac{1}{p}],
\] 
where $A_{\on{crys}}(\mathcal{S})$
is the universal $p$-adically complete divided power thickening of $\mathcal{S}/p$. Now $\mathcal{M}(G) \subset \mathcal{M}_{\on{crys}}(G)$ is characterized by being the largest submodule of $\mathcal{M}_{\on{crys}}(G)$  such that for every map $\mathcal{S} \rightarrow V$ with $V$ an integral perfectoid valuation ring with algebraically closed fraction field, its image in $\mathcal{M}_{\on{crys}}(G_V)$
lies in $\mathcal{M}(G_V)$, where $\mathcal{M}(G_V)$ is the Dieudonn\'e module if $V$ is of characteristic $p$. If $p\neq 0$ in $V$ and $V=\mathcal{O}_{C}$ with $C$ the fraction field of $V$, $\mathcal{M}(G_V)$ is constructed as in \cite[Theorem 14.4.1]{SW17}. For general $V$ not of characteristic $p$, let $k$ be the residue field of $\mathcal{O}_{C}$ with $C$ the fraction field of $V$, and $\Bar{V}\subset k$ the image of $V$ in $k$, then $\mathcal{M}(G_V)$ is defined as $\mathcal{M}(G_{\mathcal{O}_{C}})\times_{\mathcal{M}({G_k})} \mathcal{M}(G_{\Bar{V}})$, in other words, $\mathcal{M}(G_V)$ is the largest submodule of $\mathcal{M}(G_{\mathcal{O}_{C}})$ 
whose image in $\mathcal{M}(G_k)$ lies in $\mathcal{M}(G_{\Bar{V}})$. 

In summary, the construction of $\mathcal{M}(G)$ is to use v-descent to reduce to Dieudonn\'e theory and the classification over $\mathcal{O}_{C}$. The key input is then the classification in terms of Hodge--Tate periods as in \cite{ScholzeWe2013}, and Fargues' theorem  (\cite[Theorem 14.1.1]{SW17}) that establishes the equivalence between de Rham pairs and Breuil--Kisin--Fargues modules. 

We now prove that the $s_{\alpha}$'s give rise to tensors on $\mathcal{M}^{\otimes}$.

\begin{proposition} \label{orwueirueo}
Let $R^{\sharp,+}$ be an integral perfectoid ring over $\Breve{\mathbb{Z}}_p$, $f: \on{Spec}(R^{\sharp,+}) \rightarrow \mathfrak{S}_K$ a map over $\Breve{\mathbb{Z}}_p$,  and $G:=f^*A[p^{\infty}]$ the induced $p$-divisible group over $R^{\sharp,+}$. Let $\mathcal{M}$ be the $\varphi$-$A_{\on{inf}}(R^{\sharp,+})$-module corresponding to $G$, then there exists
\[
s_{\alpha, \prism}: 1 \longrightarrow \mathcal{M}^{\otimes} 
\]
such that under the identification (see \cite[Theorem 17.5.2]{SW17})
\[
\mathcal{M} \otimes_{A_{\on{inf}}(R^{\sharp,+})} A_{\on{crys}}(R^{\sharp,+}) \cong \mathcal{M}_{\on{crys}}
\]
 $s_{\alpha,\prism} \otimes 1$ is mapped to $s_{\alpha, 0}$ (restricted from $\Bar{\mathfrak{X}}_K$ to $R^{\sharp,+}/p$),
 where $\mathcal{M}_{\on{crys}}$ is the Dieudonn\'e crystal of $G$ restricted to $R^{\sharp,+}/p$. Moreover, when $R^{\sharp,+} = \mathcal{O}_{C}$, then under the identification
 \[
 \mathcal{M}\otimes_{A_{\on{inf}}(\mathcal{O}_{C})}A_{\on{inf}}(\mathcal{O}_{C}) [\frac{1}{\mu}]  \cong T_pG \otimes_{\mathbb{Z}_p} A_{\on{inf}}(\mathcal{O}_{C}) [\frac{1}{\mu}] 
 \]
from (the dual of) \cite[Theorem 14.9.1 (2)]{SW17}, 
 $s_{\alpha,\prism} \otimes 1$ is sent to $s_{\alpha,p}$.
\end{proposition}

\begin{proof}
We have 
\[
s_{\alpha,0} : 1 \longrightarrow \mathcal{M}_{\on{crys}}^{\otimes}
\]
by pulling back 
\[
s_{\alpha,0} : 1 \longrightarrow D(A)^{\otimes}
\]
along
$\on{Spec}(R^{\sharp,+}/p) \longrightarrow \bar{\mathfrak{X}}_K$, the restriction of $f$ to the characteristic $p$ fibers, and
we are aiming to show that 
$s_{\alpha,0} $
factorizes through $\mathcal{M}^{\otimes}$. By construction of $\mathcal{M}$ that we have recalled above, we need to show that for every map $R^{\sharp,+} \rightarrow V$ with V an integral perfectoid valuation ring with algebraically closed fraction field,
\[
s_{\alpha,0}|_{V}: 1 \longrightarrow \mathcal{M}_{\on{crys}}(G_V)^{\otimes}
\]
factorizes through $\mathcal{M}(G_V)^{\otimes}$. If $V$ has characteristic $p$, then
$W(V)$ is a (pro-)PD-thickening and $\mathcal{M}(G_V)=D(A_V)(W(V))$. We can evaluate the crystal $D(A_V)$ at $W(V)$  
to obtain the desired factorization. 

If $V$ has mixed characteristic and $V=\mathcal{O}_{C}$ with $C$ the fraction field of $V$, then $\mathcal{M}(G_V)$ is the Breuil--Kisin--Fargues module corresponding to the Hodge--Tate periods of $G_V$. Let us recall its construction first. We have the Hodge--Tate exact sequence
\[
0 \longrightarrow \on{Lie} \ G_V \otimes_{\mathcal{O}_{C}} C(1) \longrightarrow T_pG_V\otimes_{\mathbb{Z}_p} C \longrightarrow (\on{Lie} \ G_V^*)^* \otimes_{\mathcal{O}_{C}} C \longrightarrow 0,
\]
which gives rise to a $\mathbb{B}_{\on{dR}}^+(C)$-submodule
$\Xi \subset T_p G_V \otimes_{\mathbb{Z}_p} \mathbb{B}_{\on{dR}}^+(C)$ 
characterized by 
$\Xi \ \on{mod} \ t \subset T_p G_V \otimes_{\mathbb{Z}_p} C$ 
being 
$\on{Lie} \ G_V \otimes_{\mathcal{O}_{C}} C(1) \subset T_p G_V \otimes_{\mathbb{Z}_p} C$, where $t \in \mathbb{B}_{\on{dR}}^+(C)$ is a generator of the kernel $\mathbb{B}_{\on{dR}}^+(C) \rightarrow C$. We have now a de Rham pair $(\Xi, T_p G_V)$, which by \cite[Theorem 14.1.1]{SW17}  is equivalent to a minuscule Breuil--Kisin--Fargues module $\mathcal{M}(G_V)$. Moreover, we have by  \cite[Proposition 14.8.1]{SW17} that 
\[
\Xi \subset T_pG_V \otimes_{\mathbb{Z}_p} \mathbb{B}_{\on{dR}}(C)
\]
is the same as 
\[
\mathcal{M}_{\on{crys}}(G_V) \otimes_{A_{\on{crys}}(V)} \mathbb{B}_{\on{dR}}^+(C) \subset T_pG_V \otimes_{\mathbb{Z}_p} \mathbb{B}_{\on{dR}}(C)
\]
under a canonical identification 
\[
\mathcal{M}_{\on{crys}}(G_V) \otimes_{\on{A}_{\on{crys}}(V)} \mathbb{B}_{\on{dR}}(C) \cong T_pG_V \otimes_{\mathbb{Z}_p} \mathbb{B}_{\on{dR}}(C),
\]
where the identification actually holds over a smaller period ring $\on{B}_{\on{crys}}$. 

Now as $G_V$ is the pullback of $A[p^{\infty}]$ along a map $\on{Spec}(V) \rightarrow \mathfrak{S}_K$, we have 
\[
s_{\alpha,p} : 1 \longrightarrow (T_pG_V)^{\otimes}
\]
which by the following lemma matches with 
\[
s_{\alpha,0} : 1 \longrightarrow \mathcal{M}_{\on{crys}} (G_V)^{\otimes}
\]
under the identification 
\[
\mathcal{M}_{\on{crys}}(G_V) \otimes_{\on{A}_{\on{crys}}(V)} \mathbb{B}_{\on{dR}}(C) \cong T_pG_V \otimes_{\mathbb{Z}_p} \mathbb{B}_{\on{dR}}(C).
\]
It follows that $s_{\alpha,0}$ and $s_{\alpha,p}$ glue to an arrow of de Rham pairs
\[
s_{\alpha, \on{BdR}}: 1 \longrightarrow (\Xi,T_pG_V)^{\otimes},
\]
which by \cite[Theorem 14.1.1]{SW17} is equivalent to 
\[
s_{\alpha, \prism}: 1 \longrightarrow \mathcal{M}(G_V)^{\otimes}.
\]
The relation with $s_{\alpha,0}$ is clear from the construction. 

Lastly, if $V$ is of mixed characteristic but not necessarily equals to $\mathcal{O}_{C}$, then $V$ is the pullback of $\Bar{V} \subset k$ along $\mathcal{O}_{C} \rightarrow k$, where $k$ is the residue field of $\mathcal{O}_{C}$ and $\Bar{V}$ is the image of $V$ in $k$. We have a similar description of $\mathcal{M}(G_V)$, being 
$\mathcal{M}(G_{\mathcal{O}_{C}})\times_{\mathcal{M}({G_k})} \mathcal{M}(G_{\Bar{V}})$. We have already produced $s_{\alpha, \prism}$ for $\mathcal{M}(G_{\mathcal{O}_{C}})$, $\mathcal{M}(G_{\Bar{V}})$ and $\mathcal{M}(G_{k})$ individually, it remains to see that they glue to $\mathcal{M}(G_V)$. This is clear as they are pullback from the commutative diagram
\[
\begin{tikzcd}
& \on{Spec}(\mathcal{O}_{C}) \arrow[rd] &
\\
\on{Spec}(k) \arrow[rr] \arrow[ru] \arrow[rd] & &  \mathfrak{S}_K
\\
& \on{Spec}(\Bar{V}) \arrow[ru] &
\end{tikzcd}
\]
and we have the identification 
\[
\mathcal{M}(G_{\mathcal{O}_{C}}) \otimes_{\on{A}_{\on{inf}}(\mathcal{O}_{C})} W(k) \cong 
(\mathcal{M}(G_{\mathcal{O}_{C}}) \otimes_{\on{A}_{\on{inf}}(\mathcal{O}_{C})} \on{A}_{\on{crys}}(\mathcal{O}_{C})) \otimes_{\on{A}_{\on{crys}}} W(k)
\]
\[
\cong
\mathcal{M}_{\on{crys}}(G_{\mathcal{O}_{C}/p})  \otimes_{\on{A}_{\on{crys}}} W(k)
\cong 
\mathcal{M}_{\on{crys}}(G_k).
\]

\end{proof}

\begin{lemma}
Let $f: \on{Spec}(\mathcal{O}_{C}) \longrightarrow \mathfrak{S}_K$ be a map over $\Breve{\mathbb{Z}}_p$, and $H_{\mathcal{O}_{C}}$ be the pullback of $A[p^{\infty}]$ along $f$, then (pullback of) $s_{\alpha,0} \otimes 1$ is mapped to $s_{\alpha, p} \otimes 1$
under the canonical isomorphism 
\[
\mathcal{M}_{\on{crys}}(H_V) \otimes_{A_{\on{crys}}(V)} \mathbb{B}_{\on{dR}}(C) \cong T_pH_V \otimes_{\mathbb{Z}_p} \mathbb{B}_{\on{dR}}(C). 
\]
\end{lemma}

\begin{proof}
Using the canonical duality between crystalline cohomology (resp. étale cohomology) of abelian varieties and  Dieudonn\'e modules (resp. Tate modules) of the associated $p$-divisible groups (see \cite[Proposition 14.9.3]{SW17}), it is enough to prove the corresponding statements for abelian varieties, namely $(s_{\alpha,p})^{\vee}\otimes 1$ is mapped to $(s_{\alpha,0})^{\vee} \otimes 1$ under the étale-crystalline comparison
\[
H^1_{\text{ét}}(A_{C}, \mathbb{Z}_p) \otimes_{\mathbb{Z}_p} \mathbb{B}_{\on{dR}}(C) \cong 
H^1_{\on{crys}}(A_{\mathcal{O}_{C}/p}/ A_{\on{crys}}(\mathcal{O}_{C})) \otimes_{A_{\on{crys}}(\mathcal{O}_{C})} \mathbb{B}_{\on{dR}}(C).
\]

Since $\mathcal{O}_{C}$ is local, $f$ factorizes through an affine open $\on{Spec}(T)$ of $\mathfrak{S}_K$, which is smooth over $\Breve{\mathbb{Z}}_p$. By smoothness, we can lift $f$ to 
\[
g: \on{Spec}(A_{\on{crys}}(\mathcal{O}_{C})) \longrightarrow \on{Spec}(T),
\]
then 
\begin{equation} \label{rehjhrejhr}
H^1_{\on{crys}}(A_{\mathcal{O}_{C}/p}/ A_{\on{crys}}(\mathcal{O}_{C})) \cong 
H^1_{\on{dR}}(g^*A) 
\cong g^* (\mathcal{V}_{\on{dR}})^{\vee}
\end{equation}
where $A_{\mathcal{O}_{C}/p}$ is the restriction of $f^*A$ to $\on{Spec}(\mathcal{O}_{C}/p)$. 

Moreover, let 
\[
h: \on{Spec}(\mathbb{B}_{\on{dR}}^+(C)) \longrightarrow \on{Spec}(T)
\]
be the composition of $g$ with the canonical map 
\[
r: \on{Spec}(\mathbb{B}_{\on{dR}}^+(C)) \longrightarrow
\on{Spec}(A_{\on{crys}}(\mathcal{O}_{C})),
\]
which is a lift of (restriction of) $f $. It follows from the proof of \cite[Theorem 13.8]{Bhatt2018}   that 
\begin{equation} \label{urieuorewuir}
H^1_{\on{crys}}(A_{C}/\mathbb{B}_{\on{dR}}^+(C)) \cong
H^1_{\on{dR}}(h^*A) \cong h^* (\mathcal{V}_{\on{dR}})^{\vee}
\end{equation}
where $A_{C} := (f^*A)_{C}$, and $H^1_{\on{crys}}(A_{C}/\mathbb{B}_{\on{dR}}^+(C))$
is defined as in \cite[Theorem 13.1]{Bhatt2018}. 

Now under the identification (\ref{urieuorewuir}) and (\ref{rehjhrejhr}), the canonical isomorphism 
\begin{equation} \label{37403894380}
H^1_{\on{crys}}(A_{C}/\mathbb{B}_{\on{dR}}^+(C)) \cong 
H^1_{\on{crys}}(A_{\mathcal{O}_{C}/p}/ A_{\on{crys}}(\mathcal{O}_{C})) \otimes_{A_{\on{crys}}(\mathcal{O}_{C})} \mathbb{B}_{\on{dR}}^+(C) 
\end{equation}
from \cite[Lemma 13.11]{Bhatt2018}  is nothing but
\[
r^* g^*(\mathcal{V}_{\on{dR}})^{\vee} \cong h^*(\mathcal{V}_{\on{dR}})^{\vee}. 
\]
We have 
\[
H^1_{\on{crys}}(A_{C}/\mathbb{B}_{\on{dR}}^+(C)) \cong 
f^* \mathbb{M}_0^{\vee},
\]
under which $f^*(s_{\alpha,\on{dR}})^{\vee}$ is mapped to $h^*(s_{\alpha, \on{dR}})^{\vee}$ using the identification  (\ref{urieuorewuir}). We now abuse the notation by writing $(s_{\alpha, \on{dR}})^{\vee}$  the corresponding tensors on $H^1_{\on{crys}}(A_{C}/\mathbb{B}_{\on{dR}}^+(C))$. 
Since $(s_{\alpha,0})^{\vee}$ is $g^*(s_{\alpha,\on{dR}})^{\vee}$,
it follows that the isomorphism (\ref{37403894380}) takes $(s_{\alpha, \on{dR}})^{\vee}$ to $(s_{\alpha,0})^{\vee} \otimes 1$. 

On the other hand, we have the canonical isomorphism 
\begin{equation} \label{12380270832}
H^1_{\text{ét}}(A_{C}, \mathbb{Z}_p) \otimes_{\mathbb{Z}_p} \mathbb{B}_{\on{dR}}(C) \cong 
H^1_{\on{crys}}(A_{C}/\mathbb{B}_{\on{dR}}^+(C)) \otimes_{\mathbb{B}_{\on{dR}}^+(C)} \mathbb{B}_{\on{dR}}(C)
\end{equation}
from \cite[Theorem 13.1]{Bhatt2018}, which is nothing but the pullback of isomorphism (\ref{uiwuriwur}) along $f$. Hence $(s_{\alpha,p})^{\vee} \otimes 1$ is mapped to $(s_{\alpha,\on{dR}})^{\vee} \otimes 1$ through (\ref{12380270832}). Now the composition of isomorphisms (\ref{12380270832}) and (\ref{37403894380}) is 
\[
H^1_{\text{ét}}(A_{C}, \mathbb{Z}_p) \otimes_{\mathbb{Z}_p} \mathbb{B}_{\on{dR}}(C) \cong 
H^1_{\on{crys}}(A_{\mathcal{O}_{C}/p}/ A_{\on{crys}}(\mathcal{O}_{C})) \otimes_{A_{\on{crys}}(\mathcal{O}_{C})} \mathbb{B}_{\on{dR}}(C),
\]
which is exactly the crystalline comparison isomorphism, see \cite[Proposition 14.7.2]{SW17}. It 
matches $(s_{\alpha,p})^{\vee}\otimes 1$ with $(s_{\alpha,0})^{\vee} \otimes 1$ as the two intermediate isomorphisms do.  
\end{proof}

We now have a $\mathcal{G} \times_{\mathbb{Z}_{(p)}} A_{\on{inf}}(R^{\sharp,+})$-quasitorsor $\mathcal{E}$ over $\on{Spec}(A_{\on{inf}}(R^{\sharp,+}))$ that parametrizes isomorphisms between $Q_{\mathbb{Z}_{(p)}} \otimes_{\mathbb{Z}_{(p)}} A_{\on{inf}}(R^{\sharp,+})$ 
and $\mathcal{M}$ that preserves the Hodge tensors, i.e. for any $A_{\on{inf}}(R^{\sharp,+})$-algebra $T$,
\[
\mathcal{E}(T) = 
\{\beta: Q_{\mathbb{Z}_{(p)}} \otimes_{\mathbb{Z}_{(p)}} T \cong
\mathcal{M} \otimes_{A_{\on{inf}}(R^{\sharp,+})} T  \ \ | \beta(s_{\alpha} \otimes 1) = s_{\alpha, \prism} \otimes 1
\}.
\]

\begin{proposition}
$\mathcal{E}$ is a $\mathcal{G} \times_{\mathbb{Z}_{(p)}} A_{\on{inf}}(R^{\sharp,+})$-torsor over $\on{Spec}(A_{\on{inf}}(R^{\sharp,+}))$. 
\end{proposition}

\begin{proof}
We write $\mathcal{E}_{T}$ as the base change of $\mathcal{E}$ to $\on{Spec}(T)$ for any $A_{\on{inf}}(R^{\sharp,+})$-algebra $T$.

Let us consider the ringed site $(\mathcal{C}_{R^+},\mathcal{O})$ consists of perfect $R^+$-algebras with v-topology and sheaf of rings $\mathcal{O} (-) := W(-)$. We can view $\mathcal{E}$ as a sheaf 
\footnote{It is a sheaf as v-topology is subcanonical on perfect schemes, see  \cite[Remark 4.2]{Bhatt2016}. } 
on $(\mathcal{C}_{R^+}, \mathcal{O})$ 
with value $\mathcal{E}(W(S))$ for a perfect $R^+$-algebra $S$.  

It follows from \cite[Theorem 4.1 (ii)]{Bhatt2016}  that (finite) locally free  sheaves on $(\mathcal{C}_{R^+},\mathcal{O})$ are the same as vector bundles on $W(R^+)$, thus $\mathcal{G}$-torsors on $(\mathcal{C}_{R^+},\mathcal{O})$ (in the sense of \cite[Tag 03AH]{stacks-project}) are the same as $\mathcal{G}$-torsors on $W(R^+)$ by the Tannakian formalism.   Thus it is enough to prove that $\mathcal{E}$ is a $\mathcal{G}$-torsor on $(\mathcal{C}_{R^+},\mathcal{O})$, i.e. $\mathcal{E}_{W(S)} $
has a section for a v-cover $S$ of $R^+$. By the equivalence between torsors on  $(\mathcal{C}_S,\mathcal{O})$ and $W(S)$, it is enough to prove that  $\mathcal{E}_{W(S)} $
is a torsor for some v-cover $S$ of $R$, which 
 we can choose to be a (possibly infinite) product of perfect valuation rings with algebraically closed fraction fields, and we quickly reduce to $S=V$ a perfect valuation ring with algebraically closed fraction field.

Let $\xi \in A_{\on{inf}}(R^{\sharp,+})$ be a generator of the kernel of $A_{\on{inf}}(R^{\sharp,+}) \twoheadrightarrow R^{\sharp,+} $, then $\xi \in W(V)$ using the map $R \rightarrow V$. Taking the completion if necessary, we can assume that $V$ is $\xi \ \on{mod} \ p$-adically complete, so $V^{\sharp} =W(V)/\xi$ is integral perfectoid, which is also a valuation ring with algebraically closed fraction field. We have $R^{\sharp,+} \rightarrow V^{\sharp}$, and let $G_{V^{\sharp}}$ be the pullback of $A[p^{\infty}]$ along 
\[
\on{Spec}(V^{\sharp}) \longrightarrow \on{Spec}(R^{\sharp,+}) \longrightarrow \mathfrak{S}_K,
\]
then 
\[
\mathcal{E}_{W(V)} = \on{Isom}_{s_{\alpha}} (Q_{\mathbb{Z}_{(p)}}\otimes_{\mathbb{Z}_{(p)}} W(V), \mathcal{M}(G_{V^{\sharp}})).
\]

If the characteristic of $V^{\sharp}$ is $p$, then $\mathcal{E}_{W(V)}$ is the pullback of (the homology version and not mesmerizing trivilization of Tate twist) $\mathcal{E}_{\on{crys}}$ as defined in \cite[Corollary 7.1.14]{2017arXiv170705700X}, which is proved in $loc.cit.$ to be a torsor. 

If $V^{\sharp}$ is of mixed characteristic, and $V^{\sharp} = \mathcal{O}_{C}$ with $C$ the fraction field of $V^{\sharp}$,  then we have 
\[
\mathcal{M}(G_{V^{\sharp}}) \otimes_{A_{\on{inf}}(V^{\sharp})} 
A_{\on{inf}}(V^{\sharp}) [\frac{1}{\mu}]
\cong 
T_pG_{V^{\sharp}} \otimes_{\mathbb{Z}_p} A_{\on{inf}}(V^{\sharp}) [\frac{1}{\mu}]
\]
that matches $s_{\alpha,\prism} \otimes 1$ with $s_{\alpha, p} \otimes 1$, where $\mu = [\epsilon] -1 \in A_{\on{inf}}(V^{\sharp})$ with $\epsilon =(1, \zeta_{p}, \zeta_{p^2},\cdots) \in \mathcal{O}_{C}^{\flat}$ 
for a chosen compatible $p$-power roots of unity $\zeta_{p^{\bullet}}$. But we already know that $s_{\alpha, p}$ matches with $s_{\alpha, B}$ under Betti-étale comparison, from which we easily deduce an isomorphism 
\[
T_pG_{V^{\sharp}} \cong Q_{\mathbb{Z}_{(p)}}\otimes_{\mathbb{Z}_{(p)}} \mathbb{Z}_p
\]
that matches $s_{\alpha, p}$ with $s_{\alpha} \otimes 1$. Thus 
$\mathcal{E}_{A_{\on{inf}}(V^{\sharp}) [\frac{1}{\mu}]}$ 
is a torsor. On the other hand, 
we have another canonical isomorphism
\[
\mathcal{M}(G_{V^{\sharp}}) \otimes_{A_{\on{inf}}(V^{\sharp})} A_{\on{crys}}(V^{\sharp}) \cong 
\mathcal{M}_{\on{crys}}(G_{V^{\sharp}}), 
\]
matching $s_{\alpha,\prism} \otimes 1$ with $s_{\alpha,0}$,
so $\mathcal{E}_{A_{\on{crys}}(V^{\sharp})}$ is a torsor as it is again the pullback of torsor $\mathcal{E}_{\on{crys}}$ in \cite[Corollary 7.1.14]{2017arXiv170705700X}. We know from the proof of \cite[Corollary 7.1.14]{Bhatt2018}  that $A_{\on{inf}}(V^{\sharp}) \rightarrow A_{\on{inf}}(V^{\sharp})[\frac{1}{p}]^{\wedge}_{\mu}$ factors through $A_{\on{crys}}(V^{\sharp})$, so $\mathcal{E}_{A_{\on{inf}}(V^{\sharp})[\frac{1}{p}]^{\wedge}_{\mu}}$
is a torsor. Now Beauville--Laszlo (see \cite[Lemma 7.2]{2018arXiv180406356A}  for example) applied to $A_{\on{inf}}(V^{\sharp})[\frac{1}{p}]$ tells us that $\mathcal{E}_{A_{\on{inf}}(V^{\sharp})[\frac{1}{p}]}$ is a torsor. 

We have seen that $\mathcal{E}_{A_{\on{inf}}(V^{\sharp})}$ 
is a torsor over the open subset 
\[
U:=
\on{Spec}(A_{\on{inf}}(V^{\sharp})[\frac{1}{p}]) \bigcup \on{Spec}(A_{\on{inf}}(V^{\sharp})[\frac{1}{\mu}]),
\]
whose complement is the unique closed point of $\on{Spec}(A_{\on{inf}}(V^{\sharp}))$ 
since $A_{\on{inf}}(V^{\sharp})/(p,\mu) \cong \mathcal{O}_{C}^{\flat} / (\epsilon-1)$ where $\epsilon -1$ is a pseudouniformizer in $\mathcal{O}_{C}^{\flat}$. Now it follows from \cite[Proposition 8.5]{2018arXiv180406356A}  and \cite[Lemma 14.2.3]{SW17}  that torsors (with respect to a reductive group) on $U$ are equivalent to torsors on $\on{Spec}(A_{\on{inf}}(V^{\sharp}))$. Alternatively, we can use the Tannakian formalism,  \cite[Lemma 14.2.3]{SW17}  and the proof of \cite[Proposition 7.3]{2018arXiv180406356A}  (which shows the exactness of equivalence in  \cite[Lemma 14.2.3]{SW17}) to obtain the equivalence between torsors (without the reductive hypothesis) on U and $\on{Spec}(A_{\on{inf}}(V^{\sharp}))$. It follows that $\mathcal{E}_U$ is a trivial torsor as
$\on{Spec}(A_{\on{inf}}(V^{\sharp}))$
is strictly Henselian.

We can now conclude that 
$\mathcal{E}_{A_{\on{inf}}(V^{\sharp})}$
is a trivial torsor. As $\mathcal{E}_{A_{\on{inf}}(V^{\sharp})}$ is by construction a quasi-torsor, it is enough to show that it has a section over $A_{\on{inf}}(V^{\sharp})$. We have seen that $\mathcal{E}_{U}$ is a trivial torsor, so there is a section 
$U \overset{s}{\longrightarrow }\mathcal{E}_{A_{\on{inf}}(V^{\sharp})}$ giving rise to a map of $A_{\on{inf}}(V^{\sharp})$-algebras
\[
s^*: 
\Gamma(\mathcal{E}_{A_{\on{inf}}(V^{\sharp})}, \mathcal{O}_{\mathcal{E}_{A_{\on{inf}}(V^{\sharp})}}) \longrightarrow \Gamma(U,\mathcal{O}_U)=A_{\on{inf}}(V^{\sharp}),
\]
where the last equality follows from the proof of \cite[Lemma 4.6]{Bhatt2018}. 
Now we note that $\mathcal{E}$ 
is affine over $\on{Spec}(A_{\on{inf}}(R^{\sharp,+}))$
since it is a closed subscheme of a trivial $\on{GL}(\mathcal{M})$-torsor, so the map $s^*$ defines a section
\[
\on{Spec}(A_{\on{inf}}(V^{\sharp})) \longrightarrow \mathcal{E}_{A_{\on{inf}}(V^{\sharp})}
\]
proving the claim. 

Lastly, if $V^{\sharp}$ is of mixed characteristic without being equal to $\mathcal{O}_{C}$, then $V^{\sharp} $ is the pullback of $\overline{V^{\sharp}} \hookrightarrow k$ along $\mathcal{O}_{C} \rightarrow k$, where $k$ is the residue field of $\mathcal{O}_{C}$ and $\overline{V^{\sharp}}$ is the image of $V^{\sharp}$ in $k$. Now $\mathcal{E}_{A_{\on{inf}}(V^{\sharp})}$
is equivalent to 
\footnote{See the beginning of the proof of \cite[Theorem 17.5.2]{SW17}. }
$\mathcal{E}_{W(\overline{V^{\sharp}})} $ and 
$ \mathcal{E}_{\on{A}_{\on{inf}}(\mathcal{O}_{C})}$ 
together with an identification $(\mathcal{E}_{W(\overline{V^{\sharp}})})|_{W(k)} \cong (\mathcal{E}_{\on{A}_{\on{inf}}(\mathcal{O}_{C})})|_{W(k)}$,
which is a torsor as we have already seen $\mathcal{E}_{\on{A}_{\on{inf}}(\mathcal{O}_{C})}$ and $\mathcal{E}_{W(\overline{V^{\sharp}})}$
are torsors. 
\end{proof}

We now come back to our construction of the map 
\[
\on{loc}_p : \mathfrak{S}_K^{\diamond} \longrightarrow \on{Sht}_{\mu}.
\]
Given a characteristic $p$ affinoid perfectoid space $S=\on{Spa}(R,R^+)$ over $\Breve{\mathbb{Z}}_p$,  and an untilt $S^{\sharp}=\on{Spa}(R^{\sharp},R^{\sharp,+})$ together with a map $\on{Spec}(R^{\sharp,+}) \rightarrow \mathfrak{S}_K$ over $\Breve{\mathbb{Z}}_p$, we have constructed a $\mathcal{G} \times_{\mathbb{Z}_p} A_{\on{inf}}(R^{\sharp,+})$-torsor
over $\on{Spec}(A_{\on{inf}}(R^{\sharp,+}))$ 
parametrizing isomorphisms between $Q_{\mathbb{Z}_{(p)}} \otimes_{\mathbb{Z}_{(p)}} A_{\on{inf}}(R^{\sharp,+})$
and $\mathcal{M}$ preserving Hodge tensors. Now the $\varphi$-module structure on $\mathcal{M}$ (and $s_{\alpha, \prism}$ are compatible with the $\varphi$-structure) gives us a modification
\[
\varphi_{\mathcal{E}} : 
\varphi^*\mathcal{E} \dashrightarrow \mathcal{E}
\]
at $\varphi(\xi)$, where $\xi \in A_{\on{inf}}(R^{\sharp,+})$ is a generator of the kernel $A_{\on{inf}}(R^{\sharp,+}) \rightarrow R^{\sharp,+}$. In other words, $\varphi_{\mathcal{E}}$ is an isomorphism on $A_{\on{inf}}(R^{\sharp,+})[\frac{1}{\varphi(\xi)}]$ (meromorphicity at $\xi$ is automatic at the algebraic setting). We then obtain a $\mathcal{G}_{\mathbb{Z}_p}$-shtuka by twisting $\mathcal{E}$ by $\on{Frob}_S^{-1}$ and restricting to $\mathcal{Y}_{[0,\infty)}(S)$. More precisely, let
\[
\mathcal{P}:= (\on{Frob}_S^{-1})^* \mathcal{E} |_{\mathcal{Y}_{[0,\infty)}(S)},
\]
then $\varphi_{\mathcal{E}}$
gives rise to a modification
\[
\varphi_{\mathcal{P}}: \on{Frob}_S^*(\mathcal{P}) \dashrightarrow \mathcal{P}
\]
at the untilt $S^{\sharp}$, i.e. a shtuka with one leg at $S^{\sharp}$.

\begin{lemma}
The modification $\varphi_{\mathcal{P}}$ is bounded by $\mu$, where $\mu$ is, by abuse of notation, a dominant cocharacter in the conjugacy class of $\mu$, the restriction of the Deligne homomorphism $h_{\mathbb{C}}$ to the first factor. 
\end{lemma}

\begin{proof}
Since $\mathcal{G}$ is a subgroup of $\on{GSp}(Q,\psi)$, it is enough to prove that the corresponding modification between  $\on{GSp}(Q,\psi)$-torsors  is bounded by $\mu$ (viewed as the a cocharacter of $\on{GSp}(Q,\psi)$). In other words,  we need to prove that the resulting sympletic shtuka is bounded by the dominant cocharacter $(1^{(g)},0^{(g)})$ of $\on{Gsp(Q,\psi))}$ with $2g=\on{dim} Q$. 

Since boundedness  is a pointwise condition on $R^{\sharp,+}$, we can assume that $R^{\sharp,+} = \mathcal{O}_{C}$ for an algebraically closed field $C$. If $C$ is of mixed characteristic, then we know from \cite[Proposition 12.4.6]{SW17}  and remark 12.4.7 that the relative position of $\varphi_{\mathcal{E}}$ is governed by the Hodge--Tate period $\text{Lie} \ G_{\mathcal{O}_{C}} \otimes_{\mathcal{O}_{C}} C(1) \subset T_p  G_{\mathcal{O}_{C}} \otimes_{\mathbb{Z}_p} C$ (and the corresponding Hodge tensors), but $\text{Lie} \ G_{\mathcal{O}_{C}}$ corresponds to the Hodge filtration on the Betti homology of $A_{C}$, and there are isomorphisms respecting Hodge tensors and filtrations, see \cite[Lemma 2.3.6]{Caraiani_2017}  for example. 

If $C$ is of equal characteristic,  then $\mathcal{M}$ is the (covariant) Dieudonn\'e crystal associated to $G_{\mathcal{O}_{C}}$ evaluated at the initial  PD-thickening $W(\mathcal{O}_{C})$ of $\mathcal{O}_{C}$, which is the de Rham homology of $A_{W(\mathcal{O}_{C})}$. We can further reduce to geometric points $k$ of $\mathcal{O}_{C}$, then it is well-known that the relative position of the $F$-structure on $\mathcal{M}_{W(k)}$ is governed by the Hodge filtration on the de Rham cohomology of $A_{W(k)}$, which again can be compared to the Hodge filtration on Betti-cohomology respecting Hodge tensors. 
\end{proof} 

We have now associated a $\mathcal{G}_{\mathbb{Z}_p}$-shtuka $\mathcal{P}$ with one leg at $S^{\sharp}$ bounded by $\mu$ to each element of $\mathfrak{S}_K(R^{\sharp,+})$. This defines a functor from the presheaf on affinoid perfectoid spaces over $\overline{\mathbb{F}_p}$, which sends $S$ to untilts $S^{\sharp}=\on{Spa}(R^{\sharp}, R^{\sharp,+})$ together with elements in 
$\mathfrak{S}_K(R^{\sharp,+})$, to the v-stack of moduli spaces of $\mathcal{G}_{\mathbb{Z}_p}$-shtukas over $\Breve{\mathbb{Z}}_p$. Since $\mathfrak{S}_K^{\diamond}$ is the sheafification of the presheaf, we have  the desired
\[
\on{loc}_p : \mathfrak{S}_K^{\diamond} \longrightarrow \on{Sht}_{\mu}.
\]

Let $p-\on{Isog}$ be the presheaf on schemes over $\Breve{\mathbb{Z}}_p$  sending a scheme $T$ to the set of triples $(x,y,f)$ with two points $x,y \in \mathfrak{S}_K(T)$, and $f: A_x \rightarrow A_y$ is a $p$-quasi-isogeny from $A_x$ to $A_y$, which are the pullback of $A$ along $x$ and $y$. Moreover, we require that at every geometric point $s$, $f$ preserves the level structure, $s_{\alpha, \on{dR}}$ and $s_{\alpha, l}$ with $l\neq p$. Further, it preserves $s_{\alpha,0}$ (resp. $s_{\alpha,p}$) if $s$ is of  characteristic $p$ (resp. 0). Then $p-\on{Isog}$ is represented by a scheme, and we have a diagram
\[
\begin{tikzcd} 
& p-\on{Isog} \arrow[dr,"t"] \arrow[dl,"s"'] &
\\
\mathfrak{S}_K & & \mathfrak{S}_K
\end{tikzcd}
\]
where $s$ (resp. $t$) remembers $x$ (resp. $y$). Moreover, $s$ and $t$ are proper surjevtive, and étale on the generic fiber. See \cite[Ssection 6]{2020arXiv200611745L}  for example. 

Similar to $\on{loc}_p$, we can construct a map
\[
\on{loc}_p^{\on{H}}: p-\on{Isog}^{\diamond} \longrightarrow \on{Sht}_{\mu|\mu}.
\]
Given $\on{Spa}(R,R^+)$ affinoid pefectoid of characteristic $p$,  $\on{Spa}(R^{\sharp},R^{\sharp,+})$ an untilt and $(x,y,f)$ a point of $p-\on{Isog}(R^{\sharp,+})$, let $\mathcal{M}_x$ (resp. $\mathcal{M}_y$) be the $\varphi$-$A_{\on{inf}}(R^{\sharp,+})$-module associated to $A_x[p^{\infty}]$ (resp. $A_y[p^{\infty}]$). Then $f$ induces an isomorphism $\mathcal{M}_x[\frac{1}{p}] \overset{\sim}{\rightarrow} \mathcal{M}_y[\frac{1}{p}]$ (as $f$ is a $p$-quasi-isogeny), which  preserves $s_{\alpha,\prism}$  by proposition \ref{orwueirueo} and   $f$ preserving $s_{\alpha,0}$, so it induces a modification 
\[
f_{\mathcal{P}}: \mathcal{P}_x \dashrightarrow \mathcal{P}_y
\]
at the characteristic $p$ untilt, where $\mathcal{P}_x$ and $\mathcal{P}_y$ denotes the image of $x$ and $y$ under $\on{loc}_p$. Now $\on{loc}_p^H$ is defined by (sheafification of) sending the data of $(x,y,f)$ to $f_{\mathcal{P}}$. 

Let $\nu$ be a cocharacter of $G$. 
We note that if we restrict to the subscheme $p-\on{Isog}_{\nu}$ parametrizing quasi-isogenies of types bounded by $\nu$, i.e. the corresponding maps between shtukas assoicated to the abelian schemes with Hodge tensors are bounded by $\nu$,  then $\on{loc}_p^H$ restricts to 
\[
\on{loc}_p^{\on{H}}: p-\on{Isog}_{\nu}^{\diamond} \longrightarrow \on{Sht}_{\mu|\mu}^{\nu}.
\]

Now we have a commutative diagram
\begin{equation} \label{uriweuriouwoeu}
\begin{tikzcd} 
& p-\on{Isog}_{\nu}^{\diamond} \arrow[dr,"t"] \arrow[dl,"s"'] \arrow[dd,"\on{loc}_p^H"]&
\\
\mathfrak{S}_K^{\diamond} \arrow[dd,"\on{loc}_p"'] & & \mathfrak{S}_K^{\diamond} \arrow[dd,"\on{loc}_p"]
\\
&  \on{Sht}_{\mu|\mu}^{\nu} \arrow[dr,"p_2"] \arrow[dl,"p_1"'] &
\\
\on{Sht}_{\mu} & & \on{Sht}_{\mu}.
\end{tikzcd}
\end{equation}
By the following lemma, we can pullback cohomological correspondence along the diagram, so we have the cohomological correspondence
\begin{equation} \label{dhfoeifoie}
\on{loc}_p^* S_V: (\mathfrak{S}_K^{\diamond}, \overline{\mathbb{Q}}_\ell) \longrightarrow 
(\mathfrak{S}_K^{\diamond}, \overline{\mathbb{Q}}_\ell)
\end{equation}
supported on $p-\on{Isog}_{\nu}^{\diamond}$, 
where we have used that $\mu$ is minuscule so that the structure sheaf $\mathscr{S}_U$ on $\on{Sht}_{\mu}$ is $\overline{\mathbb{Q}}_\ell$.

\begin{lemma} \label{8993993}
The diagram factorizes through 
\[
\begin{tikzcd} 
& p-\on{Isog}_{\nu}^{\diamond} \arrow[dr,""] \arrow[dl,""] \arrow[dd,hook]&
\\
\mathfrak{S}_K^{\diamond} \arrow[dd,hook] & & \mathfrak{S}_K^{\diamond} \arrow[dd,hook]
\\
& p-\on{Isog}_{\nu}^{\dagger}
\arrow[dr,"t"] \arrow[dl,"s"'] \arrow[dd,"\overline{\on{loc}_p^H}"]&
\\
\mathfrak{S}_K^{\dagger}
\arrow[dd,"\overline{\on{loc}_p}"'] 
& & 
\mathfrak{S}_K^{\dagger}\arrow[dd,"\overline{\on{loc}_p}"]
\\
&  \on{Sht}_{\mu|\mu}^{\nu} \arrow[dr,"p_2"] \arrow[dl,"p_1"'] &
\\
\on{Sht}_{\mu} & & \on{Sht}_{\mu},
\end{tikzcd}
\]
where $\mathfrak{S}_K^{\dagger}$ 
(resp.
$p-\on{Isog}_{\nu}^{\dagger}$)
is the canonical compactification \footnote{This is different from the toroidal compactifications or any algebraic compactifications of the scheme.} of 
$\mathfrak{S}_K^{\diamond}$
(resp.
$p-\on{Isog}_{\nu}^{\diamond}$), and the vertical maps along the first row are the canonical maps into the compactifications, see \cite[Proposition 18.6]{2017arXiv170907343S}  for the definition of these compactifications.  Moreover, the vertical maps along the first line are open immersions, and
the two lower squares in the diagram are Cartesian, so we can pullback cohomological correspondences (we can pullback along smooth maps).
\end{lemma}

\begin{proof}
Recall that $\mathfrak{S}_K^{\dagger}$ is the v-sheaf sending $S=\on{Spa}(R,R^+)$ to 
\[
\mathfrak{S}_K^{\diamond}(\on{Spa}(R,R^{\circ})),
\]
in other words, it is the sheafification of the presheaf sending $\on{Spa}(R,R^+)$ to the set of untilts $\on{Spa}(R^{\sharp},R^{\sharp,+})$ of  $\on{Spa}(R,R^+)$
over $\Breve{\mathbb{Z}}_p$
together with morphisms $\on{Spec}(R^{\sharp,\circ}) \rightarrow \mathfrak{S}_K$ 
of schemes over $\Breve{\mathbb{Z}}_p$, and similarly for 
$p-\on{Isog}_{\nu}^{\dagger}$. The canonical map $\mathfrak{S}_K^{\diamond} \longrightarrow  \mathfrak{S}_K^{\dagger}$
is simply induced by the composition of the restriction map from $\on{Spec}(R^{\sharp,\circ}) \rightarrow \on{Spec}(R^{\sharp,+}) $. Since the v-stacks $\on{Sht}_{\mu}$ and $\on{Sht}_{\mu|\mu}^{\nu}$ are independent of $R^+$, the factorization follows. Indeed, it follows from the proof of \cite[Proposition 11.2.1]{SW17}  that the structure sheaf $\mathcal{O}_{\mathcal{Y}_{[0,\infty)}(S)}$ 
on 
$\mathcal{Y}_{[0,\infty)}(S)$ 
only depends on $R$ (rather than $R^+$), so the category of vector bundles on $\mathcal{Y}_{[0,\infty)}(S)$
is independent of $R^+$, so does the category of torsors on $\mathcal{Y}_{[0,\infty)}(S)$
 by the Tannakian formalism. 

It follows from \cite[Proposition 27.5]{2017arXiv170907343S}  that $\mathfrak{S}_K^{\diamond}$ and $p-\on{Isog}_{\nu}^{\diamond}$ are compactifiable over $\on{Spd}(\Breve{\mathbb{Z}}_p)$, so proposition 22.3 (1) of $loc.cit.$ tells us that the natural maps into their canonical compactifications are open immersions, i.e. the vertical maps along the first row are open immersions. 

Next, we prove the lower left square is Cartesian, the right one is similar. Let $Y$ be the fiber product of $p_1$ and $\overline{\on{loc}_p}$, then the induced map 
$p-\on{Isog}_{\nu}^{\dagger}
\longrightarrow Y$
is proper since $s$ is proper (and $Y \longrightarrow \mathfrak{S}_K^{\dagger}$
is separated, being the base change of $p_1$). It follows from \cite[Lemma 12.5]{2017arXiv170907343S}  that to prove 
$p-\on{Isog}_{\nu}^{\dagger}
\longrightarrow Y$
is an isomorphism, it is enough to prove it is an isomorphism on geometric points. Thus we can assume that 
$S^{\sharp}=\on{Spa}( C,  C^+)$
where  $C$ is an algebraically closed perfectoid field. Since both $\on{Sht}_{\mu|\mu}^{\nu}$ and $p-\on{Isog}_{\nu}^{\dagger}$
are independent of $R^+$, we can further assume that $C^+=\mathcal{O}_{C}$.

We know that
$p$-quasi-isogenies from an abelian variety $A_x$ preserving Hodge tensors and level structures are the same as $p$-quasi-isogenies from $A_x[p^{\infty}]$ preserving Hodge tensors, since $p$-quasi-isogenies do not affect level structures (the level structure at $p$ is hyperspecial) and $s_{\alpha,l}$  while the rest Hodge tensors are all detected by the corresponding $p$-divisible groups. Moreover, we know from \cite[Theorem 17.5.2]{SW17}  that $p$-quasi-isogenies between $p$-divisible groups preserving Hodge tensors are the same as $p$-quasi-isogenies between $\mathcal{M}_x$ and $\mathcal{M}_y$ preserving $s_{\alpha,\prism}$, so it remains to show these are the same as the corresponding  $p$-quasi-isogenies between  $\mathcal{E}_x$ and $\mathcal{E}_y$, whence $\mathcal{P}_x$ and $\mathcal{P}_y$. In other words, we want to know that the restriction functor from Breuil--Kisin--Fargues modules $\mathcal{M}$ (over $\mathcal{O}_C$) to $GL_n$-shtukas (over $\on{Spa}(C,\mathcal{O}_C)$) is fully faithful up to quasi-isogeny, but this follows from \cite[Theorem 13.2.1, 14.2.1]{SW17}.
\end{proof}

\begin{remark}
The reason we use the factorization as in the lemma is because the squares in (\ref{uriweuriouwoeu}) are not Cartesian, and the reason for that is we do not have fully faithfulness of the restriction functor from Breuil--Kisin--Fargues modules over $W(R^+)$ to shtukas over $\mathcal{Y}_{[0,\infty)}(\on{Spa}(R,R^+))$ in general, disproving the expectation in \cite[Remark 4.5.13]{Kedlaya2019}. Indeed, Breuil--Kisin--Fargues modules over $W(R^+)$ depend on $R^+$ while shtukas over $\mathcal{Y}_{[0,\infty)}(\on{Spa}(R,R^+))$ do not, and there are examples where the base change of Breuil--Kisin--Fargues modules (even those coming from $p$-divisible groups) from $W(R^+)$ to $W(R^{\circ})$ is not fully faithful.  This is closely related to the failure of Tate's theorem on homomorphisms between $p$-divisible groups (the characteristic $p$ case is due to de Jong) in the case of non-discrete valuation rings.

Indeed, let $K$ be the completion of an algebraic closure of $\mathbb{F}_p((T))$
with ring of integral elements $R$,  and $H=E[p^{\infty}]$ 
be the
 $p$-divisible group of an elliptic curve $E$ over $R$ which is ordinary at the generic fiber and supersingular at the special fiber.
 Then since the generic fiber is algebraically closed, $H_K \cong \mathbb{Q}_p/\mathbb{Z}_p \times \mu_{p^{\infty}} $. Now let $H'=(\mathbb{Q}_p/\mathbb{Z}_p \times \mu_{p^{\infty}})_{R}$ 
be the constant $p$-divisible group defined over $R$, then the isomorphism $H_K \cong H'_K$ cannot be extended to an isomorphism over $R$ since  they are not isomorphic on the special fiber. The Newton polygons are different  as $E$ is supersingular at the special fiber. Hence the base change to the generic fiber is not fully faithful
\footnote{The author would like to thank Kiran Kedlaya and Marco D'Addezio for discussions on this point.}.

 As a concrete example, let  
 \[
 E: y^2=x(x-1)(x-t)
 \]
 be  defined over $\mathbb{F}_p[t][\frac{1}{t(t-1)} ]$ with $p>2$, and $E$ is ordinary over the fraction field by \cite[exercise 5.8]{Silverman2009}.
 Now we pick $\lambda \in \mathbb{F}_p$ such that $E$ is supersingular at $t=\lambda$. For example, we can take $p=11$ and $\lambda=2$, see \cite[example 4.3]{Silverman2009}. Now we base change $E$ to the completion of $\mathbb{F}_p[t][\frac{1}{t(t-1)} ]$ 
 at $t-\lambda$, which is isomorphic to $\mathbb{F}_p[[T]]$, and we further base change to $R$.
\end{remark}

We now restrict to the generic fiber to have
\begin{equation} \label{802984903}
\begin{tikzcd} 
& p-\on{Isog}_{\nu, \eta}^{\diamond} \arrow[dr,"t"] \arrow[dl,"s"'] \arrow[dd,"\on{loc}_{p,\eta}^H"]&
\\
\mathcal{X}_K^{\diamond} \arrow[dd,"\on{loc}_{p,\eta}"']
& &
\mathcal{X}_K^{\diamond} \arrow[dd,"\on{loc}_{p,\eta}"]
\\
&  
\on{Sht}_{\mu|\mu, \eta}^{\nu} \arrow[dr,"p_2"] \arrow[dl,"p_1"'] 
&
\\
\on{Sht}_{\mu,\eta} & & \on{Sht}_{\mu,\eta},
\end{tikzcd}
\end{equation}
along which we can pullback the Hecke correspondence 
$\mathscr{C}_{\nu} $ in (\ref{1222}) 
to obtain 
\[ 
\mathscr{C}_{\nu}^S: (\mathcal{X}_K^{\diamond}, \overline{\mathbb{Q}}_\ell) \longrightarrow 
(\mathcal{X}_K^{\diamond}, \overline{\mathbb{Q}}_\ell),
\]
where we use that $\mu$ is minuscule so that the canonical sheaves are constant. Note that we can pullback cohomological correspondences here by lemma \ref{8993993}, which shows that it is a composition of open immersion and Cartesian squares, and we can pullback cohomological correspondences along either Cartesian squares or smooth maps.  Moreover, it follows from the étaleness of the two correspondences that $\mathscr{C}_{\nu}^S$ is nothing but
\[
\mathscr{C}_{\nu}^S: s^*\overline{\mathbb{Q}}_\ell \cong \overline{\mathbb{Q}}_\ell 
\cong
t^*\overline{\mathbb{Q}}_\ell \cong t^!\overline{\mathbb{Q}}_\ell. 
\]
Now as in section \ref{rieourioeu}, let $V$ be a representation of $\hat{G}_{\mathbb{Q}_p}$, and $h_V $ be the function corresponding to $V$ through classical Satake, 
we define
\[
\Gamma_V^{S}:= \underset{ h_V(\nu(p))\neq 0}{\cup} p-\on{Isog}_{\nu,\eta},
\]
and 
\[
T_V^S := \underset{h_V(\nu(p))\neq 0}{\sum} h_V(\nu(p)) \mathscr{C}_{\nu}^S
\]
which is a cohomological correspondence from $(\mathcal{X}_K^{\diamond}, \overline{\mathbb{Q}}_\ell)$ to itself supported on $\Gamma_V^{S,\diamond}$. Note that $T_V^S$ is exactly the (analytification of) Hecke correspondence of $\mathcal{X}_K$, so we have proved

\begin{theorem} \label{pireopwiepo}
The pullback of $S_{V,\eta}$ along
(\ref{802984903}) is exactly $T_V^S$.
\end{theorem}

\begin{proof}
This follows from the above discussion and theorem \ref{840238493}.
\end{proof}

Next we want to compare the Hecke operator and excursion operator on the compactly supported cohomology of $\bar{\mathfrak{X}}_K$, and deduce the $S=T$ conjecture of Xiao--Zhu. We begin with a well-known observation.

\begin{lemma}
Let $\pi: \mathfrak{S}_K\longrightarrow \on{Spec}(\Breve{\mathbb{Z}}_p)$ be the structure map of $\mathfrak{S}_K$, then $R^i\pi_! \overline{\mathbb{Q}}_\ell$ is a lisse sheaf on $\on{Spec}(\Breve{\mathbb{Z}}_p)$ for every $i$. 
\end{lemma}

\begin{proof}
It follows from the existence of integral toroidal compactifications (\cite{KeerthiMADAPUSIPERA2019}) that $R^j\pi_* \overline{\mathbb{Q}}_\ell$ is lisse (using the Leray spectral sequence of the open immersion into toroidal compactification and  purity of normal crossing divisors). Now Poincare duality gives the result for $R^i\pi_! \overline{\mathbb{Q}}_\ell$. 
\end{proof}

\begin{lemma}
There is a canonical isomorphism
\[
R^i(\pi^{\diamond})_!\overline{\mathbb{Q}}_\ell \cong  c_{\Breve{\mathbb{Z}}_p}^* R^i\pi_!\overline{\mathbb{Q}}_\ell,
\]
with notations as in appendix \ref{serepoir}. 
\end{lemma} 

\begin{proof}
It is enough to prove the statement with $\overline{\mathbb{Q}}_\ell$ replaced by $\Lambda$, a finite extension of $\mathbb{Z}_\ell$. It is enough to prove the identification on the special fiber and generic fiber. Over the special fiber, this is  proposition \ref{uwweerlkls}.

Let $\Bar{\eta}$ be a completed algebraically closure of $\Breve{\mathbb{Q}}_p$, then we have 
\[
(R(\pi^{\diamond})_!\Lambda)_{\Bar{\eta}} \cong R\Gamma_c((\mathfrak{S}_{K}^{\diamond})_{\Bar{\eta}}, \Lambda) \cong R\Gamma_c(\mathcal{X}_{K,\Bar{\eta}}, \Lambda) \cong R\Gamma_c(\Bar{\mathfrak{X}}_{K}, R\psi (\Lambda)) \cong 
R\Gamma_c(\Bar{\mathfrak{X}}_{K}, \Lambda)
\]
where $ R\psi$ is the nearby cycle functor with respect to the integral model $\mathfrak{S}_K / \Breve{\mathbb{Z}}_p$. 
The first isomorphism is proper base change (\cite[Proposition 22.15]{2017arXiv170907343S}), the second is \cite[Lemma 15.6]{2017arXiv170907343S}, the third is \cite[Theorem 5.7.8]{Huber2013-td}, and the last follows from  $\mathfrak{S}_K$ being smooth over $  \Breve{\mathbb{Z}}_p$. 

On the other hand, we have
\[
(R\pi_!\Lambda)_{\Bar{\eta}} \cong 
R\Gamma_c(\mathfrak{S}_{K,\Bar{\eta}},\Lambda ) \cong R\Gamma_c(\Bar{\mathfrak{X}}_{K}, R\psi (\Lambda)) \cong 
R\Gamma_c(\Bar{\mathfrak{X}}_{K}, \Lambda)
\]
where the first isomorphism is proper base change, the second follows from the existence of toroidal compactifications (\cite{KeerthiMADAPUSIPERA2019}), see \cite{lan_stroh_2018} for example, and the last follows from the smoothness of $\mathfrak{S}_K$. We have thus proved the identification over the generic fiber. 
\end{proof}

\begin{corollary}
    We have $R^i\pi^{\diamond}_! \overline{\mathbb{Q}}_\ell$ is a locally constant sheaf whose stalks are identified with $H^i_c(\Bar{\mathfrak{X}}_{K}, \overline{\mathbb{Q}}_\ell)$. 
\end{corollary}

We can now prove the $S=T$ conjecture in \cite{2017arXiv170705700X}.

\begin{theorem}
As operators on $H^i_c(\Bar{\mathfrak{X}}_{K}, \overline{\mathbb{Q}}_\ell) $,
the Hecke correspondence  $T_V^S$ 
is the same as the excursion correspondence $\on{loc}_{p}^{W,*} S_{V}^{W}$, the analytification of the excursion operator in \cite{2017arXiv170705700X} (whose action on $H^i_c(\Bar{\mathfrak{X}}_{K}, \overline{\mathbb{Q}}_\ell) $ is the same as the action as defined in \cite{2017arXiv170705700X} through comparison of cohomology of v-sheaves and perfect schemes in \cite[Section 27]{2017arXiv170907343S}).
\end{theorem}

\begin{proof}

Recall that we can pushforward cohomological correspondence (\cite[A.2.6]{2017arXiv170705700X}), so we can pushforward (\ref{dhfoeifoie}) along $R(\pi^{\diamond})_{!}$ to obtain
\[
R^i(\pi^{\diamond})_{!} \on{loc}_p^* S_V: R^i(\pi^{\diamond})_!\overline{\mathbb{Q}}_\ell \longrightarrow R^i(\pi^{\diamond})_!\overline{\mathbb{Q}}_\ell
\] 
which is a global section of the internal hom $\mathcal{H}\text{om}(R^i(\pi^{\diamond})_!\overline{\mathbb{Q}}_\ell, R^i(\pi^{\diamond})_!\overline{\mathbb{Q}}_\ell)$. Since $R^i(\pi^{\diamond})_!\overline{\mathbb{Q}}_\ell$ is locally constant, so is 
$\mathcal{H}\text{om}(R^i(\pi^{\diamond})_!\overline{\mathbb{Q}}_\ell, R^i(\pi^{\diamond})_!\overline{\mathbb{Q}}_\ell)$, 
and we see that 
$R^i(\pi^{\diamond})_{!} \on{loc}_p^* S_V$
is constant. In other words, for any geometric point $\bar{\eta}$ of $\on{Spd}(\Breve{\mathbb{Q}}_p)$, we have that the restriction of  $R^i(\pi^{\diamond})_{!} \on{loc}_p^* S_V$ to $\Bar{\eta}$,
which is nothing but
\[
R^i(\pi^{\diamond}_{\Bar{\eta}})_{!} \on{loc}_{p,\Bar{\eta}}^* S_{V,\bar{\eta}}: H^i_c(\mathcal{X}_{K,\bar{\eta}}, \overline{\mathbb{Q}}_\ell) \longrightarrow 
H^i_c(\mathcal{X}_{K,\bar{\eta}}, \overline{\mathbb{Q}}_\ell),
\] 
is identified with its restriction to the special fiber
\[
R^i(\pi^{\diamond}_{s})_{!} \on{loc}_{p,s}^* S_{V,s} : H^i_c(\Bar{\mathfrak{X}}_{K}, \overline{\mathbb{Q}}_\ell) 
\longrightarrow H^i_c(\Bar{\mathfrak{X}}_{K}, \overline{\mathbb{Q}}_\ell)
\]
under the canonical identification between $H^i_c(\Bar{\mathfrak{X}}_{K}, \overline{\mathbb{Q}}_\ell)$ 
and $H^i_c(\mathcal{X}_{K,\bar{\eta}}, \overline{\mathbb{Q}}_\ell)$, where $s:= \overline{\mathbb{F}_p}$.

Now theorem \ref{pireopwiepo} tells that $\on{loc}_{p,\Bar{\eta}}^* S_{V,\bar{\eta}} = T_V^S$ is the Hecke operator, so the action
$R^i(\pi^{\diamond}_{s})_{!} \on{loc}_{p,s}^* S_{V,s}$ 
is identified with the Hecke operator on $H^i_c(\Bar{\mathfrak{X}}_{K}, \overline{\mathbb{Q}}_\ell)$. Recall that Hecke operators on 
$H^i_c(\Bar{\mathfrak{X}}_{K}, \overline{\mathbb{Q}}_\ell)$ 
is defined through the identification of 
$H^i_c(\Bar{\mathfrak{X}}_{K}, \overline{\mathbb{Q}}_\ell)$ 
with 
$H^i_c(\mathcal{X}_{K,\bar{\eta}}, \overline{\mathbb{Q}}_\ell)$
and the Hecke action on the latter, which is exactly the way we identify the excursion operators.

It remains to identify 
$R^i(\pi^{\diamond}_{s})_{!} \on{loc}_{p,s}^* S_{V,s}$ 
with excursion operators defined by Xiao--Zhu. 
As usual, we let
\[
p-\on{Isog}_V := \underset{\nu}{\cup} \  p-\on{Isog}_{\nu}
\]
where $\nu$ varies from heighest weights of irreducible summands of $V$.

We first observe that 
\[
\begin{tikzcd} 
& p-\on{Isog}_{V,s}^{\diamond} \arrow[dr,"t"] \arrow[dl,"s"'] \arrow[dd,"\on{loc}_{p,s}^H"]&
\\
\bar{\mathfrak{X}}_{K}^{\diamond} \arrow[dd,"\on{loc}_{p,s}"']
& &
\bar{\mathfrak{X}}_{K}^{\diamond} \arrow[dd,"\on{loc}_{p,s}"]
\\
&  \on{Sht}_{\mu|\mu,s}^{V} \arrow[dr,"p_2"] \arrow[dl,"p_1"']  &
\\
\on{Sht}_{\mu,s}  & & \on{Sht}_{\mu,s}
\end{tikzcd}
\]
factorizes through
\[
\begin{tikzcd} 
& p-\on{Isog}_{V,s}^{\diamond} \arrow[dr,"t"] \arrow[dl,"s"'] \arrow[dd,"\on{loc}_{p}^{H,W}"]&
\\
\bar{\mathfrak{X}}_{K}^{\diamond} \arrow[dd,"\on{loc}_{p}^{W}"']
& &
\bar{\mathfrak{X}}_{K}^{\diamond} \arrow[dd,"\on{loc}_{p}^{W}"]
\\
&
\on{Sht}_{\mu | \mu}^{V,W} \arrow[dr,"" ] \arrow[dl, ""]  \arrow[dd]
& 
\\
\on{Sht}_{\mu}^{W} \arrow[dd]
& & 
\on{Sht}_{\mu}^{W} \arrow[dd]
\\
&  \on{Sht}_{\mu|\mu,s}^{V} \arrow[dr,"p_2"] \arrow[dl,"p_1"'] &
\\
\on{Sht}_{\mu,s} & & \on{Sht}_{\mu,s},
\end{tikzcd}
\]
where $\on{loc}_{p}^{W}$ 
sends $A$ to (the corresponding torsor) the Dieudonn\'e module associated to $A[p^{\infty}]$ together with Hodge tensors, and similarly for $\on{loc}_{p}^{H,W}$. Up to truncation, the upper half diagram is precisely the analytification of the pullback diagram considered in \cite[Section 7.3.11]{2017arXiv170705700X}. 
Then proposition \ref{oeurieurio} tells us that $\on{loc}_{p,s}^* S_{V,s}$
is the same as the pullback of $ S_V^{W}$
along the upper half of the diagram. Since $S_V^W$ is the analytification of the truncated one (proposition \ref{oeurieurio}), and $\diamond$-analytification of cohomological correspondences commutes with pullback (proposition \ref{fjdklfsjdkeie}), 
we have that $\on{loc}_{p,s}^* S_{V,s}$ is exactly the analytification of the excursion operator on the special fiber of Shimura varieties constructed in \cite{2017arXiv170705700X} (being defined as the pullback of the truncated version of $\on{loc}_p^W$ of the truncated $S_V^W$ in the category of perfect schemes). Therefore, we have proved that the Hecke operator $T_V^S$
on $H^i_c(\Bar{\mathfrak{X}}_{K}, \overline{\mathbb{Q}}_\ell) $
is the same as the action of $\on{loc}_{p}^{W,*} S_{V}^{W}$.  
\end{proof}

\begin{corollary} (\cite[Proposition 6.3.1]{2018arXiv180205299X})
The action of $T_V^S$ on $H^i_c(\Bar{\mathfrak{X}}_{K}, \overline{\mathbb{Q}}_\ell) $
satisfies Eichler--Shimura relation.
\end{corollary}

\begin{remark}
The Eichler--Shimura relation, or the Blasius--Rogawski conjecture has been proved in many special cases previously, see \cite{Faltings2010-sw} \cite{+2000+43+71}, \cite{2020arXiv200611745L}, 
\cite{+2002+133+159},
and \cite{koskivirta_2014}
for example. The method adopted in most previous works is to study the algebraic correspondences explicitly, which is very different from our categorical approach. 
\end{remark}

\appendix

\section{Algebraic stacks and v-stacks}

\subsection{Comparison of cohomology} \label{serepoir}

By \cite[Section 27]{2017arXiv170907343S}, we can associate a v-sheaf $X^{\diamond\diamond}$ to any scheme $X$ locally of finite type over a complete discrete valuation ring $\mathcal{O}$ with perfect residue field of characteristic $p$. $X^{\diamond\diamond}$
is defined by sending characteristic $p$ affinoid perfectoid  $S=\on{Spa}(R,R^+)$ 
to the set of untilts $S^{\sharp}$ over $\on{Spa}(\mathcal{O},\mathcal{O})$ together with maps $S^{\sharp} \longrightarrow X$ of locally ringed spaces over $\on{Spec}(\mathcal{O})$. Another way to view $X^{\diamond\diamond}$ is that $X^{\diamond\diamond}$ is the sheafification of the presheaf sending $S=\on{Spa}(R,R^+)$ to the set of untilts $S^{\sharp}=\on{Spa}(R^{\sharp},R^{\sharp,+})$ over $\on{Spa}(\mathcal{O},\mathcal{O})$ together with maps $\on{Spec}(R^{\sharp}) \longrightarrow X$ of schemes over $\mathcal{O}$. 

On the other hand, there is another v-sheaf $X^{\diamond}$ we can associate to any scheme $X$ locally of finite type over $\mathcal{O}$. $X^{\diamond}$ is defined to be the sheafification of the presheaf sending a characteristic $p$ affinoid perfectoid space $S=\on{Spa}(
R,R^+)$ to the set of untilts $S^{\sharp}$ together with  morphisms $\on{Spec}(
R^{\sharp,+}) \rightarrow X$ of schemes over $\mathcal{O}$. In other words, it is the v-sheaf associated to the formal scheme $X^{\wedge}$, the completion of $X$ along the special fiber of $\mathcal{O}$.  We note that there is a 
 canonical map 
\[
a_X:
 X^{\diamond} \longrightarrow X^{\diamond\diamond}
\]
sending the map $\on{Spec}(R^+) \rightarrow X$ to its restriction on $\on{Spec}(R)$, which is an open immersion. It is an isomorphism when $X$ is proper over $\mathcal{O}$. 

We extend the construction to algebraic stacks locally of finite type over $\mathcal{O}$. For clarity, we specify our convention on algebraic stacks to be  fppf sheaves of groupoids with diagonals representable by algebraic spaces, which admit representable smooth surjective maps from  schemes.   Let $X$ be such an algebraic stack, we define $X^{\diamond}$ to be the stackification of the prestack sending $S=\on{Spa}(R,R^+)$ to the groupoid of untilts $S^{\sharp}=\on{Spa}(R^{\sharp},R^{\sharp,+})$ of $S$ together with morphisms of stacks $\on{Spec}(R^{\sharp,+}) \rightarrow X$ over $\mathcal{O}$. Similarly, we define $X^{\diamond\diamond}$ to be the stackification of the prestack sending $S=\on{Spa}(R,R^+)$ to the groupoid of untilts $S^{\sharp}=\on{Spa}(R^{\sharp},R^{\sharp,+})$ of $S$ together with morphisms of stacks $\on{Spec}(R^{\sharp}) \rightarrow X$ over $\mathcal{O}$. Then as usual, there is a canonical map
\[
a_X:
X^{\diamond} \longrightarrow X^{\diamond\diamond}. 
\]

We recall that when $X$ is a scheme locally of finite type over $\mathcal{O}$, there is a natural morphism
\[
c_X:
X^{\diamond\diamond}_v \rightarrow X_{\text{ét}} 
\]
from the v-site of $X^{\diamond\diamond}$ to the étale site of $X$, which is defined by sending a scheme $U$ together with a smooth morphism $U \rightarrow X$ 
to $U^{\diamond\diamond} \rightarrow X^{\diamond\diamond}$, see \cite[Section 27]{2017arXiv170907343S}. Let $\Lambda$ be a finite ring killed by $N$ with $p$ not dividing $N$, then 
$c_X$ induces a functor 
\[
c_X^*: D(X_{\text{ét}},\Lambda) \rightarrow D_{\text{ét}}(X^{\diamond\diamond},\Lambda),
\]
and for any $f: X\rightarrow Y $ separated, there is a canonical identification 
\[
c^*_YRf_! \cong Rf_!^{\diamond\diamond} c_X^*,
\]
see \cite[Proposition 27.5]{2017arXiv170907343S}. Moreover, we have a canonical identification
\[
f^{\diamond\diamond,*} c_Y^* \cong c_X^* f^*. 
\]

For a general algebraic stack  $X$ locally of finite type over $\mathcal{O}$, we can choose a smooth surjective map  $U \rightarrow X$ with $U$ being a scheme, then it is well-known
(see \cite[Corollary 5.3.6]{2014arXiv1404.1128L}  for example)
that we have a canonical identification 
\[
\mathcal{D}(X,\Lambda) \cong \text{lim} \ \mathcal{D}(U^{\bullet}_{\text{ét}},\Lambda)
\]
where $U^{\bullet}$ is the Čech nerve of $U$, $\mathcal{D}(X,\Lambda)$ is the $\infty$-categorical enhancement of the derived category of the lisse-étale site of $X$ and similarly for the right hand side. We now can extend the comparison morphism to algebraic stacks 
\[
c_X^*: \mathcal{D}(X,\Lambda) \cong \text{lim} \ \mathcal{D}(U^{\bullet}_{\text{ét}},\Lambda) \overset{\text{lim}\ c_{U^{\bullet}}^*}{\longrightarrow} \text{lim} \ \mathcal{D}_{\text{ét}}(U^{\bullet, \diamond\diamond},\Lambda) \cong \mathcal{D}_{\text{ét}}(X^{\diamond\diamond},\Lambda),
\]
where the last isomorphism is \cite[Proposition 17.3]{2017arXiv170907343S}  and remark 17.4 (and the commutation of $^{\diamond\diamond}$ with fiber products). It is not hard to see that $c_X^*$ is independent of the choice of $U$. 

Moreover, given a separated representable map $f: X \rightarrow Y$ between algebraic stacks locally of finite type over $\mathcal{O}$, we choose a smooth surjective map $U \rightarrow Y$ from a scheme $U$, and let $V $ be the pullback of $U $ along $f$ with the induced map $g: V \rightarrow U$. Then $g$ extends to $g^{\bullet}: V^{\bullet} \rightarrow U^{\bullet}$, and the limit of
$c^*_{U^{\bullet}}Rg^{\bullet}_! \cong R(g^{\bullet,\diamond\diamond})_! c_{V^{\bullet}}^*$ 
tells us that 
\[
c^*_YRf_! \cong Rf_!^{\diamond\diamond} c_X^*.
\]
Similarly, we have 
\[
f^{\diamond\diamond,*} c_Y^* \cong c_X^* f^*. 
\]
In summary, we have the following theorem.

\begin{theorem} \label{flkdjlkfjd}
Let $\Lambda$ be a finite ring killed by $N$ with $p$ not dividing $N$, $X$ an algebraic stack locally of finite type over $\mathcal{O}$, then we have a canonical comparison map
\[
c_X^*: \mathcal{D}(X,\Lambda) \rightarrow \mathcal{D}_{\text{ét}}(X^{\diamond\diamond},\Lambda). 
\]
Moreover, for any $f: X\rightarrow Y $ representable separated map, there are canonical identifications
\[
c^*_YRf_! \cong Rf_!^{\diamond\diamond} c_X^*
\]
and 
\[
f^{\diamond\diamond,*} c_Y^* \cong c_X^* f^*. 
\]
\end{theorem} 

Lastly, we pass to adic sheaves. Let $l$ be a prime different from $p$, and $\Lambda$ be a finite extension of $\mathbb{Z}_\ell$, we know  by \cite{2014arXiv1404.1128L} that
\[
\mathcal{D}(X,\Lambda) = \underset{n}{\text{lim}} \ \mathcal{D}(X,\Lambda/l^n),
\]
where the limit is taken as $\infty$-categories. On the other hand, we also have 
\[
\mathcal{D}_{\text{ét}}(X^{\diamond\diamond},\Lambda) = \underset{n}{\text{lim}} \ \mathcal{D}_{\text{ét}}(X^{\diamond\diamond},\Lambda/l^n)
\]
by \cite[Proposition 26.2]{2017arXiv170907343S}. The analogue of theorem \ref{flkdjlkfjd} for such $\Lambda$ follows directly from the limit expression, using that pullback and $!$-pushforward are compatible with the limit expression, see \cite[Remark 26.3]{2017arXiv170907343S}. Further, we can (idempotently completely) invert $l$ and taking the limit over finite extensions of $\mathbb{Z}_\ell$ to obtain the comparison with coefficient $\overline{\mathbb{Q}}_\ell$. 

\begin{remark} \label{jfoieuoirue}
If the algebraic stack $X$ is locally of finite type over a perfect field, then $X^{\diamond\diamond}$ depends only on the perfection $X^{\on{perf}}$ of $X$, and theorem \ref{flkdjlkfjd} holds for $X^{\on{perf}}$ since  perfection does not change cohomology. This is an important case in this article. 
\end{remark}

A natural question is whether we have a similar comparison result for the $\diamond$-analytification.  The answer is no in general, but it holds in the following special case. 

\begin{proposition} \label{uwweerlkls}
Let $X$ be a scheme locally of finite type over  a perfect field $k$, and $\Lambda$ be either a finite extension of $\mathbb{Z}_\ell$ or $\overline{\mathbb{Q}}_\ell$ then we have a canonical identification 
\[
R\pi_!^{\diamond} \Lambda \cong c_{k}^* R\pi_! \Lambda
\]
where $\pi: X \rightarrow \on{Spec}(k)$
is the structure map. 
\end{proposition}

\begin{proof}
It is enough to show the result after base change to the algebraically closure of $k$, so we assume that $k$ is algebraically closed, and we are aiming to prove 
\[
R\Gamma_{c}(X^{\diamond}, \Lambda) = R\Gamma_c(X,\Lambda).
\]
By a standard Čech cohomology argument, we can assume that $X=\on{Spec}(R)$ is affine. Moreover, it is enough to prove the case when $\Lambda$ is finite, so we assume that. 

Let $C$ be a completed algebraically closure of $k((t))$, and 
by proper base change (\cite[Proposition 22.15]{2017arXiv170907343S}), we have 
\[
R\Gamma_{c}(X^{\diamond}, \Lambda) =R\Gamma_{c}(X^{\diamond}_{C}, \Lambda) 
\]
where 
$X^{\diamond}_{C} := X^{\diamond}\times_{\on{Spd}(k)} \on{Spd}(C)$. We observe that 
\[
X^{\diamond}_{k((t))} = \on{Spd}(R((t)),R[[t]])
\]
where $X^{\diamond}_{k((t))} :=X^{\diamond}\times_{\on{Spd}(k)} \on{Spd}(k((t)),k[[t]])$, which visibly is the analytification of the analytic adic space $X^{\on{ad}}_{k((t))}:= \on{Spa}(R((t)),R[[t]])$. Then by \cite[Lemma 15.6]{2017arXiv170907343S}, 
\[
R\Gamma_{c}(X^{\diamond}_{C}, \Lambda)  = 
R\Gamma_{c}((X^{\on{ad}}_{k((t))})_{C}, \Lambda). 
\]

To compute 
$R\Gamma_{c}((X^{\on{ad}}_{k((t))})_{C}, \Lambda)$, we observe that
\[
\on{Spa}(R((t)),R[[t]])=
\on{Spa}(R[[t]],R[[t]])_{\eta} := \on{Spa}(R[[t]],R[[t]]) \times_{\on{Spa}(k[[t]],k[[t]])} \on{Spa}(k((t)),k[[t]]),
\]
namely, $X^{\on{ad}}_{k((t))}$ is the generic fiber of the formal scheme $\mathfrak{X}:=\on{Spa}(R[[t]],R[[t]])$ 
over $\on{Spa}(k[[t]],k[[t]])$. Note that $\mathfrak{X}$ is the $t$-adic completion of the scheme $\on{Spec}(R) \times_{\on{Spec}(k)} \on{Spec}(k[[t]]) $ over $\on{Spec}(k[[t]])$, which has special fiber $X$.  
Then \cite[Theorem 5.7.8]{Huber2013-td}  tells us that 
\[
R\Gamma_{c}((X^{\on{ad}}_{k((t))})_{C}, \Lambda) =
R\Gamma_{c}(X,R\psi (\Lambda))
\]
where $R\psi$ is the nearby cycle functor with respect to $\on{Spec}(R[[t]]) $ over $\on{Spec}(k[[t]])$. Since $\on{Spec}(R) \times_{\on{Spec}(k)} \on{Spec}(k[[t]])$
is a constant family over $\on{Spec}(k[[t]])$, we have $R\psi (\Lambda)=\Lambda$, finishing the proof. 

\end{proof}

\subsection{Comparison of cohomological correspondences} \label{sdkkdkdkd}

We have established the comparison of cohomology, and now we want to use it to compare cohomological correspondences. Let 
\[
\mathscr{C}: s^*\mathcal{F} \longrightarrow t^!\mathcal{G}
\]
be a cohomological correspondence supported on
\[
\begin{tikzcd}
& Z \arrow[ld,"s"'] \arrow[rd,"t"]&
\\
X & & Y
\end{tikzcd}
\]
where $X, Z$ and $Y$ are algebraic stacks locally of finite type over $\mathcal{O}$, or perfections of such stacks over a perfect field, $s$ and $t$ are representable separated morphisms, and $\mathcal{F} \in D^b(X_{\text{ét}},\overline{\mathbb{Q}}_\ell)$ (resp. $\mathcal{G} \in D^b(Y_{\text{ét}},\overline{\mathbb{Q}}_\ell)$). There is a natural analytification of $\mathscr{C}$,
\[
\mathscr{C}^{\diamond\diamond}:(X^{\diamond\diamond},c_X^*\mathcal{F}) \longrightarrow (Y^{\diamond\diamond},c_Y^{*}\mathcal{G})
\]
supported on 
\[
\begin{tikzcd}
& Z^{\diamond\diamond} \arrow[ld,"s^{\diamond\diamond}"'] \arrow[rd,"t^{\diamond\diamond}"]&
\\
X^{\diamond\diamond} & & Y^{\diamond\diamond},
\end{tikzcd}
\]
which is defined by
\[
\mathscr{C}^{\diamond\diamond}: s^{\diamond\diamond,*}c_X^*\mathcal{F} \cong c_Z^* s^* \mathcal{F} \overset{c_Z^* \mathscr{C}}{\longrightarrow} c_Z^*t^!\mathcal{G} \rightarrow t^{\diamond\diamond,!} c_Y^* \mathcal{G},
\]
where the last map is the adjoint of the canonical map 
\[
t^{\diamond\diamond}_!c_Z^*t^!\mathcal{G} \cong c_Y^* t_!t^!\mathcal{G} \rightarrow  c_Y^* \mathcal{G}
\]
with the first isomorphism being theorem \ref{flkdjlkfjd} and the second the counit map. 

This gives the $\diamond\diamond$-analytification of the cohomological correspondence, but what we need in this article is the analytification with respect to $\diamond$. We can construct the desired cohomological correspondence from $\mathscr{C}^{\diamond\diamond}$ under additional hypothesis. We have naturally a commutative diagram
\[
\begin{tikzcd}
& Z^{\diamond} \arrow[dd,"a_Z"] \arrow[ld,"s^{\diamond}"'] \arrow[rd,"t^{\diamond}"]&
\\
X^{\diamond}\arrow[dd,"a_X"] & & Y^{\diamond} \arrow[dd,"a_Y"]
\\
& Z^{\diamond\diamond} \arrow[ld,"s^{\diamond\diamond}"'] \arrow[rd,"t^{\diamond\diamond}"]&
\\
X^{\diamond\diamond} & & Y^{\diamond\diamond},
\end{tikzcd}
\]
assume the right square is Cartesian, \footnote{As noted by the referee, being Cartesian fails for open immersions such as $\mathbb{G}_m \rightarrow \mathbb{A}^1$.} then we can pullback $\mathscr{C}^{\diamond\diamond}$ along it to have the cohomological correspondence
\[
\mathscr{C}^{\diamond}: (X^{\diamond}, a_X^*c_X^*\mathcal{F}) \longrightarrow (Y^{\diamond},a_Y^*c_Y^* \mathcal{G})
\]
supported on $Z^{\diamond}$. 

We record some formal properties of $\mathscr{C}^{\diamond}$. 

\begin{proposition} \label{fjdklfsjdkeie}
$\mathscr{C}^{\diamond}$ commutes with composition and pullbacks. More precisely, if $\mathscr{C}_1$ and $\mathscr{C}_2$ are composable, and assume that the right squares involving $a_Z$ and $a_Y$ of both $\mathscr{C}_1$ and $\mathscr{C}_2$ are Cartesian, then so does $\mathscr{C}_1 \circ \mathscr{C}_2$, and  we have
\[
\mathscr{C}_1^{\diamond} \circ \mathscr{C}_2^{\diamond} = (\mathscr{C}_1 \circ \mathscr{C}_2)^{\diamond}.
\]

Given a commutative diagram
\[
\begin{tikzcd}
& Z_1 \arrow[dd,"b"] \arrow[ld,"s_1"'] \arrow[rd,"t_1"]&
\\
X_1\arrow[dd,"a"'] & & Y_1 \arrow[dd,"c"]
\\
& Z_2 \arrow[ld,"s_2"'] \arrow[rd,"t_2"]&
\\
X_2 & & Y_2,
\end{tikzcd}
\]
and a cohomological correspondence 
\[
\mathscr{C}_2: (X_2,\mathcal{F}) \longrightarrow (Y_2, \mathcal{G}) 
\]
supported on $Z_2$, 
suppose that either the vertical arrows are smooth or the right square is Cartesian, so we can pullback  $\mathscr{C}_2$ along the diagram to obtain $\mathscr{C}_1$ supported on $Z_1$. Assume that both $\mathscr{C}_1^{\diamond}$ and $\mathscr{C}_2^{\diamond}$ exists, i.e. the square corresponding to $a_{Y_1}$ and $a_{Z_1}$ (resp.  $a_{Y_2}$ and $a_{Z_2}$) is Cartesian, then $\mathscr{C}_1^{\diamond}$ is the pullback of $\mathscr{C}_2^{\diamond}$ along the diagram 
\[
\begin{tikzcd}
& Z_1^{\diamond} \arrow[dd,"b^{\diamond}"] \arrow[ld,"s_1^{\diamond}"'] \arrow[rd,"t_1^{\diamond}"]&
\\
X_1^{\diamond} \arrow[dd,"a^{\diamond}"'] & & Y_1^{\diamond} \arrow[dd,"c^{\diamond}"]
\\
& Z_2^{\diamond} \arrow[ld,"s_2^{\diamond}"'] \arrow[rd,"t_2^{\diamond}"] &
\\
X_2^{\diamond} & & Y_2^{\diamond}.
\end{tikzcd}
\]
\end{proposition}

\subsection{Truncated moduli spaces of Witt vector shtukas and truncated Witt vector Hecke stacks} \label{qrerewrewrf}

We would like to compare the cohomological correspondences in \cite{2017arXiv170705700X} between moduli spaces of Witt vector shtukas  with the ones between v-stacks. More precisely, only the truncated moduli spaces of Witt vector shtukas are considered in \cite{2017arXiv170705700X} since the untruncated ones are not algebraic stacks and the cohomological formalism does not directly apply to them. 
However, the problem disappears by passing to v-stacks as the formalism developed in \cite{2017arXiv170907343S} is very general and does apply to stacks with infinite dimensional automorphisms. 

We recall the definition of truncated Witt vector Hecke stacks and moduli spaces of Witt vector shtukas in \cite{2017arXiv170705700X}. We omit the superscript $W$ for moduli spaces of Witt vector shtukas, as we will only use the truncated version for Witt vector ones. More importantly, we use the superscript $W$ for the untruncated moduli spaces of Witt vector shtukas or Hecke stacks viewed as v-stacks, whereas the truncated ones below are only defined as perfect stacks. 

Let $m \in \mathbb{N} \cup \{\infty \}$, and $L^mG$ be the perfect scheme defined by $L^mG(R)=G(W_m(R))$ (with $W_{\infty}(R)=W(R)$) for any perfect algebra $R$ over $\mathbb{F}_q$. Suppose now that  $m$ is large with respect to $\mu_{\bullet}$ so that the action of $ LG(R)$ on $Gr^W_{\mu_{\bullet}}$ factorizes through $L^mG $, where $Gr_{\mu_{\bullet}}^W$ is the (twisted products of) Witt vector affine Grassmannian, see \cite[Definition 5.1.2]{2017arXiv170705700X}.  Let \[
\on{Hecke}^{\on{loc}(m)}_{\mu_{\bullet}} := L^mG \setminus Gr^W_{\mu_{\bullet}}.
\]

We observe that 
\[
\on{Hecke}^{\on{loc}(\infty), \diamond}_{\mu_{\bullet}} =
\on{Hecke}^W_{\mu_{\bullet}}
\]
and
\[
\on{Hecke}^{\on{loc}(\infty), \diamond\diamond}_{\mu_{\bullet}} =
\on{Hecke}_{\mu_{\bullet}, s}^{\on{loc}, (\{1\}, \cdots, \{n\})}.
\]

Let $(m,n)$ be a pair of non-negative integers such that $m-n$ is $\mu_{\bullet}$-large, see \cite[Definition 3.1.6]{2017arXiv170705700X}, and 
$\on{Sht}^{\on{loc}(m,n)}_{\mu_{\bullet}}$ 
be defined as in \cite[Definition 5.3.1]{2017arXiv170705700X}, and similarly for
$\on{Sht}^{0,\on{loc}(m,n)}_{\mu_{\bullet}|\nu_{\bullet}}$. By $loc.cit.$ section 5.3.2, both 
$\on{Sht}^{\on{loc}(m,n)}_{\mu_{\bullet}}$
and 
$\on{Sht}^{0,\on{loc}(m,n)}_{\mu_{\bullet}|\nu_{\bullet}}$
can be written as quotients  by $L^mG$ of perfect schemes perfectly of finite type, so they fit into the cohomological formalism of the previous section by remark \ref{jfoieuoirue}. Now by \cite[Proposition \rom{6}.4.1]{2021arXiv210213459F}, we have canonical identifications 
\[
D_{\text{ét}}(\on{Sht}^W_{\mu_{\bullet}}, \Lambda) \cong D_{\text{ét}}(\on{Sht}^{\on{loc}(m,n),\diamond}_{\mu_{\bullet}},\Lambda),
\]
\[
D_{\text{ét}}(\on{Sht}^{0,W}_{\mu_{\bullet}|\nu_{\bullet}},\Lambda) \cong 
D_{\text{ét}}(\on{Sht}^{0,\on{loc}(m,n),\diamond}_{\mu_{\bullet}|\nu_{\bullet}}, \Lambda)
\]
and
\[
D_{\text{ét}}(\on{Hecke}^W_{\mu_{\bullet}},\Lambda) \cong 
D_{\text{ét}}(\on{Hecke}^{\on{loc}(m),\diamond}_{\mu_{\bullet}}, \Lambda).
\]
 More precisely, the conclusion of \cite[Proposition \rom{6}.4.1]{2021arXiv210213459F}   holds if we change the assumption to that the group $H$ has a filtration with graded pieces being extensions of closed unit  discs (instead of $\mathbb{A}^1$) v-locally, which the analytification $L^mG^{\diamond}$ does, and the proof in $cit.loc.$ goes without change.  
In other words, the truncated moduli spaces of Witt vector shtukas and Witt vector Hecke stacks have the same cohomology with the untruncated ones.  Similar identifications hold between moduli spaces of local Witt vector shtukas or Witt vector Hecke stacks and $\diamond\diamond$-analytification of truncated ones, for example
\[
D_{\text{ét}}(\on{Sht}^{\on{loc},W}_{\mu_{\bullet}}, \Lambda) \cong D_{\text{ét}}(\on{Sht}^{\on{loc}(m,n),\diamond\diamond}_{\mu_{\bullet}},\Lambda).
\]

\begin{remark}
The introduction of truncated moduli spaces of Witt vector shtukas or Witt vector Hecke stacks in \cite{2017arXiv170705700X} is because the untruncated ones do not directly fit into the cohomological formalism of Artin stacks. Using the new cohomology theory of v-stacks developed in \cite{2017arXiv170907343S}, we believe large parts of \cite{2017arXiv170705700X} can be simplified using the untruncated objects, with some caution on the arguments using fixed points results. 
\end{remark}

\bibliographystyle{alpha} 
\bibliography{ref}
\end{document}